\documentclass[10pt]{article}
\usepackage{amssymb,amsmath,amsfonts,amsthm,mathtools,bm,enumerate,xcolor,stmaryrd}
\usepackage{mathrsfs}
\usepackage[toc,page,title,titletoc,header]{appendix}
\usepackage{graphicx}
\usepackage{subcaption}
\usepackage{dsfont}
\usepackage{charter}
\usepackage{bbm}
\usepackage{multirow}
\usepackage{shuffle}
\usepackage[colorlinks,citecolor=blue,urlcolor=blue]{hyperref}

\usepackage[shortlabels]{enumitem}

\usepackage{multicol}
\usepackage{booktabs}
\usepackage{url}
\usepackage[outdir=./]{epstopdf} 

 \usepackage[nodayofweek]{datetime}

\usepackage{hyperref}
\hypersetup{
    colorlinks=true, 
    linktoc=all,     
    linkcolor=blue,  
}
\numberwithin{equation}{section}

\newtheorem{Theorem}{Theorem}[section]
\newtheorem{lemma}[Theorem]{Lemma}
\newtheorem{proposition}[Theorem]{Proposition}
\newtheorem{definition}[Theorem]{Definition}

\newtheorem{assumption}{Assumption H.\!\!}

\theoremstyle{definition}

\newtheorem{example}{Example}[section]

\theoremstyle{remark}
\newtheorem{remark}{Remark}[section]

\newcommand{\cO}{\mathcal{O}}

\newcommand{\bE}{\mathbb{E}}
\newcommand{\bF}{\mathbb{F}}

\newcommand{\bP}{\mathbb{P}}

\newcommand{\bR}{\mathbb{R}}

\newcommand{\bfx}{\mathbf{x}}

\newcommand{\bodx}{\boldsymbol{x}}
\newcommand{\bodX}{\boldsymbol{X}}
\newcommand{\olbodX}{\overline{ \boldsymbol{X}} }

\newcommand{\1}{\mathbbm{1}}

\newcommand{\dd}{\mathrm{d}}

\newcommand{\hx}{ \hat{X} }

\newcommand{\hco}{ \hat{\cO} }

\begin{document}

\title{
Improved weak convergence for the long time simulation of Mean-field Langevin equations
}

\author{
\normalsize Xingyuan Chen\textit{$^{a}$} \\
        \small  X.Chen-176@sms.ed.ac.uk   
\and
 \normalsize Gon\c calo dos Reis\textit{$^{a,b,}$}\footnote{G.d.R.~acknowledges support from the FCT – Fundação para a Ciência e a Tecnologia, I.P., under the scope of the projects UIDB/00297/2020 (https://doi.org/10.54499/UIDB/00297/2020) and UIDP/00297/2020 (https://doi.org/10.54499/UIDP/00297/2020) (Center for Mathematics and Applications, NOVA Math).} \\
        \small  G.dosReis@ed.ac.uk
\and
\normalsize Wolfgang Stockinger \textit{$^{c}$} \\
        \small   w.stockinger@imperial.ac.uk
\and
\normalsize Zac Wilde \textit{$^{a}$} \\
        \small   zwilde@ed.ac.uk
}

\date{%
    \footnotesize 
    $^{a}$~School of Mathematics, University of Edinburgh, The King's Buildings, Edinburgh, UK
    \\
    $^{b}$~Centro de Matem\'atica e Aplica\c c$\tilde{\text{o}}$es (Nova Math), FCT, UNL, Portugal
    \\
    $^{c}$~Imperial College London, UK
    \\
    \longdate \today \ (\currenttime)
    \vspace{-1.0cm}
}

\maketitle

\medskip 
\noindent\textbf{Abstract.} 
We study the weak convergence behaviour of the Leimkuhler--Matthews method, a non-Markovian Euler-type scheme with the same computational cost as the Euler scheme, for the approximation of the stationary distribution of a one-dimensional McKean--Vlasov Stochastic Differential Equation (MV-SDE). The  particular class under study is known as mean-field (overdamped) Langevin equations (MFL).
We provide weak and strong error results for the scheme in both finite and infinite time. 
We work under a strong convexity assumption.

Based on a careful analysis of the variation processes and the Kolmogorov backward equation for the particle system associated with the MV-SDE, we show that the method attains a higher-order approximation accuracy in the long-time limit (of weak order convergence rate $3/2$) than the standard Euler method (of weak order $1$). 
While we use an interacting particle system (IPS) to approximate the MV-SDE, we show the convergence rate is independent of the dimension of the IPS and this includes establishing uniform-in-time decay estimates for moments of the IPS, the Kolmogorov backward equation and their derivatives. The theoretical findings are supported by numerical tests. 
\medskip



\footnotesize
\tableofcontents
\normalsize

\section{Introduction}
In this paper, for $t\geq 0$, we consider a class of McKean--Vlasov Stochastic Differential Equations (MV-SDE), specifically the one-dimensional mean-field Langevin (MFL) equation: 
\begin{align}
\label{Model1_intro}
    X_t 
    = \xi 
      - \int_{0}^{t} \Big(\nabla U(X_s) + \nabla V * \mu_s(X_s) \Big) \mathrm{d}s + \sigma W_t, 
\end{align}
where $\mu_t$ is the law of $X_t$, $\sigma \in \mathbb{R}$  and $\xi \in L^{p}(\Omega,\mathbb{R})$ for some given $p \geq 2$ (i.e. the initial state is an $\mathcal{F}_0$-measurable random variable with finite $p$-th moments).  Following a statistical physics interpretation \cite{cattiaux2008probabilistic},  the map $U: \mathbb{R} \to \mathbb{R}$ is labelled the confining potential and $V:\mathbb{R} \to \mathbb{R}$  the interaction potential. $*$ denotes the usual convolution operator given by $(f* \mu)(\cdot)=\int_\bR f(\cdot-y)\mu(\mathrm{d}y)$ for some given integrable function $f:\mathbb{R}\to \mathbb{R}$. 
Of particular interest is the equation's stationary distribution $\mu^*$ 
and specifically how to efficiently generate samples from it. The latter problem has garnered additional interest due to its role in the study of training neural networks (via stochastic gradient descent algorithms) in the mean-field regime 
\cite{sirignano2022mean,deBortoli2020quantitative,claisse2023mean,chizat2018global,Mei2018,lukasz}.  
We consider the case where the functions $U, V$ satisfy suitable regularity and convexity assumptions (Assumption~\ref{assum:main}), thus the process described in \eqref{Model1_intro} admits a unique stationary distribution $\mu^*$ (e.g., \cite{cattiaux2008probabilistic,chenreisstockinger2023super3}) with a well-known implicit form satisfying 
\begin{align}
\label{eq.INTRO.ergodicDistro}
    \mu^*(x) \propto \exp\Big ( - \frac2{\sigma^2} U(x) - \frac2{\sigma^2} \int_{\mathbb{R}} V(x-y)\mu^*(\mathrm{d}y) \Big ). 
\end{align}
The mainstream method to sample from $\mu^*$ is to simulate the weakly-interacting $N$-particle SDE system (IPS) that approximates \eqref{Model1_intro} in the mean-field limit. That is, the law of the solution to \eqref{Model1_intro} is approximated, for some sufficiently large $N$, by the empirical distribution associated with the IPS $(\boldsymbol{X}^{N}_t)_{t \geq 0} \coloneqq (X_t^{1,N},\ldots,X_t^{N,N})_{t \geq 0}$ with components defined for $i \in \lbrace 1, \ldots, N \rbrace$ by
\begin{align}
\label{eq.INTRO.GenereicIPS}
X_t^{i,N} 
    &
    =
     \xi^i 
     - 
     \int_{0}^{t} \Big(
      \nabla U(X^{i,N}_s) + \frac{1}{N} \sum_{j=1}^{N} \nabla V(X^{i,N}_s - X^{j,N}_s) \Big)
     \mathrm{d}s + \sigma W^i_t,
\end{align}
where $(\xi^i,W^i)_{i \in \lbrace 1, \ldots, N \rbrace}$ are a collection of i.i.d.~copies of $(\xi,W)$. 
We briefly review some important results providing quantitative convergence guarantees for the approximation of \eqref{Model1_intro} through its IPS. 
Broadly speaking, the error of approximating \eqref{Model1_intro} by an $N$-particle system \eqref{eq.INTRO.GenereicIPS} has been widely studied in the literature of Propagation of Chaos (PoC), stemming from the seminal works \cite{Sznitman1991,Meleard1996}. In a nutshell, the accuracy of the $N$-particle approximation is known to behave like $\cO(1/N)$ in the \textit{squared $L^2$-norm or squared $W^{(2)}$-Wasserstein metric} (under Lipschitz type assumptions on the interaction and confining potentials). 
These quantitative convergence results are often referred to as strong PoC \cite{Sznitman1991,Meleard1996,bossytalay1997}. 
More recently, \cite[Theorem 2.2]{Lacker2023-Hierarchies} shows that the PoC convergence rate for models of the form \eqref{Model1_intro} (in relative entropy and finite time) can be improved from $\cO(1/N)$ to $\cO(1/N^2)$ (under certain smallness conditions), with \cite[Example 2.8]{Lacker2023-Hierarchies} showing that the rate is optimal. We refer the reader to the introductions of the papers \cite{Lacker2023-Hierarchies,lacker2023sharp} and review articles \cite{chaintron2022propagation,chaintron2021propagation} for a holistic discussion on this topic. 

When the PoC holds up to $T=\infty$, e.g., \cite{cattiaux2008probabilistic,lacker2023sharp}, it is called \textit{uniform in time PoC}  -- see our Proposition~\ref{prop:basic_estimates11} for a formulation in the strong sense. Uniform in time PoC results for the model \eqref{Model1_intro} have received significant attention in the last few years; see e.g. \cite{lacker2023sharp,chen2022uniform,zimmeruniform,cattiaux2008probabilistic} and references cited therein.

Another string of interests in terms of quantitative PoC results, is the weak convergence case.  In the context of quantitative weak PoC results (i.e. the absolute difference between expectations), we refer to the recent work \cite{Chassagneux2022WeakPoC-MasterEq} establishing finite-time higher order weak PoC results via techniques from differential calculus on the space of measures along with the study of Kolmogorov backward PDEs written on the Wasserstein space. 
Also, \cite[Theorem 3.1]{haji2021simple} establishes a finite-time weak PoC result of rate $\cO(1/N)$ based on a more classical approach, via a Talay--Tubaro expansion and an analysis of a Kolmogorov backward PDE associated with the whole IPS. Both \cite{Chassagneux2022WeakPoC-MasterEq,haji2021simple} establish their results for $T<\infty$. 
\color{black}
To the best of our knowledge, the only quantitative uniform-in-time weak PoC results that apply to our setting is the very recent \cite[Theorem 1.1]{bernou2024uniform}. From a methodological perspective \cite{bernou2024uniform} is close to \cite{Chassagneux2022WeakPoC-MasterEq} (for $T<\infty$) and somewhat close to \cite{delarue2021uniform} (we note that \cite{delarue2021uniform} obtains weak uniform-in-time PoC results but over the torus and draws on the Fokker-Planck equation). 
Loosely, the result states that the accuracy of the $N$-particle approximation to $\mu_t$ behaves like $\cO(1/N)$ uniformly in time over a metric akin to the $W^{(1)}$-Wasserstein metric.   
\color{black}

Once \eqref{eq.INTRO.GenereicIPS} and a PoC rate is established, the sampling from $\mu^*$ \eqref{eq.INTRO.ergodicDistro} is obtained via discretization of \eqref{eq.INTRO.GenereicIPS} with a convenient numerical scheme (also called numerical integrator or sampler). 
There is a growing body of contributions on the topic of sampling from the (overdamped) MFL stationary distribution using this method \cite{suzuki2023convergence,kook2024sampling}. Although said works provide a variety of quantitative PoC rate type results (under a variety of conditions), \textit{the time discretization schemes used are all of Euler type} -- see \cite[Theorem 3 and 4]{suzuki2023convergence} or \cite[Section 4]{kook2024sampling} -- the final error rates or sampling guarantees encase a leading order $1$ dependence on the time-discretization stepsize. We also mention two recent contributions on unadjusted Hamiltonian Monte Carlo for the simulation of the (underdamped or kinetic) MFL \cite{bou2023-MR4610714,bou2023nonlinear}. The main emphasis of the present work is to improve the weak convergence order (to the stationary distribution) of the standard Euler scheme using a non-Markovian version of it.

We briefly mention some (recent) contributions on numerical schemes for MV-SDEs. The seminal works \cite{bossytalay1996,bossytalay1997} investigate the convergence rates of the particle system's and Euler scheme's approximation accuracy of the cumulative distribution (in $L^1$-norm) for the Burger's type MV-SDE using density estimates or using a Malliavin calculus approach \cite{antonelli2002rate}.  
In the context of finite time-horizon simulation there are many recent contributions (e.g., Euler and Milstein schemes) focusing on the approximation error stemming from the time discretization of the IPS 
\cite{chen2022SuperMeasureIMA,chenreisstockinger2023super3,reis2018simulation,bao2021first,biswas2022explicit}. Cubature type algorithms, a class of weak approximation algorithms, for (Stratonovich) MV-SDEs have been proposed in \cite{deRaynal2015cubature,crisanmcmurray2019Cubatureformv-sdes,NAITOyAMADA2022HighWeakRatesmv-sdes}. Lastly, we mention \cite{agarwal2023numerical} and its references for numerical methods to approximate MV-SDEs directly fully avoiding the IPS approach. 

\paragraph*{Motivation, weak order schemes and the non-Markovian Euler scheme for SDEs.} 
The classical Euler scheme is an easy to implement and ubiquitous method for the numerical approximation of solutions to SDEs. In the classic overdamped Langevin context, i.e. if one sets $\nabla V=0$ in \eqref{eq.INTRO.GenereicIPS}, the Euler scheme attains a strong and weak rate $\cO(h)$ (where $h$ is the time-discretization stepsize) in either finite or infinite time horizon \cite{MilsteinTretyakov2021Book,platen2010numerical}. Informally, under certain conditions the Euler scheme's weak error at time $T=Mh$ and stepsize $h>0$ ($M\in \mathbb{N}$) can be expressed (see, Talay--Tubaro \cite{talaytubaru1990expansion,MilsteinTretyakov2021Book}) in the form 
\begin{align}
\label{eq:INtro-weakerrorExpension}
    \textrm{Weak Error}^{\textrm{Euler}}(h;T) = C_T h + \cO(h^2)
    \qquad
\textrm{where}    \quad 
\lim_{T\to \infty} C_T = \textrm{Const}>0.
\end{align}
Setting $\nabla V$=0 in \eqref{Model1_intro} and denoting by $X^{h}_{t_{m}}$ the numerical approximation to $X_{t_{m}}$, the following variant of the Euler scheme was initially proposed in 2013 by \textit{Leimkuhler and Matthews} in \cite{MR3040887} and fully analyzed in the following year in \cite{leimkuhler2014long}:
\begin{align}
\label{eq:intro: def : non-Markov Euler scheme}
    X^{h}_{t_{m+1}} 
    = 
    X^{h}_{t_{m}} 
    -\nabla U(X^{h}_{t_{m}}) h  
    + \frac\sigma2 (\Delta W_{m} + \Delta W_{m+1})
    \ \ \textrm{with}\ \
    \Delta W_{m+1} = W_{t_{m+1}} - W_{t_{m}}, 
\end{align}
\color{black} 
for any $m\in \mathbb{N}$, $t_m=mh$. 
This scheme is called the \textit{Leimkuhler–Matthews method} or the \textit{non-Markovian Euler scheme} \cite{MR3040887,leimkuhler2014long} since $X^{h}_{t_{m+1}}$ is computed using the current and past Brownian increments, $\Delta W_{m+1}$ and $\Delta W_{m}$, respectively.  
\color{black} 
It is shown in \cite{leimkuhler2014long} that \eqref{eq:INtro-weakerrorExpension} holds for $T<\infty$ (with a different $C_T$), but as $T \to \infty$ one has 
\begin{align*}
\lim_{T\to \infty} C_T = 0 
\quad \Rightarrow \quad 
\lim_{T\to \infty} \textrm{Weak Error}^{\textrm{non-Mark.~Euler}}(h;T) = \cO(h^2), 
\end{align*} 
and thus the \textit{non-Markovian Euler scheme} is a weak order-2 method as $T \to \infty$. 
An intuition behind the result is offered by \cite{vilmart2015postprocessed} through the concept of \textit{Postprocessed integrators}. There, \eqref{eq:intro: def : non-Markov Euler scheme} is re-written as a two-step method where the second step corrects the $\cO(h)$ bias of the first step in such a way that in the long time limit the weak error is $\cO(h^2)$; see \cite[Equation~(2.4)]{vilmart2015postprocessed}. 
\medskip

\textit{The focus of this work} is to study the non-Markovian Euler scheme \eqref{eq:intro: def : non-Markov Euler scheme} in the context of the overdamped MFL dynamics \eqref{Model1_intro} as way to simulate \eqref{eq.INTRO.ergodicDistro} via a higher-order weak scheme. 
As described, the MFL \eqref{Model1_intro} is first approximated by the IPS \eqref{eq.INTRO.GenereicIPS} and then the IPS is time-discretized using the non-Markovian  scheme \eqref{eq:intro: def : non-Markov Euler scheme}.

In terms of proof methodologies for weak errors, for either SDEs or MV-SDEs, it is well known since the seminal work \cite{talaytubaru1990expansion}, that weak error analysis can be tackled via the Kolmogorov backward PDE \cite[Chap.~2]{MilsteinTretyakov2021Book}. This approach for MV-SDEs and the IPS is well-reviewed in \cite{haji2021simple} and is the approach we take. An alternative method is the use of  Malliavin Calculus \cite{BallyTalay1995EulerErrorviaMallCalc,kohatsuOgawa1997WeakRAteforSDEs,NAITOyAMADA2022HighWeakRatesmv-sdes} as it  offers a path to completely bypass the analysis of the Kolmogorov backward equation. In addition, we highlight the classic backward error analysis approach drawing on It\^o- or Stratonovich--Taylor expansions \cite{KloedenPlaten1992SDENumericsbook,MilsteinTretyakov2021Book,foster2024high}. In a different spirit, results showing density approximations, via Fokker--Plank PDE analysis or Malliavin calculus have been obtained in  \cite{bossytalay1997,bossytalay1996,BallyTalay1996.I.EulerErrorRate-nonsmooth,BallyTalay1996.II.EulerErrorRate-nonsmooth}.

\paragraph*{The scheme's convergence results for the MFL class.} 
The main contribution of this paper is to establish the techniques needed to understand and quantify the weak errors for the non-Markovian Euler scheme applied to \eqref{eq.INTRO.GenereicIPS}\textit{ in a way such that the convergence rate is independent of the number of particles $N$ in the IPS}. 
In terms of the convergence results, for any $T>0$, the weak approximation error for smooth test functions $g:\bR^N\to \bR$ (satisfying Assumption~\ref{assum:main_weak error}) is defined as 
\begin{align}
 \label{eq:inro:Weak Error}
    \textrm{Weak Error}:=
    \bE\big[g\big(\boldsymbol{X}_{T}^{N}\big)\big]
    -
    \bE\big[g\big(\boldsymbol{X}^{N,h}_{T}\big)\big], 
    \qquad \textrm{ with $T=Mh$ for $h>0,M\in\mathbb{N}$}, 
\end{align}
where $\boldsymbol{X}_{T}^{N}$ denotes the solution of \eqref{eq.INTRO.GenereicIPS} and $\boldsymbol{X}^{N,h}_{T}$ denotes the $\bR^N$-valued output of $M$-steps ($T=Mh$) of the non-Markovian Euler scheme applied to \eqref{eq.INTRO.GenereicIPS} (and explicitly given in \eqref{eq: def : non-Markov Euler scheme}). 
Informally, our main result (Theorem \ref{theorem: main weak convergence}) states that 
\color{black}  
\begin{align}
\label{eq:INTRO:convResults}
          \big|\, \mathbb{E}[g(\boldsymbol{X}^{N}_{T})] 
          - \mathbb{E}[g(\boldsymbol{X}^{N,h}_{T})]\, \big| 
          & 
          \leq   \exp(-\lambda_0 T)K h + K h^{3/2}, 
\end{align}
\color{black}
for some positive constants $\lambda_0, K$ independent of $h, T,M$ and $N$. In other words, the scheme is uniformly (in the number of particles) of weak order $\cO(h^{3/2})$ as $T \to \infty$ and has standard weak order $\cO(h)$ for $T< \infty$. 
We provide an in-depth technical discussion (Remarks \ref{rem:H2-diff} and \ref{rem:Losing1/2 conv rate}) on the missing $h^{1/2}$ order in the convergence rate when comparing this to the second order weak convergence result obtained in   \cite{leimkuhler2014long,vilmart2015postprocessed}.

A secondary contribution of this work (Proposition~\ref{prop:basic_estimates22} below), is the clarification of the nuance that the higher-order weak convergence of the non-Markovian Euler scheme comes at the cost of having a \textit{uniform in time} strong $L^2$-convergence order of $\cO(h^{1/2})$. It is lower than the $\cO(h)$ strong $L^2$-convergence of the classical Euler scheme.

\paragraph*{Methodology, contributions and existing literature.}
The main methodology we follow is an involved variant of the Talay--Tubaro approach to the study of weak convergence \cite{talaytubaru1990expansion,MilsteinTretyakov2021Book}, which is also the approach used by \cite{leimkuhler2014long} (for SDEs) and \cite{haji2021simple,Chassagneux2022WeakPoC-MasterEq} (to study weak quantitative PoC over $T<\infty$). At its core, this method relates the expectations appearing in the definition of the weak error \eqref{eq:inro:Weak Error} to a Kolmogorov backward PDE with terminal condition given by the function $g$ -- see the PDE \eqref{PDE:Kolmogorov} linked to the driving SDE \eqref{eq.INTRO.GenereicIPS}, which in flow form is given in \eqref{eq: def of Bi FLOW SDE in RN} and written as  $(\boldsymbol{X}^{t,\boldsymbol{x},N}_s)_{s\geq t \geq 0}$, for $\boldsymbol{x} \in \mathbb{R}^{N}$ denoting the starting point of the IPS at time $t \geq 0$.
This analysis involves establishing certain bounds for the variation processes of the IPS \eqref{eq.INTRO.GenereicIPS}, or more precisely for the flow process $(\boldsymbol{X}^{t,\boldsymbol{x},N}_s)_{s\geq t \geq 0}$. 
In this regard, our approach is closest in spirit to that of \cite{haji2021simple} as we work with Kolmogorov backward PDEs connected to the full particle system. However, our focus is on the \textit{time-discretization} analysis \emph{uniformly in $N$ over infinite time} as opposed to the weak error analysis in the number of particles \cite{haji2021simple,delarue2021uniform,Chassagneux2022WeakPoC-MasterEq} (these works consider $T<\infty$ and deal only with the continuous time IPS equation). Our case has therefore fundamental added complexities in relation to the mentioned works, as we require estimates which are not only uniform in $N$ but also in time.
\smallskip

\emph{Technical challenges.} 
As mentioned, \eqref{eq:INTRO:convResults} is proved via a Talay--Tubaro type expansion which, for the case of the non-Markovian Euler scheme, is an involved collection of terms arising from  Taylor expansions using the Kolmogorov backward equation associated with the flow equation for the IPS \eqref{eq.INTRO.GenereicIPS}. This expansion has been given in \cite[Equation~(3.17)]{leimkuhler2014long} for SDEs and we recast it to our setting (in Lemma~\ref{lemma:Weak Expansion Leimkuhler} and also in Section \ref{section: weak error expansion} and in Appendix \ref{appendix_aux_RemainderSection6.2}). In the following, we highlight several technical elements of our work and point out crucial differences to \cite{leimkuhler2014long,haji2021simple}:
\begin{itemize}
    \item[(i)] 
    The terms arising from said Taylor expansions involve up to $6$-th order (cross)-derivatives in the spatial variable of the solution to the Kolmogorov backward PDE (see our Assumption~\ref{assum:main_weak error} and Lemma~\ref{lemma: higher-order derivative of u without exp}).     
    Critically, the usual pointwise estimates from PDE theory e.g. \cite[Equation~(3.3)]{leimkuhler2014long} (or \cite{talay1990second,vilmart2015postprocessed}), do not directly apply to our case as those would not be independent of the number of particles. It is not clear how the right-hand side in \cite[Equality (3.3)]{leimkuhler2014long} depends on the problem's dimension. 
    Therefore, we derive suitable new estimates in $L^p$-norm of the solution to the  Kolmogorov backward equation that decay exponentially in time in a non-explosive way in $N$ (see, Lemma~\ref{lemma: higher-order derivative of u without exp} for an intermediate pointwise result and Lemma~\ref{lemma: higher-order derivative of u} for the final $L^{p}$-estimates used to show the main theorem) -- this is in stark contrast to \cite{haji2021simple} (in particular their Appendix B) which establishes pointwise estimates.    
    For clarity, the derivatives of the solution to the Kolmogorov backward equation are intrinsically linked to certain moment estimates for variation processes of the IPS' flow SDE $(\boldsymbol{X}^{t,\boldsymbol{x},N}_s)_{s\geq t}$ (see Lemma~\ref{lemma: higher-order derivative of u without exp}). In order to control the time dependence of the implied constants for the moment estimates of the variation processes, a careful analysis of the terms involving the convolution kernel is needed. Consequently, we are only able to establish the bounds in Lemma~\ref{lemma: higher-order derivative of u} in an $L^{p}$-sense. The estimates of Lemma~\ref{lemma: higher-order derivative of u} are obtained in \cite{haji2021simple} in a pointwise sense but crucially without the exponential time decay component (see RHS of \eqref{eq:  high-deri on u order 2 bounded by K} and \eqref{eq:  high-deri on u order 4 bounded by K}); their analysis is carried out in finite time for which this issue is not a concern.
   
    Further, our analysis requires to study the time regularity of the solution to the Kolmogorov backward PDE, which needs estimates for the differences of the IPS' flow SDE process, $(\boldsymbol{X}^{t,\boldsymbol{x},N}_s)_{s\geq t}$, concretely differences of the form $|\boldsymbol{X}^{t,\boldsymbol{x},N}_u-\boldsymbol{X}^{s,\boldsymbol{x},N}_u|$ for $0 \leq t\leq s\leq u$. 
    
    Lastly, it is noteworthy to highlight that the weak error test function $g$ in \eqref{eq:inro:Weak Error} depends on the whole IPS (as in \cite{haji2021simple} but not as in \cite{bencheikh2019bias}\footnote{The analysis of the Kolmogorov backward PDE over a single-particle $X^{i,N}$ instead of $\boldsymbol{X}^{N}=(X^{1,N},\ldots,X^{N,N})$ on the test function $g$ enables an advantageous simplifying decoupling effect at a later point; such is not the case here.}) making the analysis much more involved.

    \item[(ii)] Before addressing the estimates for the  Kolmogorov backward equation, we derive $L^p$-norm estimates for the \textit{variation processes} of the flow of the IPS (decaying over time uniformly over $N$) up to general $n$-order (although only $6$-orders are needed). This is done in Section \ref{section: analysis of var process total}.     
    Our approach shows a way to analyze the terms arising from the interacting kernel and their recurring contributions across the different orders of the variation processes and \textit{across different particle indices} -- compare \eqref{eq: first var result addinitional} and \eqref{eq: first var result 2}  
    for the first order case and check Lemma~\ref{lemma: n-var process result} for general cases. A further crucial component of the analysis is to establish the correct decay in terms of number of particles across different orders of variation processes. This, in particular, subsequently allows to control the growth of the derivatives of the solution to the Kolmogorov backward equation.  
   The depth of the analysis is well beyond the results of \cite{haji2021simple} who carry out a related approach over finite time (or in \cite{delarue2021uniform} over the torus over infinite time horizon).  
    \item[(iii)] Regarding the strong convergence analysis of the non-Markovian Euler scheme, the one-timestep error propagation analysis requires analysing 3 sub-steps of the scheme which is in contrast to such analysis for the standard Euler scheme (loosely speaking, only 1 sub-step is analyzed) and thus, the analysis is lengthier than usual -- see e.g. the proof of Proposition~\ref{prop:basic_estimates22}.
\end{itemize}

\paragraph*{Gaps, conjectures and pathways for further study.} 
Our analysis addresses MFL dynamics in $\bR$ through $\bR^N$-valued IPS. It is believed that our results could be established in the multi-dimensional case $d>1$ if the measure dependence in \eqref{Model1_intro} was of the form $\bE[h(X_t)]$ instead of an interaction kernel. 
The tools and techniques we have employed to show our main result do not use Lions measure derivatives (due to the simplicity of the underlying model) or concentration inequalities. It seems possible, although presently unclear, that drawing on Log-Sobolev inequalities and related machinery would provide means to lift the technical constraint in dimension arising from the convolution term.
Overall, to establish higher $L^p$-moments for the variation processes in Section \ref{section: analysis of var process total}, we require the \textit{symmetrization trick} (in Remark \ref{Remark:symmetrization trick} to deal with \eqref{eq: symmetric trick}) and 
an inequality of the type $(|x|^{p-2}x -|y|^{p-2}y ) \cdot \left( x-y \right) \geq 0$ to hold -- this would not naturally hold in the $d>1$ case (this issued is hinted at in \cite{chen2022SuperMeasureIMA} and appears explicitly in \cite{chenreisstockinger2023super3}). 
A possible alternative methodology to establish our main result is the postprocessed integrators machinery presented in \cite{vilmart2015postprocessed}. To use it, we benefit from all the results shown in this work. However, some others would still need to be established, e.g. one has to derive results that imply Assumption~2.3 or Theorem 4.1 of \cite{vilmart2015postprocessed} in the IPS \eqref{eq.INTRO.GenereicIPS} setting that remain uniform in the particle number; this is yet to be explored and left for future research. 
In addition, it would be interesting to see if techniques from Malliavin calculus could be used \cite{BallyTalay1995EulerErrorviaMallCalc,kohatsuOgawa1997WeakRAteforSDEs,NAITOyAMADA2022HighWeakRatesmv-sdes}, even in the standard SDE context, to establish the weak convergence results shown in this article. It is interesting to question if the weak results in our manuscript could be extended to the difference between the densities of \eqref{Model1_intro} and \eqref{eq.INTRO.GenereicIPS} as in \cite[Corollary 2.1]{BallyTalay1996.II.EulerErrorRate-nonsmooth} (as the diffusion of \eqref{eq.INTRO.GenereicIPS} is uniformly elliptic) -- in fact, the question is also pertinent in the context of standard SDEs itself (do the results of \cite{leimkuhler2014long,vilmart2015postprocessed} also hold for densities as in \cite{BallyTalay1996.II.EulerErrorRate-nonsmooth}).

Further afield and more broadly is if these results could be established under the setting of common-noise MFL dynamics \cite{maillet2023note}, or in the context of the kinetic/underdamped MFL \cite{MR4333408,chen2023uniformKINETIC-EJP,bernou2024uniform}. 
Our work also paves the way to study stochastic gradient descent convergence \cite{suzuki2023convergence,kook2024sampling} but using the non-Markovian Euler scheme as the update instead of the standard Euler one.

\paragraph{Paper organization.} This paper is organized as follows.
In Section \ref{section: framework and scheme}, we state  the main assumptions, introduce the non-Markovian scheme and state basic results regarding wellposedness of the underlying model. In Section \ref{section: Weak Error expansion}, we present our  weak error expansions based on the result in \cite{leimkuhler2014long}. We state the main technical difficulties when applying the scheme to the IPS and explain why we cannot reach weak rate of order $2$ in the case for classical SDEs. All proofs of Section \ref{section: framework and scheme} and \ref{section: Weak Error expansion} are postponed to the final part of the paper. Section \ref{section: analysis of var process total} contains the results relating to the analysis of several variation processes, while Section \ref{section: analysis of the dus} contains the decay estimates for the solution to the Kolmogorov backward equation. Section \ref{section: weak error expansion} contains the proof of the weak error result (of Section \ref{section: Weak Error expansion}). 
 An illustrative numerical example is provided in Section \ref{sec:Numericalexample}.

\paragraph*{Acknowledgments.} The authors are grateful to (in no particular order): B.~Leimkuhler (U.~of Edinburgh), A.~Teckentrup (U.~of Edinburgh), D.~Higham (U.~of Edinburgh), M.~Tretyakov (U.~of Nottingham), F.~Delarue (U.~C\^ote d'Azur) and A.-L. Haji-Ali (Heriot-Watt U.) for the helpful discussions.

\section{Framework and numerical scheme}
\label{section: framework and scheme}
\subsection{Notation and spaces}
For a vector in $\mathbb{R}^d$, $d \geq 1$, we will write $\boldsymbol{x} \coloneqq (x_1,\ldots, x_d) \in \mathbb{R}^d$. The inner product of two vectors $\boldsymbol{a}, \boldsymbol{b} \in \mathbb{R}^d$ is denoted by $\left \langle \boldsymbol{a}, \boldsymbol{b} \right \rangle$ and for standard Euclidean norm we will use the notation $|\cdot|$. Throughout this article $\mathcal{O}(\cdot)$ refers to the standard Landau (big `O') 
 notation. 
For a twice continuously differentiable function $f:\mathbb{R}^{d} \to \mathbb{R}$, we denote by $\partial_{x_i}f: \mathbb{R}^{d} \to \mathbb{R}$ the partial derivative with respect to the $i$-th component, by $\nabla f: \mathbb{R}^{d} \to \mathbb{R}^{d}$ its gradient $\nabla f=(\partial_{x_1}f,\ldots,\partial_{x_d}f)$, 
and by $\nabla^{2}f:\mathbb{R}^{d} \to \mathbb{R}^{d \times d}$ its Hessian. 
For a multi-index $\ell =(\ell_1, \ldots, \ell_d) \in \mathbb{N}^{d}$, we denote higher-order derivatives as \begin{align*}
    \partial^{d}_{x_{\ell_1},\ldots, x_{\ell_d}} f.
\end{align*} 
The sup-norm of $f$ will be denoted by
$|f|_\infty \coloneqq \sup_{\boldsymbol{x} \in \mathbb{R}^d} |f(\boldsymbol{x})|$.

 Let our probability space be a completion of $(\Omega, \bF, \mathcal{F},\bP)$ with $\bF=\lbrace \mathcal{F}_t \rbrace_{t\geq 0}$ being the natural filtration of the one-dimensional Brownian motion $W =(W_t)_{t \geq 0}$, augmented with a sufficiently rich sub $\sigma$-algebra $\mathcal{F}_0$ independent of $W$. We denote by $\bE[\cdot]=\bE^\bP[\cdot]$, the expectation with respect to $\bP$. 

For any $ p\geq 2$, we define $L^p(\Omega, \bR^d)$ as the space of $\bR^d$-valued measurable random variables such that $\bE[\,|X|^p  ]^\frac1p <\infty$. Let $\mathcal{P}_p(\bR^d)$ denote the space of probability measures $\mu$ on $\bR^d$ such that $\int_{\bR^d} |x|^p \mu(\mathrm{d}x) <\infty$. Let
\begin{align*} 
W^{(p)}(\mu,\nu) := \inf_{\pi\in\Pi(\mu,\nu)} \Big(\int_{\bR^d\times \bR^d} |x-y|^p\pi(\dd x,\dd y)\Big)^\frac1p, \quad \mu,\nu\in \mathcal{P}_p(\bR^d),
\end{align*}  
be the Wasserstein distance, where $\Pi(\mu,\nu)$ is the set of couplings for $\mu$ and $\nu$ such that $\pi\in\Pi(\mu,\nu)$ is a probability measure on $\bR^d\times \bR^d$ with $\pi(\cdot\times \bR^d)=\mu$ and $\pi(\bR^d \times \cdot)=\nu$. 

\subsection{Theoretical framework and preliminary results}

We consider the following one-dimensional MV-SDE, for $t \geq 0$, 
\begin{align}
\label{Model1xx}
    X_t 
    = \xi 
      - \int_{0}^{t} \Big(\nabla U(X_s) + \nabla V * \mu_s(X_s) \Big) \mathrm{d}s + \sigma W_t, 
\end{align}
where $\sigma \in  \mathbb{R}$, 
and $\xi \in L^{p}(\Omega,\mathbb{R})$ for some given $p \geq 2$. $U: \mathbb{R} \to \mathbb{R}$ is the confining potential and  $V:\mathbb{R} \to \mathbb{R}$ is the interaction potential, with $*$ denoting the usual convolution operator where $(f* \mu)(\cdot)=\int_\bR f(\cdot-y)\mu(\mathrm{d}y)$. We impose the following standard assumptions on $U$ and $V$.

\begin{assumption}
\label{assum:main}
Let 
$U: \mathbb{R} \to \mathbb{R}$ and $V:\mathbb{R} \to \mathbb{R}$ be twice continuously differentiable functions with globally Lipschitz continuous gradients. Further suppose that
\begin{enumerate}[(1)]
\item $U$ is uniformly convex in the sense that there exists $\lambda >0$ such that for all $x,y \in \mathbb{R}$, 
\begin{align}
\label{eq: def of lambda}
     \big( \nabla U(x)- \nabla U(y)\big) \big(x-y \big) \geq \lambda | x-y |^2,
\end{align}
which implies $ \nabla^2 U \geq \lambda $. 
\item $V$ is even (thus $\nabla V$ is odd), and convex, i.e., for all $x,y \in \mathbb{R}$, 
\begin{equation*}
     \big(\nabla V(x)- \nabla V(y)\big) (x-y ) \geq 0,
\end{equation*}
and there exists $K_V>0$ such that  $ |\nabla^2 V|_\infty \leq K_V$.
\end{enumerate}
\end{assumption}

\paragraph*{The Interacting particle system (IPS).} Define the $\mathbb{R}^N$-valued map $B$ as
\begin{equation*}
\mathbb{R}^N \ni \boldsymbol{x} = (x_1, \ldots, x_N)  
\mapsto
    B(\boldsymbol{x}) := \big(B_1(x_1, \ldots,x_N), \ldots, B_N(x_1, \ldots,x_N)\big),
\end{equation*}
where 
\begin{align}
\label{eq:def of func B}
    B_i(\boldsymbol{x})=B_i(x_1, \ldots,x_N) \coloneqq - \nabla U(x_i) - \frac{1}{N} \sum_{j=1}^{N} \nabla V(x_i - x_j).
\end{align}
Let $(\xi^{i},W^{i})$ for $i \in \lbrace 1, \ldots, N \rbrace$ be i.i.d.\ copies of $(\xi,W)$ and define the IPS associated with \eqref{Model1xx} to be
\begin{align}
\label{ModelIPS}
        X_t^{i,N} 
    &
    =
     \xi^i + \int_{0}^{t} B_i(X^{1,N}_s, \ldots, X^{N,N}_s) \mathrm{d}s + \sigma W^i_t, 
     \\
     \label{ModelIPS as in RN}   
     \boldsymbol{X}_t^{N} &
    =
     \boldsymbol{\xi} + \int_{0}^{t} B(\boldsymbol{X}_s^{N} ) \mathrm{d}s + \sigma \boldsymbol{W}_t,
\end{align}
with solution process $(\boldsymbol{X}^{N}_t)_{t \geq 0} \coloneqq (X_t^{1,N},\ldots,X_t^{N,N})_{t \geq 0}$, where we introduced $\boldsymbol{\xi}=(\xi^1,\ldots,\xi^N)$ and $(\boldsymbol{W}_t)_{t \geq 0}:=(W^1_t,\ldots,W^N_t)_{t \geq 0}$.

\subsubsection*{Preliminary results.}
The next proposition collects some basic properties of the MFL equation \eqref{Model1xx} and the IPS \eqref{ModelIPS}. 
\begin{proposition}
\label{prop:basic_estimates11}
Let Assumption~\ref{assum:main} hold and let $\xi \in L^{p}(\Omega,\mathbb{R})$ for some $p \geq 2$. 
Then the following hold:
\begin{enumerate}[(1)]
\item 
The MV-SDE \eqref{Model1xx} and the IPS \eqref{ModelIPS} each admit a unique strong solution. 
There exist constants $\kappa \in (0,\lambda)$ and $K \geq 0$ {\color{black}($K,\kappa$ are independent of $t$ and $N$)} such that for any $t \geq 0$
\begin{equation*}
 \mathbb{E}\big[\,|X_t|^{p}\big] \leq K\big(1 
 + \mathbb{E}[\,|\xi|^{p}] e^{-p\kappa t}\big)
 \quad\textrm{and}\quad
 \max_{i \in \lbrace 1, \ldots, N \rbrace } \mathbb{E}\big[\,|X_t^{i,N}|^{p}\big] \leq K\big(1 + \mathbb{E}[\,|\xi|^{p}] e^{-p\kappa t}\big).
\end{equation*}
\color{black}
\item Uniform in time strong propagation of chaos (PoC) holds, i.e., there exist constants $\kappa \in (0,\lambda)$ and $K \geq 0$ ($K,\kappa$ are independent of $t$ and $N$) such that for every $N \geq 1$
\begin{equation*}
    \max_{i \in \lbrace 1, \ldots, N \rbrace }  \sup_{t \geq 0} \mathbb{E}\big[\,|X^{i}_t-X_t^{i,N}|^2\big] \leq \frac{K}{N}\Big(1 
 + \mathbb{E}[\,|\xi|^{2}] e^{-2\kappa t}\Big),
\end{equation*}
\color{black}
where $X^{i}$ is the solution of \eqref{Model1xx} with $(\xi,W)$ replaced by $(\xi^{i},W^{i})$ (i.e., the so-called non-interacting particle system).

\item There exists a unique stationary distribution for  \eqref{Model1xx} and \eqref{ModelIPS}, denoted by $\mu^{*}$ and $\mu^{N,*}$, respectively. Moreover, $W^{(2)}(\mu_t,\mu^*)\to 0$ and $W^{(2)}(\mu_t^N,\mu^{N,*})\to 0$ as $t \to  \infty$.

 \end{enumerate}
\end{proposition}
\begin{proof}
See Appendix \ref{appendix_proof_basic11}.
\end{proof}

\subsection{The non-Markovian Euler scheme}

Let $h\in(0,1)$ denote the timestep and take $m \in \lbrace 0, \ldots, M-1 \rbrace$ for a given $M \in \mathbb{N}$. Inspired by \cite[Equation~(1.7)]{leimkuhler2014long} (also \cite{vilmart2015postprocessed}), we introduce the following non-Markovian Euler scheme 
\begin{align}
\label{eq: def : non-Markov Euler scheme}
    X^{i,N,h}_{t_{m+1}} = X^{i,N,h}_{t_{m}} - \Big(  \nabla U(X^{i,N,h}_{t_{m}}) + \frac{1}{N}\sum_{j=1}^{N} \nabla V(X_{t_{m}}^{i,N,h} - X_{t_m}^{j,N,h}) \Big)  h   + \frac\sigma2
    (\Delta W_{m}^i + \Delta W_{m+1}^i),
\end{align}
with $X^{i,N,h}_{t_0} = X^{i,N}_{t_0}$, to approximate the IPS \eqref{ModelIPS}, where we set the time grid points as $t_m := mh$ up to some time $T:=t_M=Mh$, and the random increments as $\Delta W_{m+1}^i = W^i_{t_{m+1}} - W^i_{t_{m}}$ with $\Delta W^{i}_{0} = 0$. 
In analogy to the IPS, we write for the solution process 
$( \boldsymbol{X}^{N,h}_{t_m})_{m \in \lbrace 0, \ldots,M \rbrace} \coloneqq (X^{1,N,h}_{t_m},\ldots, X^{N,N,h}_{t_m})_{m \in \lbrace 0, \ldots,M \rbrace}$. We aim to analyze the behaviour of this scheme as $T\to \infty$.

The following result establishes some fundamental properties for the non-Markovian Euler scheme \eqref{eq: def : non-Markov Euler scheme}: moment estimates and $L^2$-strong convergence (we were unable to find a proof in the literature regarding the $L^2$-strong convergence for this scheme (even for SDEs) and we thus provide it here). Critically, the moment estimates obtained are independent of the time horizon (i.e., the constant $K$ appearing below is independent of $T$). Lastly, as in \cite{leimkuhler2014long} or \cite{vilmart2015postprocessed}, the result holds for a sufficiently small timestep $h$. 
\begin{proposition}
\label{prop:basic_estimates22}
Let Assumption~\ref{assum:main} hold, let $\xi \in L^{p}(\Omega,\mathbb{R})$ for some $p \geq 2$, let $N,M\in \mathbb{N}$, $h>0$ and set $T = t_M = Mh$.   
Then the following statements hold for the process defined in \eqref{eq: def : non-Markov Euler scheme}.  
\begin{enumerate}[(1)]
\item 
There exist $\kappa,K >0$ (both are independent of $h,T,M$ and $N$) such that for any sufficiently small timestep $h$ (chosen independently of  $T,M,N$) and satisfying $0 < h < \min \{1/2\lambda,1\}$ it holds for all  $m \in \lbrace 0, \ldots, M \rbrace$ that
\begin{equation*}
    \max_{i \in \lbrace 1, \ldots, N \rbrace } \mathbb{E}\big[\,|X_{t_m}^{i,N,h}|^{p}\big] 
    \leq K \big(1 + \mathbb{E}\big[\,|\xi|^{p}\big] e^{-\kappa  t_m}\big).
\end{equation*}
\color{black}
\item \textit{$L^2$-strong error}.  
There exists $K >0$ (independent of $h,T,M$ and $N$) such that for any sufficiently small timestep $h$ (chosen independently of  $T,M,N$)  and satisfying $0 < h < \min \{1/2\lambda,1\}$, it holds that 
\begin{align*} 
        \max_{i \in \lbrace 1, \ldots, N \rbrace}
        \max_{m \in \lbrace 0, \ldots, M \rbrace}
        \bE    \big[\,|  X_{t_m}^{i,N} -   X_{t_m}^{i,N,h}  |^2  \big]  \leq Kh\big(1 + \mathbb{E}\big[\,|\xi|^{2}\big]  \big),
    \end{align*}
\color{black}
where $X^{i,N}$ and $X^{i,N,h}$ are the processes defined in \eqref{ModelIPS} and \eqref{eq: def : non-Markov Euler scheme} respectively.    
\end{enumerate}
\end{proposition}
\begin{proof}
See Appendix  \ref{appendix_proof_basic22}.
\end{proof}

\section{The weak error expansion}
\label{section: Weak Error expansion}
Before presenting the framework for the weak error analysis and our main result, we require the following definition which will be helpful to characterize and analyze the higher-order variation processes. 
\begin{definition}
    \label{def: set of sequence pi} Let $n, m,N \in \mathbb{N}$ with $N\gg n, m$ be  given integers.
    Define the set of multi-indices
\begin{align*}
    \Pi_n^N:=\big\{ \gamma=(\gamma_1,\dots,\gamma_n):   \gamma_i\in \lbrace 1, \ldots, N \rbrace ~\textit{for all}~ i\in \lbrace 1, \ldots,n \rbrace\big\}, 
\end{align*}
    with $ \Pi_0^N \coloneqq \emptyset$ denoting the empty set. For a subset $\overline{\gamma} \subseteq \gamma$, let $|\overline{\gamma}|$ be its length. 
    
    For a given $\gamma \in \Pi_n^{N}$, let
    $\hat{\gamma}$ be a set of length $N$ counting the frequency of each $j \in \lbrace 1, \ldots, N \rbrace$ in $\gamma$, and define
    $\hco(\gamma):=\{ \text{number of non-zero values in}~\hat{\gamma} \}$. 
 We also use the following three operations for the multi-indices: for  $\gamma^{(1)} \in \Pi_n^{N}$ and $\gamma^{(2)} \in \Pi_m^{N}$, the difference $ \gamma^{(1)}\setminus \gamma^{(2)} \in \Pi_k^{N}$ is specified through the counting set (as $|\hat \gamma^{(1)}|=|\hat \gamma^{(2)}|=N$) 
 \begin{align*}
     \hat{\gamma}^{\textrm{diff}} \coloneqq \lbrace \max \lbrace \hat{\gamma}^{(1)}_1- \hat{\gamma}^{(2)}_1,0 \rbrace, \ldots, \max \lbrace \hat{\gamma}^{(1)}_N- \hat{\gamma}^{(2)}_N,0 \rbrace \rbrace,
 \end{align*} 
 where $k$ is the number of non-zero elements in\footnote{In this work the difference $\gamma^{(1)}\setminus \gamma^{(2)}$ ( also, the intersection $\gamma^{(1)}\bigcap\gamma^{(2)}$) is \textit{never} used directly, only the quantity $\hco( \gamma^{(1)}\setminus \gamma^{(2)})$ (also, $\hco( \gamma^{(1)}\bigcap \gamma^{(2)})$) will be used and thus we require only $\hat{\gamma}^{\textrm{diff}}$ (also, $ \hat{\gamma}^{\textrm{intersec}}$). For practical purposes, one can think of the $\gamma$ as being ordered vectors (in increasing order) -- see Example \ref{example:Counting gamma multiindices}.}  $\hat{\gamma}^{\textrm{diff}}$.  

\color{black}
Similarly, the intersection $ \gamma^{(1)}\bigcap\gamma^{(2)}\in \Pi_k^{N}$ is specified through the counting set 
 \begin{align*}
     \hat{\gamma}^{\textrm{intersec}} \coloneqq \lbrace \min \lbrace \hat{\gamma}^{(1)}_1, \hat{\gamma}^{(2)}_1 \rbrace, \ldots, \min \lbrace \hat{\gamma}^{(1)}_N, \hat{\gamma}^{(2)}_N \rbrace \rbrace,
 \end{align*} 
 where $k$ is the number of non-zero elements. 
\color{black}
  
 The union is defined by 
 \begin{align*}
     \gamma^{(1)}\bigcup\gamma^{(2)}=(\gamma^{(1)}_1,\dots,\gamma^{(1)}_{n}, \gamma^{(2)}_1,\dots,\gamma^{(2)}_{m}) \in \Pi_{n+m}^{N}.
 \end{align*}

 For two sets of multi-indices, $\Pi_n^{N}$ and $\Pi_m^{N}$, with $n \neq m$, the union is defined as
\begin{align*}
         \Pi_n^{N} \bigcup \Pi_m^{N} = \lbrace \gamma: \gamma \in \Pi_n^{N} \text{ or } \gamma \in \Pi_m^{N} \rbrace.
\end{align*}
\end{definition}

The shuffle product\footnote{
\color{black}
The shuffle product $\shuffle$ is the binary operation that takes two ordered sequences (multi-indices for our purposes) and creates all possible ways of interleaving them while preserving the relative order of elements in each sequence. 
In a nutshell, the shuffle operation ensures that the relative order of elements within each original sequence is preserved and all possible interleavings are included  -- see Example \ref{example:Counting gamma multiindices}.
\color{black}
} 
(see, \cite{shuffledefinition}) for two multi-indices $\gamma^{(1)}\in \Pi_n^{N}$ and $\gamma^{(2)} \in \Pi_m^{N}$ is denoted by $\gamma^{(1)} \shuffle \gamma^{(2)}$. We write $\gamma^{(1)} \simeq \gamma^{(2)}$ if $m = n $ and there exists a permutation $\pi \in S_n$ such that $(\gamma^{(1)}_1,\ldots,\gamma^{(1)}_n)=(\gamma^{(2)}_{\pi(1)},\ldots,\gamma^{(2)}_{\pi(n)})$, where $S_n$ is the symmetric group on the set $\{1,\ldots,N\}$. 
\begin{example}
\label{example:Counting gamma multiindices}
   We present the following examples to make Definition \ref{def: set of sequence pi} clearer: For $N =3$, we have:
    \begin{align*}
        \Pi_1^3=\big\{( 1 &),( 2 ),( 3 )\big\},
        \quad
        \Pi_2^3=\big\{(1,1 ),(1,2 ) ,( 1,3),(2,1 ),(2 , 2 ),( 2 , 3 ),( 3 , 1 ),(3  , 2),(  3, 3 )\big\}.
    \end{align*}
    For $N \geq 3$, let  $(1,1,1 )$ (i.e., $\hat{\gamma} =(3,0,\dots,0),~|\hat{\gamma}|=N$), $(1,1,2 )$ (i.e., $\hat{\gamma} =(2,1,0,\dots,0)$) , $(1,2,2 ) \in \Pi_3^{N} $ (i.e., $\hat{\gamma} =(1,2,0,\dots,0)$). Then we have  
\begin{align*}
    \hco\big(( 1,1,1 )\big)=1, 
        \quad
        \hco\big((1,1,2 )\big)=2,
        \quad
        \hco\big((1,2,3 )\big)=3.
\end{align*}

     With regards to the set operations, we offer an example for $\alpha=(1,2,3,3),\beta = (3,5)$:
      \begin{align*}
        \alpha \setminus \beta =  ( 1,2,3 ),
        \quad 
         \alpha \bigcup \beta =  ( 1,2,3,3,3,5 ).
    \end{align*}
    In regards to the shuffle product, take for example $\alpha=(1,2)$, $\beta = (3)$, then 
    \begin{align*}
        \alpha \shuffle \beta 
        & = \big\{ ( 3,1,2 ),(1,3,2  ) ,( 1,2,3) \big\}.
    \end{align*}
\color{black}
It should be apparent that $\hco$ yields the same result for all elements in the shuffle product $\alpha \shuffle \beta$, and we can understand $( 1,2,3) \simeq (1,3,2  )
        \simeq ( 2,1,3)\simeq ( 2,3,1) \simeq ( 3,1,2 )\simeq( 3,2,1)$.
\color{black}
        
Lastly, the following summation of the elements in an $N\times N$ matrix with elements   $a_{i,j}$ for $i,j \in \lbrace 1, \ldots, N \rbrace $,  demonstrates the meaning of $\hco(\cdot)$ in the definition for the case $n=2$:
    \begin{align*}
        \sum_{i=1}^N \sum_{j=1}^N a_{i,j} 
        = \sum_{\gamma\in \Pi_2^N} a_{{\gamma_1},{\gamma_2}}
        = 
        \sum_{ \substack{ \gamma\in \Pi_2^N,   \hco(\gamma )=1 }   } a_{{\gamma_1},{\gamma_2}}
        +
        \sum_{ \substack{ \gamma\in \Pi_2^N,  \hco(\gamma )=2 }   } a_{{\gamma_1},{\gamma_2}},
    \end{align*}
    where we partitioned the summation into diagonal and off-diagonal elements. 
\end{example} 
For our analysis, we impose the following assumption.  
\begin{assumption}
\label{assum:main_weak error}
Assumption~\ref{assum:main} holds. Further, suppose that:  
\begin{enumerate}[(1)]
    \item The potentials $U, V \in \mathcal{C}^{7}(\mathbb{R})$, 
    and all derivatives of $\nabla U, \nabla V$ are uniformly bounded. (This in particular implies that $\nabla U, \nabla V$ are Lipschitz continuous.)
\item The convexity parameters $\lambda,K_V$ satisfy $\lambda \geq  7 K_V$.  
\item 
 \color{black}
Let $N \in \mathbb{N}$ with $N\gg 6$.  
For any $n \in \lbrace 1, \ldots, 6 \rbrace$ and $ ( \gamma_1,\dots,\gamma_{|\gamma|} )= \gamma \in \bigcup_{k=1}^n \Pi_k^N$, with integers $\gamma_j \in \lbrace 1, \ldots, N \rbrace$, 
there exists a constant $K$ independent of $N$ such that 
the function $g:\mathbb{R}^N \to \bR$, satisfies 
$|\partial^{|\gamma|}_{{x_{\gamma_1},\dots,x_{\gamma_{|\gamma|}}}} g|_{\infty} \leq  K \cdot N^{-\hco(\gamma)}$.
\color{black}
\item The function $g$ and its derivatives up to order $n$ are Lipschitz. (Note that item (3) implies that the function $g$ and its derivatives up to order $n-1$ are Lipschitz.) 
\end{enumerate}
\end{assumption}

\begin{remark}(On point (1) of Assumption~\ref{assum:main_weak error}\label{rem:H2-diff}): 
Our analysis has a small reduction regarding the order of regularity when compared to \cite{leimkuhler2014long} who require the drift of the underlying model to be 8-times continuously differentiable. In \cite{leimkuhler2014long}, in an intermediate step, weak convergence of order $\mathcal{O}(h)$ is first established for the non-Markovian Euler scheme (which only requires the drift to be 6-times continuously differentiable). 

The intermediate result is then employed to show weak convergence of a certain term where the test function involves 2nd order derivatives of the drift and the solution to the Kolmogorov PDE. In our notation this test function is denoted by $L$ and is precisely defined in \eqref{eq:def of B0} below.
Since $L$ already involves second order derivatives (of the potentials), the higher regularity is needed. We are not able to show that $L$ possesses sufficient regularity properties (i.e., item (3) in Assumption~\ref{assum:main_weak error}) to apply the intermediate result concerning the first order weak convergence and resort to a strong convergence result instead (as a consequence we only derive the weak convergence rate 1.5 instead of 2); see Remark \ref{rem:Losing1/2 conv rate} for full details.  
\end{remark}

\begin{remark}[On point (2) of Assumption~\ref{assum:main_weak error}] 
The constraint $\lambda \geq  7K_V$ is not sharp and relates to the need of sufficient convexity as we analyze the $n$-th order variation processes (for the process defined in \eqref{eq: def of Bi FLOW SDE in RN}).  
For instance, Lemma~\ref{lemma: second variation bound noiid details} establishes moment bounds of the $2$nd variation processes of \eqref{eq: def of Bi FLOW SDE in RN} and for it we require $\lambda > (2+1/N)K_V$ in \eqref{eq: second-var I2pst result}. For the moment bounds of the $6$-{th} variation processes (in Lemma~\ref{lemma: n-var process result}), the requirement ends up being $\lambda> (6+1/N)K_V$ and we streamline it to $\lambda \geq 7K_V$. This is a technical constraint of the analysis stemming from the interplay between the confinement potential and the interaction kernel functions and will be made more precise in the proofs of the lemmas. Lastly, this assumption is not comparable to those in  \cite{leimkuhler2014long} or \cite{vilmart2015postprocessed} as neither have  interaction kernels; only confining potentials (see \eqref{Model1xx}). 
\end{remark}

\begin{remark}[On point (3) of Assumption~\ref{assum:main_weak error}]
\label{rem:ON Assumtion H2 part (3)}   
\color{black}
   Typical examples for $g$ satisfying the above assumptions would be $g(\boldsymbol{x}) = \tilde{g}\left( \frac{1}{N} \sum_{i=1}^{N} f(x_i) \right)$, for some functions $f, \tilde{g}: \mathbb{R} \to \mathbb{R}$ that are sufficiently often differentiable with bounded derivatives. We provide a few examples.

   \begin{itemize}
       \item    For instance, consider the case $n = 3$   
   and let $ \gamma \in    \Pi_{1}^N \bigcup \Pi_2^{N} \bigcup \Pi_{3}^N$ for which $\hat{\mathcal{O}}( \gamma)=3$ (e.g.,
   $ \gamma = ( 1,2,3 )$). Therefore, our assumption requires $|\partial^{3}_{x_1, x_2,x_3} {g}|_{\infty} = \mathcal{O}(N^{-3})$, which is satisfied for regular enough functions $\tilde{g}$ and $f$. 
   
   As a further example, consider    $ \gamma = ( 1,1,3 )$ for which $\hat{\mathcal{O}}( \gamma)=2$, and hence, $|\partial^{3}_{x_1, x_1,x_3} {g}|_{\infty} = \mathcal{O}(N^{-2})$. 

   \item  If $f=\textrm{id} $ then for any $| \gamma|$-order derivative, one has $|\partial^{| \gamma|}_{{x_{ \gamma_1},\dots,x_{ \gamma_{| \gamma|}}}} {g}|_{\infty} = \mathcal{O}(N^{-| \gamma|})$.

   \item Let $\gamma$ be $k$-dimensional with $k\geq 2$ and set $\tilde g=\textrm{id}$. 
  
   Then, any cross derivative of $g$ satisfies $\partial^2_{x_i x_j}g(\boldsymbol{x})=0$ for $i\neq j$ and only when $i=j$ we have $\partial^2_{x_i x_i}g(\bfx)= \frac1N f''(x_i)$. Thus $|\partial^{| \gamma|}_{{x_{ \gamma_1},\dots,x_{ \gamma_{| \gamma|}}}} {g}|_{\infty} = 0$ whenever $\hat{\mathcal{O}}( \gamma)\geq 2$ and $|\partial^{| \gamma|}_{{x_{ \gamma_1},\dots,x_{ \gamma_{| \gamma|}}}} {g}|_{\infty} = \cO(N^{-1})$ whenever $\hat{\mathcal{O}}( \gamma)=1$.

  \item These observations hold if adapted to the multi-dimensional case at the cost of an increased notational complexity. Without offering too many details, assume the particle system is set in $\boldsymbol{x} = (x_1, \ldots, x_N) \in (\bR^d)^N$ where $x_i=(x_{i,1},\ldots,x_{i,d})\in \bR^d$ and let $g:(\bR^d)^N\to \bR$, $\tilde g:\bR\to \bR$ and $f:\bR^d \to \bR$. Let $i,j\in\{1,\ldots,N\}$.

  A few easy calculations show immediately, using the $\nabla$ gradient notation (as maps are multidimensional), that
  \[
  |\nabla_{x_i} g(\boldsymbol{x})|_\infty = \mathcal{O}(N^{-1})
  \quad \textrm{and}\quad 
  |\nabla_{x_j}\big(\nabla_{x_i} g(\boldsymbol{x}) \big)|_\infty = \mathcal{O}(N^{-1})\1_{i=j}+\mathcal{O}(N^{-2}).   
  \]
   \end{itemize}
Lastly, for the reader familiar with the language of Lions derivatives and flows of measures, it is possible to find (sufficiency) conditions on the derivatives of the measure functional $G:\mathcal P_2(\bR)\to \bR$ so that $g(\boldsymbol{x}):=G\big(\frac{1}{N} \sum_{i=1}^{N} \delta_{x_{i}}\big)$  satisfies the Assumption~3.3 (3). For the case $n=2$ (with $ \gamma \in    \Pi_{1}^N \bigcup \Pi_2^{N}$), this is easily seen using the idea of empirical projection map and \cite[Proposition 5.91 and Proposition 5.35]{CarmonaDelarue2017book1}. In essence, for the case $n=2$, a (strong) sufficient condition would be for the derivatives $\partial_v \partial_{\mu} G$ and $\partial^2_{\mu\mu} G$ to exist and be uniformly bounded. 
\color{black}
\end{remark}   

\paragraph*{Weak discretization error.} 
We define the weak discretization error induced by the non-Markovian scheme (see, \eqref{eq: def : non-Markov Euler scheme}) approximating the IPS $\boldsymbol{X}_{T}^{N}$ \eqref{ModelIPS} as follows: Let the test function  $g:\bR^N\to \bR$ satisfy Assumption~\ref{assum:main_weak error}. For any $T>0$, the weak approximation error is the worst case (over $g$) of 
\begin{align}
\label{eq:Weak Error}
    \textrm{Weak (discretization) Error}:=
    \bE\big[g\big(\boldsymbol{X}_{T}^{N}\big)\big]
    -
    \bE\big[g\big(\boldsymbol{X}^{N,h}_{T}\big)\big].
\end{align}

\paragraph*{The Kolmogorov backward equation for the flow and the weak error expansion.}

We study the weak error \eqref{eq:Weak Error} via an analysis of the Kolmogorov backward equation for the stochastic flow equation associated with the dynamics of \eqref{ModelIPS as in RN}. Concretely, let $ N \in \mathbb{N}$, $\boldsymbol{x} \in \mathbb{R}^N$,  and $0 \leq t \leq s $. Then we introduce  $\boldsymbol{X}^{t,\boldsymbol{x},N}_s = (X^{t,x_1,1,N}_s,\ldots,X^{t,x_N,N,N}_s)$, where
\begin{align}
    \label{eq: def of Bi}
    X^{t,x_i,i,N}_s 
    & 
    = x_i + \int_{t}^{s} B_i( X^{t,x_1,1,N}_u,\ldots, X^{t,x_N,N,N}_u) \mathrm{d}u + \sigma (W^{i}_s-W^{i}_t), 
    \quad i \in \lbrace 1, \ldots, N\rbrace,
\\ \label{eq: def of Bi FLOW SDE in RN}
     \boldsymbol{X}_s^{t, \boldsymbol{x},  N} &
    =
     \boldsymbol{x} 
     + \int_{t}^{s} B( \boldsymbol{X}_r^{t, \boldsymbol{x},  N} ) \mathrm{d}r 
     + \sigma (\boldsymbol{W}_s-\boldsymbol{W}_t).
\end{align}
The wellposedness of \eqref{eq: def of Bi} or \eqref{eq: def of Bi FLOW SDE in RN} under our assumptions is clear. The generator for \eqref{eq: def of Bi FLOW SDE in RN} is defined by  
\begin{align*}
    \mathcal{L}_N 
    &= \sum_{i=1}^{N} B_i \partial_{x_i} 
    + \frac{1}{2} \sigma^2 \partial^2_{x_i, x_i} ,
\end{align*}
where $B_i$ is the drift term for the $i$-th particle as in \eqref{eq: def of Bi}. 
Now, for $u:[0,T] \times \mathbb{R}^N \to \mathbb{R}$, we introduce the Kolmogorov backward equation:  
\begin{equation}
\label{PDE:Kolmogorov}
 \partial_t u + \mathcal{L}_N u = 0, \quad t \in [0,T), \quad u(T,\boldsymbol{x} ) = g(\boldsymbol{x}),
\end{equation}
for the above test function $g: \mathbb{R}^{N} \to \mathbb{R}$.
Under the above assumptions, the solution of the above PDE is given by the Feynman-Kac formula \cite{MilsteinTretyakov2021Book,F1971booksdeaaa}:
\begin{align}
    \label{eq: def of u in Ex form}
     u(t,\boldsymbol{x} ) 
     =   
     \mathbb{E} \left[ g(
     \boldsymbol{X}_T^{N})
      |   
      X_t^{i,N} = x_i, i \in \lbrace 1, \ldots ,N \rbrace \right].
\end{align}

To analyze the weak error \eqref{eq:Weak Error}, we need to expand it akin to a Talay--Tubaro expansion \cite{talaytubaru1990expansion} (see also  \cite{MilsteinTretyakov2021Book,haji2021simple}), but with certain fundamental differences. The following expansion is shown in \cite{leimkuhler2014long}.  
\begin{lemma}[Weak error expansion, Equation~(3.17) in \cite{leimkuhler2014long}]
\label{lemma:Weak Expansion Leimkuhler}

Let Assumption~\ref{assum:main_weak error} hold, then the following expansion of the weak error holds for the processes defined in \eqref{ModelIPS} and \eqref{eq: def : non-Markov Euler scheme}. 

Take any sufficiently small timestep $h$ (chosen independently of  $T,M,N$) and satisfying $0 < h < \min \{1/2\lambda,1\}$. Set $m \in \lbrace 0, \ldots, M-1 \rbrace$ for a given $M \in \mathbb{N}$ (recall $T = t_M = Mh$), we then have
\begin{align}
    \label{eq: weak error exp}
    \mathbb{E}[g(\boldsymbol{X}^{N}_{T})] - \mathbb{E}[g(\boldsymbol{X}^{N,h}_{T})] &= h^2 \mathbb{E} \left[\sum_{m=0}^{M-1} L(t_m,\boldsymbol{X}^{N,h}_{t_m}) \right]  +  \mathbb{E} \left[ \sum_{m=0}^{M-1} R(t_m,\boldsymbol{X}^{N,h}_{t_m}) \right],
\end{align}
where $L:\bR_{+}\times \bR^N \to \bR$ is defined via the map $u$ defined in \eqref{eq: def of u in Ex form} and the drifts $( B_i )_{i \in \{1,\ldots,N\}}$ in \eqref{ModelIPS} as 
\begin{align} 
    \nonumber
    L(t,\boldsymbol{x}) &= \frac{1}{2} \Big[ \sum_{i,j=1}^{N}B_{j}(\boldsymbol{x}) \partial_{x_j} B_{i}(\boldsymbol{x}) \partial_{x_i}u(t,\boldsymbol{x}) 
      +\frac{\sigma^2}{2} \sum_{i,j=1}^{N} \partial_{x_j}B_i(\boldsymbol{x})  \partial^2_{x_i, x_j} u(t,\boldsymbol{x}) 
    \\
\label{eq:def of B0}
    & \quad \quad \quad + \frac{\sigma^2}{2} \sum_{i,j=1}^{N} \partial^2_{x_j, x_j}B_i(\boldsymbol{x})  \partial_{x_i} u(t,\boldsymbol{x}) \Big],
    \quad t\geq 0 ,~ \boldsymbol{x}\in  \bR^N, 
\end{align}
and $R(\cdot,\cdot)$ is a collection of remainder terms (discussed and analyzed in Section \ref{sec:Remainder terms R}).
\end{lemma} 

\begin{proof}
This expansion is derived and presented in \cite[Equation~(3.17)]{leimkuhler2014long} and we do not reproduce it here. Our function $L$ given in \eqref{eq:def of B0} is denoted as $B_0$ in \cite{leimkuhler2014long} (see their Theorem 3.4). 
The sum of remainders $R(t_m,\cdot)$ in \eqref{eq: weak error exp} corresponds to the second sum of remainders $h^3 r(t_m,\cdot)$ in  \cite[Equation~(3.17)]{leimkuhler2014long} -- where the $h^3 r(t_m,\cdot)$ itself is a linear combination of remainders $h^3 r_j$ for $j \in \lbrace 1,\ldots,8 \rbrace$ appearing in Equations (3.8), (3.10)--(3.16) in \cite[p7-9]{leimkuhler2014long}. This derivation is discussed in more depth in Section \ref{sec:Remainder terms R}. 
\end{proof}
We aim to control the growth of $L$ and the remainder $R$ in terms of quantities that do not grow in $N$ -- this is a central technical difference to \cite{leimkuhler2014long}. A key point in the growth analysis for $L$ (and $R$) is the suitable control of moment bounds for the variation processes of $(\boldsymbol{X}^{t,\boldsymbol{x},N}_s)_{s \geq t \geq 0}$, which will be discussed in the next section.
We end this section with this manuscript's main result. Its proof is given in Section \ref{section: weak error expansion} after establishing a large collection of auxiliary results in Sections \ref{section: analysis of var process total} and \ref{section: analysis of the dus}, and the appendix.
\begin{Theorem}
\label{theorem: main weak convergence}
\color{black}
    Let Assumption~\ref{assum:main_weak error} hold, let  $\xi \in L^{10}(\Omega, \bR)$ and take a timestep $h$ (chosen independently of  $T,M,N$) sufficiently small satisfying $0 < h < \min \{1/2\lambda,1\}$. 
    Then, the following expansion for the weak error for the processes defined in \eqref{ModelIPS} and \eqref{eq: def : non-Markov Euler scheme} holds: for any $N,M\in \mathbb{N}$ (with $T$ defined as $T = Mh$), 
\begin{align*}
        \big|  \mathbb{E}[g(\boldsymbol{X}^{N}_{T})] 
          &- \mathbb{E}[g(\boldsymbol{X}^{N,h}_{T})] \big| \leq |C_0(T)| h + K h^{3/2},
    \end{align*}
where 
\begin{align*}
    C_0(T) & \coloneqq \mathbb{E} \left[\int_{0}^{T} L(t,\boldsymbol{X}^{N}_t) \mathrm{d}t \right], 
\quad \textrm{and} \quad 
|C_0(T)|  \leq K \exp(-\lambda_0 T)+Kh^{1/2},
\end{align*}
\color{black}
for some positive constants $\lambda_0, K$ independent of $h, T,M$ and $N$.
\end{Theorem}

\begin{proof}
    This result follows as a consequence of Lemmas \ref{lemma: analysis of the b0 term summation integration} and \ref{lemma: analysis of the residual term} in Section \ref{section: weak error expansion}.
\end{proof}
It is clear from the main statement that, as $T \to \infty$, the weak error is of order ${3/2}$ uniformly in $N$, and at finite time $T<\infty$ it is of order $1$ uniformly in $N$. We flag that for standard SDEs (without concern for uniformity in $N$), both \cite{leimkuhler2014long,vilmart2015postprocessed} obtain an error of order $h^2$ (as  $T \to \infty$). This gap in our result is of technical nature (appearing in Section \ref{section: weak error expansion}) and is detailed in Remark \ref{rem:Losing1/2 conv rate}. The assumption that $\xi \in L^{10}(\Omega,\mathbb{R})$ stems from the analysis of the remainder terms $R$ appearing in the weak error expansion \eqref{eq: weak error exp} (see Section \ref{sec:Remainder terms R}). The analysis of the $L$ term only requires $\xi \in L^{4}(\Omega,\mathbb{R})$ (see Section \ref{section: proof of B0 terms}). 
\medskip 

\color{black}
Before moving into proofs and auxiliary results, we flag that Theorem \ref{theorem: main weak convergence} controls the weak discretization error between the non-Markovian scheme \eqref{eq: def : non-Markov Euler scheme} and the IPS' solution \eqref{ModelIPS as in RN}. To control the weak error between the non-Markovian scheme \eqref{eq: def : non-Markov Euler scheme} and the initial MV-SDE \eqref{Model1xx} one requires only a triangle inequality with the IPS' solution \eqref{ModelIPS as in RN} alongside a quantitative uniform in time weak propagation of chaos result controlling the error from \eqref{ModelIPS as in RN} to \eqref{Model1xx}. The latter is provided by \cite[Theorem 1.1]{bernou2024uniform} stating that such error is of order $\cO(1/N)$ uniformly over time $T$. In the time limit $T\to \infty$, the non-Markovian scheme approximates the original MV-SDE (in weak error) as $\cO(N^{-1})+\cO(h^{3/2})$.
\color{black}

\section{Analysis of Variation processes}
\label{section: analysis of var process total}
The results for the variation processes established below are key to studying the uniform-in-$N$ and uniform-in-time decay of the solution to the Kolmogorov backward equation. For completeness, we state the following lemma regarding wellposedness of the multiple variation process used throughout this section and drawing from classical SDE theory. The subsequent results of the section are devoted to establishing $L^p$-estimates of these processes that decay exponentially in time in a non-explosive way in $N$. 

\begin{lemma}
Let Assumption~\ref{assum:main_weak error} hold, and let $\mathbb{N} \ni n \leq 6 $ and $T>0$. For any $\boldsymbol{x} \in \mathbb{R}^N$, and $T \geq s \geq t \geq 0$, let $\boldsymbol{X}^{t,\boldsymbol{x},N}_s$ be defined by \eqref{eq: def of Bi}. Then its first $n$-variation processes given by
\eqref{first_var_process noiid},
\eqref{second_var_process noiid}, and \eqref{n_var_process noiid} have unique solutions.
\end{lemma}
\begin{proof}
For any fixed $N$ and $T>0$, Assumption~\ref{assum:main_weak error} implies Assumption~(A) (p108), condition (5.12) and condition (5.15) (at any higher-order; see pages 120 and 122, respectively) in \cite{F1971booksdeaaa}. This suffices to ensure the wellposedness of the first \eqref{first_var_process noiid}, second \eqref{second_var_process noiid} and higher-order \eqref{n_var_process noiid} variation processes via Theorem 5.3 and Theorem 5.4 in \cite{F1971booksdeaaa} 
(see the comment after the proof of \cite[Theorem 5.5 (p123)]{F1971booksdeaaa} regarding the extension of Theorem 5.4 to higher order derivatives). Note that the analysis in the later sections of this article is carried out for an arbitrary $T>0$, and $T \to \infty$ is only considered in the final step; in particular, a wellposedness result for the variation processes in finite time suffices for our purposes.
\end{proof}

\subsection{First Variation process}
Here and below let $T>0$ be an arbitrary terminal time and let $T \geq s \geq t\geq 0$, $N\in \mathbb{N}$. The first variation process of $(\boldsymbol{X}^{t,\boldsymbol{x},N}_s)_{s \geq t\geq 0 }$ defined in \eqref{eq: def of Bi}, is given by 
    \begin{align}
        \label{first_var_process noiid}
        X^{t,x_i,i,N}_{s,x_j} 
        = \delta_{i,j} 
        + \int_{t}^{s} \sum_{l=1}^{N} \partial_{x_l}B_i(\bodX_{u}^{t,\bodx,N} )  X^{t,x_l,l,N}_{u,x_j} \mathrm{d}u,
    \end{align}
where  $\delta_{i,j}$ is the usual Kronecker symbol. The subindex $x_j$ in $X^{t,x_i,i,N}_{s,x_j}$ indicates the perturbation with respect to the $j$-th component of the initial data $\boldsymbol{x} \in \mathbb{R}^N$ of the flow process $(\boldsymbol{X}^{t,\boldsymbol{x},N}_s)_{s \geq t\geq 0 }$. Note that the processes $(X^{t,x_i,i,N}_{s,x_j})_{s \geq t \geq 0}$ (for different indices $i,j$) are, in general, not identically distributed. However, if the starting positions $x_i$, for $i \in \lbrace 1, \ldots, N \rbrace$, are all sampled from the same distribution, then the `diagonal' elements of $(\boldsymbol{X}^{t,\boldsymbol{x},N}_s)_{s \geq t\geq 0 }$ are identically distributed. (The same argument applies to the off-diagonal ones). This first lemma accounts for the different behaviours of $L^p$-moments for \eqref{first_var_process noiid}.
\begin{lemma}
\label{lemma: first variation bound noiid}
  Let Assumption~\ref{assum:main_weak error} hold and let $p \geq 2$.   
    Consider the first variation process $( X^{t,x_i,i,N}_{s,x_j})_{i,j \in \lbrace 1, \ldots, N \rbrace}$ defined by \eqref{first_var_process noiid} for $T \geq s \geq t \geq 0$. Then there exist constants $\lambda_1 \in (0,\lambda)$ and  $K>0$ (both are independent of  $s,t,T$ and $N$) such that for any $T \geq  s\geq t \geq 0$  
    \begin{equation}
    \label{eq: first var result}
          \sum_{i=1}^{N}  \mathbb{E}\Big[|X^{t,x_i,i,N}_{s,x_j}|^p\Big] \leq  K e^{- {\lambda p }(s-t)}
          \qquad \textrm{and} \qquad 
          \sum_{i=1, i\neq j}^{N}  \mathbb{E}\Big[|X^{t,x_i,i,N}_{s,x_j}|^p\Big] \leq  \frac{K}{N^{p-1}} e^{-\lambda_1 p(s-t)}.
    \end{equation} 
\end{lemma}

This lemma and the earlier Remark \ref{rem:ON Assumtion H2 part (3)} highlight the main difficulty faced in this manuscript's analysis. The above inequalities suggest that the term $i=j$ delivers the $\cO(1)$ behaviour while all other (cross-derivative $i\neq j$) elements decay proportionally to the number of particles. Our analysis throughout this section is involved, as this behaviour needs to be tracked across  higher-order variation processes.

Lastly, since \eqref{first_var_process noiid} is a linear ODE with random coefficients (bounded in $\bodX^{t ,\bodx,N}$) and initial condition $\delta_{i,j}$ we are able to obtain all $L^p$-moments without imposing further constraints on the integrability of $\bodX^{t ,\bodx,N}$.

\begin{remark}[The `symmetrization trick']
\label{Remark:symmetrization trick}
    We employ a recurring argument in this proof which we coin as the \textit{symmetrization trick}. This trick  exploits that $\nabla^2 V$ is an even function.  That is, for any $x \in \mathbb{R}$ we have $\nabla^2 V(x)=\frac{1}{2}(\nabla^2 V(x)+\nabla^2 V(-x))$. 
\end{remark}

\begin{proof}
    Note that in the following proof, the positive constant $K$ is independent of $s,t,T,N$ and may change line by line. The essence of this proof is the application of It\^o's formula followed by standard domination arguments. 
    
    Applying It\^{o}'s formula yields for any $i, j \in \lbrace 1, \ldots, N \rbrace$  
    \begin{align*}
    &e^{\lambda p (s-t)}\mathbb{E}\Big[|X^{t,x_i,i,N}_{s,x_j}|^p\Big]
        \\
        &
        \leq  \delta_{i,j} + p \int_{t}^{s} e^{\lambda p (u-t)}
        \bigg( 
        \mathbb{E} \Big[   X^{t,x_i,i,N}_{u,x_j}\cdot	\Big(  \sum_{l=1}^{N} \partial_{x_l}B_i(\bodX_{u}^{t,\bodx,N}
        ) X^{t,x_l,l,N}_{u,x_j} \Big)  |X^{t,x_i,i,N}_{u,x_j}|^{p-2} \Big]
        + \lambda   \mathbb{E}\Big[|X^{t,x_i,i,N}_{u,x_j}|^p\Big]
        \bigg) 
        \mathrm{d}u \\
    & = \delta_{i,j} - p \int_{t}^{s} e^{\lambda p (u-t)}
        \bigg(  \mathbb{E}\Big[  X^{t,x_i,i,N}_{u,x_j} \cdot	\Big( \nabla^{2} U(X^{t,x_i,i,N}_u) X^{t,x_i,i,N}_{u,x_j}    \Big) |X^{t,x_i,i,N}_{u,x_j}|^{p-2} \Big] 
        - \lambda   \mathbb{E}\Big[|X^{t,x_i,i,N}_{u,x_j}|^p\Big]
        \bigg)
        \mathrm{d}u 
    \\
    & \quad   -   p \int_{t}^{s}   e^{\lambda p (u-t)}\mathbb{E} \Big[  X^{t,x_i,i,N}_{u,x_j}\cdot	\Big(  \sum_{l=1}^{N} \frac{1}{N}\sum_{k=1}^{N} 
\partial_{x_l} \nabla V(X^{t,x_i,i,N}_u-X^{t,x_k,k,N}_u)  X^{t,x_l,l,N}_{u,x_j}  \Big)   |X^{t,x_i,i,N}_{u,x_j}|^{p-2} \Big] \mathrm{d}u 
    \\
    & \leq \delta_{i,j} 
    -  p \int_{t}^{s} e^{\lambda p (u-t)} \mathbb{E} \Big[ X^{t,x_i,i,N}_{u,x_j} \cdot	\Big(  \sum_{l=1}^{N} \frac{1}{N}\sum_{k = 1}^{N} \partial_{x_l} \nabla V(X^{t,x_i,i,N}_u-X^{t,x_k,k,N}_u)  X^{t,x_l,l,N}_{u,x_j} \Big)   |X^{t,x_i,i,N}_{u,x_j}|^{p-2}\Big] \mathrm{d}u ,
    \end{align*}
where we used \eqref{eq: def of lambda} of Assumption~\ref{assum:main} to obtain
\begin{align}
\label{eq: first var func V}
    -  X^{t,x_i,i,N}_{u,x_j}\cdot	\Big(  \nabla^{2} U(X^{t,x_i,i,N}_u) X^{t,x_i,i,N}_{u,x_j} \Big)  \leq -\lambda |X^{t,x_i,i,N}_{u,x_j}|^2.
\end{align}
We further note that by the chain rule,
\begin{align}
    \label{eq: first var: sum derivative on convolution}
    \nonumber 
    \sum_{l=1}^{N}
     \frac{1}{N}\sum_{k = 1}^{N} & \partial_{x_l} \nabla V(X^{t,x_i,i,N}_u
      -X^{t,x_k,k,N}_u)  X^{t,x_l,l,N}_{u,x_j}  
     \\
     & 
     =
     \frac{1}{N} \sum_{l=1}^{N}  
        \nabla^2 V(X^{t,x_i,i,N}_u-X^{t,x_l,l,N}_u)  ( X^{t,x_i,i,N}_{u,x_j}-X^{t,x_l,l,N}_{u,x_j}).
\end{align}
Hence taking summation over $i  \in \lbrace 1, \ldots, N\rbrace$ in \eqref{eq: first var: sum derivative on convolution}, we have 
\allowdisplaybreaks
\begin{align}
\nonumber
&\sum_{i = 1}^{N} \mathbb{E} \Big[   
        X^{t,x_i,i,N}_{u,x_j} \cdot \Big( 
        \sum_{l=1}^{N}     \frac{1}{N}\sum_{k = 1}^{N} \partial_{x_l} \nabla V(X^{t,x_i,i,N}_u-X^{t,x_k,k,N}_u)  X^{t,x_l,l,N}_{u,x_j}\Big)  
       |X^{t,x_i,i,N}_{u,x_j}|^{p-2} \Big]
     \\
     \nonumber
   & = \bE \Big[ \frac{1}{N} \sum_{i = 1}^{N}\sum_{l=1}^{N} \Big( | X^{t,x_i,i,N}_{u,x_j}|^{p-2}X^{t,x_i,i,N}_{u,x_j} \cdot	
        \nabla^2 V(X^{t,x_i,i,N}_u-X^{t,x_l,l,N}_u)  ( X^{t,x_i,i,N}_{u,x_j}-X^{t,x_l,l,N}_{u,x_j}) 
       \Big)  \Big]
    \\
    \nonumber
   & = \bE \Big[ \frac{1}{2N} \sum_{i = 1}^{N}\sum_{l=1}^{N} \Big( | X^{t,x_i,i,N}_{u,x_j}|^{p-2}X^{t,x_i,i,N}_{u,x_j} \cdot	
        \nabla^2 V(X^{t,x_i,i,N}_u-X^{t,x_l,l,N}_u)  ( X^{t,x_i,i,N}_{u,x_j}-X^{t,x_l,l,N}_{u,x_j}) 
      \Big)   \Big]
    \\
   & \quad + \bE \Big[ \frac{1}{2N} \sum_{l=1}^{N}\sum_{i = 1}^{N} \Big( | X^{t,x_l,l,N}_{u,x_j}|^{p-2}X^{t,x_l,l,N}_{u,x_j} \cdot	
        \nabla^2 V(X^{t,x_l,l,N}_u-X^{t,x_i,i,N}_u)  ( X^{t,x_l,l,N}_{u,x_j}-X^{t,x_i,i,N}_{u,x_j}) 
       \Big)  \Big]
       \label{eq: symmetric trick}
    \\
    \nonumber
   & = \bE \Big[ \frac{1}{2N} \sum_{i = 1}^{N}\sum_{l=1}^{N}  \Big(~\big( | X^{t,x_i,i,N}_{u,x_j}|^{p-2}X^{t,x_i,i,N}_{u,x_j}-| X^{t,x_l,l,N}_{u,x_j}|^{p-2}X^{t,x_l,l,N}_{u,x_j} \big)
   \\
   \label{eq: evenness of laplacian of V}
   & \hspace{4cm} \cdot	
        \nabla^2 V(X^{t,x_i,i,N}_u-X^{t,x_l,l,N}_u)  ( X^{t,x_i,i,N}_{u,x_j}-X^{t,x_l,l,N}_{u,x_j}) 
       \Big) \Big] 
    \geq 0,
\end{align}
where in \eqref{eq: evenness of laplacian of V}, we used the inequality $ \left( |x|^{p-2}x -|y|^{p-2}y \right) \cdot \left( x-y \right) \geq 0$ for all $x,y \in \bR$ and the fact that $ \nabla^2 V(x)\geq 0$, for all $x \in \mathbb{R}$. 
In accordance to Remark \ref{Remark:symmetrization trick}, we used the \textit{symmetrization trick} to derive \eqref{eq: symmetric trick}. Hence, taking summation over $i \in \lbrace 1, \ldots, N \rbrace$, we deduce that
\begin{align*}
e^{\lambda p (s-t)}
    \sum_{i = 1}^{N} \mathbb{E}\Big[|X^{t,x_i,i,N}_{s,x_j}|^p\Big] &\leq 
    \sum_{i = 1}^{N}  \delta_{i,j} 
    = 
    1
\end{align*}
and consequently for all $j \in \lbrace 1, \ldots, N \rbrace$, we have
\begin{align}
    \label{eq: first var result 1}
    \sum_{i = 1}^{N} \mathbb{E}\Big[|X^{t,x_i,i,N}_{s,x_j}|^p\Big]  \leq e^{- \lambda p(s-t)}.
\end{align}
To prove the claim of the second result in \eqref{eq: first var result}, we first derive for any $j  \in \lbrace 1, \ldots, N \rbrace$ , $\lambda_1 
 \in (0,\lambda)$ (using  It\^{o}'s formula and \eqref{eq: first var func V}), 
\begin{align*}
    e^{\lambda_1 p(s-t)}&\sum_{i=1,i\neq j}^N
    \mathbb{E}\Big[|X^{t,x_i,i,N}_{s,x_j}|^p\Big]
\leq    p \int_{t}^{s}  (\lambda_1-\lambda)e^{\lambda_1 (u-t)}\sum_{i=1,i\neq j}^N \mathbb{E} \Big[ |X^{t,x_i,i,N}_{u,x_j}|^{p} \Big]  
\\
& \hspace{-0.7cm}  -  e^{\lambda_1 (u-t)} \sum_{i=1,i\neq j}^N
\mathbb{E} \Big[ |X^{t,x_i,i,N}_{u,x_j}|^{p-2} 
\\
& \hspace{2.0cm}
\times X^{t,x_i,i,N}_{u,x_j} \cdot	\Big( 
\sum_{l=1}^{N} \frac{1}{N}\sum_{k = 1}^{N} \partial_{x_l}\nabla V(X^{t,x_i,i,N}_u-X^{t,x_k,k,N}_u)  X^{t,x_l,l,N}_{u,x_j} \Big)  \Big] \mathrm{d}u.
\end{align*}
For the last term above, we note that
\begin{align}
\nonumber
&- \sum_{i=1,i\neq j}^{N} \mathbb{E} \Big[ |X^{t,x_i,i,N}_{u,x_j}|^{p-2} 
        X^{t,x_i,i,N}_{u,x_j} \cdot	\Big( 
        \sum_{l=1}^{N}     \frac{1}{N}\sum_{k = 1}^{N} \partial_{x_l} \nabla V(X^{t,x_i,i,N}_u-X^{t,x_k,k,N}_u)  X^{t,x_l,l,N}_{u,x_j}  
       \Big) \Big]
     \\
     \nonumber
   & = - \bE \Big[ \frac{1}{N} \sum_{i=1,i\neq j}^{N}
   \sum_{l=1,l\neq j}^{N} \Big(|X^{t,x_i,i,N}_{u,x_j}|^{p-2}  X^{t,x_i,i,N}_{u,x_j}\cdot	
        \nabla^2 V(X^{t,x_i,i,N}_u-X^{t,x_l,l,N}_u)  ( X^{t,x_i,i,N}_{u,x_j}-X^{t,x_l,l,N}_{u,x_j}) 
        \Big) \Big]
    \\
    \nonumber
    & \quad  - \bE \Big[ \frac{1}{N} \sum_{i=1,i\neq j}^{N}\Big( |X^{t,x_i,i,N}_{u,x_j}|^{p-2}\cdot	  X^{t,x_i,i,N}_{u,x_j},
        \nabla^2 V(X^{t,x_i,i,N}_u-X^{t,x_j,j,N}_u)  ( X^{t,x_i,i,N}_{u,x_j}-X^{t,x_j,j,N}_{u,x_j}) 
       \Big)  \Big]
    \\
    \nonumber
   & \leq   -\bE \Big[ \frac{1}{2N} \sum_{i = 1,i\neq j}^{N}\sum_{l=1,l\neq j}^{N}  \Big( ~\big(  |X^{t,x_i,i,N}_{u,x_j}|^{p-2}X^{t,x_i,i,N}_{u,x_j}-|X^{t,x_l,l,N}_{u,x_j}|^{p-2} X^{t,x_l,l,N}_{u,x_j}\big)\cdot	
        \nabla^2 V(X^{t,x_i,i,N}_u-X^{t,x_l,l,N}_u)  
     \\
     \label{eq: symmetrization trick 2}
    & \quad   \cdot	( X^{t,x_i,i,N}_{u,x_j}-X^{t,x_l,l,N}_{u,x_j}) 
      \Big)  \Big] + \bE \Big[ \frac{1}{N} \sum_{i=1,i\neq j}^{N}  
        |\nabla^2 V(X^{t,x_i,i,N}_u-X^{t,x_j,j,N}_u)|   ~|X^{t,x_i,i,N}_{u,x_j}|^{p-1} ~
        |X^{t,x_j,j,N}_{u,x_j}|  
      \Big]
     \\
     \label{eq: convexity Kv}
    & \leq K_V 
      \bE \Big[  \sum_{i=1,i\neq j}^{N}  
        ~|X^{t,x_i,i,N}_{u,x_j}|^{p-1}~
         \frac{|X^{t,x_j,j,N}_{u,x_j}| }{N}  
      \Big]
    \\
    \label{eq: first var separation trick}
    &\leq 
    \varepsilon  \Big( \sum_{i=1,i\neq j}^{N}
    \bE \Big[|X^{t,x_i,i,N}_{u,x_j}|^p
    \Big] \Big) + \frac{K }{N^{p-1}}  
    \bE \Big[|X^{t,x_j,j,N}_{u,x_j}|^p
    \Big],
\end{align}
where in \eqref{eq: symmetrization trick 2}, we used the \textit{symmetrization trick} again. \eqref{eq: convexity Kv} follows from Assumption~\ref{assum:main}(2) and \eqref{eq: first var separation trick} is a consequence of Young's inequality, where $\varepsilon$ is some  positive constant which can be chosen to be arbitrarily small. Note that one can choose $\varepsilon, K$ to be positive constants (both independent of $s,t,N$)  satisfying $\varepsilon \in (0,\lambda-\lambda_1)$  for any $\lambda_1 \in (0,\lambda)$.
 Thus, we conclude that  
\begin{align*}
\nonumber 
    e^{\lambda_1 p (s-t)}&\sum_{i=1,i\neq j}^N\mathbb{E}\Big[|X^{t,x_i,i,N}_{s,x_j}|^p\Big]
    \\ 
& 
\leq   p \int_{t}^{s}  
 e^{\lambda_1 p (u-t)} \bigg((\lambda_1+\varepsilon - \lambda )
\sum_{i=1,i\neq j}^N \mathbb{E} \Big[ |X^{t,x_i,i,N}_{u,x_j}|^{p} \Big]  
+ 
 \frac{K}{N^{p-1}}
      \mathbb{E} \Big[ |X^{t,x_j,j,N}_{u,x_j}|^{p} \Big] \bigg)\mathrm{d}u 
\\\nonumber 
& 
\leq  
 \frac{Kp}{N^{p-1}}
\int_{t}^{s}     e^{p (\lambda_1-\lambda)(u-t)} \mathrm{d}u
\leq  
 \frac{K}{N^{p-1}   (\lambda-\lambda_1)  }  
,
\end{align*} 
where we noted that $\varepsilon$ can be chosen to be arbitrarily small, such that $\lambda_1+\varepsilon-\lambda$ remains negative (so the summation term can be upper bounded by zero).
We used \eqref{eq: first var result 1} to bound $\mathbb{E}  \big[ |X^{t,x_j,j,N}_{u,x_j}|^{p}  \big]$ and then noted that $\int_{t}^{s}     e^{p (\lambda_1-\lambda)(u-t)} \mathrm{d}u \leq 1/p(\lambda-\lambda_1)  $ to conclude the result. 

Therefore, for all   $\lambda_1 \in (0,\lambda)$, we have 
\begin{align*}
 \sum_{i=1,i\neq j}^N\mathbb{E}\Big[|X^{t,x_i,i,N}_{s,x_j}|^p\Big]
\leq \frac{K}{N^{p-1}}  e^{ -\lambda_1 p (s-t)} .
\end{align*} 
\end{proof}


The following proposition provides $L^2$-estimates for the differences of the processes defined in \eqref{first_var_process noiid} with the same initial points, but at different starting times. The results are used in Section \ref{section: proof of B0 terms} to establish time-regularity estimates for the derivatives of the function $u$. 
\begin{proposition}
    \label{propsition:  first variation bound prop v2}
    Let Assumption~\ref{assum:main_weak error} hold. Consider the first variation process with components $ (X^{t,x_i,i,N}_{s,x_j}    )_{ s \geq t \geq 0}$ defined by \eqref{first_var_process noiid} for $i,j \in \lbrace 1, \ldots, N \rbrace$ and assume that the starting positions $x_i\in L^4(\Omega, \bR)$ {are $\mathcal{F}_t$-measurable}  random variables  that are identically distributed over all $ i  \in \lbrace 1, \ldots, N \rbrace$.     
    Then there exist 
     $\lambda_2\in (0,\min\{   \lambda-2K_V,\lambda_1\} )$, $\lambda_3\in (0,\min\{   \lambda-2K_V,\lambda_2\} )$, and  
    $K>0$ (all independent of $s,t,T,N$) such that for all  $T \geq s \geq t \geq  0$  
    with $s-t < 1$ ,
    \begin{align}
    \label{eq: first var result addinitional}
          \sum_{i=1}^{N}  \mathbb{E}\Big[|X^{t,x_i,i,N}_{T,x_j}-X^{s,x_i,i,N}_{T,x_j}|^2\Big] 
          &
          \leq  K(s-t)  e^{- 2{ \lambda_2 }(T-s)},
          \\
          \label{eq: first var result 2}
          \sum_{i=1, i\neq j}^{N}  \mathbb{E}\Big[|X^{t,x_i,i,N}_{T,x_j}-X^{s,x_i,i,N}_{T,x_j}|^2\Big] 
          &
          \leq  \frac{K (s-t)}{N } e^{-2\lambda_3 (T-s)}.
    \end{align} 
\end{proposition}

\begin{proof}
Note that in the following proof, the positive constant $K$ is independent of $s,t,T,N$ and may change line by line.

\textit{Part 1: Preliminary manipulations.}  
Similar to the calculations in the proof of Lemma~\ref{lemma: first variation bound noiid}, we derive that (recalling that \eqref{first_var_process noiid} is an ODE with random coefficients), for all $i,j\in \{ 1,\dots,N\},~T\geq s \geq t \geq 0$, $\lambda_2\in (0,\min\{   \lambda-2K_V,\lambda_1\} )$,
    \begin{align*} e^{2\lambda_2  (T-s)}&
        |X^{t,x_i,i,N}_{T,x_j}-X^{s,x_i,i,N}_{T,x_j}|^2
        =
        |X^{t,x_i,i,N}_{s,x_j}-X^{s,x_i,i,N}_{s,x_j}|^2
        +
        2\lambda_2 
        \int_{0}^{T-s}  e^{2\lambda_2  u}   |X^{t,x_i,i,N}_{s+u,x_j}-X^{s,x_i,i,N}_{s+u,x_j}|^2    \mathrm{d}u
        \\
        & 
        -
         2\int_{0}^{T-s}  e^{2\lambda_2  u}  \Big( X^{t,x_i,i,N}_{s+u,x_j}-X^{s,x_i,i,N}_{s+u,x_j} \Big)	\cdot	\Big( 
         \nabla^{2} U(X^{t,x_i,i,N}_{s+u}) X^{t,x_i,i,N}_{s+u,x_j} 
         -\nabla^{2} U(X^{s,x_i,i,N}_{s+u})X^{s,x_i,i,N}_{s+u,x_j}     \Big)    \mathrm{d}u 
         \\
        & 
        -
         2  \int_{0}^{T-s} e^{2\lambda_2  u} \Big(   X^{t,x_i,i,N}_{s+u,x_j}-X^{s,x_i,i,N}_{s+u,x_j} \Big)	\cdot	\Big( 
         \frac{1}{N} \sum_{l=1}^{N}  
        \nabla^2 V(X^{t,x_i,i,N}_{s+u}-X^{t,x_l,l,N}_{s+u})  ( X^{t,x_i,i,N}_{{s+u},x_j}-X^{t,x_l,l,N}_{{s+u},x_j})
         \\
        &\quad\qquad \qquad 
        - \frac{1}{N} \sum_{l=1}^{N} \nabla^2 V(X^{s,x_i,i,N}_{s+u}-X^{s,x_l,l,N}_{s+u})  ( X^{s,x_i,i,N}_{{s+u},x_j}-X^{s,x_l,l,N}_{{s+u},x_j})
             \Big)   \mathrm{d}u.
   \end{align*}   
Therefore, using \eqref{eq: def of lambda},  we have 
\begin{align}
 &e^{2\lambda_2  (T-s)}
 |X^{t,x_i,i,N}_{T,x_j}-X^{s,x_i,i,N}_{T,x_j}|^2     \leq 
          |X^{t,x_i,i,N}_{s,x_j}-X^{s,x_i,i,N}_{s,x_j}|^2
          +2 (\lambda_2-\lambda)  \int_{0}^{T-s} e^{2\lambda_2  u}|X^{t,x_i,i,N}_{s+u,x_j}-X^{s,x_i,i,N}_{s+u,x_j}|^2 \dd u \nonumber
          \\
          & 
          \label{eq:hessian u term}
          \quad-2 \int_{0}^{T-s}  e^{2\lambda_2  u}\Big(  X^{t,x_i,i,N}_{s+u,x_j}-X^{s,x_i,i,N}_{s+u,x_j}  \Big)	\cdot	\Big(
         \nabla^{2} U(X^{t,x_i,i,N}_{s+u}) X^{s,x_i,i,N}_{s+u,x_j} 
         -\nabla^{2} U(X^{s,x_i,i,N}_{s+u})X^{s,x_i,i,N}_{s+u,x_j}         \Big)\mathrm{d}u 
          \\
        &\nonumber\quad 
        -
         2 \int_{0}^{T-s}  e^{2\lambda_2  u}\Big(   X^{t,x_i,i,N}_{s+u,x_j}-X^{s,x_i,i,N}_{s+u,x_j}  \Big)	\cdot	\bigg( ~\Big( 
         \frac{1}{N} \sum_{l=1}^{N}  
        \nabla^2 V(X^{t,x_i,i,N}_{s+u}-X^{t,x_l,l,N}_{s+u})  ( X^{t,x_i,i,N}_{{s+u},x_j}-X^{t,x_l,l,N}_{{s+u},x_j})
        \\
        &
        \label{eq: hessian V term}
        \qquad  - R^{i,t,s}_{s+u}\Big)
        +\Big( R^{i,t,s}_{s+u}-
         \frac{1}{N} \sum_{l=1}^{N}  
        \nabla^2 V(X^{s,x_i,i,N}_{s+u}-X^{s,x_l,l,N}_{s+u})  ( X^{s,x_i,i,N}_{{s+u},x_j}-X^{s,x_l,l,N}_{{s+u},x_j})  \Big)~\bigg)\mathrm{d}u
        ,
    \end{align}
    where we added and subtract the following auxiliary term:
    \begin{align*}
      R^{i,t,s}_{s+u} \coloneqq  \frac{1}{N} \sum_{l=1}^{N} \left[ \nabla^2 V(X^{t,x_i,i,N}_{s+u}-X^{t,x_l,l,N}_{s+u})  ( X^{s,x_i,i,N}_{{s+u},x_j}-X^{s,x_l,l,N}_{{s+u},x_j}) 
         \right].
    \end{align*} 
The result in \eqref{eq: first var result} and the fact that the starting positions $x_i$ are identically distributed, yield that for all $i,j\in \{ 1,\dots,N\}, i\neq j$, $\lambda_1 \in (0,\lambda)$ 
\begin{align}
\label{eq: aux result for differences 1}
\bE \Big[       \big|    X^{s,x_i,i,N}_{s+u,x_i}     \big|^4 \Big] \leq K e^{-4\lambda_1 u} ,\qquad 
 \bE \Big[       \big|    X^{s,x_i,i,N}_{s+u,x_j}     \big|^4 \Big] \leq \frac{K}{N^4} e^{-4\lambda_1 u}
 . 
\end{align}

\textit{Part 2: Establishing \eqref {eq: first var result addinitional}. }We further estimate the term involving $\nabla^2 U$, \eqref{eq:hessian u term}: under Assumption~\ref{assum:main}, we have
    \begin{align}
    \nonumber
          \sum_{i=1}^{N} \bE \Big[         &    \Big(      X^{t,x_i,i,N}_{s+u,x_j}-X^{s,x_i,i,N}_{s+u,x_j}
          \Big)	\cdot	\Big( 
         \nabla^{2} U(X^{t,x_i,i,N}_{s+u}) X^{s,x_i,i,N}_{s+u,x_j} 
         -\nabla^{2} U(X^{s,x_i,i,N}_{s+u})X^{s,x_i,i,N}_{s+u,x_j}     \Big) \Big]
         \\
         \nonumber
         &\leq \sum_{i=1}^{N} {  \color{black}   K   
         \bE \Big[       \big| X^{t,x_i,i,N}_{s+u,x_j}-X^{s,x_i,i,N}_{s+u,x_j} \big| ~\big| X^{t,x_i,i,N}_{s+u}- X^{s,x_i,i,N}_{s+u}       \big|~\big|      X^{s,x_i,i,N}_{s+u,x_j}  \big| \Big] }
         \\
         \label{eq: first var diff V }
         &\leq \varepsilon  \sum_{i=1}^{N} \bE \Big[       \big| X^{t,x_i,i,N}_{s+u,x_j}-X^{s,x_i,i,N}_{s+u,x_j} \big| ^2 \Big]
         + K  \sum_{i=1}^{N} \sqrt{    \bE \Big[       \big|     X^{t,x_i,i,N}_{s+u}- X^{s,x_i,i,N}_{s+u}   \big|^4 \Big]	\bE \Big[       \big|    X^{s,x_i,i,N}_{s+u,x_j}     \big|^4 \Big],	    }
    \end{align}
      {\color{black}  where we employed Young's inequality (with constants $\varepsilon,K>0$ independent of $s,t$ and $N$) and the Cauchy--Schwarz inequality. We further bound \eqref{eq: first var diff V } 
      by applying Lemma~\ref{lemma: 4 moment difference} 
  \begin{align}
    \nonumber  
 \eqref{eq: first var diff V }
 &\leq \varepsilon  \sum_{i=1}^{N} \bE \Big[       \big| X^{t,x_i,i,N}_{s+u,x_j}-X^{s,x_i,i,N}_{s+u,x_j} \big| ^2 \Big]
 \\ \nonumber
 & \qquad \qquad 
         + K (s-t) e^{- 2\lambda_2 u}  \bigg( \sqrt{    	\bE \Big[       \big|    X^{s,x_j,j,N}_{s+u,x_j}     \big|^4 \Big]	    }
        +     \sum_{i=1,i\neq j}^{N} \sqrt{       	\bE \Big[       \big|    X^{s,x_i,i,N}_{s+u,x_j}     \big|^4 \Big]	    }  \bigg) 
          \\
         \label{eq: bound on fourth moments}  
         &\leq \varepsilon  \sum_{i=1}^{N} \bE \Big[       \big| X^{t,x_i,i,N}_{s+u,x_j}-X^{s,x_i,i,N}_{s+u,x_j} \big| ^2 \Big]
        +   K (s-t) e^{- 2\lambda_2 u}       \Big( e^{- 2\lambda_1 u} +   \frac{1}{N^2}  \sum_{i=1, i\neq j}^{N} e^{- 2\lambda_1 u}  \Big) 	    
          \\
          \label{eq: differences U result}
         &\leq  \varepsilon  \sum_{i=1}^{N} \bE \Big[       \big| X^{t,x_i,i,N}_{s+u,x_j}-X^{s,x_i,i,N}_{s+u,x_j} \big| ^2 \Big]
         + K (s-t) e^{-4\lambda_2 u},
    \end{align}
      for some $\lambda_2\in (0,\min \{   \lambda-2K_V,\lambda_1\} )$, where we injected the estimate \eqref{eq: aux result for differences 1}  and used that the processes {\color{black}$ (     X^{t,x_i,i,N}_s )_{ s \geq t \geq 0}$} for $i \in \lbrace 1, \ldots, N \rbrace$  are identically distributed (due to the assumption in \ref{propsition:  first variation bound prop v2} on the starting positions $x_i$ being identically distributed over $i\in\{1,\ldots,N\}$) in the second inequality.}
      
As for the term involving $\nabla^2 V$, \eqref{eq: hessian V term}, after taking the expectation and summing over $i  \in \lbrace 1, \ldots, N\rbrace$, 
    \begin{align}
        &-\sum_{i=1}^{N}  \bE \Big[ \Big( X^{t,x_i,i,N}_{s+u,x_j}-X^{s,x_i,i,N}_{s+u,x_j}
         \Big)	\cdot	\Big( 
         \frac{1}{N} \sum_{l=1}^{N}  
        \nabla^2 V(X^{t,x_i,i,N}_{s+u}-X^{t,x_l,l,N}_{s+u})  ( X^{t,x_i,i,N}_{{s+u},x_j}-X^{t,x_l,l,N}_{{s+u},x_j}) - R^{i,t,s}_{s+u}  \Big)\Big]\nonumber
        \\
        &= -\frac{1}{2N} \sum_{i=1}^{N}\sum_{l=1}^{N} 
        \bE \Big[  \Big( \big(X^{t,x_i,i,N}_{s+u,x_j}-X^{s,x_i,i,N}_{s+u,x_j}\big) -   \big(X^{t,x_l,l,N}_{s+u,x_j}-X^{s,x_l,l,N}_{s+u,x_j}\big)    \Big)	\cdot 
        \nabla^2 V(X^{t,x_i,i,N}_{s+u}-X^{t,x_l,l,N}_{s+u})  \nonumber
        \\
        &
        \label{eq: symmetrization yields nonpositive}
        \qquad \qquad 	\cdot \Big(  \big( X^{t,x_i,i,N}_{{s+u},x_j}-X^{t,x_l,l,N}_{{s+u},x_j}\big)  -  \big(X^{s,x_i,i,N}_{{s+u},x_j}-
        X^{s,x_l,l,N}_{{s+u},x_j}\big) \Big) \Big]
        \leq 
        0,
    \end{align}
    where we once again use the \textit{symmetrization trick} and that $ \nabla^2 V(x)\geq 0$ for all $x \in \mathbb{R}$. Similar to the analysis  involving $\nabla^{2}U$, we obtain 
     \begin{align}
        &\sum_{i=1}^{N}  \bE \Big[ \Big(  X^{t,x_i,i,N}_{s+u,x_j}-X^{s,x_i,i,N}_{s+u,x_j}
         \Big)	\cdot	\Big( 
        R^{i,t,s}_{s+u} -
         \frac{1}{N} \sum_{l=1}^{N}  
        \nabla^2 V(X^{s,x_i,i,N}_{s+u}-X^{s,x_l,l,N}_{s+u})  ( X^{s,x_i,i,N}_{{s+u},x_j}-X^{s,x_l,l,N}_{{s+u},x_j})    \Big) \Big] \nonumber
        \\
        & \leq  \frac{1}{N}
        \sum_{i=1}^{N}
        \sum_{l=1}^{N}
        K_V \bE \Big[\big|X^{t,x_i,i,N}_{s+u,x_j}-X^{s,x_i,i,N}_{s+u,x_j}
        \big|~\big| 
        (X^{t,x_i,i,N}_{s+u}-X^{t,x_l,l,N}_{s+u})
        -( X^{s,x_i,i,N}_{s+u}-X^{s,x_l,l,N}_{s+u})\big| \nonumber
        \\
        &
        \label{eq:bound on 4.14}
        \quad \cdot
        ~\big|X^{s,x_i,i,N}_{{s+u},x_j}-X^{s,x_l,l,N}_{{s+u},x_j} \big| \Big] 
        \leq 
        \varepsilon \sum_{i=1}^{N} \bE \Big[\big|X^{t,x_i,i,N}_{s+u,x_j}-X^{s,x_i,i,N}_{s+u,x_j}
        \big|^2 \Big]
        +K (s-t) e^{-4\lambda_2 u},
    \end{align}
     where we applied similar calculations as in \eqref{eq: bound on fourth moments} and \eqref{eq: differences U result}. After taking the expectation, summing over $i  \in \lbrace 1, \ldots, N\rbrace$, and injecting our established estimates \eqref{eq: differences U result}, \eqref{eq: symmetrization yields nonpositive} and \eqref{eq:bound on 4.14}, we have for an arbitrary small $\varepsilon > 0$,
\begin{align}
\nonumber
         e^{2\lambda_2  (T-s)}
         & 
         \sum_{i=1}^N
         \bE \Big[ |X^{t,x_i,i,N}_{T,x_j}-X^{s,x_i,i,N}_{T,x_j}|^2 \Big]  
        \\ 
        \nonumber 
        &
       \leq   
        \sum_{i=1}^N
       \bE \Big[|X^{t,x_i,i,N}_{s,x_j}-X^{s,x_i,i,N}_{s,x_j}|^2 \Big]  
        +K (s-t) \int_{0}^{T-s} e^{(2\lambda_2 - 4\lambda_2) u} \dd u 
        \\
        \label{eq: first var diff sum 1 }
        &
        \qquad 
        +2(2\varepsilon+\lambda_2-\lambda )
        \int_{0}^{T-s} e^{2\lambda_2  u} 
        \sum_{i=1}^N\bE \Big[ | X^{t,x_i,i,N}_{s+u,x_j}-X^{s,x_i,i,N}_{s+u,x_j}|^2 \Big] 
        \dd u.
    \end{align}
Further notice that for the first summation term of \eqref{eq: first var diff sum 1 },  we obtain 
\allowdisplaybreaks
\begin{align}
  \nonumber
       \sum_{i=1}^N& \bE  \Big[|X^{t,x_i,i,N}_{s,x_j}-X^{s,x_i,i,N}_{s,x_j}|^2 \Big]
        \\ \nonumber 
       & =
       \sum_{i=1}^N\bE  \bigg[ \bigg|
        \int_0^{s-t}  \Big( \frac{1}{N}
       \sum_{l=1}^N 
       \nabla^2V(X_{t+u}^{t,x_i,i,N}-X_{t+u}^{t,x_l,l,N}) (X^{t,x_i,i,N}_{t+u,x_j}-X^{t,x_l,l,N}_{t+u,x_j})
       \\ \nonumber 
       &\hspace{4cm}
       + 
       \nabla^2U(X_{t+u}^{t,x_i,i,N}) X^{t,x_i,i,N}_{t+u,x_j}  \Big) \dd u 
       \bigg|^2 
       \bigg]
       \\ \nonumber 
       & 
       \leq K(s-t) \int_0^{s-t} \bigg( \frac{1}{N }  \sum_{i=1}^N \sum_{l=1}^N \bE  \Big[ |\nabla^2V(X_{t+u}^{t,x_i,i,N}-X_{t+u}^{t,x_l,l,N})|^2~|X^{t,x_i,i,N}_{t+u,x_j}-X^{t,x_l,l,N}_{t+u,x_j} |^2 \Big] 
        \\ \label{eq: jensen hessian V} 
        &
        \hspace{4cm}
        + \sum_{i=1}^N\bE  \Big[ |\nabla^2U(X_{t+u}^{t,x_i,i,N})|^2~|X^{t,x_i,i,N}_{t+u,x_j} |^2 \Big]  ~\bigg)~
        \dd u
        \\ \label{eq: Kv convexity 3}
        &\leq K\big(  4K_V +\lambda  \big)
        \int_0^{s-t}  \sum_{i=1}^N  
         \bE\Big[ |X^{t,x_i,i,N}_{t+u,x_j}  |^2 \Big]  \dd u 
         \\
         &
         \label{eq: first var diff initial diff}
        \leq 
        K(s-t) \int_0^{s-t}  e^{-2\lambda_1 u} \dd u \leq K(s-t),
  \end{align}
 where we used Jensen's inequality in \eqref{eq: jensen hessian V}, Assumption~\ref{assum:main} in \eqref{eq: Kv convexity 3}  {\color{black}  and  Lemma~\ref{lemma: first variation bound noiid} with $\lambda_1\in(0,\lambda)$} to establish \eqref{eq: first var diff initial diff}.   Consequently, for all $j \in \lbrace 1, \ldots, N \rbrace$, 
        we have (by choosing $\varepsilon $ arbitrarily small)  
\begin{align*}
    e^{2\lambda_2(T-s)}\sum_{i=1}^N 
    & \bE  \Big[ |X^{t,x_i,i,N}_{T,x_j}-X^{s,x_i,i,N}_{T,x_j}|^2 \Big] 
    \\\nonumber 
    &\quad ~ \leq
    K (s-t)
    + 
     K(s-t) \int_0^{T-s}  e^{-2\lambda_2 u} \dd u
     \\
     & 
     \qquad \qquad 
    +
    2(2\varepsilon+\lambda_2-\lambda )
        \int_{0}^{T-s} e^{2\lambda_2 u}
        \sum_{i=1}^N\bE \Big[ | X^{t,x_i,i,N}_{s+u,x_j}-X^{s,x_i,i,N}_{s+u,x_j}|^2 \Big].
\end{align*}
Using $\int_0^{T-s}  e^{-2\lambda_2 u} \dd u \leq {1}/{(2\lambda_2)}$
and $(2\varepsilon+\lambda_2-\lambda )<0$, we deduce that 
\begin{align} \label{eq: first var result 1 additional v2} 
     \sum_{i=1}^N\bE  \Big[ |X^{t,x_i,i,N}_{T,x_j}-X^{s,x_i,i,N}_{T,x_j}|^2 \Big] 
     &\leq
    K (s-t)e^{- {2\lambda_2  }(T-s)}.
\end{align}
This concludes the first part of the statement \eqref{eq: first var result addinitional}.\\

\textit{Part 3: Establishing \eqref{eq: first var result 2}.} Mimicking the estimates in \eqref{eq: symmetrization trick 2}--\eqref{eq: first var separation trick}, we first establish a result to deal with the term involving $\nabla^{2}V$: 
\begin{align}
\nonumber
        &-\sum_{i=1,i\neq j}^{N}  \bE \Big[ \Big( 
        X^{t,x_i,i,N}_{s+u,x_j}-X^{s,x_i,i,N}_{s+u,x_j}
         \Big)	\cdot	\Big( 
         \frac{1}{N} \sum_{l=1}^{N}  
        \nabla^2 V(X^{t,x_i,i,N}_{s+u}-X^{t,x_l,l,N}_{s+u})  ( X^{t,x_i,i,N}_{{s+u},x_j}-X^{t,x_l,l,N}_{{s+u},x_j}) - R^{i,t,s}_{s+u} \Big)\Big]
        \\\nonumber
        &= -\frac{1}{2N} \sum_{i=1,i\neq j}^{N}\sum_{l=1,l\neq j}^{N} 
        \bE \Big[   \Big( \big(X^{t,x_i,i,N}_{s+u,x_j}-X^{s,x_i,i,N}_{s+u,x_j}\big) -   \big(X^{t,x_l,l,N}_{s+u,x_j}-X^{s,x_l,l,N}_{s+u,x_j}\big)  \Big)	\cdot	\Big( 
        \nabla^2 V(X^{t,x_i,i,N}_{s+u}-X^{t,x_l,l,N}_{s+u})  
        \\\nonumber
        & \qquad \qquad \big( X^{t,x_i,i,N}_{{s+u},x_j}-X^{t,x_l,l,N}_{{s+u},x_j} -  X^{s,x_i,i,N}_{{s+u},x_j}+
        X^{s,x_l,l,N}_{{s+u},x_j}\big) \Big)  \Big]
        \\\nonumber
        &\quad 
        - \frac{1}{N} \sum_{i=1,i\neq j}^{N} 
        \bE \Big[ \Big(  X^{t,x_i,i,N}_{s+u,x_j}-X^{s,x_i,i,N}_{s+u,x_j}
         \Big)	\cdot	
        \nabla^2 V(X^{t,x_i,i,N}_{s+u}-X^{t,x_j,l,N}_{s+u}) 
        \\ \label{symmetrization trick 3}
        &\hspace{5cm} \cdot 
        \Big( \big(      X^{t,x_i,i,N}_{{s+u},x_j}-X^{t,x_j,j,N}_{{s+u},x_j} \big) - 
        \big(     X^{s,x_i,i,N}_{{s+u},x_j}
        -
        X^{s,x_j,j,N}_{{s+u},x_j}\big)~\Big)   \Big]
        \\
        &
        \label{convexity Kv 2}
        \leq  
        K_V \sum_{i=1,i\neq j}^{N}
        \bE \Big[ \Big(    \big|  X^{t,x_i,i,N}_{s+u,x_j}-X^{s,x_i,i,N}_{s+u,x_j}\big|\Big)\cdot \Big(    \frac{1}{N} \big| 
        X^{t,x_j,j,N}_{{s+u},x_j}-X^{s,x_j,j,N}_{{s+u},x_j}
        \big|\Big)
        \Big]
        \\
        \label{eq: proof second 22}
         &\leq 
          \varepsilon  \sum_{i=1,i\neq j}^{N} \bE \Big[\big|   X^{t,x_i,i,N}_{s+u,x_j}-X^{s,x_i,i,N}_{s+u,x_j}  \big| ^2	\Big]
         +
         \frac{K}{N^2} \sum_{i=1,i\neq j}^{N} 
           \bE \Big[\big|   X^{t,x_j,j,N}_{s+u,x_j}-X^{s,x_j,j,N}_{s+u,x_j}  \big| ^2\Big],
    \end{align}
where we once again used the \textit{symmetrization trick} in \eqref{symmetrization trick 3}, \ref{assum:main}(2) in \eqref{convexity Kv 2} and Young's inequality with constants $\varepsilon,K$ chosen such that $\varepsilon < \lambda$ in \eqref{eq: proof second 22}. Similarly, we have 
\begin{align}
        \nonumber
        &-\sum_{i=1,i\neq j}^{N}  \bE \Big[ \Big(  X^{t,x_i,i,N}_{s+u,x_j}-X^{s,x_i,i,N}_{s+u,x_j}
         \Big)	\cdot	\Big( 
        R^{i,t,s}_{s+u} -          \frac{1}{N} \sum_{l=1}^{N}  
        \nabla^2 V(X^{s,x_i,i,N}_{s+u}-X^{s,x_l,l,N}_{s+u})  ( X^{s,x_i,i,N}_{{s+u},x_j}-X^{s,x_l,l,N}_{{s+u},x_j})   
        \Big)\Big]
        \\
                \nonumber
        & \leq  \frac{1}{N}
        \sum_{i=1,i\neq j}^{N}
        \sum_{l=1,l\neq j}^{N}
        K_V \bE \Big[\big|X^{t,x_i,i,N}_{s+u,x_j}-X^{s,x_i,i,N}_{s+u,x_j}
        \big|~\big| 
        (X^{t,x_i,i,N}_{s+u}-X^{t,x_l,l,N}_{s+u})
        -( X^{s,x_i,i,N}_{s+u}-X^{s,x_l,l,N}_{s+u})
        \big|
        \\
        \nonumber
        &\quad \hspace{3.5cm} 
        \cdot \big|X^{s,x_i,i,N}_{{s+u},x_j}-X^{s,x_l,l,N}_{{s+u},x_j} \big| \Big]
        \\
        \nonumber
        &\quad +
        \frac{1}{N}
        \sum_{i=1,i\neq j}^{N} 
         K_V \bE \Big[\big|X^{t,x_i,i,N}_{s+u,x_j}-X^{s,x_i,i,N}_{s+u,x_j}
        \big|~\big| 
        (X^{t,x_i,i,N}_{s+u}-X^{t,x_j,j,N}_{s+u})
        -( X^{s,x_i,i,N}_{s+u}-X^{s,x_j,j,N}_{s+u})
        \big|
        \\
        \nonumber
        &\quad \hspace{3cm} 
        \cdot \big|X^{s,x_i,i,N}_{{s+u},x_j}-X^{s,x_j,j,N}_{{s+u},x_j} \big| \Big]
        \\
        \nonumber
        &\leq 
         \frac{1}{N}
        \sum_{i=1,i\neq j}^{N}
        \sum_{l=1,l\neq j}^{N}
        \Big( \varepsilon
        \bE \Big[\big|X^{t,x_i,i,N}_{s+u,x_j}-X^{s,x_i,i,N}_{s+u,x_j}
        \big|^2 \Big] +  \frac{K (s-t)}{N^2} e^{- 4 \lambda_2 u}
        \Big)
        \\
        \label{eq: proof second 2233 injected bounds}
        &\quad +
        \sum_{i=1,i\neq j}^{N}
        \Big( \varepsilon
        \bE \Big[\big|X^{t,x_i,i,N}_{s+u,x_j}-X^{s,x_i,i,N}_{s+u,x_j}
        \big|^2 \Big] + \frac{ K(s-t)}{N^2} e^{- 4 \lambda_2 u}
        \Big)
        \\
        \label{eq: proof second 2233}
        &\leq 
        2\varepsilon \Big(\sum_{i=1,i\neq j}^{N} \bE \Big[\big|X^{t,x_i,i,N}_{s+u,x_j}-X^{s,x_i,i,N}_{s+u,x_j}
        \big|^2 \Big]~ \Big)
        +
        \frac{K (s-t)}{N }    e^{- 4\lambda_2 u},
    \end{align}
        where we used the moment bounds established in \eqref{eq: aux result for differences 1} to get \eqref{eq: proof second 2233 injected bounds} and applied similar calculations as in \eqref{eq: differences U result}.
Taking summation over $i, i\neq j$ and collecting the estimates  in  \eqref{eq: proof second 22} and \eqref{eq: proof second 2233},  
     \begin{align}
        \nonumber
        e^{2\lambda_3  (T-s)}&\sum_{i=1,i\neq j}^N\bE \Big[  |X^{t,x_i,i,N}_{T,x_j}-X^{s,x_i,i,N}_{T,x_j}|^2 \Big] 
        \leq 
        \sum_{i=1,i\neq j}^N\bE \Big[ |X^{t,x_i,i,N}_{s,x_j}-X^{s,x_i,i,N}_{s,x_j}|^2 \Big] 
        \\
        &
        \label{eq: lambda3+2epsilon-lambda term}
        +2(\lambda_3+2\varepsilon-\lambda)
        \int_{0}^{T-s} 
        e^{2\lambda_3  u}
        \sum_{i=1,i\neq j}^N\bE \Big[ | X^{t,x_i,i,N}_{s+u,x_j}-X^{s,x_i,i,N}_{s+u,x_j}|^2 \Big] 
        \dd u
        \\ 
        \nonumber 
         & 
         +
        \frac{2K}{N^2}
         \int_{0}^{T-s} e^{2\lambda_3  u}
        \sum_{i=1,i\neq j}^N\bE \Big[ |  X^{t,x_j,j,N}_{s+u,x_j}-X^{s,x_j,j,N}_{s+u,x_j}|^2 \Big] 
        \dd u 
        \\ 
        \label{eq: 1/N^2 sum term}
        & 
        \hspace{4cm}
         +\frac{K(s-t)}{N} \int_{0}^{T-s} {  \color{black}      e^{(2\lambda_3- 4 \lambda_2) u}  }\dd u .      
    \end{align}
Note that \eqref{eq: lambda3+2epsilon-lambda term}$\leq 0$ since $\lambda_3<\lambda$ and $\varepsilon$ can be chosen to be arbitrarily small, so that this term remains negative. Implementing a crude upper bound \eqref{eq: first var result 1 additional v2} and using that $\lambda_3 <\lambda_2$ (hence the integrals remain bounded as $T$ gets large), we have
\begin{align*}
    \eqref{eq: 1/N^2 sum term}\leq  \frac{K^2(s-t)}{N(\lambda_2-\lambda_3)}+\frac{K(s-t)}{2N(2\lambda_2-\lambda_3)}.
    \end{align*}
    Joining together the terms and estimates terms, we have
    \begin{align}
    \nonumber 
        \label{eq: summation diff first var i neq j}
        e^{2\lambda_3  (T-s)}\sum_{i=1,i\neq j}^N
         \bE \Big[  |X^{t,x_i,i,N}_{T,x_j}&-X^{s,x_i,i,N}_{T,x_j}|^2 \Big]
        \\
        & \leq  \sum_{i=1,i\neq j}^N\bE \Big[ |X^{t,x_i,i,N}_{s,x_j}-X^{s,x_i,i,N}_{s,x_j}|^2 \Big]
         + \frac{K(s-t)}{N}.
    \end{align}
   
    \color{black} 
    
    To analyze the summation term in {\color{black} \eqref{eq: summation diff first var i neq j}, we first provide the following estimate: For all $u\geq 0$ and $i \in \{1,\dots,N\}$: 
    \begin{align}
    \bE \bigg[&
       \bigg|\frac{1}{N }
       \sum_{l=1}^N  \nabla^2V(X_{t+u}^{t,x_i,i,N}-X_{t+u}^{t,x_l,l,N})~(X^{t,x_i,i,N}_{t+u,x_j}-X^{t,x_l,l,N}_{t+u,x_j}) \bigg|^2 \bigg]
        \nonumber\\&\leq  2
       \bE \bigg[ 
       \bigg|\frac{1}{N }
       \sum_{l=1,l\neq j }^N  \nabla^2V(X_{t+u}^{t,x_i,i,N}-X_{t+u}^{t,x_l,l,N})~(X^{t,x_i,i,N}_{t+u,x_j}-X^{t,x_l,l,N}_{t+u,x_j}) \bigg|^2 \bigg] 
       \nonumber\\&
       \nonumber
       \quad + \frac{2}{N^2}
       \bE \bigg[   
        \bigg|
        \nabla^2V(X_{t+u}^{t,x_i,i,N}-X_{t+u}^{t,x_j,j,N})~(X^{t,x_i,i,N}_{t+u,x_j}-X^{t,x_j,j,N}_{t+u,x_j}) \bigg|^2
       \bigg] 
        \\&\nonumber
        \leq  
       \frac{K}{N }
       \sum_{l=1,l\neq j }^N \bE \Big[\Big| 
       \nabla^2V(X_{t+u}^{t,x_i,i,N}-X_{t+u}^{t,x_l,l,N})~
       (X^{t,x_i,i,N}_{t+u,x_j}-X^{t,x_l,l,N}_{t+u,x_j}) \Big|^2 \Big] 
       \\&\nonumber
       \hspace{2cm} 
       +\frac{K}{N^2} 
        \bE \Big[\Big| 
         X^{t,x_i,i,N}_{t+u,x_j}-X^{t,x_j,j,N}_{t+u,x_j} \Big|^2 \Big] 
        \\&\nonumber\leq
       {\color{black}  \frac{K}{N}\sum_{i=1,i\neq j}^{N}
        \mathbb{E} \Big[ |X^{t,x_i,i,N}_{t+u,x_j}|^2  
          \Big]  
          +  \frac{K}{N^2} \bE\Big[ |X^{t,x_j,j,N}_{t+u,x_j}|^2  
          \Big]  
        \leq \frac{K}{N^2} e^{-2\lambda_1u}},
\end{align} 
where we isolated the $l=j$ term, applied $(a+b)^2\leq 2(a^2+b^2)$ for $a,b \in \mathbb{R}$, before applying Jensen's inequality and Assumption~\ref{assum:main}. Therefore, for the summation term in \eqref{eq: summation diff first var i neq j}}:    
    \begin{align}
    \nonumber
         &\sum_{i=1,i\neq j}^N\bE \Big[ |X^{t,x_i,i,N}_{s,x_j}-X^{s,x_i,i,N}_{s,x_j}|^2 \Big]
          \\ \nonumber ~       & 
       \leq K(s-t) \int_0^{s-t} \bigg(   \sum_{i=1,i\neq j}^N\bE  \bigg[
       \bigg|\frac{1}{N }
       \sum_{l=1}^N  \nabla^2V(X_{t+u}^{t,x_i,i,N}-X_{t+u}^{t,x_l,l,N})~(X^{t,x_i,i,N}_{t+u,x_j}-X^{t,x_l,l,N}_{t+u,x_j}) \bigg|^2 \bigg] 
        \\ \nonumber 
        &
        \hspace{5cm}
        + \sum_{i=1,i\neq j}^N\bE  \Big[ |\nabla^2U(X_{t+u}^{t,x_i,i,N})|^2~|X^{t,x_i,i,N}_{t+u,x_j} |^2 \Big]  ~\bigg)~
        \dd u
     \\ \label{eq: fisrt var diff 2 intial} 
        &
        \leq 
        K(s-t) \int_0^{s-t} \bigg( 
          \Big( \sum_{i=1,i\neq j}^N 
          \frac{e^{-2\lambda_1 u }}{N^2}
          \Big) 
          + \frac{e^{-2\lambda_1 u }}{N}
          \bigg)~
          \dd u
          \leq
           \frac{K  (s-t)}{N } \int_0^{s-t} e^{-2\lambda_1 u } \dd u 
          \leq
         \frac{K  (s-t)}{N }  ,
    \end{align}    
    where we used Lemma~\ref{lemma: first variation bound noiid} in the last line. Consequently, substituting \eqref{eq: fisrt var diff 2 intial} into \eqref{eq: summation diff first var i neq j},  we conclude   
\begin{align*}
    \sum_{i=1, i\neq j}^{N}  
    \mathbb{E}\big[\, |X^{t,x_i,i,N}_{T,x_j}-X^{s,x_i,i,N}_{T,x_j}|^2\big] 
    \leq 
    \frac{K  (s-t)}{N } e^{-2\lambda_3  (T-s)}.
\end{align*} 
\end{proof}
\subsection{Second Variation process }
Let $T \geq s\geq t \geq 0$, $N\in \mathbb{N}$. The second variation process of $  ( \boldsymbol{X}^{t,\boldsymbol{x},N}_s )_{s \geq t\geq 0}$ 
is defined, for $i,j,k \in \lbrace 1, \ldots, N \rbrace$, as
\begin{align}
\label{second_var_process noiid}
X^{t,x_i,i,N}_{s,x_j,x_k} 
&=   \int_{t}^{s} \sum_{l=1}^{N} \partial_{x_l}B_i(\bodX_{u}^{t,\bodx,N})  X^{t,x_l,l,N}_{u,x_j,x_k}\mathrm{d}u \notag
\\
         & \qquad \qquad +\int_{t}^{s} \sum_{l=1}^{N} \sum_{l'=1}^{N} \partial^2_{x_l,x_{l'}} B_i(\bodX_{u}^{t,\bodx,N})  X^{t,x_l,l,N}_{u,x_j} X^{t,x_{l'},l',N}_{u,x_k} \mathrm{d}u.
    \end{align}
The following lemma proceeds the results in  Lemma~\ref{lemma: first variation bound noiid} and accounts for the different behaviours of $L^p$-moments for the second order variation processes defined in \eqref{second_var_process noiid}, which is needed in Lemma~\ref{lemma: n-var process result} and contributes to the analysis in Section \ref{section: weak error expansion}. 
\begin{lemma} 
    \label{lemma: second variation bound noiid details} 
    Let Assumption~\ref{assum:main_weak error} hold and let $p \geq 2$.
    Consider the second variation process \eqref{second_var_process noiid} and 
    assume that the starting positions $x_i\in L^2(\Omega, \bR)$ are $\mathcal{F}_t$-measurable  random variables  that are identically distributed over all $ i  \in \lbrace 1, \ldots, N \rbrace$. 
    Then there  exist $\lambda_4 \in (0,\min \{   \lambda-(2+1/N)K_V,\lambda_3\} )$ and  $K > 0$ (both independent of $s,t,T$ and $N$)  such that  for any $T \geq s\geq t \geq 0$ and $i\in \lbrace 1, \ldots, N \rbrace $ 
    \begin{align*}
         \ \mathbb{E}\Big[|X^{t,x_i,i,N}_{s,x_i,x_i}|^p\Big] 
         \leq K e^{-\lambda_4  p (s-t)},
        \qquad \quad 
        \sum_{i,j,k=1,~i\neq j \neq k}^{N} \mathbb{E}\Big[|X^{t,x_i,i,N}_{s,x_j,x_k}|^p\Big] 
        &\leq 
        \frac{K}{N^{2p-3}} e^{-\lambda_4  p(s-t)},
        \\
        \textrm{and}\quad 
        \sum_{i,k=1,~i\neq k}^{N}  
        \Big( 
        \mathbb{E}\Big[|X^{t,x_i,i,N}_{s,x_k,x_k}|^p\Big]
        + 
        \mathbb{E}\Big[|X^{t,x_i,i,N}_{s,x_i,x_k}|^p\Big] 
        + \mathbb{E}\Big[|X^{t,x_i,i,N}_{s,x_k,x_i}|^p\Big] 
        \Big)
        &\leq \frac{K}{N^{p-2}} e^{-\lambda_4  p(s-t)}
        .
    \end{align*} 
\end{lemma}
{\color{black}  	
\begin{remark}
This lemma continues to highlight the main difficulty faced in this manuscript's analysis. The above inequalities suggest that the second-order variation process, where $\hco((i,j,k))=1$, i.e., $i=j=k$, yields the $\cO(1)$ behaviour, while all other  elements (i.e., the cross-derivatives with $\hco((i,j,k))\geq 2$) decay differently with respect to the number of particles.  We refer to Lemma~\ref{lemma: n-var process result} for a general result. 
\end{remark}
}
\begin{proof} 
{\color{black} Let $p\geq 2$ be a given integer. Note that in the following proof, the positive constant $K$ is independent of $s,t,T,N$ and may change line by line}.

\textit{Part 1: Preliminary manipulations.} For $\lambda_4 \in (0,\min \{   \lambda-(2+1/N)K_V,\lambda_3 \} )$, we define $I_{t,s}^{2,p}$ for all $s\geq t \geq 0$ as
\begin{align}
\nonumber
    I_{t,s}^{2,p}
    :=
    e^{p \lambda_4 (s-t)} \bigg( 
     \frac{1}{N} 
     &
     \sum_{i=1}^{N}\mathbb{E}\Big[|X^{t,x_i,i,N}_{s,x_i,x_i}|^p 
         \Big]
    \\
    & \nonumber +
     N^{p-2}\sum_{i,k=1,~i\neq k}^{N}\Big(
      \mathbb{E}\Big[|X^{t,x_i,i,N}_{s,x_k,x_k}|^p\Big]
     +\mathbb{E}\Big[|X^{t,x_i,i,N}_{s,x_i,x_k}|^p\Big]  
     +\mathbb{E}\Big[|X^{t,x_i,i,N}_{s,x_k,x_i}|^p\Big]
     \Big)
     \\
     \label{eq: second var Ipst}
     & \qquad 
    +
        N^{2p-3}\sum_{i,j,k=1,~i\neq j \neq k}^{N}
    \bE \Big[ 
    | X^{t,x_i,i,N}_{s,x_j,x_k}|^{p }    \Big]
     \bigg),
\end{align}
for which we will aim to show that we can upper bound $ I_{t,s}^{2,p} \leq K$. We start by analysing each of the second variation processes:  
For any $i,j, k \in \lbrace 1, \ldots, N \rbrace$, $\lambda_4\in (0,\min \{   \lambda-(2+1/N)K_V,\lambda_3\} )$, $s\geq t \geq 0$,  we have that 
\allowdisplaybreaks
 \begin{align}
 \nonumber
     &e^{p \lambda_4 (s-t)} \mathbb{E}\Big[|X^{t,x_i,i,N}_{s,x_j,x_k}|^p\Big]  
     \\ \nonumber
        & =  - p \int_{t}^{s} e^{p \lambda_4 (u-t)}\mathbb{E} \bigg[  \big(   X^{t,x_i,i,N}_{u,x_j,x_k}
        \big)   \cdot  	\big(   
        \nabla^{2} U(X^{t,x_i,i,N}_u) X^{t,x_i,i,N}_{u,x_j,x_k} \big) 
        |X^{t,x_i,i,N}_{u,x_j,x_k}|^{p-2} \bigg] \mathrm{d}u 
    \\\nonumber
    &  \quad -  p \int_{t}^{s} e^{p \lambda_4 (u-t)} \mathbb{E} \bigg[     \Big( 
    X^{t,x_i,i,N}_{u,x_j,x_k}
   \Big)   \cdot	\Big(  
    \sum_{l=1}^{N} \partial_{x_l}\frac{1}{N}\sum_{q=1}^{N} \nabla V(X^{t,x_i,i,N}_u-X^{t,x_q,q,N}_u)  X^{t,x_l,l,N}_{u,x_j,x_k}  \Big)    |X^{t,x_i,i,N}_{u,x_j,x_k}|^{p-2} \bigg] \mathrm{d}u 
    \\\nonumber
    & \quad + p\int_{t}^{s} \lambda_4  e^{p \lambda_4 (u-t)} \mathbb{E}\Big[|X^{t,x_i,i,N}_{u,x_j,x_k}|^p\Big] \mathrm{d}u 
    \\\nonumber
    & \quad +   p \int_{t}^{s} e^{p \lambda_4 (u-t)}\mathbb{E} \bigg[  \Big(    X^{t,x_i,i,N}_{u,x_j,x_k}
     \Big)   \cdot	\Big(   
    \sum_{l=1}^{N} \sum_{l'=1}^{N}\partial_{x_l,x_{l'}}^2 B_i(\bodX_{u}^{t,\bodx,N}) X^{t,x_l,l,N}_{u,x_j}  X^{t,x_{l'},l',N}_{u,x_k}  \Big) |X^{t,x_i,i,N}_{u,x_j,x_k}|^{p-2} \bigg] \mathrm{d}u
     \\\label{eq: hessian u term bound 2}
        & \leq 
    p\int_{t}^{s} (-\lambda+\lambda_4) e^{p \lambda_4 (u-t)} \mathbb{E}\Big[|X^{t,x_i,i,N}_{u,x_j,x_k}|^p\Big] \mathrm{d}u
     \\\label{eq: o2 convolution term}
    &  \quad
    -  p \int_{t}^{s} e^{p \lambda_4 (u-t)} \mathbb{E} \bigg[   \Big( X^{t,x_i,i,N}_{u,x_j,x_k}
       \Big)   \cdot	\Big(   
    \frac{1}{N}\sum_{l=1}^{N}  {  \color{black}       \nabla^2 V }(X^{t,x_i,i,N}_u-X^{t,x_l,l,N}_u)  (X^{t,x_i,i,N}_{u,x_j,x_k}-X^{t,x_l,l,N}_{u,x_j,x_k})  \Big)
    \\   \nonumber 
    & \hspace{4cm} \cdot
          |X^{t,x_i,i,N}_{u,x_j,x_k}|^{p-2} \bigg] \mathrm{d}u 
    \\ \label{eq: o2 lower order var term}
    & \quad +   p \int_{t}^{s} e^{p \lambda_4 (u-t)} \mathbb{E} \Big[ 
     |X^{t,x_i,i,N}_{u,x_j,x_k}|^{p-1}\sum_{l=1}^{N} \sum_{l'=1}^{N}
     |\partial_{x_l,x_{l'}}^2 
 B_i(\bodX_{u}^{t,\bodx,N})|
     ~| X^{t,x_l,l,N}_{u,x_j}| ~| X^{t,x_{l'},l',N}_{u,x_k}|    \Big] \mathrm{d}u,
 \end{align}
where we used Assumption~\ref{assum:main} on the first term to obtain \eqref{eq: hessian u term bound 2}. In what follows, the last two terms, which we will refer to as the convolution term \eqref{eq: o2 convolution term} and the lower order variation term \eqref{eq: o2 lower order var term}, will be investigated in more detail. \\

\textit{Part 2: Analysis of the convolution term  \eqref{eq: o2 convolution term}.}
We analyze the convolution term \eqref{eq: o2 convolution term} by considering five different cases: $i=j=k,~i\neq j =k,~~i=j \neq k,~~ i = k \neq j$ and $i\neq j \neq k$, where we will see that different methodologies need to be implemented based on the values $\hco ((i,j,k))$ of the second variation processes. {\color{black} Note that  $i \neq j \neq k$ indicates that none of the indices are identical. } 
\\ 
\textbf{Case $i=j=k$:} The convolution term, after summing over all $i \in \lbrace 1, \ldots, N \rbrace$, simplifies to
\begin{align}
    \nonumber 
    &-\sum_{i =1 }^{N} \frac{1}{N}  \sum_{l=1}^{N}
    \bE \left[\Big(    | X^{t,x_i,i,N}_{u,x_i,x_i}|^{p-2}X^{t,x_i,i,N}_{u,x_i,x_i}
      \Big)   \cdot	\Big(  
        \nabla^2 V(X^{t,x_i,i,N}_u-X^{t,x_l,l,N}_u)  ( X^{t,x_i,i,N}_{u,x_i,x_i}-X^{t,x_l,l,N}_{u,x_i,x_i}) 
        \Big)  \right]
    \\\nonumber 
    & = - \frac{1}{N} \sum_{i,l =1,~i\neq l}^{N}  \bE \Big[ 
  \Big( | X^{t,x_i,i,N}_{u,x_i,x_i}|^{p-2}X^{t,x_i,i,N}_{u,x_i,x_i}
      \Big)   \cdot	\Big(   
        \nabla^2 V(X^{t,x_i,i,N}_u-X^{t,x_l,l,N}_u)   X^{t,x_i,i,N}_{u,x_i,x_i} \Big) 
       \Big]
    \\\nonumber 
    &\quad + \frac{1}{N} \sum_{i,l =1,~i\neq l}^{N}  \bE \Big[ 
  \Big(  | X^{t,x_i,i,N}_{u,x_i,x_i}|^{p-2}X^{t,x_i,i,N}_{u,x_i,x_i}
     \Big)   \cdot	\Big(   
        \nabla^2 V(X^{t,x_i,i,N}_u-X^{t,x_l,l,N}_u)   X^{t,x_l,l,N}_{u,x_i,x_i}  \Big)
      \Big]
        \\ \label{eq:kv/N bound}
   & \leq {\color{black} \frac{K_V}{N}
   \sum_{i,l =1,~i\neq l }^{N}
   \bE \Big[
  |X^{t,x_i,i,N}_{u,x_i,x_i}|^{p-1}\         \cdot	    |X^{t,x_l,l,N}_{u,x_i,x_i}|  \Big] }
   \\\label{eq: second var res1}
   & \leq 
  \frac{K_V(p-1)}{p} \sum_{i =1 }^{N} \bE \Big[ 
  |X^{t,x_i,i,N}_{u,x_i,x_i}|^{p}   \Big] 
  +  \frac{K_V }{Np}   \sum_{i,l =1,~i\neq l }^{N}
  \bE \Big[ 
  |X^{t,x_l,l,N}_{u,x_i,x_i}|^{p}   \Big],
\end{align}
where we used $\nabla^2 V(x)\geq 0$ for all $x \in \mathbb{R}$ and Assumption~\ref{assum:main} to derive \eqref{eq:kv/N bound} and \eqref{eq: second var res1} is a consequence of Young's inequality. 
\\ \\
\textbf{Case $i\neq j =  k$:}  In this situation, after summing over all $i, k \in \lbrace 1, \ldots, N \rbrace,~i \neq k$ and splitting the summation over $l$, we derive using the \textit{symmetrization trick} 
that 
\begin{align}
\nonumber 
    &-\sum_{i,k=1,~i\neq k}^{N}\bE \bigg[ \frac{1}{N}  \sum_{l=1}^{N}\Big(  | X^{t,x_i,i,N}_{u,x_k,x_k}|^{p-2}X^{t,x_i,i,N}_{u,x_k,x_k}
       \Big)   \cdot	\Big(   
        \nabla^2 V(X^{t,x_i,i,N}_u-X^{t,x_l,l,N}_u)  ( X^{t,x_i,i,N}_{u,x_k,x_k}-X^{t,x_l,l,N}_{u,x_k,x_k}) 
    \Big)     \bigg]
    \\
    \nonumber 
   & = -\frac{1}{2N}
   \sum_{i,k,l=1,~i,l\neq k}^{N}
   \bE \Big[ 
   \Big(  | X^{t,x_i,i,N}_{u,x_k,x_k}|^{p-2}X^{t,x_i,i,N}_{u,x_k,x_k}-| X^{t,x_l,l,N}_{u,x_k,x_k}|^{p-2}X^{t,x_l,l,N}_{u,x_k,x_k}
       \Big)   
   \\
   \nonumber 
   & \qquad \qquad  \cdot	\Big(   
        \nabla^2 V(X^{t,x_i,i,N}_u-X^{t,x_l,l,N}_u)  ( X^{t,x_i,i,N}_{u,x_k,x_k}-X^{t,x_l,l,N}_{u,x_k,x_k}
        ) \Big)
       \Big] 
    \\
    \nonumber 
    &\quad 
    -\bE \bigg[ \frac{1}{N} \sum_{i,k=1,~i\neq k}^{N} \Big(    | X^{t,x_i,i,N}_{u,x_k,x_k}|^{p-2}X^{t,x_i,i,N}_{u,x_k,x_k}
      \Big)   \cdot	\Big(   
        \nabla^2 V(X^{t,x_i,i,N}_u-X^{t,x_k,k,N}_u)  ( X^{t,x_i,i,N}_{u,x_k,x_k}-X^{t,x_k,k,N}_{u,x_k,x_k}) 
       \Big)  \bigg]
    \\
    \label{eq:kv/N bound22}
    &\leq 
     {\color{black}   K_V 
      \sum_{i,k=1,~i\neq k }^{N}
      \bE \Big[    
  |X^{t,x_i,i,N}_{u,x_k,x_k}|^{p-1}  \cdot    \frac{|X^{t,x_k,k,N}_{u,x_k,x_k}|}{N}  \Big] }
  \\
  \label{eq: second var res2}
    &\leq
     \frac{K_V(p-1)}{p}\sum_{i,k=1,~i\neq k}^{N}
    \bE \Big[ 
    | X^{t,x_i,i,N}_{u,x_k,x_k}|^{p }    \Big]
    + 
    \frac{K_V}{pN^{p-1} } \sum_{k=1 }^{N}
    \bE \Big[  |X^{t,x_k,k,N}_{u,x_k,x_k}|^p    \Big],
\end{align}
 where we used $\nabla^2 V(x)\geq 0$ for all $x \in \mathbb{R}$ and Assumption~\ref{assum:main} to derive \eqref{eq:kv/N bound22}. \eqref{eq: second var res2} is a consequence of Young's inequality. 
\\

\textbf{Cases} $i=j\neq k$ and  $i = k \neq j $: These two cases share similar calculations and we show the first of these. Summing over all $i, k \in \lbrace 1, \ldots, N \rbrace,~i \neq k$,  we have  
\begin{align}
\nonumber 
   & -\sum_{i,k=1,~i\neq k}^{N}  \bE \bigg[ \frac{1}{N} \sum_{l=1}^{N}  \Big(   | X^{t,x_i,i,N}_{u,x_i,x_k}|^{p-2}X^{t,x_i,i,N}_{u,x_i,x_k} 
    \Big)   \cdot	\Big(   
        \nabla^2 V(X^{t,x_i,i,N}_u-X^{t,x_l,l,N}_u)  ( X^{t,x_i,i,N}_{u,x_i,x_k}-X^{t,x_l,l,N}_{u,x_i,x_k}) 
      \Big)  \bigg]
        \\\nonumber 
   & \leq  
   -\frac{1}{N}
   \sum_{i,k,l=1,~i\neq l\neq k}^{N}
   \bE \Big[ \left |
   \Big(  | X^{t,x_i,i,N}_{u,x_i,x_k}|^{p-2}X^{t,x_i,i,N}_{u,x_i,x_k} 
      \Big)   \cdot	\Big(   
        \nabla^2 V(X^{t,x_i,i,N}_u-X^{t,x_l,l,N}_u)   
        X^{t,x_l,l,N}_{u,x_i,x_k}
       \Big) \right |  \Big] 
    \\\nonumber 
    &\qquad +\frac{1}{N} \sum_{i,k=1,~i\neq k}^{N}
    \bE \left[  \left | \Big(   | X^{t,x_i,i,N}_{u,x_i,x_k}|^{p-2}X^{t,x_i,i,N}_{u,x_i,x_k} 
       \Big)   \cdot	\Big(   
        \nabla^2 V(X^{t,x_i,i,N}_u-X^{t,x_k,k,N}_u)   X^{t,x_k,k,N}_{u,x_i,x_k}  
    \Big) \right |   \right]
    \\ \nonumber
    & \leq 
    \frac{K_V}{N}
   \sum_{i,k,l=1,~i\neq l\neq k}^{N} \bigg(   
     \frac{p-1}{p}\bE \Big[
     | X^{t,x_i,i,N}_{u,x_i,x_k}|^{p}  \Big]  
     + \frac{1}{p}
     \bE \Big[| X^{t,x_l,l,N}_{u,x_i,x_k}|^{p} 
     \Big]      \bigg) 
     \\ 
     \label{eq: young and H1 implementation} 
     &\hspace{5cm} 
     +
     \frac{K_V}{N}
   \sum_{i,k=1,~i\neq k}^{N} \bigg(  
     \frac{p-1}{p}\bE \Big[
     | X^{t,x_i,i,N}_{u,x_i,x_k}|^{p}  \Big]  
     + \frac{1}{p}
     \bE \Big[| X^{t,x_k,k,N}_{u,x_i,x_k}|^{p} 
     \Big]      \bigg)  
    \\\nonumber 
    & 
    \leq    \frac{K_V(p-1)}{p}   
    \sum_{i,k=1,~i\neq k}^{N}
     \bE \Big[ 
    | X^{t,x_i,i,N}_{u,x_i,x_k}|^{p} 
      \Big] 
      +
      \frac{K_V}{Np}
      \sum_{i,k,l=1,~i\neq l\neq k}^{N}
      \bE \Big[ 
        |X^{t,x_l,l,N}_{u,x_i,x_k}|^{p} 
      \Big]
      \\ 
      \label{eq: second var res3}
    & \qquad + 
      \frac{K_V}{N}
   \sum_{i,k=1,~i\neq k}^{N} \bigg(  
     \frac{p-1}{p}\bE \Big[
     | X^{t,x_i,i,N}_{u,x_i,x_k}|^{p}  \Big]  
     + \frac{1}{p}
     \bE \Big[| X^{t,x_k,k,N}_{u,x_i,x_k}|^{p} 
     \Big]     \bigg)  ,
 \end{align}
where once again, we split up the summation over $l$ and repeatedly apply Assumption~\ref{assum:main} and Young's inequality to obtain \eqref{eq: young and H1 implementation} and merely rearranging terms yields \eqref{eq: second var res3}. \\ \\
\noindent
\textbf{Case $i\neq j \neq k$:} Summing over all $i, j,k \in \lbrace 1, \ldots, N \rbrace,~i \neq j \neq  k$, we obtain
\begin{align}
\nonumber  
&- \sum_{i,j,k=1,~i\neq j \neq k}^{N}  \bE \bigg[ \frac{1}{N} \sum_{l=1}^{N}  \Big( | X^{t,x_i,i,N}_{u,x_j,x_k}|^{p-2}X^{t,x_i,i,N}_{u,x_j,x_k}
   \Big)   \cdot	\Big(  
        \nabla^2 V(X^{t,x_i,i,N}_u-X^{t,x_l,l,N}_u)  ( X^{t,x_i,i,N}_{u,x_j,x_k}-X^{t,x_l,l,N}_{u,x_j,x_k}) 
       \Big)   \bigg]
        \\ \nonumber 
   & = -\frac{1}{2N}
   \sum_{i,j,k,l=1,~i\neq j \neq k,~l\neq j \neq k}^{N}
   \bE \Big[ 
    \Big(   | X^{t,x_i,i,N}_{u,x_j,x_k}|^{p-2}X^{t,x_i,i,N}_{u,x_j,x_k}-| X^{t,x_l,l,N}_{u,x_j,x_k}|^{p-2}X^{t,x_l,l,N}_{u,x_j,x_k} \Big)  
   \\ \nonumber 
   & \qquad \qquad     \cdot	\Big(   
        \nabla^2 V(X^{t,x_i,i,N}_u-X^{t,x_l,l,N}_u)  ( X^{t,x_i,i,N}_{u,x_j,x_k}-X^{t,x_l,l,N}_{u,x_j,x_k}
        ) 
        \Big) \Big] 
    \\ \nonumber 
    &\quad 
    -\bE \bigg[ \frac{1}{N} \sum_{i,j,k=1,~i\neq j \neq k}^{N}  \Big(  | X^{t,x_i,i,N}_{u,x_j,x_k}|^{p-2}X^{t,x_i,i,N}_{u,x_j,x_k}
       \Big)   \cdot	\Big(   
        \nabla^2 V(X^{t,x_i,i,N}_u-X^{t,x_j,j,N}_u)  ( X^{t,x_i,i,N}_{u,x_j,x_k}-X^{t,x_j,j,N}_{u,x_j,x_k}) 
     \Big)   \bigg]
    \\ \label{eq: summation over l split up terms} 
    &\quad 
    -\bE \bigg[ \frac{1}{N} \sum_{i,j,k=1,~i\neq j \neq k}^{N} \Big(   | X^{t,x_i,i,N}_{u,x_j,x_k}|^{p-2}X^{t,x_i,i,N}_{u,x_j,x_k}
       \Big)   \cdot	\Big(   
        \nabla^2 V(X^{t,x_i,i,N}_u-X^{t,x_k,k,N}_u)  ( X^{t,x_i,i,N}_{u,x_j,x_k}-X^{t,x_k,k,N}_{u,x_j,x_k}) 
     \Big)   \bigg]
    \\ \label{eq: convexity estimate combine terms} 
    &\leq {\color{black} K_V  
    \sum_{i,j,k=1,~i\neq j \neq k}^{N}
    \bE \Big[     
    | X^{t,x_i,i,N}_{u,x_j,x_k}|^{p-1} \cdot \Big(    \frac1N|X^{t,x_j,j,N}_{u,x_j,x_k}| +\frac1N
    |X^{t,x_k,k,N}_{u,x_j,x_k}|
       \Big) \Big]}
    \\ \label{eq: implement youngs inequality on p-1 th moment} 
    &\leq K_V
    \sum_{i,j,k=1,~i\neq j \neq k}^{N}
    \bigg(      \frac{2(p-1)}{p}
    \bE \Big[ 
    | X^{t,x_i,i,N}_{u,x_j,x_k}|^{p }    \Big]
    + \frac{1}{pN^p} \Big(
    \bE \Big[  |X^{t,x_j,j,N}_{u,x_j,x_k}|^p    \Big]
    +
    \bE \Big[  |X^{t,x_k,k,N}_{u,x_j,x_k}|^p    \Big]
    \Big)
     \bigg)
     \\
     \label{eq: second var res4}
    &\leq
    \frac{2K_V(p-1)}{p}
    \sum_{i,j,k=1,~i\neq j \neq k}^{N}
    \bE \Big[ 
    | X^{t,x_i,i,N}_{u,x_j,x_k}|^{p }    \Big]
    + 
    \frac{K_V}{pN^{p-1}} \sum_{j,k=1,j \neq k}^{N}\Big(
    \bE \Big[  |X^{t,x_j,j,N}_{u,x_j,x_k}|^p    \Big]
    +
    \bE \Big[  |X^{t,x_k,k,N}_{u,x_j,x_k}|^p    \Big]
    \Big).
\end{align} 
This is once again established via splitting up the sum over $l \in \lbrace 1,\ldots,N \rbrace$ and using the \textit{symmetrization trick} to obtain \eqref{eq: summation over l split up terms} and Assumption~\ref{assum:main} to obtain \eqref{eq: convexity estimate combine terms}. Young's inequality yields \eqref{eq: implement youngs inequality on p-1 th moment} before a final rearrangement of terms provides the final inequality \eqref{eq: second var res4}.\\

\textit{Part 3: Analysis of the lower order variation term \eqref{eq: o2 lower order var term}.} Lemma~\ref{lemma: first variation bound noiid} with $\lambda_1 \in(\lambda_2,\lambda)$ 
implies that there exists $K>0$ such that for all $i,j\in \lbrace 1, \ldots, N\rbrace,i\neq j,~s\geq t \geq 0 $,  
\begin{align}
    \label{eq: proof of second var iid aux for first}
      \mathbb{E}[|X^{t,x_i,i,N}_{s,x_i}|^p]
     &\leq 
    \sum_{k=1}^{N}  \mathbb{E}[|X^{t,x_k,k,N}_{s,x_i}|^p] \leq  K e^{- \lambda_1 p (s-t)}  ,
    \\
      \quad  
       \mathbb{E}[|X^{t,x_i,i,N}_{s,x_j}|^p]
       &=
       \frac{1}{N-1}
       \sum_{k=1,~k\neq j}^{N}  \mathbb{E}[|X^{t,x_k,k,N}_{s,x_j}|^p] \leq  \frac{K}{N^{p}} e^{-  \lambda_1 p  (s-t)}, \notag
\end{align} 
where for the equality in the second line we used that the starting points $x_i$ are identically distributed.
{\color{black}For the  lower-order variation terms \eqref{eq: o2 lower order var term}, using Hölder's inequality, we have that for all $i,j,k\in \lbrace 1, \ldots, N\rbrace $ and $u\geq t \geq 0 $, 
\begin{align}
\nonumber
 e^{p \lambda_4 (u-t)}
 &\mathbb{E} \Big[ 
 |X^{t,x_i,i,N}_{u,x_j,x_k}|^{p-1}\sum_{l=1}^{N} \sum_{l'=1}^{N}
 |\partial_{x_l,x_{l'}}^2 B_i(\bodX_{u}^{t,\bodx,N})|
 ~| X^{t,x_l,l,N}_{u,x_j}| ~| X^{t,x_{l'},l',N}_{u,x_k}|    \Big]
 \\\nonumber
 &\leq e^{p \lambda_4 (u-t)}
         \sum_{l=1}^{N} \sum_{l'=1}^{N} 
     |\partial_{x_l,x_{l'}}^2 B_i|_\infty    \left(\mathbb{E}\left[|X^{t,x_i,i,N}_{u,x_j,x_k}|^{p} \right]\right)^{(p-1)/p}
    \cdot
     \left(\mathbb{E}\left[| X^{t,x_l,l,N}_{u,x_j}  |^p \ |X^{t,x_{l'},l',N}_{u,x_k}|^p \right] \right)^{1/p}    
     \\
    \label{eq: e4 E I ijk}
 &  =: e^{p \lambda_4 (u-t)}                    \left(\mathbb{E}\left[|X^{t,x_i,i,N}_{u,x_j,x_k}|^{p} \right]\right)^{(p-1)/p} 
      \cdot
      I^{1,i,j,k}_{t,u}.
\end{align}}
For all $i,j,k\in \lbrace 1, \ldots, N\rbrace,~u\geq t \geq 0 $, we defined the following:   
\begin{align}
    \nonumber I^{1,i,j,k}_{t,u}
    &:=\sum_{l=1}^{N} \sum_{l'=1}^{N} 
     |\partial_{x_l,x_{l'}}^2 B_i|_\infty
     \left(\mathbb{E}\left[| X^{t,x_l,l,N}_{u,x_j}  |^p \ |X^{t,x_{l'},l',N}_{u,x_k}|^p \right] \right)^{1/p}
     \\
   \nonumber \leq & 
     \sum_{l=1}^{N} \sum_{l'=1}^{N} 
     |\partial_{x_l,x_{l'}}^2 B_i|_\infty
     \left(\mathbb{E}\left[| X^{t,x_l,l,N}_{u,x_j}  |^{2p} \right]  \ 
     \mathbb{E}\left[|X^{t,x_{l'},l',N}_{u,x_k}|^{2p} \right] \right)^{1/2p}
     \\\nonumber
     =& |\partial_{x_i,x_{i}}^2  B_i|_\infty
     \left(\mathbb{E}\left[| X^{t,x_i,i,N}_{u,x_j}  |^{2p} \right]  \ 
     \mathbb{E}\left[|X^{t,x_{i},i,N}_{u,x_k}|^{2p} \right] \right)^{1/2p}
     \\ \nonumber
     \quad&
     +
     \sum_{l'=1, l'\neq i}^{N}
     |\partial^2_{x_i,x_{l'}}  B_i|_\infty
     \left(\mathbb{E}\left[| X^{t,x_i,N}_{u,x_j}  |^{2p} \right]  \ 
     \mathbb{E}\left[|X^{t,x_{l'},l',N}_{u,x_k}|^{2p} \right] \right)^{1/2p}
     \\ \nonumber
     \quad&
     +
     \sum_{l=1, l\neq i}^{N}
     |\partial^2_{x_l,x_{i}}  B_i|_\infty
     \left(\mathbb{E}\left[| X^{t,x_l,N}_{u,x_j}  |^{2p} \right]  \ 
     \mathbb{E}\left[|X^{t,x_{i},i,N}_{u,x_k}|^{2p} \right] \right)^{1/2p}
     \\\nonumber
     \quad& 
     + 
      \sum_{l=1,~i\neq l}^{N}  
      \sum_{l'=1,~i\neq l'}^{N}
     |\partial^2_{x_l,x_{l'}}  B_i|_\infty
     \left(\mathbb{E}\left[| X^{t,x_l,l,N}_{u,x_j}  |^{2p} \right]  \ 
     \mathbb{E}\left[|X^{t,x_{l'},l',N}_{u,x_k}|^{2p} \right] \right)^{1/2p},
\end{align}
where we implemented the Cauchy--Schwarz inequality in the first inequality. Now using Assumption~\ref{assum:main_weak error},
\begin{align*}
    |  \partial^2_{x_l,x_{l'}} B_i|_\infty = 
    \begin{cases}
         \mathcal{O}(1), \quad &i = l = l', \\
         \mathcal{O}(N^{-1}), \quad &i = l  \neq l' ~\textit{or} ~ i = l'  \neq l~\textit{or} ~ i \neq l  = l',\\
         0,  
         \quad &i \neq l \neq l',
    \end{cases}
\end{align*}
and the results in \eqref{eq: proof of second var iid aux for first} show that there exists $K>0$ such that 
\begin{align}
 \label{eq: second var res5}
    I^{1,i,j,k}_{t,u} \leq 
    K e^{- 2 \lambda_1 (u-t)} \cdot
    \begin{cases}
         1, \quad &i = j = k,\\
          N^{-1} , \quad &i = j  \neq k ~\textit{or} ~ i = k \neq j~\textit{or} ~ i \neq j  = k,\\
          N^{-2} , \quad &i \neq j \neq k.
    \end{cases}
\end{align}
Having established this general estimate for $ I^{1,i,j,k}_{t,u}$ in \eqref{eq: e4 E I ijk},  we distinguish the following scenarios: \\ \\
\noindent 
\medskip
\textit{Case 1: $i=j=k$.} We have 
\begin{align}
     \sum_{i=1}^{N} e^{p \lambda_4 (u-t)}         
     &    
     \left(\mathbb{E} \left[|X^{t,x_i,i,N}_{u,x_i,x_i}|^{p} \right]\right)^{(p-1)/p} 
      \cdot
      I^{1,i,i,i}_{t,u}
    \nonumber
     \\
     \nonumber
     &\leq 
     K e^{p \lambda_4 (u-t)} \sum_{i=1}^{N} 
     \left(\mathbb{E}\left[|X^{t,x_i,i,N}_{u,x_i,x_i}|^{p} \right]\right)^{(p-1)/p}   \cdot  e^{-  2\lambda_1  (u-t)} 
     \\
     \nonumber
     &= 
     K e^{ (\lambda_4- 2\lambda_1 ) (u-t)}  \cdot \sum_{i=1}^{N} 
\left(\mathbb{E}\left[|X^{t,x_i,i,N}_{u,x_i,x_i}|^{p} \right]\right)^{(p-1)/p} e^{(p-1) \lambda_4 (u-t)}     
      \\
      \label{eq: second var res6}
     &\leq
     K e^{ (\lambda_4- 2\lambda_1 ) (u-t)}  \cdot
     \Big( N +  e^{p \lambda_4 (u-t)}\sum_{i=1}^{N}
     \mathbb{E}\left[|X^{t,x_i,i,N}_{u,x_i,x_i}|^{p} \right] 
     \Big),
\end{align}
where we deploy \eqref{eq: second var res5} to obtain the first inequality. 
Using Young's inequality $ab\leq a^{q_1}/q_1 + b^{q_2}/q_2 $ with $a =  \left(\mathbb{E}\left[|X^{t,x_i,i,N}_{u,x_i,x_i}|^{p} \right]\right)^{(p-1)/p} e^{(p-1) \lambda_4 (u-t)}$, $b = 1$, $q_1=p/(p-1)$, $q_2 = p$ yields \eqref{eq: second var res6}.
\\ \\
\textit{Case 2: $i\neq j =  k$.} We have 
\begin{align}
    \sum_{i,k=1,~i\neq k}^{N}     e^{p \lambda_4 (u-t)}            &\left(\mathbb{E}\left[|X^{t,x_i,i,N}_{u,x_k,x_k}|^{p} \right]\right)^{(p-1)/p} 
      \cdot
      I^{1,i,k,k}_{t,u}
\nonumber
     \\
     \nonumber
     &\leq 
     e^{p \lambda_4 (u-t)} \sum_{i,k=1,~i\neq k}^{N}
     \left(\mathbb{E}\left[|X^{t,x_i,i,N}_{u,x_k,x_k}|^{p} \right]\right)^{(p-1)/p}   \cdot \frac{K}{N} e^{- 2\lambda_1 (u-t)} 
      \\
      \label{eq: second var res7}
     &\leq
     K e^{ (\lambda_4- 2\lambda_1) (u-t)}  \cdot
     \Big( \frac{1}{N^{p-2}} +  e^{p \lambda_4 (u-t)}\sum_{i,k=1,~i\neq k}^{N}
     \mathbb{E}\left[|X^{t,x_i,i,N}_{u,x_k,x_k}|^{p} \right]
     \Big),
\end{align}
where we used Young's inequality $ab\leq a^{q_1}/q_1 + b^{q_2}/q_2 $ with $a =  \left(\mathbb{E}\left[|X^{t,x_i,i,N}_{u,x_k,x_k}|^{p} \right]\right)^{(p-1)/p} e^{(p-1) \lambda_1 (u-t)}$, $b = 1/N$, $q_1=p/(p-1)$, $q_2 = p$ to derive \eqref{eq: second var res7}.

Similar calculations apply to the cases \textbf{$i = j \neq  k$} and { \textbf{$i= k \neq  j$}}: 
\begin{align*}
    \sum_{i,k=1,~i\neq k}^{N}     e^{p \lambda_4 (u-t)}           
    &
    \left(\mathbb{E}\left[|X^{t,x_i,i,N}_{u,x_i,x_k}|^{p} \right]\right)^{(p-1)/p} 
      \cdot
      I^{1,i,i,k}_{t,u}
      \\
      \nonumber
     &\leq
     K e^{ (\lambda_4- 2\lambda_1) (u-t)}  \cdot
     \Big( \frac{1}{N^{p-2}} +  e^{p \lambda_4 (u-t)}\sum_{i,k=1,~i\neq k}^{N}
     \mathbb{E}\left[|X^{t,x_i,i,N}_{u,x_i,x_k}|^{p} \right] 
     \Big),
     \\
     \sum_{i,j=1,~i\neq j}^{N}  e^{p \lambda_4 (u-t)}            
     &
     \left(\mathbb{E}\left[|X^{t,x_i,i,N}_{u,x_j,x_j}|^{p} \right]\right)^{(p-1)/p} 
      \cdot
      I^{1,i,j,i}_{t,u}
      \\
     \nonumber
     &\leq
     K e^{ (\lambda_4- 2\lambda_1) (u-t)}  \cdot
     \Big( \frac{1}{N^{p-2}} +  e^{p \lambda_4 (u-t)}\sum_{i,k=j,~i\neq j}^{N}
     \mathbb{E}\left[|X^{t,x_i,i,N}_{u,x_j,x_i}|^{p} \right] 
     \Big).
\end{align*}
\textit{Case 3: $i\neq j \neq k  $.} We have 
\begin{align}
     \sum_{i,j,k=1,~i\neq j \neq k}^{N} e^{p \lambda_4 (u-t)}            
     &
     \left(\mathbb{E}\left[|X^{t,x_i,i,N}_{u,x_j,x_k}|^{p} \right]\right)^{(p-1)/p} 
      \cdot
      I^{1,i,j,k}_{t,u}
\nonumber
     \\\nonumber
     &\leq 
     e^{p \lambda_4 (u-t)} \sum_{i,j,k=1,~i\neq j \neq k}^{N}
     \left(\mathbb{E}\left[|X^{t,x_i,i,N}_{u,x_j,x_k}|^{p} \right]\right)^{(p-1)/p}   \cdot  \frac{K}{N^2} e^{- 2\lambda_1 (u-t)}  
      \\
      \label{eq: second var res10}
     &\leq
     K e^{ (\lambda_4- 2\lambda_1) (u-t)}  \cdot
     \Big( \frac{1}{N^{2p-3}} +  e^{p \lambda_4 (u-t)}\sum_{i,j,k=1,~i\neq j \neq k}^{N}
     \mathbb{E}\left[|X^{t,x_i,i,N}_{u,x_j,x_k}|^{p} \right] 
     \Big),
\end{align}
where we used Young's inequality $ab\leq a^{q_1}/q_1 + b^{q_2}/q_2 $ with $a =   (\mathbb{E} [|X^{t,x_i,i,N}_{u,x_j,x_k}|^{p}  ] )^{(p-1)/p} e^{(p-1) \lambda_4 (u-t)}$, $b = 1/N^2 $, $q_1=p/(p-1)$, $q_2 = p$ to derive \eqref{eq: second var res10}.\\

\textit{Part 4: Collecting estimates and conclusion.} {\color{black}Gathering up all the results in \eqref{eq: second var res1}--\eqref{eq: second var res10} and recalling the definition of $ I_{t,s}^{2,p}$ in \eqref{eq: second var Ipst}, we have for all $\lambda_4 \in (0,\min \{  \lambda-(2+1/N)K_V,\lambda_3 \})$,  
\begin{align*}
     I_{t,s}^{2,p}
    &\leq 
     p 
     \int_{t}^{s} \Big(-\lambda+\lambda_4
     +
     \frac{2K_V(p-1)}{p}
     + K e^{ (\lambda_4- 2\lambda_1) (u-t)}
     \Big) 
   I_{t,u}^p 
    \dd u
     +
      pK  
      \int_{t}^{s}   e^{ (\lambda_4- 2\lambda_1) (u-t)}
    \dd u
    \\
    &\quad 
    +
      p  \int_{t}^{s} \frac{K_V}{N^2p}  
    \sum_{i,k=1,~i\neq k}^{N}
      \bE \Big[ 
   |X^{t,x_i,i,N}_{u,x_k,x_k}|^{p}  
      \Big] \dd u
    +
      p  \int_{t}^{s} \frac{K_V}{N  p}   
    \sum_{i=1}^{N}
      \bE \Big[ 
   |X^{t,x_i,i,N}_{u,x_i,x_i}|^{p}  
      \Big] \dd u
      \\
    &\quad 
    +
      p  \int_{t}^{s} \bigg( \frac{2K_VN^{ p-3}}{p}  
    \sum_{i,j,k=1,~i\neq j \neq  k}^{N}
      \bE \Big[ 
   |X^{t,x_i,i,N}_{u,x_j,x_k}|^{p}  
      \Big] 
   \\  
   & \hspace{3cm}
      +
      K_V N^{ p-3}  
    \sum_{i,k=1,~i \neq  k}^{N} \Big( 
      \bE \Big[ 
   |X^{t,x_i,i,N}_{u,x_i,x_k}|^{p}  
      \Big] 
      + \bE \Big[    |X^{t,x_i,i,N}_{u,x_k,x_i}|^{p}        \Big]
      \Big)~\bigg)~\dd u
       \\
    &\quad 
    +
    p  \int_{t}^{s} \frac{K_V}{p} N^{p-2}
    \sum_{i,k=1,~i \neq k}^{N}\Big(
    \bE \Big[  |X^{t,x_i,i,N}_{u,x_i,x_k}|^p    \Big]
    +
    \bE \Big[  |X^{t,x_k,k,N}_{u,x_i,x_k}|^p    \Big]
    \Big) \dd u.
    \end{align*}
Combining the terms above, we further have 
    \begin{align*}
     I_{t,s}^{2,p} &\leq 
     p 
     \int_{t}^{s} \Big(-\lambda+\lambda_4
     +
     \frac{2K_V(p-1)}{p}
     + K e^{ (\lambda_4- 2\lambda_1) (u-t)}
     \Big) 
   I_{t,u}^p 
    \dd u
     +
      pK  
      \int_{t}^{s}   e^{ (\lambda_4- 2\lambda_1) (u-t)}
    \dd u
     \\
    & \quad + p  \int_{t}^{s} \frac{K_V}{  p}   
    \bigg( \frac{1}{N}\sum_{i=1}^{N}
      \bE \Big[ 
   |X^{t,x_i,i,N}_{u,x_i,x_i}|^{p}  
      \Big] 
      +
       N^{p-2}
    \sum_{i,k=1,~i \neq k}^{N}\Big(
    \bE \Big[  |X^{t,x_i,i,N}_{u,x_i,x_k}|^p    \Big]
    +
    \bE \Big[  |X^{t,x_k,k,N}_{u,x_i,x_k}|^p    \Big]
    \Big)
      \bigg) \dd u
    \\
    & \quad+ p  \int_{t}^{s} \frac{K_V}{  N }
     \bigg(
     N^{p-2}
     \sum_{i,k=1,~i \neq  k}^{N} \Big( 
      \bE \Big[ 
   |X^{t,x_i,i,N}_{u,x_k,x_k}|^{p}  
      \Big] 
      +
      \bE \Big[ 
   |X^{t,x_i,i,N}_{u,x_i,x_k}|^{p}  
      \Big] + \bE \Big[    |X^{t,x_i,i,N}_{u,x_k,x_i}|^{p}        \Big]
      \Big)
       \\
    & \quad \qquad  \qquad \qquad \qquad \qquad + \frac{2}{pN^{p-1}}N^{2p-3}
       \sum_{i,j,k=1,~i\neq j \neq  k}^{N}
      \bE \Big[ 
   |X^{t,x_i,i,N}_{u,x_j,x_k}|^{p}  
      \Big] 
     \bigg) \dd u.
    \end{align*}
Hence, using $I^{2,p}_{t,s}$ in \eqref{eq: second var Ipst} to dominate the sums of expectations (in each integral) we have 
\begin{align}    
    \label{eq: second-var I2pst result}
     I_{t,s}^{2,p}&\leq 
     p 
     \int_{t}^{s} \Big(-\lambda+\lambda_4
     + 2K_V
     +\frac{K_V}{N } 
     + K e^{ (\lambda_4- 2\lambda_1) (u-t)}
     \Big) 
   I_{t,u}^{2,p} 
    \dd u
     +
      pK  
      \int_{t}^{s}   e^{ (\lambda_4- 2\lambda_1) (u-t)}
    \dd u. 
\end{align}

Recalling that is $\lambda_4$ chosen such that $\lambda_4 \in (0,\min \{   \lambda-(2+1/N)K_V,\lambda_3\} )$  and $0<\lambda_3<\lambda_1 <\lambda$, which in turn implies  $\lambda_4- 2\lambda_1<0$,  
we have 
\begin{align*}
    I_{t,s}^{2,p}& \leq pK\int_{t}^{s}  
       e^{ (\lambda_4- 2\lambda_1) (u-t)}
   I_{t,u}^{2,p} 
    \dd u
     +
      pK  
      \int_{t}^{s}   e^{ (\lambda_4- 2\lambda_1) (u-t)}
    \dd u \leq K,
\end{align*}
where for the last inequality, we used Gronwall's inequality (see Lemma~\ref{lemma: appendix: gronwall}) with 
\begin{align*}
    \alpha(s) = pK e^{(\lambda_4- 2\lambda_1)(s-t)}
    , \quad
   \beta(s) = pK  \int_{t}^{s} e^{(\lambda_4- 2\lambda_1)(u-t)}    \mathrm{d}u \leq  { pK } / ({2\lambda_1- \lambda_4})
    \quad \textrm{and} 
\end{align*}
$u(s) = I_{t,s}^{2,p}$ in combination with the estimate $\int_t^u \alpha(s) \mathrm{d}s \leq   { pK }/ ({2\lambda_1- \lambda_4})$. 
}

\end{proof}

\allowdisplaybreaks
\subsection[The higher-order variation process]{The n-Variation process}
Let $T \geq s\geq t\geq 0$, $N\in \mathbb{N}$ and $1<n \leq 6$ be an integer.
Recalling the quantity $\Pi_n^N$ in Definition \ref{def: set of sequence pi},  the $n$-variation process of $  ( \boldsymbol{X}^{t,\boldsymbol{x},N}_s )_{ s \geq t \geq 0}$ is given for $  i\in \lbrace 1, \ldots, N \rbrace$, $s \geq t$, $\gamma\in \Pi_n^N$, by
\begin{align}
\label{n_var_process noiid}
X^{t,x_i,i,N}_{s,x_{\gamma_1},\dots,x_{\gamma_n}} 
& 
= 
\int_{t}^{s}  \Big( \sum_{l=1}^{N} \partial_{x_l}B_i(\bodX_{u}^{t,\bodx,N})  X^{t,x_l,l,N}_{u,x_{\gamma_1}} \Big)_{ x_{\gamma_2},\dots,x_{\gamma_n}   }
\mathrm{d}u
\\
\nonumber 
&
= \int_{t}^{s} \sum_{l=1}^{N} \partial_{x_l}B_{i}(\bodX_{u}^{t,\bodx,N})  X^{t,x_l,l,N}_{u,x_{\gamma_1},\dots,x_{\gamma_{n}}}\mathrm{d}u
\\
\nonumber 
&
\qquad   +  
    \sum_{  
    \substack{  \alpha,\beta \in \bigcup_{k=0}^{n-1} \Pi_k^N
,\\  |\alpha|>0,~
\gamma \setminus (\gamma_1) \in \alpha\shuffle\beta 
}
     }
    \int_{t}^{s}  
    \sum_{l=1}^{N} \Big( \partial_{x_l}B_{ i}(\bodX_{u}^{t,\bodx,N})  
    \Big)_{x_{\alpha_1},\dots, x_{\alpha_{|\alpha|}}   }
    \Big( 
    X^{t,x_l,l,N}_{u,x_{\gamma_1}} \Big)_{ x_{\beta_1},\dots, x_{\beta_{|\beta|}} }
    \mathrm{d}u, 
\end{align} 
where $  \alpha = (\alpha_1,\dots, \alpha_{|\alpha|}), ~ \beta = (\beta_1,\dots, \beta_{|\beta|})$ with $  \alpha_i \in \{1,\dots,N\}$, and $ \beta_i \in \{1,\dots,N\}$ for $i\in \{1,\dots, |\beta|\}$. {\color{black}To clarify the notation, we present the following examples for $n\in \{2,3\}$: 
\begin{align*}
     \Big \{& \alpha,\beta \in \bigcup_{k=0}^{1} \Pi_k^N
 :  |\alpha|>0,~
 \gamma \setminus (\gamma_1) \in \alpha\shuffle\beta  
     \Big \}
     = 
     \Big \{    \big\{ \alpha=(\gamma_2), \beta = \emptyset\big\}     \Big \},
    \end{align*}
     when $\gamma = (\gamma_1,\gamma_2)$; for the case $n=3$ and setting $\gamma = (\gamma_1,\gamma_2,\gamma_3)$
     \begin{align*}
     \Big \{& \alpha,\beta \in \bigcup_{k=0}^{2} \Pi_k^N
 :  |\alpha|>0,~
 \gamma \setminus (\gamma_1) \in \alpha\shuffle\beta  
     \Big \}
     \\
     &\qquad = 
     \Big \{    \big\{ \alpha=(\gamma_2,\gamma_3), \beta = \emptyset\big\},   
     \big\{ \alpha=(\gamma_2), \beta = (\gamma_3)\big\},
     \big\{ \alpha=(\gamma_3), \beta = (\gamma_2)\big\}
     \Big \} .
\end{align*}

}	 

The following lemma is a generalization of Lemma~\ref{lemma: first variation bound noiid} and Lemma~\ref{lemma: second variation bound noiid details}, which describes the behaviour of the higher-order variation processes and is needed in the proofs in Section \ref{section: analysis of the dus} and Section \ref{section: weak error expansion}. 
\begin{lemma} 
    \label{lemma: n-var process result}
    Let the assumptions of Lemma~\ref{lemma: second variation bound noiid details} hold and let $p \geq 2$ be given.
    Consider the $n$-variation process with components
    $( X^{t,x_{\gamma_1},\gamma_1,N}_{s,x_{\gamma_2},\dots,x_{\gamma_{n+1}}}  )_{ s \geq t \geq 0}$  
    defined by \eqref{n_var_process noiid} for $T \geq s \geq t\geq 0$, $n,N\in \mathbb{N}$, $\gamma  \in \Pi_{n+1}^N$, $1 \leq n \leq 6$. Then for each $1 \leq n \leq 6$, there exist constants $\lambda_0^{(n)} \in (0,\lambda)$ and 
     {\color{black}    $K_1,K_2 > 0$			 } (all independent of $s,t,T$ and $N$) 
    such that for any $m \in \{1,\ldots, n+1\}$, we have 
    \begin{align*}
         \sum_{\gamma\in \Pi_{n+1}^N,~\hco(\gamma) = m}
    \mathbb{E}\Big[|X^{t,x_{\gamma_1},\gamma_1,N}_{s,x_{\gamma_2},\dots,x_{\gamma_{n+1}}}|^p \Big]
        \leq 
        \frac{K_1}{N^{p(m-1)-m}} e^{-\lambda_0^{(n)} p(s-t)}.
    \end{align*}
      In particular, this implies that, for all $\gamma\in \Pi_{n+1}^N$, such that $\hco(\gamma) = m,~m\in\{1,\dots,n+1\}$, we have 
     \begin{align*}
    \mathbb{E}\Big[|X^{t,x_{\gamma_1},\gamma_1,N}_{s,x_{\gamma_2},\dots,x_{\gamma_{n+1}}}|^p \Big]
        \leq 
        \frac{K_2}{N^{p(m-1)}} e^{-\lambda_0^{(n)} p(s-t)}.
    \end{align*}
\end{lemma} 

\begin{remark}
    We take $ \gamma  \in \Pi_{n+1}^N$ here in Lemma~\ref{lemma: n-var process result}: this still corresponds to the highest order of the derivatives remaining as 6. We add the extra index since $\gamma_1$ corresponds to the index of the starting position, not an index which corresponds to any derivatives being taken.
\end{remark}

\begin{remark}
    As in Lemma~\ref{lemma: first variation bound noiid}, where the constant appearing in the bound of the second variation process $\lambda_4$ was strictly less than $\lambda_1$ (the constant which arose bounding the first variation process), we choose $\lambda_0^{(n)}$ such that the sequence $\lambda_0^{(i)}$, $i \in \lbrace 1, \ldots, 6 \rbrace$ is a strictly decreasing sequence. Note that here, our $\lambda_0^{(1)}$ and $\lambda_0^{(2)}$ subsumes our $\lambda_1$ and $\lambda_4$ in the previous proofs respectively.
\end{remark}

\color{black}

\begin{remark}
    \label{remark: gamma seq ineq implies eq}
     For any  $\gamma\in \Pi_{n+1}^N$ with $n\geq 1, ~\hco(\gamma) = m,~m\in\{1,\dots,n+1\}$, there are $\cO(N^m)$ sequences of indices that have the same pattern as $\gamma$. This means, for the processes $( X^{t,x_{\gamma_1},\gamma_1,N}_{s,x_{\gamma_2},\dots,x_{\gamma_{n+1}}}  ) $ with given $s,t$ in Lemma \ref{lemma: n-var process result}, 
     there are  $\cO(N^m)$ identical processes in the summation of the first result in Lemma \ref{lemma: n-var process result}, which implies the second inequality. For example, consider $X^{t,x_{1},1,N}_{s,x_{2},x_{3}}$: There are $N(N-1)(N-2)$ random variables that are distributed identically to $X^{t,x_{1},1,N}_{s,x_{2},x_{3}}$ such as $X^{t,x_{1},1,N}_{s,x_{3},x_{4}},\dots$. Thus, we derive the second inequality by dividing $N^3$ and adjust the constant $K$.  
\end{remark}

\color{black}

\begin{proof} 
    Note that in the following proof, the positive constant $K$ is independent of $s,t,T,N$ and may change line by line.

    \textit{Part 1: Preliminary manipulations.} 
    We shall prove this by induction.  
    The result follows for $n =1$ and $n = 2$, from Lemma~\ref{lemma: first variation bound noiid} and Lemma~\ref{lemma: second variation bound noiid details} respectively.    

    Now, suppose that the claim holds for some $n_1\in \mathbb{N},~n_1 < 6$, i.e., there exists a sufficiently small constant  $\lambda_0^{(n_1)} \in (0,\lambda)$, such that for $m \in\{1,\ldots,n_1+1\}$,
    \begin{align}
    \label{eq: aux induction hypothesis}
         \sum_{\gamma\in \Pi_{n_1+1}^N,~\hco(\gamma) = m}
    \mathbb{E}\Big[|X^{t,x_{\gamma_1},\gamma_1,N}_{s,x_{\gamma_2},\dots,x_{\gamma_{n_1+1}}}|^p \Big]
        \leq 
        \frac{K}{N^{p(m-1)-m}} e^{-\lambda_0^{(n_1)} p(s-t)}.
    \end{align}
    We  need to prove the statement for $ m \in \{1,\ldots,n_1+2\}$, i.e., there exists some constant $\lambda_0^{(n_1+1)} \in (0,\lambda^{(n_1)}) $ such that
    \begin{align*}
         \sum_{\gamma\in \Pi_{n_1+2}^N,~\hco(\gamma) = m}
    \mathbb{E}\Big[|X^{t,x_{\gamma_1},\gamma_1,N}_{s,x_{\gamma_2},\dots,x_{\gamma_{n_1+2}}}|^p \Big]
        \leq 
        \frac{K}{N^{p(m-1)-m}} e^{-\lambda_0^{(n_1+1)} p(s-t)}.
    \end{align*}
    \color{black}
    We shall use the induction hypothesis \eqref{eq: aux induction hypothesis} in \textit{Part 2.2}. 
    
    From \eqref{n_var_process noiid}, we have for all $ \gamma\in \Pi_{n_1+2}^N$
    \begin{align*} 
   & X^{t,x_{\gamma_1},\gamma_1,N}_{s,x_{\gamma_2},\dots,x_{\gamma_{n_1+2}}} 
   \\ 
    & \qquad = 
    \int_{t}^{s}  \Big( \sum_{l=1}^{N} \partial_{x_l}B_{\gamma_1}(\bodX_{u}^{t,\bodx,N})  X^{t,x_l,l,N}_{u,x_{\gamma_2}} \Big)_{ x_{\gamma_3},\dots,x_{\gamma_{n_1+2}}   }
    \mathrm{d}u
    \\
    & \qquad =
    \int_{t}^{s} \sum_{l=1}^{N} \partial_{x_l}B_{\gamma_1}(\bodX_{u}^{t,\bodx,N})  X^{t,x_l,l,N}_{u,x_{\gamma_2},\dots,x_{\gamma_{n_1+2}}}\mathrm{d}u
    \\
    & \qquad \qquad +  
    \sum_{  \substack{ \alpha,\beta \in 
    \bigcup_{k=0}^{n_1} \Pi_k^N,
    \\ |\alpha|>0 ,~
\gamma \setminus (\gamma_1,\gamma_2) \in \alpha\shuffle\beta
    } }
    \int_{t}^{s}  
    \sum_{l=1}^{N} \Big( \partial_{x_l}B_{\gamma_1}(\bodX_{u}^{t,\bodx,N})  
    \Big)_{x_{\alpha_1},\dots , x_{\alpha_{|\alpha|}} 		 }
    \Big( 
    X^{t,x_l,l,N}_{u,x_{\gamma_2}} \Big)_{ x_{\beta_1},\dots , x_{\beta_{|\beta|}} }
    \mathrm{d}u.
    \end{align*}
An application of It\^{o}'s formula yields for all $s \geq t$, $\lambda_0^{(n_1+1)} >0$, 
\allowdisplaybreaks
    \begin{align}
    \nonumber
    e&^{p \lambda_0^{(n_1+1)} (s-t)} \mathbb{E}\Big[   |X^{t,x_{\gamma_1},{\gamma_1},N}_{s,x_{\gamma_2},\dots,x_{\gamma_{n_1+2}}} |^p\Big]
    \\ \nonumber
    & =p \int_{t}^{s} e^{p \lambda_0^{(n_1+1)}  (u-t)} \mathbb{E}
\bigg[|X^{t,x_{\gamma_1},{\gamma_1},N}_{u,x_{\gamma_2},\dots,x_{\gamma_{n_1+2}}}|^{p-2} 
\Big(    X^{t,x_{\gamma_1},{\gamma_1},N}_{u,x_{\gamma_2},\dots,x_{\gamma_{n_1+2}}}
 \Big)   \cdot	\Big(  
    \sum_{l=1}^{N} \partial_{x_l}B_{\gamma_1}(\bodX_{u}^{t,\bodx,N})  X^{t,x_l,l,N}_{u,x_{\gamma_2},\dots,x_{\gamma_{n_1+2}}}  
    \\\nonumber
    & \qquad +  
    \sum_{  \substack{ \alpha,\beta \in
    \bigcup_{k=0}^{n_1} \Pi_k^N,
    \\ |\alpha|>0,~
\gamma \setminus (\gamma_1,\gamma_2) \in \alpha\shuffle\beta  }}
    \sum_{l=1}^{N} 
    \Big( \partial_{x_l}B_{\gamma_1}(\bodX_{u}^{t,\bodx,N})  
    \Big)_{x_{\alpha_1},\dots , x_{\alpha_{|\alpha|}} 		 }
    \Big( 
    X^{t,x_l,l,N}_{u,x_{\gamma_2}} \Big)_{ x_{\beta_1},\dots, x_{\beta_{|\beta|}}  }
     \Big)
    \bigg]
    \mathrm{d}u
    \\\nonumber
    & \qquad + p\lambda_0^{(n_1+1)} \int_{t}^{s}   e^{p \lambda_0^{(n_1+1)}  (u-t)} \mathbb{E}\Big[   | X^{t,x_{\gamma_1},{\gamma_1},N}_{u,x_{\gamma_2},\dots,x_{\gamma_{n_1+2}}} |^p\Big] \mathrm{d}u
    \\ \label{eq: aux first integral term lambda0-lambda}
    &\leq 
    p\int_{t}^{s} (\lambda_0^{(n_1+1)} -\lambda)  e^{p \lambda_0^{(n_1+1)}  (u-t)} \mathbb{E}\Big[   | X^{t,x_{\gamma_1},{\gamma_1},N}_{u,x_{\gamma_2},\dots,x_{\gamma_{n_1+2}}} |^p\Big] \mathrm{d}u
    \\ \nonumber 
    & \qquad 
    -
    p \int_{t}^{s} e^{p \lambda_0^{(n_1+1)}  (u-t)} \mathbb{E}
\bigg[|X^{t,x_{\gamma_1},{\gamma_1},N}_{u,x_{\gamma_2},\dots,x_{\gamma_{n_1+2}}}|^{p-2} 
 \Big( X^{t,x_{\gamma_1},{\gamma_1},N}_{u,x_{\gamma_2},\dots,x_{\gamma_{n_1+2}}}     \Big)   
    \\ 
    \label{eq: n-var - convolution }
    &\qquad \qquad   \cdot	\bigg(   
    \frac{1}{N}\sum_{l=1}^{N} \nabla^{2} V(X^{t,x_{\gamma_1},{\gamma_1},N}_u-X^{t,x_l,l,N}_u)
    \
    \big( X^{t,x_{\gamma_1},{\gamma_1},N}_{u,x_{\gamma_2},\dots,x_{\gamma_{n_1+2}}}
    - X^{t,x_l,l,N}_{u,x_{\gamma_2},\dots,x_{\gamma_{n_1+2}}}
    \big)\bigg)\bigg] \dd u
\\ \notag
&
\qquad-
    p \int_{t}^{s} e^{p \lambda_0^{(n_1+1)}  (u-t)} \mathbb{E}
\bigg[|X^{t,x_{\gamma_1},{\gamma_1},N}_{u,x_{\gamma_2},\dots,x_{\gamma_{n_1+2}}}|^{p-2} 
 \Big( X^{t,x_{\gamma_1},{\gamma_1},N}_{u,x_{\gamma_2},\dots,x_{\gamma_{n_1+2}}}     \Big)    
    \\ \label{eq: n-var - lower var }
    &\qquad \qquad
   \cdot \bigg(\sum_{  \substack{ \alpha,\beta \in 
    \bigcup_{k=0}^{n_1} \Pi_k^N,
    \\ |\alpha|>0,~
\gamma \setminus (\gamma_1,\gamma_2) \in \alpha\shuffle\beta  }}
    \sum_{l=1}^{N} 
    \Big( \partial_{x_l}B_{\gamma_1}(\bodX_{u}^{t,\bodx,N})  
    \Big)_{x_{\alpha_1},\dots  , x_{\alpha_{|\alpha|}} 		}
    \Big( 
    X^{t,x_l,l,N}_{u,x_{\gamma_2}} \Big)_{ x_{\beta_1},\dots , x_{\beta_{|\beta|}} }
    \bigg)
    \bigg]
    \mathrm{d}u.
    \end{align}
The analysis of the convolution term \eqref{eq: n-var - convolution }  and the lower order variation term \eqref{eq: n-var - lower var },  mimic that of the proof 
of Lemma~\ref{lemma: second variation bound noiid details} - all arguments are complete generalizations. We present these for clarity.

 For a given $p \geq 2$, we define the process $I_{t,s}^{n_1+1,p}$ (similar to the term $I_{t,s}^{2,p}$ from \eqref{eq: second var Ipst}): 
\begin{align}
    \label{eq: def of itsnpp}
    I_{t,s}^{n_1+1,p}
    := &  
    e^{p\lambda_0^{(n_1+1)} (s-t)} \sum_{m=1}^{n_1+2} \bigg( 
    N^{(m-1)p-m}  \sum_{\gamma\in \Pi_{n_1+2}^N,~\hco(\gamma) = m}
    \mathbb{E}\Big[|X^{t,x_{\gamma_1},\gamma_1,N}_{s,x_{\gamma_2},\dots,x_{\gamma_{n_1+2}}}|^p \Big]
    \bigg)
    \\ \notag
    &
    \leq \sum_{m=1}^{n_1+2} \bigg( 
    N^{(m-1)p-m}  \sum_{\gamma\in \Pi_{n_1+2}^N,~\hco(\gamma) = m} \eqref{eq: aux first integral term lambda0-lambda}+\eqref{eq: n-var - convolution }+\eqref{eq: n-var - lower var }\bigg),
\end{align}
which we now work towards bounding $I_{t,s}^{n_1+1,p} \leq K$, by separating our analysis into that of the convolution term \eqref{eq: n-var - convolution }  and the lower order variation term \eqref{eq: n-var - lower var }.
\\


\textit{Part 2.1: Analysis of the convolution term \eqref{eq: n-var - convolution }.}  
We highlight the main steps that deal specifically with the convolution term \eqref{eq: n-var - convolution } (similar to \eqref{eq: second var res1}--\eqref{eq: second var res4} in the proof steps for $n=2$). Summing over     $\gamma\in \Pi_{n_1+2}^N,~\hco(\gamma) = m$, gives 
\begin{align}
    \nonumber
    -&\sum_{ \gamma\in \Pi_{n_1+2}^N,~\hco(\gamma) = m}   \bE 
    \bigg[ \frac{1}{N} \sum_{l=1}^{N}  \Big(   | X^{t,x_{\gamma_1},\gamma_1,N}_{s,x_{\gamma_2},\dots,x_{\gamma_{n_1+2}}}|^{p-2}X^{t,x_{\gamma_1},\gamma_1,N}_{s,x_{\gamma_2},\dots,x_{\gamma_{n_1+2}}} 
    \Big)   	
    \\
    \nonumber
    &\qquad \qquad  \cdot \Big(   
        \nabla^2 V(X^{t,x_{\gamma_1},\gamma_1,N}_s-X^{t,x_l,l,N}_s)  ( X^{t,x_{\gamma_1},\gamma_1,N}_{s,x_{\gamma_2},\dots,x_{\gamma_{n_1+2}}}- X^{t,x_l,l,N}_{s,x_{\gamma_2},\dots,x_{\gamma_{n_1+2}}})
      \Big)   \bigg]
      \\
          \nonumber
      =-& \bigg( \sum_{ \substack{\gamma\in \Pi_{n_1+2}^N,~\hco(\gamma) = m 
        \\
        \hco(\gamma\setminus  (\gamma_1)) = m  
     }}
        + \sum_{ \substack{\gamma\in \Pi_{n_1+2}^N,~\hco(\gamma) = m 
        \\
        \hco(\gamma\setminus  (\gamma_1)) = m-1  
     }}
     \bigg)
     \bE 
    \bigg[ \frac{1}{N} \sum_{l=1}^{N}  \Big(   | X^{t,x_{\gamma_1},\gamma_1,N}_{s,x_{\gamma_2},\dots,x_{\gamma_{n_1+2}}}|^{p-2}X^{t,x_{\gamma_1},\gamma_1,N}_{s,x_{\gamma_2},\dots,x_{\gamma_{n_1+2}}} 
    \Big)   	
    \\
     \label{eq: n-var convolution 1}
    &\qquad \qquad  \cdot \Big(   
        \nabla^2 V(X^{t,x_{\gamma_1},\gamma_1,N}_s-X^{t,x_l,l,N}_s)  ( X^{t,x_{\gamma_1},\gamma_1,N}_{s,x_{\gamma_2},\dots,x_{\gamma_{n_1+2}}}- X^{t,x_l,l,N}_{s,x_{\gamma_2},\dots,x_{\gamma_{n_1+2}}})
      \Big)   \bigg].
\end{align}
To continue from \eqref{eq: n-var convolution 1}, we need to consider 
whether or not $\gamma_1$ is  a unique index in $\gamma$, i.e., the two sums appearing in \eqref{eq: n-var convolution 1}. 
 \smallskip

\textit{Case 1: $\gamma_1$ is not a unique index in $\gamma$ (i.e., $\hco(\gamma) =\hco(\gamma\setminus  (\gamma_1)) = m $)}. Note that this case is only meaningful for $m \in \lbrace 1, \ldots, n_1 +1 \rbrace$, as $\gamma$ has at most $n_1 + 1$ different elements (note that $\gamma_1$ appears at least twice). Further, we remark that  $\hco(\gamma)=m$ but $ \hco(\gamma \setminus (\gamma_1) \bigcup (l) )\in\{m,m+1\}$, and therefore we need to consider $X^{t,x_{\gamma_1},\gamma_1,N}_{\cdot}$ 
and $X^{t,x_{l},l,N}_{\cdot}$  separately
(in a similar fashion to \eqref{eq: second var res1} and \eqref{eq: second var res3}). Hence, we write  
\begin{align}
\nonumber
     -&\sum_{ \substack{\gamma\in \Pi_{n_1+2}^N,~\hco(\gamma) = m 
        \\
        \hco(\gamma\setminus  (\gamma_1)) = m  
     }}  \bE 
    \bigg[ \frac{1}{N} \sum_{l=1}^{N}  \Big(   | X^{t,x_{\gamma_1},\gamma_1,N}_{s,x_{\gamma_2},\dots,x_{\gamma_{n_1+2}}}|^{p-2}X^{t,x_{\gamma_1},\gamma_1,N}_{s,x_{\gamma_2},\dots,x_{\gamma_{n_1+2}}} 
    \Big)   	
    \\ \nonumber
    &\qquad \qquad  \cdot \Big(   
        \nabla^2 V(X^{t,x_{\gamma_1},\gamma_1,N}_s-X^{t,x_l,l,N}_s)  ( X^{t,x_{\gamma_1},\gamma_1,N}_{s,x_{\gamma_2},\dots,x_{\gamma_{n_1+2}}}- X^{t,x_l,l,N}_{s,x_{\gamma_2},\dots,x_{\gamma_{n_1+2}}})
      \Big)   \bigg] 
      \\ \nonumber
      & =
      -\frac1N \sum_{ \substack{\gamma\in \Pi_{n_1+2}^N,~\hco(\gamma) = m  
        \\ \nonumber
        \hco(\gamma\setminus  (\gamma_1)) = m  
     }}  
     \bigg( 
     \sum_{\substack{     l=1, \\    \hco(\gamma \setminus (\gamma_1) \bigcup (l) )  = m } }^{N}
     +
     \sum_{\substack{     l=1, \\    \hco(\gamma \setminus (\gamma_1) \bigcup (l) )  = m+1} }^{N}
     \bigg)
     \bE 
    \bigg[   \Big(   | X^{t,x_{\gamma_1},\gamma_1,N}_{s,x_{\gamma_2},\dots,x_{\gamma_{n_1+2}}}|^{p-2}X^{t,x_{\gamma_1},\gamma_1,N}_{s,x_{\gamma_2},\dots,x_{\gamma_{n_1+2}}} 
    \Big)   	
    \\ \label{eq: gamma1 not unique hessV term}
    &\qquad \qquad  \cdot \Big(   
        \nabla^2 V(X^{t,x_{\gamma_1},\gamma_1,N}_s-X^{t,x_l,l,N}_s)  ( X^{t,x_{\gamma_1},\gamma_1,N}_{s,x_{\gamma_2},\dots,x_{\gamma_{ n_1+2}}}- X^{t,x_l,l,N}_{s,x_{\gamma_2},\dots,x_{\gamma_{n_1+2}}})
      \Big)   \bigg].
\end{align}

We address each of the inner sums over $l$ in  \eqref{eq: gamma1 not unique hessV term} separately. Concretely: for the first inner sum, where the summation is over $\hco(\gamma \setminus (\gamma_1) \bigcup (l) )  = m$ with $l$ taking values in $\gamma\setminus  (\gamma_1)$ (as per the outer summation), we use the \textit{symmetrization trick} to get \eqref{eq: n-var convolution syme1} and we substitute the upper indices  $\gamma_1 \mapsto l_1$ and $l \mapsto l_2$ for added clarity. For the second inner sum, where the summation is over $\hco(\gamma \setminus (\gamma_1) \bigcup (l) )  = m+1$ with $l\notin \gamma$, we apply Young's inequality to get \eqref{eq: n-var convolution diff1}.   
We point out that similar arguments were employed in the estimation of the second variation process -- refer to  \eqref{eq: second var res1} and \eqref{eq: second var res3} for example. We then have, 
\begin{align}
    \nonumber 
      & \eqref{eq: gamma1 not unique hessV term}
      \\\nonumber 
      &~
      \leq 
       -\frac{1}{2N}\sum_{ \substack{\gamma \in \Pi_{n_1+1}^N,~\hco(\gamma) =m  
        \\
        l_1,l_2 \in \gamma   
     }}   
      \bE 
    \bigg[   \Big(   | X^{t,x_{l_1},l_1,N}_{s,x_{\gamma_1},\dots,x_{\gamma_{n_1+1}}}|^{p-2}X^{t,x_{l_1},l_1,N}_{s,x_{\gamma_1},\dots,x_{\gamma_{n_1+1}}} 
    - 
     | X^{t,x_{l_2},l_2,N}_{s,x_{\gamma_1},\dots,x_{\gamma_{n_1+1}}}|^{p-2}X^{t,x_{l_2},l_2,N}_{s,x_{\gamma_1},\dots,x_{\gamma_{n_1+1}}}
    \Big) 
    \\    \label{eq: n-var convolution syme1}
    &\qquad \qquad  \cdot \Big(   
        \nabla^2 V(X^{t,x_{l_1},l_1,N}_s-X^{t,x_{l_2},l_2,N}_s)  ( X^{t,x_{l_1},l_1,N}_{s,x_{\gamma_1},\dots,x_{\gamma_{n_1+1}}}- X^{t,x_{l_2},l_2,N}_{s,x_{\gamma_1},\dots,x_{\gamma_{n_1+1}}})
      \Big)   \bigg] 
     \\    \label{eq: n-var convolution diff1} 
     &\quad + 
     \frac{K_V }{N} \sum_{ \substack{\gamma\in \Pi_{n_1+2}^N,~\hco(\gamma) = m 
        \\
        \hco(\gamma\setminus  (\gamma_1)) = m  
     }} 
     \sum_{\substack{     l=1, \\    \hco(\gamma \setminus (\gamma_1) \bigcup (l) )  = m+1 } }^{N}
     \Big( 
     \frac{p-1}{p}
     \bE \Big[ \Big|    X^{t,x_{\gamma_1},\gamma_1,N}_{s,x_{\gamma_2},\dots,x_{\gamma_{n_1+2}}}\Big|^p 	\Big] 
     + \frac{1}{p}
     \bE \Big[ \Big|    X^{t,x_l,l,N}_{s,x_{\gamma_2},\dots,x_{\gamma_{n_1+2}}}\Big|^p 	\Big]
     \Big) 
     \\
     \label{eq: n-var convolution r1}
     &~\leq 
     \frac{ K_V(p-1)}{p}\sum_{ \substack{\gamma\in \Pi_{n_1+2}^N,~\hco(\gamma) = m 
        \\
        \hco(\gamma\setminus  (\gamma_1)) = m  
     }}
     \bE \Big[ \Big|    X^{t,x_{\gamma_1},\gamma_1,N}_{s,x_{\gamma_2},\dots,x_{\gamma_{n_1+2}}}\Big|^p 	\Big]
     +\frac{K_V }{Np}
     \sum_{ \substack{\gamma\in \Pi_{n_1+2}^N,~\hco(\gamma) = m+1 
        \\
        \hco(\gamma\setminus  (\gamma_1)) = m  
     }} 
     \bE \Big[ \Big|    X^{t,x_{\gamma_1},\gamma_1,N}_{s,x_{\gamma_2},\dots,x_{\gamma_{n_1+2}}}\Big|^p 	\Big],
\end{align}
where \eqref{eq: n-var convolution r1} follows from the fact $\nabla^2 V(x)\geq 0$ for all $x \in \mathbb{R}$ (implying that \eqref{eq: n-var convolution syme1} is bounded above by zero), and an application of Jensen's inequality is used to get \eqref{eq: n-var convolution diff1}. 
This will provide acceptable weights in \eqref{eq: n-var convolution rrr} below. 

\medskip
\color{black}
\textit{Case 2: 
$\gamma_1$ is a unique index in $\gamma$ (i.e., $\hco(\gamma) =m,~ \hco(\gamma\setminus  (\gamma_1)) = m-1$)}. 
Note that this case is only meaningful for $m \in \lbrace 2, \ldots, n_1 +2 \rbrace$, as $\gamma$ has at least $2$ distinct elements (at least one of $\gamma_2, \ldots, \gamma_{n_1+2}$ has to be different to $\gamma_1$). Further, we remark that  $\hco(\gamma)=m$ but $ \hco(\gamma \setminus (\gamma_1) \bigcup (l) )\in\{m,m-1\}$, and therefore we need to consider $X^{t,x_{\gamma_1},\gamma_1,N}_{\cdot}$ 
and $X^{t,x_{l},l,N}_{\cdot}$  separately
(in a similar fashion to \eqref{eq: second var res2} and \eqref{eq: second var res4}).  Hence, we write  
\begin{align}
    \nonumber
     -&\sum_{ \substack{\gamma\in \Pi_{n_1+2}^N,~\hco(\gamma) = m 
        \\
        \hco(\gamma\setminus  (\gamma_1)) = m-1  
     }}  \bE 
    \bigg[ \frac{1}{N} \sum_{l=1}^{N}  \Big(   | X^{t,x_{\gamma_1},\gamma_1,N}_{s,x_{\gamma_2},\dots,x_{\gamma_{n_1+2}}}|^{p-2}X^{t,x_{\gamma_1},\gamma_1,N}_{s,x_{\gamma_2},\dots,x_{\gamma_{n_1+2}}} 
    \Big)   	
    \\
        \nonumber
    &\qquad \qquad  \cdot \Big(   
        \nabla^2 V(X^{t,x_{\gamma_1},\gamma_1,N}_s-X^{t,x_l,l,N}_s)  ( X^{t,x_{\gamma_1},\gamma_1,N}_{s,x_{\gamma_2},\dots,x_{\gamma_{n_1+2}}}- X^{t,x_l,l,N}_{s,x_{\gamma_2},\dots,x_{\gamma_{n_1+2}}})
      \Big)   \bigg] 
      \\
          \nonumber
      & =
      -\frac1N \sum_{ \substack{\gamma\in \Pi_{n_1+2}^N,~\hco(\gamma) = m  
        \\
        \hco(\gamma\setminus  (\gamma_1)) = m-1  
     }}  
     \bigg( 
     \sum_{\substack{     l=1, \\    \hco(\gamma \setminus (\gamma_1) \bigcup (l) )  = m } }^{N}
     +
     \sum_{\substack{     l=1, \\    \hco(\gamma \setminus (\gamma_1) \bigcup (l) )  = m-1} }^{N}
     \bigg)
     \bE 
    \bigg[   \Big(   | X^{t,x_{\gamma_1},\gamma_1,N}_{s,x_{\gamma_2},\dots,x_{\gamma_{n_1+2}}}|^{p-2}  	
    \\
    \label{eq: n-var convolution syme2-old}
    &\qquad \qquad  \cdot X^{t,x_{\gamma_1},\gamma_1,N}_{s,x_{\gamma_2},\dots,x_{\gamma_{n_1+2}}} 
    \Big) \cdot \Big(   
        \nabla^2 V(X^{t,x_{\gamma_1},\gamma_1,N}_s-X^{t,x_l,l,N}_s)  ( X^{t,x_{\gamma_1},\gamma_1,N}_{s,x_{\gamma_2},\dots,x_{\gamma_{n_1+2}}}- X^{t,x_l,l,N}_{s,x_{\gamma_2},\dots,x_{\gamma_{n_1+2}}})
      \Big)   \bigg].    
\end{align}
We exploit the \textit{symmetrization trick} for the summation over $\hco(\gamma \setminus (\gamma_1) \bigcup (l) )  = m$  where $l$ takes values out of  $\gamma\setminus  (\gamma_1)$,   and   Young's inequality for the summation over $\hco(\gamma \setminus (\gamma_1) \bigcup (l) )  = m-1$ where $l \in (\gamma\setminus (\gamma_1))$ (check  \eqref{eq: second var res2} and \eqref{eq: second var res4} for example). As before,  substitute the new upper indices  \eqref{eq: n-var convolution syme2-old}: $\gamma_1 \mapsto l_1$ and $l \mapsto l_2$ to obtain 
\begin{align}
 \nonumber 
      & \eqref{eq: n-var convolution syme2-old}\leq 
       -\frac{1}{2N}\sum_{ \substack{\gamma \in \Pi_{n_1+1}^N,~\hco(\gamma) =m-1  
        \\
        l_1,l_2 \notin \gamma   
     }}   
      \bE 
    \bigg[   \Big(   | X^{t,x_{l_1},l_1,N}_{s,x_{\gamma_1},\dots,x_{\gamma_{n_1+1}}}|^{p-2}X^{t,x_{l_1},l_1,N}_{s,x_{\gamma_1},\dots,x_{\gamma_{n_1+1}}} 
    - 
     | X^{t,x_{l_2},l_2,N}_{s,x_{\gamma_1},\dots,x_{\gamma_{n_1+1}}}|^{p-2}  
    \\    \label{eq: n-var convolution syme2}
    &\qquad \qquad  \cdot X^{t,x_{l_2},l_2,N}_{s,x_{\gamma_1},\dots,x_{\gamma_{n_1+1}}}
    \Big) \cdot \Big(   
        \nabla^2 V(X^{t,x_{l_1},l_1,N}_s-X^{t,x_{l_2},l_2,N}_s)  ( X^{t,x_{l_1},l_1,N}_{s,x_{\gamma_1},\dots,x_{\gamma_{n_1+1}}}- X^{t,x_{l_2},l_2,N}_{s,x_{\gamma_1},\dots,x_{\gamma_{n_1+1}}})
      \Big)   \bigg] 
     \\
     \label{eq: n-var convolution diff2}
     &\quad + 
     K_V  \sum_{ \substack{\gamma\in \Pi_{n_1+2}^N,~\hco(\gamma) = m 
        \\
        \hco(\gamma\setminus  (\gamma_1)) = m-1  
     }}  
     \sum_{\substack{     l=1, \\    \hco(\gamma \setminus (\gamma_1) \bigcup (l) )  = m-1 } }^{N}
     \Big( 
     \frac{p-1}{p}
     \bE \Big[ \Big|    X^{t,x_{\gamma_1},\gamma_1,N}_{s,x_{\gamma_2},\dots,x_{\gamma_{n_1+2}}}\Big|^p 	\Big] 
     + \frac{1}{pN^p}
     \bE \Big[ \Big|    X^{t,x_l,l,N}_{s,x_{\gamma_2},\dots,x_{\gamma_{n_1+2}}}\Big|^p 	\Big]
     \Big) 
     \\
     \label{eq: n-var convolution r2}
    & \leq 
      \frac{6K_V(p-1)}{p}\sum_{ \substack{\gamma\in \Pi_{n_1+2}^N,~\hco(\gamma) = m 
        \\
        \hco(\gamma\setminus  (\gamma_1)) = m-1  
     }}  
     \bE \Big[ \Big|    X^{t,x_{\gamma_1},\gamma_1,N}_{s,x_{\gamma_2},\dots,x_{\gamma_{n_1+2}}}\Big|^p 	\Big]
     + \frac{6K_V}{pN^p}
    \sum_{ \substack{\gamma\in \Pi_{n_1+2}^N,~\hco(\gamma) = m-1
        \\
        \hco(\gamma\setminus  (\gamma_1)) = m-1  
     }}  
     \bE \Big[ \Big|    X^{t,x_{\gamma_1},\gamma_1,N}_{s,x_{\gamma_2},\dots,x_{\gamma_{n_1+2}}}\Big|^p 	\Big],
\end{align}
where we  used $\nabla^2 V(x)\geq 0$ to bound \eqref{eq: n-var convolution syme2} above by 0, then apply Jensen's inequality to \eqref{eq: n-var convolution diff2} to match the acceptable weights on $N$ in \eqref{eq: n-var convolution rrr}. The constant 6 arises from the fact that there are at most 6 different values for $l \in \gamma\setminus (\gamma_1), m\leq n_1+2 \leq 7$ such that $ \hco(\gamma \setminus (\gamma_1) \bigcup (l) )  =\hco( \gamma \setminus (\gamma_1))= m-1$. 

\medskip
\textit{Combining case 1 and case 2:}
We inject the established estimates in \eqref{eq: n-var convolution r1} and \eqref{eq: n-var convolution r2} into \eqref{eq: n-var convolution 1}  and consider the summation \eqref{eq: def of itsnpp}, (noting the inclusion of the  $N^{(m-1)p-m}$ term in the summand) to obtain 
\begin{align}
    \nonumber
    &\sum_{m=1}^{n_1+2} \Bigg( 
    N^{(m-1)p-m}  \sum_{ \gamma\in \Pi_{n_1+2}^N,~\hco(\gamma) = m}   \bE 
    \bigg[ \frac{1}{N} \sum_{l=1}^{N}  \Big(   | X^{t,x_{\gamma_1},\gamma_1,N}_{s,x_{\gamma_2},\dots,x_{\gamma_{n_1+2}}}|^{p-2}X^{t,x_{\gamma_1},\gamma_1,N}_{s,x_{\gamma_2},\dots,x_{\gamma_{n_1+2}}} 
    \Big)   	
    \\
    \nonumber
    &\qquad \qquad  \cdot \Big(   
        \nabla^2 V(X^{t,x_{\gamma_1},\gamma_1,N}_s-X^{t,x_l,l,N}_s)  ( X^{t,x_{\gamma_1},\gamma_1,N}_{s,x_{\gamma_2},\dots,x_{\gamma_{n_1+2}}}- X^{t,x_l,l,N}_{s,x_{\gamma_2},\dots,x_{\gamma_{n_1+2}}})
      \Big)   \bigg] \quad 
    \Bigg) 
    \\\nonumber
    &
    \leq   
    \Big( \frac{6K_V(p-1)}{p} + \frac{K_V}{Np}+\frac{6K_V}{Np}\Big)
    \sum_{m=1}^{n_1+2} \bigg( 
    N^{(m-1)p-m}  \sum_{ \gamma\in \Pi_{n_1+2}^N,~\hco(\gamma) = m}
    \bE \Big[ \Big|    X^{t,x_{\gamma_1},\gamma_1,N}_{s,x_{\gamma_2},\dots,x_{\gamma_{n_1+2}}}\Big|^p 	\Big]
    \bigg) 
    \\
    \label{eq: n-var convolution rrr}
    &\leq 
    K_V\Big(6+\frac{1}{Np}\Big) 
    \sum_{m=1}^{n_1+2} 
    \bigg( 
    N^{(m-1)p-m}  \sum_{ \gamma\in \Pi_{n_1+2}^N,~\hco(\gamma) = m}
    \bE \Big[ \Big|    X^{t,x_{\gamma_1},\gamma_1,N}_{s,x_{\gamma_2},\dots,x_{\gamma_{n_1+2}}}\Big|^p 	\Big]
    \bigg)
    \\ \label{final estimate KV(6+1/Np)}
    &
    =K_V
    \Big(6 +\frac{1}{Np}\Big) 
    e^{-\lambda_0^{(n_1+1)} (s-t)}I_{s,t}^{n_1+1,p}  .
\end{align}\\ 
\textit{Part 2.2: Analysis of the lower order variation term \eqref{eq: n-var - lower var }.}
We provide the following details on the derivatives of the function $B_{\gamma_1}$:  We notice that, from Assumption~\ref{assum:main_weak error}, for any $\bodx \in \mathbb{R}^N$,  
    \begin{align*}
         \partial_{x_l}B_{\gamma_1} 
    = \begin{cases}
         -\nabla^{2} U(x_{\gamma_1})  - \frac{1}{N}\partial_{x_l}\sum_{j=1}^N \nabla V(x_{\gamma_1}-x_j)   
        = 
         -\nabla^{2} U(x_{\gamma_1})
         - \frac{1}{N}\sum_{j=1}^N \nabla^{2} V(x_{\gamma_1}-x_j),
          & {\gamma_1}=l,
         \\
          - \frac{1}{N}\partial_{x_l}\sum_{j=1}^N \nabla V(x_{\gamma_1}-x_j)  
          =   \frac{1}{N}\nabla^{2} V(x_{\gamma_1}-x_l),  
          & {\gamma_1}\neq l.
    \end{cases}
    \end{align*}
 Thus, we have     
    \begin{align}
    \label{eq: prop of deriv of B}
    &|\partial_{x_l}B_{\gamma_1}|_\infty
    \leq \begin{cases}
         |\nabla^{2} U|_\infty + \frac{1}{N}\sum_{j=1}^N |\nabla^{2} V  |_\infty =  \cO(1),    & {\gamma_1}=l,
         \\
          \frac{1}{N} |\nabla^{2} V  |_\infty
         = \cO( N^{-1}),  & {\gamma_1}\neq l,
    \end{cases}
    \end{align}
    and  
    \begin{align} 
    \nonumber
    & |\partial^2_{x_{l},x_{l'}} B_{\gamma_1}|_\infty
    \leq  \begin{cases}
   |\nabla^{3} U|_\infty + \frac{1}{N}\sum_{j=1}^N |\nabla^{3} V  |_\infty =  \cO(1),  \quad & {\gamma_1}=l=l',
        \\
         \frac{1}{N}  |\nabla^{3} V  |_\infty=  \cO( N^{-1}), \quad & {\gamma_1}=l\neq l',
        \\
          \frac{1}{N}  |\nabla^{3} V  |_\infty =   \cO( N^{-1}), \quad & {\gamma_1}=l'\neq l,
         \\
         \frac{1}{N}  |\nabla^{3} V  |_\infty=   \cO( N^{-1}), \quad &{\gamma_1}\neq l=l',
         \\
         0, \quad & {\gamma_1}\neq l\neq l'.
    \end{cases}
    \end{align} 
   Using this methodology, one can easily establish that for $n\geq 1,~\eta \in \Pi_{n}^N$
    \begin{align}
    \label{eq: bounds of the summation of derivatives of B noiid}
        |\partial^n_{x_{\eta_{1}},\dots,x_{\eta_{n}} }  B_{\gamma_1}|_\infty &= \begin{cases}
            \cO( N^{1-\hco( \eta \bigcup  (\gamma_1 ))} ),
            \qquad &\hco( \eta \bigcup  (\gamma_1 ) ) \in \{1,2\},
            \\
            0, \qquad &\hco( \eta \bigcup  (\gamma_1 )) \geq 3,
        \end{cases}
    \end{align}
\color{black}
and further using that only tuples of the form $(  \gamma_1,\eta_1, \ldots, \eta_{n} )$, where elements take at most two different values, yield the non-zero contribution in \eqref{eq: bounds of the summation of derivatives of B noiid}.  
\color{black}
Proceeding exactly as in the manner of \eqref{eq: second var res5}--\eqref{eq: second var res10}, applying our induction hypothesis \eqref{eq: aux induction hypothesis}, we can establish the existence of positive constants $K$ and $\alpha$ such that 
\begin{align}
\label{eq:aux sum 4.69}
\sum_{m=1}^{n_1+2} \bigg( 
    N^{(m-1)p-m}  \sum_{\gamma\in \Pi_{n_1+2}^N,~\hco(\gamma) = m} \eqref{eq: n-var - lower var }\bigg) \leq K \int_t^s 
     e^{-\alpha (u-t)}  
     \dd u+ K \int_t^s 
     e^{-\alpha (u-t)}  
    I_{t,u}^{n_1+1,p}\dd u .
    \end{align}

\textit{Part 3: Collecting terms and conclusion}.
By collecting \eqref{eq: aux first integral term lambda0-lambda}, \eqref{final estimate KV(6+1/Np)} and \eqref{eq:aux sum 4.69}, we have that 
\begin{align*}
    I_{t,s}^{n_1+1,p} 
    \leq 
    & K \int_t^s 
     e^{-\alpha (u-t)} \dd u 
     \\
     & + p\Big(\lambda_0^{(n_1+1)} -\lambda + K_V(6+\frac{1}{Np})\Big) \int_t^s I_{t,u}^{n_1+1,p}\dd u + K \int_t^s e^{-\alpha (u-t)}I_{t,u}^{n_1+1,p}\dd u.
\end{align*}
Using the convexity assumption $\lambda \geq7K_V$, we can conclude the existence of $\lambda_0^{(n_1+1)} \in (0,\lambda)$ such that the term in front of the second integral remains negative. The result follows from Gronwall's inequality (as was done in the case $n=2$). 
\end{proof}

\section{Decay estimates for the Kolmogorov backward equation}
\label{section: analysis of the dus}
We establish the following estimates for the derivatives of the solution to the  Kolmogorov backward equation in terms of moments of the variation processes. This will further help us to apply the results developed in Section \ref{section: analysis of var process total} to the weak error analysis in Section \ref{section: weak error expansion}. 
\begin{lemma}
\label{lemma: first derivative of u}
Let $u$ be the solution to the Kolmogorov backward equation \eqref{PDE:Kolmogorov} with $g$ as in Assumption~\ref{assum:main_weak error} and let $T\geq t \geq 0, N\in \mathbb{N}$. Then there exists a constant $K>0$ (independent of $t,T,N$), such that for any $j \in \lbrace 1, \ldots, N \rbrace$ and $\bodx \in \mathbb{R}^N$
\begin{align*}
|\partial_{x_j} u(t, \bodx)|^2 
\leq     
\frac{K}{N^2}  \bE \Big[ | X_{T,x_j}^{t,x_j,j,N} | ^2  \Big]
+
 \frac{K}{N} \sum_{i=1,~i\neq j}^{N}  
  \mathbb{E}\left[| X_{T,x_j}^{t,x_i,i,N} |^2 \right]
.
\end{align*}   
\end{lemma}
\begin{proof} 
From the definition of $u$, using  \cite[Theorem 5.5 in Chapter 5]{F1971booksdeaaa} as in \cite[Proposition B.3]{haji2021simple}, we deduce 
\begin{align*}
   |\partial_{x_j}  & u(t, \bodx)|^2 
 \\
    &=
   \Big| \mathbb{E}\Big[ ~\sum_{i=1}^{N} \big(   \partial_{x_i} g(\bodX_T^{t,\boldsymbol{x},N}) 
     \big)   \cdot  	\big(   
   X_{T,x_j}^{t,x_i,i,N} \big)   \Big] \Big|^2 
    \\
    & \leq 
    2 \Big|    \bE \Big[ | \partial_{x_j} g(\bodX_T^{t,\boldsymbol{x},N})| \ | X_{T,x_j}^{t,x_j,j,N} |  \Big] 
    \Big|^2
    + 
    2\Big| \mathbb{E}\Big[ ~\sum_{i=1,~i\neq j}^{N}  \big(   \partial_{x_i} g(\bodX_T^{t,\boldsymbol{x},N}) 
     \big)   \cdot  	\big(   
    X_{T,x_j}^{t,x_i,i,N} \big)  \Big] \Big|^2
    \\
    & \leq 
    \frac{K}{N^2}  \bE \Big[ | X_{T,x_j}^{t,x_j,j,N} | ^2  \Big]
     +  K N 
     \sum_{i=1,~i\neq j}^{N} \mathbb{E}\left[  \Big|| \partial_{x_i} g(\bodX_T^{t,\boldsymbol{x},N})| \ | X_{T,x_j}^{t,x_i,i,N} | \Big|^2 \right] 
    \\
    & \leq  \frac{K}{N^2}  \bE \Big[ | X_{T,x_j}^{t,x_j,j,N} | ^2  \Big] 
    +
    \frac{K}{N} \sum_{i=1,~i\neq j}^{N}  
  \mathbb{E}\left[| X_{T,x_j}^{t,x_i,i,N} |^2 \right]    , 
\end{align*}
where we used Jensen's inequality and the growth of derivatives of $g$ in  Assumption~\ref{assum:main_weak error}. 
 \end{proof}

The next result, Lemma~\ref{lemma: higher-order derivative of u without exp}, generalizes Lemma~\ref{lemma: first derivative of u} to derivatives of order $ 1 < n \leq 6$, whose proof is involved and relies on carefully applying Jensen's inequality. 
The next calculations aim to clarify that partitioning summations in particular ways before the application of Jensen's inequality can yield sharper upper bounds.

For each $\gamma =(\gamma_1,\ldots,\gamma_n )  \in \Pi_n^{N}$, let $x_{\gamma_1,\ldots,\gamma_n}$ be a real number. Then applying Jensen's inequality directly we have 
\begin{align}
     \nonumber
    \Big| \sum_{\gamma \in \Pi_1^N} x_{\gamma_1} \Big|^2
    &=\Big| \sum_{i=1}^N x_i \Big|^2= N^2  \Big| \frac{1}{N }\sum_{i=1}^N x_i \Big|^2
    \leq N   \sum_{i=1}^N  |x_i|^2,
    \\
    \label{eq: jensen's eg1}
    \Big| \sum_{\gamma \in \Pi_n^N}  x_{\gamma_1,\dots,\gamma_{n}} \Big|^2
    &
    = 
    N^{2n} \Big| \frac{1}{N^n} \sum_{\gamma \in \Pi_n^N}  x_{\gamma_1,\dots,\gamma_{n}} \Big|^2
    \leq
      \frac{N^{2n}}{N^n} \sum_{\gamma \in \Pi_n^N}  |x_{\gamma_1,\dots,\gamma_{n}}  |^2
    =
    N^n  \sum_{\gamma \in \Pi_n^N}  |x_{\gamma_1,\dots,\gamma_{n}}  |^2.
\end{align} 
Consider now the specific two-dimensional example  of $x_{\gamma_1,\gamma_{2}}=N^{1-\hco(\gamma)}$ {\color{black}(corresponding to an $N\times N$ matrix with diagonal entries $1$ and otherwise $1/N$) }. Using  \eqref{eq: jensen's eg1} we have that
\begin{align*} 
    \Big| \sum_{\gamma \in \Pi_2^N} x_{ \gamma_1,\gamma_2} \Big|^2
    &
    \leq N^2   \sum_{i,j=1}^N  |x_{i,j}|^2
    = N^2  \sum_{i=1}^N |x_{i,i}|^2 + N^2  \sum_{i,j=1,i \neq j  }^N  |x_{i,j}|^2 = N^3+ N^2 \leq 2N^3. 
\end{align*}
This estimate is too naive and can be improved, as we can instead consider 
\begin{align*}
    \Big| \sum_{\gamma \in \Pi_2^N} x_{\gamma_1,\gamma_2} \Big|^2
    &
    \leq 2\Big|\sum_{i=1}^N x_{i,i}\Big|^2 +2 \Big| \sum_{i,j=1,i \neq j }^N   x_{i,j}  \Big|^2
    \leq  
    2N \sum_{i=1}^N |x_{i,i}|^2+ 2N^2  \sum_{i,j=1,i \neq j }^N |x_{i,j}|^2
    \\
    & 
    =2N^2 + \frac{2N^3(N-1)}{N^2} \leq 4N^2   
    ,
\end{align*}
which is indeed a sharper upper bound. 

This argument can be extended to the general $n$-dimensional case; that is, take $x_{\gamma_1,\ldots,\gamma_{n}}=N^{1-\hco(\gamma)}$ as the entries of a $\bigotimes_{i=1}^n\bR^N$-valued $n$-tensor. Observe that 
\begin{align}
\label{eq-aux:countingNumbers}
|\{\gamma \in \Pi_n^N: \hco(\gamma)=m\}|=\cO(N^m)  ,  
\end{align}
a result which follows from a simple combinatorial argument: regardless of the length $n$ of $\gamma$, we have $m$ distinct values chosen out of the range $\{1,\ldots,N\}$; that is $ \cO(N^m)$ 
possibilities. We have
\begin{align*}
    \nonumber 
     \Big| \sum_{\gamma \in \Pi_n^N}  x_{\gamma_1,\dots,\gamma_{n}} \Big|^2
    &
    = 
     \Big| \sum_{ m=1}^{n} \sum_{ \substack{\gamma \in \Pi_n^N,
        \\     
        \nonumber 
     \hco(\gamma)=m
     }}   x_{\gamma_1,\dots,\gamma_{n}} \Big|^2
     \leq 
     K\sum_{ m=1}^{n}  \Big|\sum_{ \substack{\gamma \in \Pi_n^N,
        \\ 
        \nonumber 
     \hco(\gamma )=m
     }}   x_{ \gamma_1,\dots,\gamma_{n}} \Big|^2
     \\ 
     &\leq 
     K  \sum_{ m=1}^{n} N^m 
      \sum_{ \substack{\gamma \in \Pi_n^N,
        \\ 
     \hco(\gamma )=m
     }}   |x_{ \gamma_1,\dots,\gamma_{n}}  |^2
     \leq 
      K  \sum_{ m=1}^{n} N^{2m} N^{2-2m}
      \leq KN^2,
\end{align*}
with $K$ a constant, changing across the inequalities, independent of $N$ but dependent on $n$ where $n\leq 6$. 
This is also a sharper upper bound than applying Jensen's inequality directly. We now state the generalization of Lemma~\ref{lemma: first derivative of u} for $ 1 \leq n \leq 6$. 
\begin{lemma}
\label{lemma: higher-order derivative of u without exp}
Let $u$ satisfy the Kolmogorov backward equation \eqref{PDE:Kolmogorov} with $g$ as in Assumption~\ref{assum:main_weak error} and let $T\geq t \geq 0, N\in \mathbb{N}$. Then there exists a constant $K>0$ (independent of $t,T,N$),  such that for any $n\in \mathbb{N}, 1 \leq n \leq 6$, $\gamma\in \Pi_n^N $,  and  
 $ \boldsymbol{x} \in \mathbb{R}^N$ 
\begin{align}
\nonumber
 |\partial&_{x_{\gamma_1},\dots,x_{\gamma_n}}^n u(t, \bodx)|^2
   \\ 
   \label{eq: bounds for high-deri on u}
    &\leq 
      K \sum_{m=0}^n ~
        ~	
    \sum_{ \substack{    \ell \in  
    \bigcup_{k=1}^{n} \Pi_k^N,
    \\   \hco(\ell \bigcup \gamma) = \hco(\gamma)+m   }   }  
    ~N^{m-2\hco(\ell) }~ {\color{black} \sum_{
    \substack{     
    \alpha_1,\dots,\alpha_{|\ell|} \in  
    \bigcup_{k=1}^{n} \Pi_k^N, \\
    \bigcup_{i=1}^{|\ell|} \alpha_i \simeq \gamma          }} 
    ~\bE \Big[ ~ 
     \prod_{i=1}^{|\ell|}
    \Big|
    {\color{black}
    X_{T,
    x_{\alpha_{i,1}} ,\dots ,x_{\alpha_{i,|\alpha_i |}}  }^{t,x_{\ell_i},\ell_i,N}
    }
    \Big|^2 
    ~ \Big]},
\end{align}
where {\color{black}$\alpha_i = (  
\alpha_{i,1},\dots , \alpha_{i,|\alpha_i |})$
and  $ \alpha_{i,j} \in \{1,\dots,N\}$ for $j\in \{1,\dots, |\alpha_i|\}$.
} Further, we have 
\begin{align}
\nonumber
 |\partial&_{x_{\gamma_1},\dots,x_{\gamma_n}}^n u(t, \bodx)|^4
   \\ 
    \label{eq: bounds for high-deri on u order 4}
    &\leq 
      K \sum_{m=0}^n ~
        ~	
    \sum_{\substack{     \ell \in  
    \bigcup_{k=1}^{n} \Pi_k^N,\\
    \hco(\ell \bigcup \gamma) = \hco(\gamma)+m    }  }  
    ~N^{3m-4\hco(\ell) }~ 
    \sum_{ \substack{   \alpha_1,\dots,\alpha_{| \ell|} \in  
     \bigcup_{k=1}^{n} \Pi_k^N, 
    \\   
    \bigcup_{i=1}^{| \ell|} \alpha_i \simeq \gamma}
    } 
    ~\bE \Big[ ~ 
     \prod_{i=1}^{| \ell|}
    \Big|
    {\color{black}
    X_{T,
    x_{\alpha_{i,1}} ,\dots ,x_{\alpha_{i,|\alpha_i |}}  }^{t,x_{\ell_i},\ell_i,N} 
    }
    \Big|^4 
    ~ \Big].
\end{align}
\end{lemma}
\begin{proof}
Note that in the following proof, the positive constant $K$ is independent of $t,T,N$ and may change line by line. 
Before establishing the main results (\textit{Part 2} and \textit{Part 3}), we present (\textit{Part 1}) a study for the second derivatives of $u$ (i.e., for the case $n=2$) to demonstrate the approach used and its nuances in regards to estimate \eqref{eq: bounds for high-deri on u} of the lemma. In the following text, we use $i\neq j \neq k$ to denote all three indices are distinct from one another. 

\textit{Part 1: Comparing the direct calculation with \eqref{eq: bounds for high-deri on u}  
for the example case $n=2$}. For all $j,k\in \{ 1,\dots,N\}$, $j\neq k $, using  \cite[Theorem 5.5 in Chapter 5]{F1971booksdeaaa}, we have
\allowdisplaybreaks
\begin{align}
\nonumber
    & |\partial_{x_j, x_k}^2 u(t, \bodx)|^2 
    \\ \nonumber
    &=  \bigg|  \mathbb{E}\Big[ \sum_{i=1}^{N} \partial_{x_i} g  (\bodX_T^{t,\boldsymbol{x},N}) X^{t,x_i,i,N}_{T,x_j,x_k}  \Big]  
    + 
        \mathbb{E}\Big[ \sum_{i=1}^{N} \sum_{i'=1}^{N}  \partial^2_{x_i, x_{i'}} g(\bodX_T^{t,\boldsymbol{x},N}) X^{t,x_i,i,N}_{T,x_j}  X^{t,x_{i'},{i'},N}_{T,x_k}  \Big] \bigg| ^2 
        \\ 
    &\leq K \Bigg(  \bigg| \label{eq: lemma5.2 line 3}
    \bE \Big[  \partial_{x_j  } g (\bodX_T^{t,\boldsymbol{x},N}) X^{t,x_j,j,N}_{T,x_j,x_k}
    +
    \partial_{ x_k} g (\bodX_T^{t,\boldsymbol{x},N}) X^{t,x_k,k,N}_{T,x_j,x_k}\Big]  \bigg|^2 
    \\ \label{eq: lemma5.2 line 4}
    &\quad 
    +
   \bigg| \mathbb{E}\Big[ \sum_{i=1,~i\neq j \neq k}^{N} \partial_{x_i}g (\bodX_T^{t,\boldsymbol{x},N}) X^{t,x_i,i,N}_{T,x_j,x_k}  ~\Big]  \bigg|^2
    \\   \label{eq: lemma5.2 line 5}
    &\quad + \bigg|\bE \Big[  \partial^2_{x_j, x_j}g (\bodX_T^{t,\boldsymbol{x},N}) X^{t,x_j,j,N}_{T,x_j}  X^{t,x_{j},{j},N}_{T,x_k}
    +  \partial^2_{x_k, x_k}
    g (\bodX_T^{t,\boldsymbol{x},N}) X^{t,x_k,k,N}_{T,x_j}  X^{t,x_{k},{k},N}_{T,x_k}   
          \Big] \bigg|^2
     \\ \label{eq: lemma5.2 line 7}
    &\quad + \bigg|\bE \Big[  \partial^2_{x_j, x_k}g (\bodX_T^{t,\boldsymbol{x},N}) X^{t,x_j,j,N}_{T,x_j}  X^{t,x_{k},{k},N}_{T,x_k}
    +  \partial^2_{x_k, x_j}g (\bodX_T^{t,\boldsymbol{x},N}) X^{t,x_k,k,N}_{T,x_j}  X^{t,x_{j},{j},N}_{T,x_k}
          \Big] \bigg|^2
    \\ \label{eq: lemma5.2 line 8}
    &\quad +\bigg|\bE \Big[ 
          \sum_{i=1,~i\neq j \neq k}^{N}   \partial^2_{x_i, x_i}g (\bodX_T^{t,\boldsymbol{x},N}) X^{t,x_i,i,N}_{T,x_j}  X^{t,x_{i},{i},N}_{T,x_k}
          \Big]\bigg|^2
    \\ 
     \nonumber
   &\quad  +  
    \bigg|\bE \Big[  
      \sum_{i=1,~i\neq j \neq k}^{N}    
      \Big( 
    \partial^2_{x_j, x_i}g (\bodX_T^{t,\boldsymbol{x},N}) X^{t,x_j,j,N}_{T,x_j}  X^{t,x_{i},{i},N}_{T,x_k} 
    +
    \partial^2_{x_i, x_j}g (\bodX_T^{t,\boldsymbol{x},N}) X^{t,x_i,i,N}_{T,x_j}  X^{t,x_{j},{j},N}_{T,x_k}
    \\ 
     \label{eq: lemma5.2 line 10}
   &\hspace{1cm}   +
      \partial^2_{x_k, x_i}
      g (\bodX_T^{t,\boldsymbol{x},N}) X^{t,x_k,k,N}_{T,x_j}  X^{t,x_{i},{i},N}_{T,x_k} 
    +
      \partial^2_{x_i, x_k}
      g (\bodX_T^{t,\boldsymbol{x},N}) X^{t,x_i,i,N}_{T,x_j}  X^{t,x_{k},{k},N}_{T,x_k} 
      \Big)
      \Big]\bigg|^2
     \\
     \label{eq: lemma5.2 line 11}
     &\quad  + 
     \bigg|\bE \Big[
     \sum_{i,i'=1,~j\neq i\neq i '\neq k}^{N}  \partial^2_{{x_i, x_{i'}}} g(\bodX_T^{t,\boldsymbol{x},N}) X^{t,x_i,i,N}_{T,x_j}  X^{t,x_{i'},{i'},N}_{T,x_k}
    \Big] 
    \bigg|^2\bigg),
    \end{align}
where the inequality follows simply by separating the summation terms by number of distinct derivatives of $g$ matching also to the order of decay in $N$, and then applying Jensen's inequality (but still leaving the square outside the expectations).

Under Assumption~\ref{assum:main_weak error} on the derivatives of $g$ and observing that \eqref{eq: lemma5.2 line 4}, \eqref{eq: lemma5.2 line 8} and \eqref{eq: lemma5.2 line 10} have  $\cO(N)$ terms and that \eqref{eq: lemma5.2 line 11} has  $\cO(N^2)$ terms in the summand, we apply Jensen's inequality once more to get
    \begin{align}
    \eqref{eq: lemma5.2 line 3}+\eqref{eq: lemma5.2 line 5}+\eqref{eq: lemma5.2 line 11}
    & 
    \leq 
    \frac{K}{N^2}\bE \Big[   |X^{t,x_j,j,N}_{T,x_j,x_k} |^2 + |X^{t,x_k,k,N}_{T,x_j,x_k}|^2   
    +
    |X^{t,x_j,j,N}_{T,x_j}  X^{t,x_{j},{j},N}_{T,x_k}|^2
    + |X^{t,x_k,k,N}_{T,x_j}  X^{t,x_{k},{k},N}_{T,x_k}|^2 \Big]\nonumber
    \\
    & \nonumber
    \hspace{3.15cm}+\sum_{i,i'=1,~j\neq i\neq i '\neq k}^{N} \Big[ | X^{t,x_i,i,N}_{T,x_j}  X^{t,x_{i'},{i'},N}_{T,x_k} |^2\Big]
    \\
    \nonumber
    \eqref{eq: lemma5.2 line 4}+\eqref{eq: lemma5.2 line 8}
    & 
    \leq\frac{K}{N}  
    \sum_{i=1,~i\neq j \neq k}^{N}
     \bE\Big[   | X^{t,x_i,i,N}_{T,x_j,x_k}     |^2 
    +
       | X^{t,x_i,i,N}_{T,x_j}  X^{t,x_{i},{i},N}_{T,x_k}  |^2 \Big] 
    \\\nonumber
    \eqref{eq: lemma5.2 line 7}
    & \leq
    \frac{K}{N^4}\bE \Big[   
    |X^{t,x_j,j,N}_{T,x_j}  X^{t,x_{k},{k},N}_{T,x_k}|^2
    + |X^{t,x_k,k,N}_{T,x_j}  X^{t,x_{j},{j},N}_{T,x_k}|^2
    \Big]
     \\
    \nonumber
    \eqref{eq: lemma5.2 line 10}
    &\leq  
    \frac{K}{N^3} \sum_{i=1,~i\neq j \neq k}^{N}
    \bE \Big[   |X^{t,x_j,j,N}_{T,x_j}  X^{t,x_{i},{i},N}_{T,x_k}  |^2   +|X^{t,x_i,i,N}_{T,x_j}  X^{t,x_{j},{j},N}_{T,x_k}  |^2
    \\ \nonumber 
    & \hspace{3.15cm}
    + 
    |X^{t,x_k,k,N}_{T,x_j}  X^{t,x_{i},i,N}_{T,x_k}  |^2
    +
    |X^{t,x_i,i,N}_{T,x_j}  X^{t,x_{k},{k},N}_{T,x_k}  |^2  \Big].
\end{align}
We now demonstrate how these estimates \eqref{eq: lemma5.2 line 3}--\eqref{eq: lemma5.2 line 11}, when combined together, can be upper bounded by the form given in \eqref{eq: bounds for high-deri on u} for $n=2$,  $\gamma=(j,k), ~j\neq k$:
\begin{align}
    \nonumber
    |\partial&_{x_{j},x_{k}}^2 u(t, \bodx)|^2
   \\\nonumber
    &\leq 
      K \sum_{m=0}^2 ~
        ~	
    \sum_{ \substack{    \ell \in  
    \bigcup_{k=1}^{2} \Pi_k^N,
    \\   \hco(\ell \bigcup \gamma) = \hco(\gamma)+m   }   }  
    ~N^{m-2\hco(\ell) }~ {\color{black} \sum_{
    \substack{     
    \alpha_1,\dots,\alpha_{|\ell|} \in  
    \bigcup_{k=1}^{2} \Pi_k^N, \\
    \bigcup_{i=1}^{|\ell|} \alpha_i \simeq \gamma          }} 
    ~\bE \Big[ ~ 
     \prod_{i=1}^{|\ell|}
    \Big|X_{T,
    x_{\alpha_{i,1}} ,\dots ,x_{\alpha_{i,|\alpha_i |}}  }^{t,x_{\ell_i},\ell_i,N} \Big|^2 
    ~ \Big]}
    \\\nonumber
    & = 
    K
    \bigg( \frac{1}{N^2} \bE \Big[   |X^{t,x_j,j,N}_{T,x_j,x_k} |^2 + |X^{t,x_k,k,N}_{T,x_j,x_k}|^2   
    \Big] 
    + 
    \frac{1}{N^2} \bE \Big[   |X^{t,x_j,j,N}_{T,x_j}  |^2~|X^{t,x_j,j,N}_{T,x_k}|^2 
    +  
    |X^{t,x_j,j,N}_{T,x_k} |^2~| X^{t,x_j,j,N}_{T,x_j}|^2 
    \\
    \nonumber 
    &\quad 
    +|X^{t,x_k,k,N}_{T,x_j} |^2~|X^{t,x_k,k,N}_{T,x_k}|^2   
    +|X^{t,x_k,k,N}_{T,x_k} |^2~|X^{t,x_k,k,N}_{T,x_j}|^2 
    \Big]    
     + 
     \frac{1}{N^4} 
     \bE \Big[   |X^{t,x_j,j,N}_{T,x_j} |^2~|X^{t,x_k,k,N}_{T,x_k}|^2 
      \\\nonumber
    &\quad 
     +|X^{t,x_j,j,N}_{T,x_k} |^2~|X^{t,x_k,k,N}_{T,x_j}|^2   
     +
      |X^{t,x_k,k,N}_{T,x_j} |^2~|X^{t,x_j,k,N}_{T,x_k}|^2
      +
      |X^{t,x_k,k,N}_{T,x_k} |^2~|X^{t,x_j,j,N}_{T,x_j}|^2 
    \Big]    
    \\
    \nonumber
    &\quad  + \frac{1}{N} \sum_{i=1,i\neq j \neq k}^N \bE \Big[   |X^{t,x_i,i,N}_{T,x_j,x_k} |^2  
    + |X^{t,x_i,i,N}_{T,x_k,x_j} |^2
    \Big]
    + 
    \frac{1}{N^3} \sum_{i=1,i\neq j \neq k}^N \bE \Big[   |X^{t,x_j,j,N}_{T,x_j} |^2~|X^{t,x_i,i,N}_{T,x_k}|^2
    \\
    \nonumber
    &\quad  
    +|X^{t,x_j,j,N}_{T,x_k} |^2~|X^{t,x_i,i,N}_{T,x_j}|^2 
    +|X^{t,x_i,i,N}_{T,x_j} |^2~|X^{t,x_j,j,N}_{T,x_k}|^2
    +|X^{t,x_i,i,N}_{T,x_k} |^2~|X^{t,x_j,j,N}_{T,x_j}|^2 
    + |X^{t,x_k,k,N}_{T,x_j} |^2~|X^{t,x_i,i,N}_{T,x_k}|^2
    \\
    \nonumber
    &\quad 
    +|X^{t,x_k,k,N}_{T,x_k} |^2~|X^{t,x_i,i,N}_{T,x_j}|^2
     + |X^{t,x_i,i,N}_{T,x_j} |^2~|X^{t,x_k,k,N}_{T,x_k}|^2
    +|X^{t,x_i,i,N}_{T,x_k} |^2~|X^{t,x_k,k,N}_{T,x_j}|^2
    \Big] 
    \\
    \nonumber
    &\quad 
    +  \frac{1}{N} \sum_{i=1,i\neq j \neq k}^N
    \bE \Big[   |X^{t,x_i,i,N}_{T,x_j} |^2~|X^{t,x_i,i,N}_{T,x_k}|^2
    +  |X^{t,x_i,i,N}_{T,x_k} |^2~|X^{t,x_i,i,N}_{T,x_j}|^2
    \Big]
     \\
     \label{eq: second deri of u example 222}
    &\quad  {  \color{black}      
    +  \frac{1}{N^2} \sum_{ 
    \substack{i,i'=1, i \neq j \neq k \\ i\neq i', i' \neq j \neq k  }
     }^{N}   } \bE \Big[ | X^{t,x_i,i,N}_{T,x_j}  |^2~|X^{t,x_{i'},{i'},N}_{T,x_k} |^2
    + 
    | X^{t,x_{i'},{i'},N}_{T,x_j}  |^2~|X^{t,x_i,i,N}_{T,x_k} |^2
    \Big]    
     \bigg). 
\end{align}
Comparing to the results of \eqref{eq: lemma5.2 line 3}--\eqref{eq: lemma5.2 line 11}, we can see that \eqref{eq: second deri of u example 222} contains more terms. For example, consider the term 
\begin{align*}
    \frac{1}{N} \sum_{i=1,i\neq j \neq k}^N \bE \Big[   |X^{t,x_i,i,N}_{T,x_j,x_k} |^2  
    + |X^{t,x_i,i,N}_{T,x_k,x_j} |^2
    \Big].
\end{align*}
For the case $m=1,\ell \in \Pi_1^N$, we not only take summation over $|X^{t,x_i,i,N}_{T,x_j,x_k}|^2$, but also consider $|X^{t,x_i,i,N}_{T,x_k,x_j}|^2$, so that the sum of the terms \eqref{eq: lemma5.2 line 3}--\eqref{eq: lemma5.2 line 11} is bounded above by \eqref{eq: second deri of u example 222}, verifying our result in the case $n=2$, $\gamma=(j,k), ~j\neq k$. 
\newline

\textit{Part 2. The bound \eqref{eq: bounds for high-deri on u}.} 
From the above estimates, one can see that the idea is to essentially partition the sums based on the relationship between indices $i, i',j,k$, to keep consistent orders of $N$. Having this separation trick in mind, for any $\gamma\in \Pi_n^N, ~ 1 \leq n \leq 6  $, we consider (recall the notation for $\hco(\cdot)$ and $|\cdot|$ in Definition \ref{def: set of sequence pi}):
\begin{align}
        & |\partial^n_{x_{\gamma_1},\dots,x_{\gamma_n}} u(t, \bodx)|^2 \notag
    \\ \notag
    &=   \Bigg| \mathbb{E}\Bigg[ ~ 
    \sum_{  \substack{     \alpha,\beta \in 
    \bigcup_{k=0}^{n-1} \Pi_k^N, \\
\gamma \setminus (\gamma_1) \in \alpha\shuffle\beta }  }
     \sum_{i=1}^{N}  \Big( \partial_{x_i}g(\bodX_T^{t,\boldsymbol{x},N}) \Big)_{x_{\alpha_1},\dots , x_{\alpha_{|\alpha|}} 	  }
     \big(
     X^{t,x_i,i,N}_{T,x_{\gamma_1}}   
     \big)_{x_{\beta_1},\dots, x_{\beta_{|\beta|}}   }
     \Bigg] \Bigg|^2
    \\ \notag
    &\leq 
    \Bigg|
     \bE \Bigg[ ~ 
    \sum_{m=1}^n 
    \sum_{\ell \in  \Pi_m^N}  
     \bigg(  \Big| \partial^m_{{x_{\ell_1},\dots,x_{\ell_m}} }g (\bodX_T^{t,\boldsymbol{x},N})  \Big|
    \sum_{
    \substack{    \alpha_1,\dots,\alpha_m \in 
    \bigcup_{k=1}^{n} \Pi_k^N, \\
    \bigcup_{i=1}^m \alpha_i \simeq \gamma }
       } 
     \prod_{i=1}^m
    \Big|X_{T,x_{\alpha_{i,1}},\dots,x_{\alpha_{i,|\alpha_i |}} }^{t,x_{\ell_i},\ell_i,N} \Big| 
    ~
    \bigg)
     \Bigg] \Bigg|^2
      \\ \label{eq: n variation line 4}
    &=
    \Bigg|
     \bE \Bigg[ ~
    \sum_{m=1}^n ~
    \sum_{
    \substack{    \ell \in  \bigcup_{k=1}^{n} \Pi_k^N,  \\   \hco(\ell) = m   }
       }  ~
     \bigg(  \Big| \partial^{|\ell|}_{{x_{\ell_1},\dots,x_{\ell_{|\ell|}} }}g
     (\bodX_T^{t,\boldsymbol{x},N})  \Big|
    \sum_{
    \substack{   \alpha_1,\dots,\alpha_{|\ell|} \in  \bigcup_{k=1}^{n} \Pi_k^N,   \\   \bigcup_{i=1}^{|\ell|} \alpha_i \simeq \gamma    }
      } 
    ~ \prod_{i=1}^{|\ell|}
    \Big|X_{T,x_{\alpha_{i,1}},\dots ,x_{\alpha_{i,|\alpha_i |}} }^{t,x_{\ell_i},\ell_i,N} \Big| 
    ~
    \bigg)
     \Bigg]\Bigg|^2
  \\ \label{eq: n variation line 5}
    &=
    \Bigg|
     \bE \Bigg[ ~ 
    \sum_{m=0}^n ~
    \sum_{\substack{     \ell \in  \bigcup_{k=1}^{n} \Pi_k^N,  \\    \hco(\ell \bigcup \gamma) = \hco(\gamma)+m }   }  
    ~
     \bigg(  \Big| \partial^{|\ell|}_{{x_{\ell_1},\dots,x_{\ell_{|\ell|}} }}g(\bodX_T^{t,\boldsymbol{x},N})  \Big|
    \sum_{\substack{   \alpha_1,\dots,\alpha_{|\ell|} \in  \bigcup_{k=1}^{n} \Pi_k^N,   \\   \bigcup_{i=1}^{|\ell|} \alpha_i \simeq \gamma   }  } 
    ~ \prod_{i=1}^{|\ell|}
    \Big|X_{T,x_{\alpha_{i,1}},\dots ,x_{\alpha_{i,|\alpha_i |}} }^{t,x_{\ell_i},\ell_i,N} \Big| 
    ~
    \bigg)
     \Bigg]\Bigg|^2
      \\ \notag
    &\leq 
    K \sum_{m=0}^n  N^{2m}~
     \bE \Bigg[ \Bigg|\frac{1}{N^{m}}
    \sum_{\substack{     \ell \in  \bigcup_{k=1}^{n} \Pi_k^N,  \\    \hco(\ell \bigcup \gamma) = \hco(\gamma)+m }   }  
     \bigg(  \Big| \partial^{|\ell|}_{{x_{\ell_1},\dots,x_{\ell_{|\ell|}} }}g(\bodX_T^{t,\boldsymbol{x},N})  \Big|
    \\ \notag
    & \quad \hspace{6cm} \cdot 
    \sum_{\substack{   \alpha_1,\dots,\alpha_{|\ell|} \in  \bigcup_{k=1}^{n} \Pi_k^N,   \\   \bigcup_{i=1}^{|\ell|} \alpha_i \simeq \gamma   }  } 
    ~ \prod_{i=1}^{|\ell|}
    \Big|X_{T,x_{\alpha_{i,1}},\dots ,x_{\alpha_{i,|\alpha_i |}} }^{t,x_{\ell_i},\ell_i,N} \Big| 
    ~
    \bigg)\Bigg|^2~
     \Bigg]
     \\ \label{eq: n variation jensen application}
    &\leq 
    K \sum_{m=0}^n ~
        ~	
    \sum_{\substack{     \ell \in  \bigcup_{k=1}^{n} \Pi_k^N,  \\    \hco(\ell \bigcup \gamma) = \hco(\gamma)+m  }   }  
    ~N^{m }~
    \bE \Bigg[ \Big| \partial^{|\ell|}_{{x_{\ell_1},\dots,x_{\ell_{|\ell|}} }}g(\bodX_T^{t,\boldsymbol{x},N})  \Big|^2
      \sum_{\substack{   \alpha_1,\dots,\alpha_{|\ell|} \in  \bigcup_{k=1}^{n} \Pi_k^N,   \\   \bigcup_{i=1}^{|\ell|} \alpha_i \simeq \gamma    }  } 
    ~ \prod_{i=1}^{|\ell|}
    \Big|X_{T,x_{\alpha_{i,1}},\dots ,x_{\alpha_{i,|\alpha_i |}} }^{t,x_{\ell_i},\ell_i,N} \Big|^2 
    ~ \Bigg]
    \\ 
    &\leq 
    K \sum_{m=0}^n ~
        ~	
    \sum_{\substack{     \ell \in  \bigcup_{k=1}^{n} \Pi_k^N ,  \\    \hco(\ell \bigcup \gamma) = \hco(\gamma)+m  }   }  
    ~N^{m-2\hco(\ell) }~
    \bE \Bigg[ 
      \sum_{\substack{   \alpha_1,\dots,\alpha_{|\ell|} \in  \bigcup_{k=1}^{n} \Pi_k^N,   \\   \bigcup_{i=1}^{|\ell|} \alpha_i \simeq \gamma     }  } 
    ~ \prod_{i=1}^{|\ell|}
    \Big|X_{T,x_{\alpha_{i,1}},\dots ,x_{\alpha_{i,|\alpha_i |}} }^{t,x_{\ell_i},\ell_i,N} \Big|^2 
    ~ \Bigg], \notag
\end{align}
{\color{black}where in \eqref{eq: n variation line 4} and \eqref{eq: n variation line 5}, we regroup the summation over all $\ell \in \bigcup_{k=1}^{n} \Pi_k^N $ based on the magnitude of $\hco(\ell) $ and $\hco(\ell \bigcup \gamma)$}.  
In \eqref{eq: n variation jensen application}, we apply Jensen's inequality to the second summation where the set $\{ \ell \in  \bigcup_{k=1}^{n} \Pi_k^N:  \hco(\ell \bigcup \gamma) = \hco(\gamma)+m\}$  has $\cO(N^m)$ terms ($\ell$ has $m$ degrees of freedom), thus we end up with a factor of $N^m$ after calculation. For the last line, we used Assumption~\ref{assum:main_weak error}: $\big|\partial^{|\ell|}_{x_{\ell_1},\dots,x_{\ell_{|\ell|}}} g \big|_{\infty} = \cO(N^{-\hco(\ell)})$ . 
\newline
 
 \textit{Part 3: The bound \eqref{eq: bounds for high-deri on u order 4}.}  
One proves \eqref{eq: bounds for high-deri on u order 4} using the same arguments as in \textit{Part 2}. We have 
\begin{align*}
        & |\partial^n_{x_{\gamma_1},\dots,x_{\gamma_n}} u(t, \bodx)|^4 
    \\
    &\leq 
    K \sum_{m=0}^n ~
        ~	
    \sum_{\substack{      \ell \in  \bigcup_{k=1}^{n} \Pi_k^N , 
    \\   
    \hco(\ell \bigcup \gamma) = \hco(\gamma)+m  }  }   
    ~N^{3m }~
    \bE \Bigg[ \Big| \partial^{|\ell|}_{{x_{\ell_1},\dots,x_{\ell_{|\ell|}} }}g(\bodX_T^{t,\boldsymbol{x},N})  \Big|^4
      \sum_{\substack{     \alpha_1,\dots,\alpha_{|\ell|} \in  \bigcup_{k=1}^{n} \Pi_k^N,  \\      \bigcup_{i=1}^{|\ell|} \alpha_i \simeq \gamma  } } 
    ~ \prod_{i=1}^{|\ell|}
    \Big|X_{T,x_{\alpha_{i,1}},\dots ,x_{\alpha_{i,|\alpha_i |}} }^{t,x_{\ell_i},\ell_i,N} \Big|^4 
    ~ \Bigg]
    \\
    &{\color{black}\leq 
    K \sum_{m=0}^n ~
        ~	
    \sum_{\substack{      \ell \in  \bigcup_{k=1}^{n} \Pi_k^N ,
    \\      \hco(\ell \bigcup \gamma) = \hco(\gamma)+m }  }  
    ~N^{3m-4\hco(\ell) }~
    \bE \Bigg[ 
      \sum_{\substack{      \alpha_1,\dots,\alpha_{|\ell|} \in  \bigcup_{k=1}^{n} \Pi_k^N,
      \\    \bigcup_{i=1}^{|\ell|} \alpha_i \simeq \gamma  }   } 
    ~ \prod_{i=1}^{|\ell|}
    \Big|X_{T,x_{\alpha_{i,1}},\dots ,x_{\alpha_{i,|\alpha_i |}} }^{t,x_{\ell_i},\ell_i,N} \Big|^4 
    ~ \Bigg].}
\end{align*}
 \end{proof}

\begin{lemma}
\label{lemma: higher-order derivative of u}
Let $u$ satisfy the Kolmogorov backward equation \eqref{PDE:Kolmogorov} with $g$ as in Assumption~\ref{assum:main_weak error},  
 let $T\geq t \geq 0, N\in \mathbb{N}$ and assume that the starting points $x_i$ are $\mathcal{F}_t$-measurable random variables in $L^{2}(\Omega,\mathbb{R})$ sampled from the same distribution for all $ i  \in \lbrace 1, \ldots, N \rbrace$.
Then {\color{black}there exist positive constants $K$, $\lambda_0\in(0,\lambda)$ (both are independent of $t,T,N$)} such that for any $1 \leq n \leq 6$, $\gamma\in \Pi_n^N $
\begin{align}
\label{eq:  high-deri on u order 2 bounded by K}
  \bE \Big[   \big|\partial^n_{x_{\gamma_1},\dots,x_{\gamma_n}} u(t, \bodx)\big|^2\Big]
    & 
    \leq  K e^{- \lambda_0 (T-t)} N^{-2  \hco(\gamma)}, 
\\
\label{eq:  high-deri on u order 4 bounded by K}
  \bE \Big[   \big|\partial^n_{x_{\gamma_1},\dots,x_{\gamma_n}} u(t, \bodx)\big|^4\Big]
    &\leq 
    K e^{- \lambda_0 (T-t)} N^{-4  \hco(\gamma)}.
\end{align}

\end{lemma}

\begin{proof}
    Note that in the following proof, the positive constants $K,\lambda_0$ are independent of $t,T,N$ and may change line by line. 
    Recall the results and notations in Lemma~\ref{lemma: higher-order derivative of u without exp}, after taking the expectation, we have 
    \begin{align*}
         \bE \Big[  \big|\partial&_{x_{\gamma_1},\dots,x_{\gamma_n}}^n u(t, \bodx) \big|^2 \Big] 
          \\\nonumber
    &\leq 
      K \sum_{m=0}^n ~
        ~	
    \sum_{ \substack{    \ell \in  
    \bigcup_{k=1}^{n} \Pi_k^N,  
    \\   \hco(\ell \bigcup \gamma) = \hco(\gamma)+m   }   }  
    ~N^{m-2\hco(\ell) }~ \sum_{
    \substack{     
    \alpha_1,\dots,\alpha_{|\ell|} \in  
    \bigcup_{k=1}^{n} \Pi_k^N, \\
    \bigcup_{i=1}^{|\ell|} \alpha_i \simeq \gamma }} 
    ~\bE \Big[ ~ 
     \prod_{i=1}^{|\ell|}
    \Big|X_{T,
    x_{\alpha_{i,1}} ,\dots ,x_{\alpha_{i,|\alpha_i |}}  }^{t,x_{\ell_i},\ell_i,N} \Big|^2 
    ~ \Big]
   \\\nonumber
    &\leq 
      K \sum_{m=0}^n ~
        ~	
    \sum_{\substack{     \ell \in  \bigcup_{k=1}^{n} \Pi_k^N,   \\      \hco(\ell \bigcup \gamma) = \hco(\gamma)+m}  }  
    ~N^{m-2\hco(\ell) }~ \sum_{\substack{    \alpha_1,\dots,\alpha_{|\ell|} \in  \bigcup_{k=1}^{n} \Pi_k^N,  
    \\      \bigcup_{i=1}^{|\ell|} \alpha_i \simeq \gamma  }  } 
    \prod_{i=1}^{|\ell|}
    \Big( \bE \Big[ ~  
    \Big|X_{T,x_{\alpha_{i,1}},\dots ,x_{\alpha_{i,|\alpha_i |}} }^{t,x_{\ell_i},\ell_i,N} \Big|^{2|\ell|} 
    ~ \Big] \Big)  ^{1/|\ell|},
    \end{align*}
  where we employed H\"older's inequality. We have the following estimate for the variation processes in the above product:
    \begin{align*}
    \Big( \bE \Big[ ~  
    \Big|X_{T,x_{\alpha_{i,1}},\dots ,x_{\alpha_{i,|\alpha_i |}} }^{t,x_{\ell_i},\ell_i,N} \Big|^{2|\ell|} 
    ~ \Big] \Big)  ^{1/|\ell|}
    &
    \leq 
    \bigg( 
    \frac{K}{N^{\hco ( (\ell_i) \bigcup \alpha_i  )}} \sum_{\substack{      \beta \in \Pi_{   |\alpha_i|+1  }^N,
    \\      ~\hco(\beta) = \hco (  (\ell_i) \bigcup \alpha_i  )}}
    \mathbb{E}\Big[~    \Big|X^{t,x_{\beta_1},\beta_1,N}_{T,x_{\beta_2},\dots,x_{\beta_{|\beta |}}}
    \Big|^{2|\ell|} \Big]
    \bigg)  ^{1/|\ell|}
    \\
    &\leq 
        \Big( \frac{K}{N^{2|\ell|(\hco ( (\ell_i) \bigcup \alpha_i  )-1)}} e^{-\lambda_0 2|\ell|(T-t)}\Big)  ^{1/|\ell|}
        \\
        &
        \leq K e^{-2\lambda_0 (T-t)} N^{-2(\hco (  (\ell_i) \bigcup \alpha_i  )-1) },
    \end{align*}
    where we used the second and first part of Lemma~\ref{lemma: n-var process result} with $p=2|\ell|$, $m=\hco (  (\ell_i) \bigcup \alpha_i  )$ to obtain the first two inequalities. {\color{black}The summation in the first inequality  contains $\cO(N^{\hco ( (\ell_i) \bigcup \alpha_i  )})$ terms that are identically distributed to the left-hand side process, and it also contains more terms.}
    \footnote{\label{footnote4} { \color{black}For example, the process $X_{T,x_{\alpha_{i,1}},\dots ,x_{\alpha_{i,|\alpha_i |}} }^{t,x_{\ell_i},\ell_i,N}$  
    with $\ell_i =1 ,\alpha_i=(1,2,3)$ is identically distributed to  the process with $\ell_i = 2, \alpha_i=(2,3,4)$ but not necessarily identically distributed to the process with $\ell_i = 3, \alpha_i=(4,2,2)$, though they both satisfy $\hco(\ell_i\cup \alpha_i) =3$ due to the different derivative patterns. }}

    We now relate the orders $(\hco ( (\ell_i) \bigcup \alpha_i  )-1)$ from the previous estimate and $\hco(\gamma)$ appearing in  
    \eqref{eq:  high-deri on u order 2 bounded by K}, by first showing that 
    $ \hco( \gamma \setminus \ell   ) \leq \sum_{i=1}^{|\ell|} \Big(\hco (  (\ell_i) \bigcup \alpha_i  ) -1\Big)$. Concretely, the constraint in the second summation, $  \bigcup_{i=1}^{|\ell|} \alpha_i \simeq \gamma $, implies 
 \begin{align*}
        \sum_{i=1}^{|\ell|} \Big(\hco ( (\ell_i) \bigcup \alpha_i  ) -1\Big)
        &=
        \sum_{i=1}^{|\ell|} \bigg(\hco \Big( (\ell_i) \bigcup
        \Big(   \alpha_i \bigcap (\gamma \setminus \ell) \Big)
         \bigcup
        \Big(   \alpha_i \setminus (\gamma \setminus \ell) \Big)
        \Big) 
        -1 
        \bigg)
        \\
        &
        =
        \sum_{i=1}^{|\ell|} \bigg(\hco \Big(
        \alpha_i \bigcap (\gamma \setminus \ell) 
        \Big)
        +
        \hco \Big( (\ell_i)
         \bigcup
        \Big(   \alpha_i \setminus (\gamma \setminus \ell) \Big)
        \Big) 
        -1
        \bigg)
        \\
        &\geq 
        \sum_{i=1}^{|\ell|}  \hco \Big(   \alpha_i \bigcap (\gamma \setminus \ell) \Big)  
         \geq  
        \hco\Big(   ( \bigcup_{i=1}^{|\ell|}\alpha_i) \bigcap (\gamma \setminus \ell) \Big)
        = 
        \hco( \gamma \setminus \ell   ).
    \end{align*}  
    Then using the constraint $ \hco(\ell \bigcup \gamma) = \hco(\gamma)+m$ of the first summation we infer 
\begin{align*}
    \sum_{i=1}^{|\ell|} \Big(\hco ( (\ell_i)  \bigcup \alpha_i  ) -1\Big)
        \geq  
        {  \color{black}      
        \hco( \gamma \setminus \ell   )
        \geq 
        \hco(\ell \bigcup \gamma)- \hco (\ell)  }
        =
        \hco(\gamma)+m - \hco (\ell).
\end{align*}
Hence,
\allowdisplaybreaks
    \begin{align*}
         \bE \Big[  |\partial&_{x_{\gamma_1},\dots,x_{\gamma_n}}^n u(t, \bodx)|^2 \Big] 
   \\ 
    &\leq 
      K e^{- 2\lambda_0 (T-t)} \sum_{m=0}^n ~
        ~	
    \sum_{\substack{     \ell \in  \bigcup_{k=1}^{n} \Pi_k^N,  \\ \hco(\ell \bigcup \gamma) = \hco(\gamma)+m     }   }  
    ~N^{m-2\hco(\ell) }~  
    N^{-2 \sum_{i=1}^{|\ell|} \big(\hco ( (\ell_i) \bigcup \alpha_i)  -1\big)}
     \\ 
    &
    \leq  K e^{- 2\lambda_0 (T-t)}
    \sum_{m=0}^n
    N^{2m-2\hco(\ell)-2 \big(  \hco(\gamma) - \hco (\ell)+m \big) }
    \leq  K e^{- 2\lambda_0 (T-t)} N^{-2  \hco(\gamma)},
    \end{align*}
which yields \eqref{eq:  high-deri on u order 2 bounded by K}, as sought. Similar calculations deliver \eqref{eq:  high-deri on u order 4 bounded by K}. That is, we have 
\allowdisplaybreaks
     \begin{align*}
         \bE \Big[  |\partial&_{x_{\gamma_1},\dots,x_{\gamma_n}}^n u(t, \bodx)|^4 \Big] 
   \\\nonumber
    &\leq 
      K \sum_{m=0}^n ~
        ~	
    \sum_{\substack{      \ell \in  \bigcup_{k=1}^{n} \Pi_k^N, 
    \\     \hco(\ell \bigcup \gamma) = \hco(\gamma)+m  }  }  
    ~N^{3m - 4\hco(\ell) }~ \sum_{\substack{      \alpha_1,\dots,\alpha_{|\ell|} \in  \bigcup_{k=1}^{n} \Pi_k^N,\\     \bigcup_{i=1}^{|\ell|} \alpha_i \simeq \gamma }  } 
    \prod_{i=1}^{|\ell|}
    \Big( \bE \Big[ ~  
    \Big|X_{T,x_{\alpha_{i,1}},\dots,x_{\alpha_{i,|\alpha_i |}} }^{t,x_{\ell_i},\ell_i,N} \Big|^{4|\ell|} 
    ~ \Big] \Big)  ^{1/|\ell|}
    \\
    &\leq 
      K e^{- 2\lambda_0 (T-t)} \sum_{m=0}^n ~
        ~	
    \sum_{\substack{    \ell \in \bigcup_{k=1}^{n} \Pi_k^N,    \\      \hco(\ell \bigcup \gamma) = \hco(\gamma)+m}  }  
    ~N^{3m-4\hco(\ell) }~  
    N^{-4 \sum_{i=1}^{|\ell|} \big(\hco ( (\ell_i) \bigcup \alpha_i)  -1\big) }
     \\ 
    &
    \leq  K e^{- 2\lambda_0 (T-t)} 
    \sum_{m=0}^n
    N^{4m-4\hco(\ell)-4 \big(  \hco(\gamma) - \hco (\ell)+m \big) }
    \leq  
    K e^{- 2\lambda_0 (T-t)} N^{-4  \hco(\gamma)}.
    \end{align*}
\end{proof}

\section{Weak error expansion and its analysis}
\label{section: weak error expansion}

For the convenience of the reader, we recall Lemma~\ref{lemma:Weak Expansion Leimkuhler} which provides an expansion for the global weak error \eqref{eq:Weak Error} under Assumption~\ref{assum:main_weak error}, for the processes defined in \eqref{ModelIPS} and \eqref{eq: def : non-Markov Euler scheme} as follows:
\begin{align}
    \label{eq: weak error expansion-ori1}
    \mathbb{E} \Big[g(\boldsymbol{X}^{N}_{T}) \Big] - \mathbb{E} \Big[g(\boldsymbol{X}^{N,h}_{T}) \Big] &= h^2 \mathbb{E} \left[\sum_{m=0}^{M-1} L(t_m,\boldsymbol{X}^{N,h}_{t_m}) \right]  +   \mathbb{E} \left[ \sum_{m=0}^{M-1} R(t_m,\boldsymbol{X}^{N,h}_{t_m}) \right],
\end{align}
where  the map $L:\bR_{+}\times \bR^N \to \bR$ is defined via the maps $u$ and $( B_i )_{i \in \lbrace 1,\ldots,N \rbrace}$: 
\allowdisplaybreaks
\begin{align} 
    \nonumber
    L(t,\boldsymbol{x}) &= \frac{1}{2} \Big[ \sum_{i,j=1}^{N}B_{j}(\boldsymbol{x}) \partial_{x_j} B_{i}(\boldsymbol{x}) \partial_{x_i}u(t,\boldsymbol{x}) 
      +\frac{\sigma^2}{2} \sum_{i,j=1}^{N} \partial_{x_j}B_i(\boldsymbol{x})  \partial_{x_i, x_j}^2 u(t,\boldsymbol{x}) 
    \\
    \label{eq: B0 formula}
    & \quad \quad \quad + \frac{\sigma^2}{2} \sum_{i,j=1}^{N} \partial^2_{x_j, x_j}B_i(\boldsymbol{x})  \partial_{x_i} u(t,\boldsymbol{x}) \Big].
\end{align}
The remainder term $R(\cdot,\cdot)$ will later be written as a linear combination of 8 remainder terms, which we will analyze in Section \ref{sec:Remainder terms R}.

\subsection[Estimates for the leading error term L]{Estimates for the leading error term {$L$}}
\label{section: proof of B0 terms}

We consider the first term in \eqref{eq: weak error expansion-ori1} expressed as a telescoping sum:  
\begin{align}
    \label{eq: C0 term}
     h^2\bE \Big[  \sum_{m=0}^{M-1} L(t_m,\bodX_{t_m}^{N,h})  \Big]
     & =
      h\bE \Big[   \int_{0}^T  L(s,\bodX_{s}^{N}) \dd s \Big]
      \\  \label{eq: rbh term}
      &\qquad+ h \sum_{m=0}^{M-1}
      \bE \Big[   \int_{t_m}^{t_{m+1}}
        \Big(  L(t_m,\bodX_{t_m}^{N,h})
        -  L(s,\bodX_{s}^{N})
        \Big)
      \dd s \Big], 
\end{align}
and derive the following result. 
\begin{lemma}
    \label{lemma: analysis of the b0 term summation integration}
    Let  Assumption~\ref{assum:main_weak error} hold and let $\xi \in L^{4}(\Omega,\mathbb{R})$.
    Let $L$ be defined in \eqref{eq: B0 formula}. Then there exists a  positive constant $\lambda_0\in (0,\lambda)$ such that
    \begin{align}
    \label{eq: B0 terms results}
     h^2\bE \Big[  &\sum_{m=0}^{M-1} L(t_m,\bodX_{t_m}^{N,h})  \Big]
    \leq  
     Kh^{3/2}+Khe^{-\lambda_0 T}.
\end{align}
\end{lemma}

\begin{proof}
Note that in the following proof, the positive constants $K,\lambda_0$ are independent of $t,T,N$ and may change line by line. The proof is carried out by analysing \eqref{eq: C0 term} and \eqref{eq: rbh term} separately to reach \eqref{eq: B0 terms results}. \\

\textit{Part 1: Estimating \eqref{eq: C0 term}.}  
Let $(\overline{ X}_{0}^{1,N}, \ldots,\overline{ X}_{0}^{N,N}) = \olbodX_0^{N}\sim \mu^{N,*}$, where $\mu^{N,*}$ is the stationary distribution of the IPS (viewed as a $\mathbb{R}^N$-valued SDE). 
\color{black}
The next result is shown in Lemma \eqref{lemma:ZeroMeanResult} for the benefit of the reader and it follows using integration by parts -- the argument can be found in \cite[Equation~(3.29)]{leimkuhler2014long} and leverages the known closed form (up to a scaling constant) of the invariant density map $\mu^{N,*}$ \eqref{eq:NparticleGibbsDistro} and the definition of the operator $L$ in \eqref{eq: B0 formula}. 
\color{black}
For any $t\geq 0$,  we then have  
\begin{align*}
      \bE \Big[ L(t,\olbodX_t^{N})  \Big] = \bE \Big[ L(0,\olbodX_0^{N})  \Big] = \int_{\bR^N} L(0,\boldsymbol{x}) \mu^{N,*}(\dd \boldsymbol{x}) =0,  
\end{align*}
where we recall the density $\mu^{N,*}$ for the particle system (that corresponds to \eqref{eq.INTRO.ergodicDistro}) 
\begin{align}
\label{eq:NparticleGibbsDistro}
        \mu^{N,*}(  \boldsymbol{x})
        \propto
        \exp \Big(  -\frac{2}{\sigma^2}  \sum_{i=1}^{N} U(x_i) - \frac{1}{\sigma^2N} \sum_{i=1}^{N}\sum_{j=1}^{N}  V(x_i - x_j) \Big).  
\end{align}

We can now start the proof of our result. We may then write
\begin{align*}
    \bE \Big[   \int_{0}^T  L(s,\bodX_{s}^{N}) \dd s \Big] 
    &= 
       \int_{0}^T  \bE \Big[ L(s,\bodX_{s}^{N})  - 
        L(s,\olbodX_s^{N})
       \Big] \dd s
    \\
    & = 
    \int_{0}^T 
    \int_{0}^1  
    {  \color{black}       
    \bE \Big[
         \big\langle   \partial_{\bodx} L\big(s,
         \rho\bodX_{s}^{N}+(1-\rho)\olbodX_s^N\big)
          ,
            ~
         \bodX_s^{N} -\olbodX_s^{N}   \big\rangle      
    \Big]  }\dd \rho ~\dd s.
\end{align*}
Let $\bodX_{s,\rho}^{N} \coloneqq \rho\bodX_{s}^{N}+(1-\rho)\olbodX_s^{N}$. Using the chain rule, we deduce the following: 
\allowdisplaybreaks
\begin{align}
\nonumber
    \bE \Big[   & \Big\langle
            \partial_{\bodx} L(s,\bodX_{s,\rho}^{N})   ,   
            ~
         \bodX_s^{N} -\olbodX_s^{N}  \Big\rangle     
    \Big]
    =
     \frac{1}{2}\bE\bigg[  
     \sum_{i,j,k=1}^{N}
       \Big(  
     \partial_{x_k}B_{j}(\bodX_{s,\rho}^{N}) \partial_{x_j} B_{i}(\bodX_{s,\rho}^{N}) \partial_{x_i}u(s,\bodX_{s,\rho}^{N})
    \\\nonumber& \quad\quad 
     +
     B_{j}(\bodX_{s,\rho}^{N}) \partial^2_{x_j,x_k} B_{i}(\bodX_{s,\rho}^{N}) \partial_{x_i}u(s,\bodX_{s,\rho}^{N}) 
     +
     B_{j}(\bodX_{s,\rho}^{N}) \partial_{x_j} B_{i}(\bodX_{s,\rho}^{N}) \partial_{x_i,x_k}^2u(s,\bodX_{s,\rho}^{N}) 
       \Big)   \cdot	\Big(   
       X_{s}^{k,N} - \overline{ X}_s^{k,N}
     \Big)
     \\\nonumber
     & \quad\quad 
     +\frac{\sigma^2}{2} \sum_{i,j,k=1}^{N} \Big(      \partial^2_{x_j,x_k}B_i(\bodX_{s,\rho}^{N})  \partial^2_{x_i, x_j} u({s},\bodX_{s,\rho}^{N})
     +
     \partial_{x_j}B_i(\bodX_{s,\rho}^{N})  \partial^3_{x_i, x_j,x_k} u({s},\bodX_{s,\rho}^{N})
       \Big)   \cdot	\Big(  
       X_{s}^{k,N} - \overline{ X}_s^{k,N}
      \Big)
     \\\nonumber
     & \quad\quad
     + \frac{\sigma^2}{2} \sum_{i,j,k =1}^{N}  \Big(  \partial^3_{x_j,x_j,x_k}B_i(\bodX_{s,\rho}^{N})  \partial_{x_i} u({s},\bodX_{s,\rho}^{N})
     +
     \partial^2_{x_j,x_j}B_i(\bodX_{s,\rho}^{N})  \partial^2_{x_i,x_k} u({s},\bodX_{s,\rho}^{N})
       \Big)   \cdot	\Big(   
       X_{s}^{k,N} - \overline{ X}_s^{k,N}   
      \Big)
     \bigg]  
     \\\nonumber
     &\leq 
     K  \sum_{i,j,k =1}^{N}    
       \sqrt{\bE \Big[       \big|  X_{s}^{k,N} - \overline{ X}_s^{k,N}  \big|^2 \Big]  }
    \cdot  \bigg\{ 
     \sqrt{\bE \Big[       \big|     \partial_{x_k}B_{j}(\bodX_{s,\rho}^{N}) \partial_{x_j} B_{i}(\bodX_{s,\rho}^{N}) \partial_{x_i}u(s,\bodX_{s,\rho}^{N})   \big|^2 \Big]  }
    \\\nonumber
    & \quad\quad +
    \sqrt{\bE \Big[       \big|        B_{j}(\bodX_{s,\rho}^{N}) \partial^2_{x_j,x_k} B_{i}(\bodX_{s,\rho}^{N}) \partial_{x_i}u(s,\bodX_{s,\rho}^{N}) 
       \big|^2 \Big]  }
    +       \sqrt{\bE \Big[       \big|    B_{j}(\bodX_{s,\rho}^{N}) \partial_{x_j} B_{i}(\bodX_{s,\rho}^{N}) \partial^2_{x_i,x_k}u(s,\bodX_{s,\rho}^{N})    \big|^2 \Big]  }
    \\\nonumber
    & \quad\quad +
    \sqrt{\bE \Big[       \big|     \partial^2_{x_j,x_k}B_i(\bodX_{s,\rho}^{N})  \partial_{x_i, x_j}^2 u({s},\bodX_{s,\rho}^{N})  \big|^2 \Big]  }
    +
    \sqrt{\bE \Big[       \big|     \partial_{x_j}B_i(\bodX_{s,\rho}^{N})  \partial_{x_i,x_j,x_k}^3 u({s},\bodX_{s,\rho}^{N})
   \big|^2 \Big]  }
    \\
    \label{eq-aux:7-root-Expectations}
    & \quad\quad +
    \sqrt{\bE \Big[       \big|     \partial^3_{x_j,x_j,x_k}B_i(\bodX_{s,\rho}^{N})  \partial_{x_i} u({s},\bodX_{s,\rho}^{N})   \big|^2 \Big]  }
    +
    \sqrt{\bE \Big[       \big|     \partial^2_{x_j,x_j}B_i(\bodX_{s,\rho}^{N})  \partial_{x_i,x_k}^2 u({s},\bodX_{s,\rho}^{N})   \big|^2 \Big]  }~\bigg\}, 
\end{align}
where we used  the Cauchy--Schwarz inequality.  
We now work through \eqref{eq-aux:7-root-Expectations}. 
As for the $L^2$-distance to the invariant distribution (the first expectation in the sum), note that \eqref{eq:  2 moment difference } in Lemma~\ref{lemma: 2 moment difference} implies that for any $s \geq 0$ we have $\bE [ |  X_{s}^{k,N} - \overline{ X}_s^{k,N}   |^2] \leq K e^{-2\lambda s} $. 
The approach to deal with the remaining seven terms is more or less identical. We inject the estimate \eqref{eq:  high-deri on u order 2 bounded by K} of Lemma~\ref{lemma: higher-order derivative of u} for the derivatives of $u$ and the bounds for the derivatives of $B$ established in \eqref{eq: bounds of the summation of derivatives of B noiid}. The inequality below preserves the exact ordering of \eqref{eq-aux:7-root-Expectations}, and we highlight that obtaining \eqref{eq-aux:Extra Argument for linear growth of Bj} requires the additional use of the linear growth of $B_j$, the Cauchy--Schwarz inequality, the $L^4$-estimates of $\bodX^N$ (in Proposition~\ref{prop:basic_estimates11}) and the $L^4$-estimate \eqref{eq:  high-deri on u order 4 bounded by K} for the derivatives of $u$ (recalling that $\xi \in L^{4}(\Omega,\bR)$); for some positive constant $\lambda_0 \in (0,\lambda)$ chosen small enough we have 
\begin{align}
\nonumber
\bE \Big[   & \Big\langle
            \partial_{\bodx} L(s,\bodX_{s,\rho}^{N})   ,   
            ~
         \bodX_s^{N} -\olbodX_s^{N}  \Big\rangle     
    \Big] 
   \\ \nonumber 
   & \leq 
    K \sqrt{
   e^{-2\lambda s}  } \cdot  \sqrt{ e^{-\lambda_0 (T-s)}
    } \sum_{\gamma \in \Pi^N_3} 
    \Big( 
    \frac{1}{N^{  (\hco((\gamma_2,\gamma_3))-1)
                 +(\hco((\gamma_1,\gamma_2))-1) 
                 +1 }  
            }
     \\
    \label{eq-aux:Extra Argument for linear growth of Bj}
    & 
    \hspace{5cm}
    + \frac{1}{N^{  ( \hco( \gamma)-1)+1 } }
   + \frac{1}{N^{  (\hco((\gamma_1,\gamma_2))-1) +  \hco((\gamma_1,\gamma_3))}  }
    \\
    \nonumber 
    & 
    \hspace{5cm}
    + \frac{1}{N^{  (\hco( \gamma)-1) +  \hco((\gamma_1,\gamma_2))}  }
    + \frac{1}{N^{    (\hco((\gamma_1,\gamma_2))-1)+ \hco( \gamma)}  }
        \\
    \nonumber 
    & 
    \hspace{5cm}
    + \frac{1}{N^{  (\hco( \gamma)-1) +1} } 
    + \frac{1}{N^{    (\hco((\gamma_1,\gamma_2))-1)
        +  \hco((\gamma_1,\gamma_3))  }} \Big) 
    \\  
    \label{eq: rb: result of partial x B0 part1} 
    & 
    \leq  
    K \sqrt{
   e^{-2\lambda s}     e^{-\lambda_0 (T-s)}
    } 
    \sum_{\gamma \in \Pi^N_3} \frac{1}{N^{  \hco( \gamma)  } }
    \\
    \label{eq: rb: result of partial x B0}
    & 
    \leq K \sqrt{
   e^{-2\lambda s}     e^{-\lambda_0 (T-s)}
    }
    ,
\end{align}
where the inequality in \eqref{eq: rb: result of partial x B0 part1} follows from the fact that $\hco((\gamma_1,\gamma_2))+  \hco((\gamma_1,\gamma_3))-1\geq \hco(\gamma)$ for any $\gamma \in \Pi_3^N$ (seen by checking the cases). 
The final result \eqref{eq: rb: result of partial x B0} follows from recalling \eqref{eq-aux:countingNumbers}, in turn implying that the summation term in \eqref{eq: rb: result of partial x B0 part1} is indeed $\cO(1)$.

\color{black}

To conclude this first part of the proof, we gather our estimates and obtain  
\begin{align*}
    \bE \Big[   \int_{0}^T  L(s,\bodX_{s}^{N}) \dd s \Big] 
   &
   \leq
    K
    \int_{0}^T
    \int_{0}^1
    \sqrt{
    e^{-\lambda_0 (T-s)}   e^{-2\lambda s} 
    }~
    \dd \rho 
    \mathrm{d}s 
    \\
    & 
    =  
    K
    e^{-\frac{\lambda_0}{2} T}  
    \int_{0}^T
    e^{\frac{\lambda_0-2\lambda}{2} s} \mathrm{d}s 
    \leq 
    K e^{-\frac{\lambda_0}{2} T}.
\end{align*}

\textit{Part 2: Estimating \eqref{eq: rbh term}.} 
For the term \eqref{eq: rbh term}, we have 
\begin{align}
     \nonumber
    \sum_{m=0}^{M-1} 
      \bE \Big[   \int_{t_m}^{t_{m+1}}
        \Big(  L(t_m,\bodX_{t_m}^{N,h})
&
        -  L(s,\bodX_{s}^{N})
        \Big)
      \dd s \Big]
      \\ 
         \label{eq:  rb sum term1}
      &= \sum_{m=0}^{M-1}  
         \int_{t_m}^{t_{m+1}}\bE \Big[ 
           L(t_m,\bodX_{t_m}^{N,h})
        -  L(t_m,\bodX_{t_m}^{N})
         \Big]
      \dd s 
      \\    \label{eq:  rb sum term2}
      &\quad +
      \sum_{m=0}^{M-1}  
          \int_{t_m}^{t_{m+1}}\bE \Big[
            L(t_m,\bodX_{t_m}^{N})
        -   L(s,\bodX_{t_m}^{N})
         \Big] 
      \dd s
      \\    \label{eq:  rb sum term3}
      &\quad+
      \sum_{m=0}^{M-1}  
        \int_{t_m}^{t_{m+1}}\bE \Big[  
            L(s,\bodX_{t_m}^{N})
        -   L(s,\bodX_{s}^{N})
           \Big] 
      \dd s. 
\end{align} 

\textit{Part 2.1: Estimating \eqref{eq:  rb sum term1} and \eqref{eq:  rb sum term3}.}  
For \eqref{eq:  rb sum term1}, similar to the calculations \eqref{eq-aux:7-root-Expectations}--\eqref{eq: rb: result of partial x B0} (in the previous part of the proof), we derive
\begin{align}
    \nonumber
    \bE \Big[ &
          L(t_m,\bodX_{t_m}^{N,h})
        -  L(t_m,\bodX_{t_m}^{N})
         \Big]
         \\
         \nonumber
        &= \int_{0}^1     
        \bE \Big[  
            \Big\langle  \partial_{\bodx} L\big(t_m,\rho\bodX_{t_m}^{N,h}+(1-\rho)\bodX_{t_m}^{N}\big)
          ,  
            ~
         \bodX_{t_m}^{N,h} -\bodX_{t_m}^{N}  \Big\rangle   
        \Big] 
        \dd \rho 
        \\
        \label{eq: b0 term 1 result}
        &\leq 
        \frac{K}{N} \sum_{i=1}^N
        ~
        \sqrt{  e^{-\lambda_0 (T-t_m)}   \bE \Big[  
        \big|X_{t_m}^{i,N,h}-X_{t_m}^{i,N}   \big|^2 
        \Big]  ~} 
        \leq K  h^{1/2} e^{-\lambda_0 (T-t_{m+1})/2},
\end{align}
where we used Proposition~\ref{prop:basic_estimates22} for the strong error rate.{ Similarly, for \eqref{eq:  rb sum term3}, we have for $s \in [t_m,t_{m+1}]$},
\begin{align}
    \nonumber
    \bE \Big[  
          L(s,\bodX_{t_m}^{N})
        -  L(s,\bodX_{s}^{N})
         \Big]         
         &
        = \int_{0}^1 
        \bE \Big[   \Big\langle  
        \partial_{\bodx} L\big(s, \rho\bodX_{t_m}^{N}
        + (1- \rho) \bodX_{s}^{N}
        \big)
         ,~  
         \bodX_{s}^{N} -\bodX_{t_m}^{N}   \Big\rangle  
        \Big]
        \dd \rho 
        \\
        \nonumber 
        \label{eq: b0 term 2 result 3}
        &\leq 
        \frac{K}{N} \sum_{i=1}^N
        ~ 
        \sqrt{  e^{-\lambda_0 (T-s)}   \bE \Big[  
        \big|X_{s}^{i,N}-X_{t_m}^{i,N}   \big|^2 
        \Big]  ~} 
        \\
        & 
        \leq 
        K  \sqrt{s-t_m} e^{-\lambda_0 (T-s)/2}
        \leq K  h^{1/2} e^{-\lambda_0 (T-t_{m+1})/2},
\end{align}
where we used Proposition~\ref{prop:basic_estimates11xx} and that $h\geq s-t_m$. 
\medskip

\textit{Part 2.2: Estimating \eqref{eq:  rb sum term2}.} 
For the term \eqref{eq:  rb sum term2}, we get for all $m$ and $s\in[t_m,t_{m+1}]$
\begin{align}
\nonumber 
    \bE \Big[  &
            L(t_m, \bodX_{t_m}^{N})
        -   L(s,\bodX_{t_m}^{N})
           \Big] 
 \\ 
 \nonumber
 & =
    \frac{1}{2} \bE\bigg[ ~\sum_{i,j=1}^{N}
    \bigg( 
    B_{j}(\bodX_{t_m}^{N}) \partial_{x_j} B_{i}(\bodX_{t_m}^{N}) \Big(  \partial_{x_i}u(t_m,\bodX_{t_m}^{N})- \partial_{x_i}u(s,\bodX_{t_m}^{N})\Big)
    \\
    \nonumber
    & 
    \qquad 
    +\frac{\sigma^2}{2}    \partial_{x_j}B_i(\bodX_{t_m}^{N})  \Big(  \partial_{x_i, x_j}^2u(t_m,\bodX_{t_m}^{N})- \partial^2_{x_i ,x_j}u(s,\bodX_{t_m}^{N})\Big)
 \\
 \nonumber 
 &
 \qquad 
 +\frac{\sigma^2}{2} 
    \partial^2_{x_j,x_j}B_i(\bodX_{t_m}^{N})  \Big(  \partial_{x_i }u(t_m,\bodX_{t_m}^{N})- \partial_{x_i }u(s,\bodX_{t_m}^{N})\Big)   \bigg]
    \\\label{eq: B0 extended 1} 
    &\leq 
    K \sum_{i,j=1}^{N}
    \sqrt{  \bE \Big[ \Big|  B_{j}(\bodX_{t_m}^{N}) \partial_{x_j} B_{i}(\bodX_{t_m}^{N})  \Big|^2 \Big] 
    \bE \Big[ \Big| \partial_{x_i}u(t_m,\bodX_{t_m}^{N})- \partial_{x_i}u(s,\bodX_{t_m}^{N}) \Big|^2 \Big]}
     \\
    \label{eq: B0 extended 2} 
     &
    \qquad  +K \sum_{i,j=1}^{N}
    \sqrt{  \bE \Big[ \Big| \partial_{x_j}B_i(\bodX_{t_m}^{N})   \Big|^2 \Big] 
    \bE \Big[ \Big| \partial_{x_i, x_j}^2u(t_m,\bodX_{t_m}^{N})- \partial^2_{x_i ,x_j}u(s,\bodX_{t_m}^{N}) \Big|^2 \Big]}
    \\
    \label{eq: B0 extended 3}
     &
     \qquad +K \sum_{i,j=1}^{N}
    \sqrt{  \bE \Big[ \Big|  \partial^2_{x_j,x_j}B_i(\bodX_{t_m}^{N})    \Big|^2 \Big] 
    \bE \Big[ \Big| \partial_{x_i }u(t_m,\bodX_{t_m}^{N})- \partial_{x_i }u(s,\bodX_{t_m}^{N}) \Big|^2 \Big]}
    .
\end{align} 

We first study the differences of first-order derivatives of $u$ in \eqref{eq: B0 extended 1} and \eqref{eq: B0 extended 3}, and then study the difference of second-order derivatives of $u$ \eqref{eq: B0 extended 2}. Collecting these estimates and using the bounds on the derivatives of $B$ yields the upper bound on \eqref{eq:  rb sum term2}. Then we will be in a position to conclude the main result.\\

\textit{Part 2.2.1: The first order derivative terms.} We first derive the following estimate: 
\begin{align}
\nonumber 
    \bE \Big[  &    \big|   \partial_{x_i}u(t_m,\bodX_{t_m}^{N})- \partial_{x_i}u(s,\bodX_{t_m}^{N})     \big|^2 \Big]
     \\ \nonumber 
     &   = 
    \bE \bigg[\Big| ~
    ~\sum_{j=1}^{N}\bE \Big[    
    \partial_{x_j} g(\bodX_T^{t_m,\bodx ,N}) 
          \cdot  	 
    X_{T,x_i}^{t_m,x_j,j,N}     
    -
        \partial_{x_j} g(\bodX_T^{s,\bodx ,N}) 
          \cdot  	  
    X_{T,x_i}^{s,x_j,j,N}  
    \Big]
      ~\bigg|_{\substack{\bodx = \bodX_{t_m}^{N}
              } } ~ \Big|^2\bigg]    
    \\ \nonumber 
    & \leq K 
    \bE \bigg[~\Big| ~
    \bE \Big[ ~\sum_{j=1, j\neq i}^{N} \Big(  
    \partial_{x_j} g(\bodX_T^{t_m,\bodx ,N}) 
       \cdot  	 
    X_{T,x_i}^{t_m,x_j,j,N}    
    -
      \partial_{x_j} g(\bodX_T^{s,\bodx ,N}) 
        \cdot    
    X_{T,x_i}^{s,x_j,j,N} 
    \Big)     
    \Big]  ~\bigg|_{\substack{\bodx = \bodX_{t_m}^{N}
              } }~
     \Big|^2  \bigg] 
     \\ \nonumber 
   & \qquad\qquad +
   K 
    \bE \bigg[~\Big| ~
    \bE \Big[  
    \partial_{x_i} g(\bodX_T^{t_m,\bodx,N}) 
        \cdot    
    X_{T,x_i}^{t_m,x_i,i,N}   
    -
    \partial_{x_i} g(\bodX_T^{s,\bodx,N}) 
         \cdot  	   
    X_{T,x_i}^{s,x_i,i,N}  
    \Big]  ~\bigg|_{\substack{\bodx = \bodX_{t_m}^{N}
              } } ~
     \Big|^2 \bigg]
     \\ \nonumber 
    & \leq K 
    \bE \bigg[~\Big| ~
     \sum_{j=1, j\neq i}^{N} \Big(  
    \partial_{x_j} g(\bodX_T^{t_m,\bodX_{t_m}^{N} ,N}) 
      \cdot   
    X_{T,x_i}^{t_m,X_{t_m}^{j,N},j,N}    
    -
      \partial_{x_j} g(\bodX_T^{s,\bodX_{t_m}^{N} ,N}) 
         \cdot     
    X_{T,x_i}^{s,X_{t_m}^{j,N},j,N} 
    \Big)   ~   
     \Big|^2 ~\bigg] 
     \\ \label{eq: removing iterated expectations} 
   & \qquad\qquad +
   K 
    \bE \bigg[~\Big| ~ 
    \partial_{x_i} g(\bodX_T^{t_m,\bodX_{t_m}^{N},N}) 
         \cdot     
    X_{T,x_i}^{t_m,X_{t_m}^{i,N},i,N}    
    -
       \partial_{x_i} g(\bodX_T^{s,\bodX_{t_m}^{N},N}) 
       \cdot   
    X_{T,x_i}^{s,X_{t_m}^{i,N},i,N} 
     ~
     \Big|^2 ~\bigg] 
     \\ \nonumber 
    &  
    \leq 
    KN\sum_{j=1, j\neq i}^{N} \bE \bigg[ ~ 
    \Big|~
    \Big(   \partial_{x_j} g(\bodX_T^{t_m,\bodX_{t_m}^{N},N}) -\partial_{x_j} g(\bodX_T^{s,\bodX_{t_m}^{N},N})
       \Big)   \cdot	   
    X_{T,x_i}^{t_m,X_{t_m}^{j,N},j,N}   
    \Big|^2 \bigg]
    \\ \nonumber 
   & \qquad\qquad +
    KN\sum_{j=1, j\neq i}^{N} \bE \bigg[ ~ 
        \Big|   \partial_{x_j} g(\bodX_T^{s,\bodX_{t_m}^{N},N}) 
             \cdot	\Big(   
        X_{T,x_i}^{t_m,X_{t_m}^{j,N},j,N}  
    -
    X_{T,x_i}^{s,X_{t_m}^{j,N},j,N} \Big)  ~\Big|^2 ~
   \bigg]
   \\ \nonumber 
    & \qquad\qquad +
    K \bE \bigg[  ~
    \Big|~
     \Big(  \partial_{x_i} g(\bodX_T^{t_m,\bodX_{t_m}^{N},N}) -\partial_{x_i} g(\bodX_T^{s,\bodX_{t_m}^{N},N}) 
       \Big)   \cdot  	   
     X_{T,x_i}^{t_m,X_{t_m}^{i,N},i,N}  ~  
    \Big|^2 ~
   \bigg]
   \\ \label{eq-aux-diff of 1st order derives time}
    & \qquad\qquad +
    K \bE \bigg[  ~
        \Big|~    \partial_{x_i} g(\bodX_T^{s,\bodX_{t_m}^{N},N}) 
           \cdot  	\Big(  
        X_{T,x_i}^{t_m,X_{t_m}^{i,N},i,N}  
    -
    X_{T,x_i}^{s,X_{t_m}^{i,N},i,N}  \Big) 
    ~\Big|^2 ~
   \bigg],
\end{align}
where we employed the Cauchy--Schwarz and Jensen inequalities, as well as the tower property for conditional expectations to obtain \eqref{eq: removing iterated expectations}. Inequality \eqref{eq-aux-diff of 1st order derives time} follows from a standard re-arrangement of \eqref{eq: removing iterated expectations}.

An application of H\"older's inequality and Assumption~\ref{assum:main_weak error} on the function $g$  further yields
\allowdisplaybreaks
\begin{align}
\nonumber
    \eqref{eq-aux-diff of 1st order derives time}
    & 
    \leq  
     \frac{ K }{N}\sum_{j=1, j\neq i}^N 
     \sqrt{ 
      \bE \Big[
    \big|     
    X_{T}^{t_m,X_{t_m}^{j,N},j,N}  
    - X_{T}^{s,X_{t_m}^{j,N},j,N}  
    \big|^4 \Big]
    ~
     \bE \Big[
    \big|     
    X_{T,x_i}^{t_m,X_{t_m}^{j,N},j,N}  
    \big|^4 \Big] }
     \\
     \nonumber
    &\qquad\qquad +
    \frac{ K }{N}\sum_{j=1, j\neq i}^N \bE \Big[
      \big|  
    X_{T,x_i}^{t_m,X_{t_m}^{j,N},j,N} - X_{T,x_i}^{s,X_{t_m}^{j,N},j,N}
    \big|^2 \Big]    
    \\
    \nonumber
    &\qquad\qquad +
    \frac{ K }{N^2}\sqrt{ 
      \bE \Big[
    \big|     
    X_{T}^{t_m,X_{t_m}^{i,N},i,N}  
    - X_{T}^{s,X_{t_m}^{i,N},i,N}  
    \big|^4 \Big]
    ~
     \bE \Big[
    \big|     
    X_{T,x_i}^{t_m,X_{t_m}^{i,N},i,N}  
    \big|^4 \Big] }
    \\ \nonumber
    &\qquad\qquad
    +
    \frac{ K }{N^2}
    \bE \Big[
      \big|  
    X_{T,x_i}^{t_m,X_{t_m}^{i,N},i,N} - X_{T,x_i}^{s,X_{t_m}^{i,N},i,N}
    \big|^2 \Big]
    \\
    \label{eq: partial u res 1}
    &\leq   
    \frac{K}{N^2} (s-t_m) e^{-2 
    \lambda_2 (T-s)},
\end{align}
\color{black}
 where we  used $|\partial^{|\gamma|}_{{x_{\gamma_1},\dots,x_{\gamma_{|\gamma|}}}} g|_{\infty} = \cO(N^{-\hco(\gamma)})$ (a consequence of Assumption~\ref{assum:main_weak error}), 
 \color{black}
 and we used Lemma~\ref{lemma: n-var process result}  with $n=1$ (note that $\lambda_0$ can be replaced by $\lambda_1$ in this case), Lemma~\ref{lemma: 4 moment difference}  and Proposition~\ref{propsition:  first variation bound prop v2} with $\lambda_2 \in (0,\min\{   \lambda-2K_V,\lambda_1\} )$ for the final inequality.  

\textit{Part 2.2.2: The second order derivative terms.}
 Similarly,  for the differences of the second order derivatives of $u$ in \eqref{eq: B0 extended 2}, applying the tower property of the conditional expectation once more we have 
\begin{align}
    \nonumber
    \bE &\Big[       \big|   \partial_{x_i,x_j}^2u(t_m,\bodX_{t_m}^{N})- \partial_{x_i,x_j}^2u(s,\bodX_{t_m}^{N})     \big|^2 \Big]
     \\ 
     \nonumber 
     &   = 
    \bE \bigg[~\bigg| ~
    ~\sum_{k=1}^{N}\bE \Big[    
    \partial_{x_k} g(\bodX_T^{t_m,\bodx ,N}) 
          \cdot  	    
    X_{T,x_i,x_j}^{t_m,x_k,k,N}    
    -
       \partial_{x_k} g(\bodX_T^{s,\bodx ,N}) 
          \cdot  	   
    X_{T,x_i,x_j}^{s,x_k,k,N}  
    \Big]
    \\ \nonumber
    &\qquad   +   
     \sum_{k, k'=1}^N   \bE \Big[  \partial_{x_k,x_{k'}} g( \bodX_T^{ t_m,\bodx,N}  ) 
     \cdot 
    X_{T,x_i}^{t_m,x_k,k,N}  X_{T,x_j}^{t_m,x_{k'},k',N} 
    \\ \nonumber
    &\hspace{4cm}
    - 
    \partial_{x_k,x_{k'}} g( \bodX_T^{ s,\bodx,N}  )
    \cdot 
    X_{T,x_i}^{s,x_{k},k,N}  X_{T,x_j}^{s,x_{k'},k',N}
     \Big] 
     ~\bigg|_{\substack{\bodx = \bodX_{t_m}^{N}
              } }
              ~  \bigg|^2
              \bigg]
    \\ \label{eq: uxx term1}
    &\leq  K 
     \bE \Big[\Big| ~~
     \sum_{k=1}^N   \Big( \partial_{x_k} g( \bodX_T^{ t_m,\bodX_{t_m}^{N},N}  ) - 
    \partial_{x_k} g( \bodX_T^{ s,\bodX_{t_m}^{N},N}  ) \Big)\cdot 
    X_{T,x_i,x_j}^{t_m,X_{t_m}^{k,N},k,N} \Big|^2 \Big]
    \\ \label{eq: uxx term2}
    & \qquad+K 
     \bE \Big[\Big|    \sum_{k=1}^N   \Big(
       X_{T,x_i,x_j}^{t_m,X_{t_m}^{k,N},k,N}
    - X_{T,x_i,x_j}^{s,X_{t_m}^{k,N},k,N}
    \Big) \cdot \partial_{x_k} g( \bodX_T^{ s,\bodX_{t_m}^{N},N}  )
    \Big|^2 \Big]
    \\ \label{eq: uxx term3}
    &\qquad +K 
     \bE \Big[\Big| ~~
     \sum_{k,k'=1}^N \Big( \partial_{x_k,x_{k'}} g( \bodX_T^{ t_m,\bodX_{t_m}^{N},N}  ) 
    - 
    \partial_{x_k,x_{k'}} g( \bodX_T^{ s,\bodX_{t_m}^{N},N}  )\Big)
    \cdot 
    X_{T,x_i}^{t_m,X_{t_m}^{k,N},k,N}  X_{T,x_j}^{t_m,X_{t_m}^{k',N},k',N}
    \Big|^2 \Big]
    \\ \label{eq: uxx term4}
    &\qquad +K 
     \bE \Big[\Big| ~~
     \sum_{k,k'=1}^N \Big( 
    X_{T,x_i}^{t_m,X_{t_m}^{k,N},k,N}   -  
    X_{T,x_i}^{s,X_{t_m}^{k,N},k,N}  
    \Big) \cdot \partial_{x_k,x_{k'}} g( \bodX_T^{ s,\bodX_{t_m}^{N},N}  )\cdot X_{T,x_j}^{t_m,X_{t_m}^{k',N},k',N}
    \Big|^2 \Big]
    \\ \label{eq: uxx term5}
    &\qquad +K  
     \bE \Big[\Big| ~~
     \sum_{k,k'=1}^N \Big(  X_{T,x_j}^{t_m,X_{t_m}^{k',N},k',N}
    -  X_{T,x_j}^{s,X_{t_m}^{k',N},k',N}
    \Big) \cdot \partial_{x_k,x_{k'}} g( \bodX_T^{ s,\bodX_{t_m}^{N},N}  ) \cdot
    X_{T,x_i}^{s,X_{t_m}^{k,N},k,N} 
    \Big|^2 \Big].
    \end{align} 
    We are required to analyze each of the terms \eqref{eq: uxx term1}--\eqref{eq: uxx term5} separately. By further applying Jensen's inequality and H\"older's inequality, we derive the following estimates for the first two terms \eqref{eq: uxx term1} and \eqref{eq: uxx term2} based on the values of $\hco((i,j,k))$  
    \begin{align}
    \nonumber
   &      \eqref{eq: uxx term1}
   \\ \nonumber 
   &\leq 
    K \sqrt{   
     \bE \Big[\Big| ~ 
         \partial_{x_i} g( \bodX_T^{ t_m,\bodX_{t_m}^{N},N}  ) - 
    \partial_{x_i} g( \bodX_T^{ s,\bodX_{t_m}^{N},N}  )  
      \Big|^4  \Big] 
    \bE \Big[\Big|  
    X_{T,x_i,x_j}^{t_m,X_{t_m}^{i,N},i,N} \Big|^4 
    ~ \Big] }
    \\\nonumber
    &\qquad 
    +K 
      \sqrt{   
     \bE \Big[\Big| ~ 
         \partial_{x_j} g( \bodX_T^{ t_m,\bodX_{t_m}^{N},N}  ) - 
    \partial_{x_j} g( \bodX_T^{ s,\bodX_{t_m}^{N},N}  )  
      \Big|^4  \Big] 
    \bE \Big[\Big|  
    X_{T,x_i,x_j}^{t_m,X_{t_m}^{j,N},j,N} \Big|^4 
    ~ \Big] }
    \\ \label{eq-aux:wrap-up to eq: uxx term1}
    &\qquad +
    N K \sum_{k=1,k\notin\{i,j\}}^N
    \sqrt{   
     \bE \Big[\Big| ~ 
         \partial_{x_k} g( \bodX_T^{ t_m,\bodX_{t_m}^{N},N}  ) - 
    \partial_{x_k} g( \bodX_T^{ s,\bodX_{t_m}^{N},N}  )  
      \Big|^4  \Big] 
    \bE \Big[\Big|  
    X_{T,x_i,x_j}^{t_m,X_{t_m}^{k,N},k,N} \Big|^4 
    ~ \Big] }.
    \end{align}
    Now, similar to \eqref{eq: partial u res 1}, an application of Assumption~\ref{assum:main_weak error} for the function $g$ and then injecting the bounds from Lemma~\ref{lemma: n-var process result} with $n=2$ (note that $\lambda_0$ can be replaced by $\lambda_4$ in this case) and Lemma~\ref{lemma: 4 moment difference} yield  
    \begin{align}
        \nonumber
        \eqref{eq-aux:wrap-up to eq: uxx term1} &\leq
            \frac{K}{N^2} \sqrt{ 
          \bE \Big[
        \big|     
        X_{T}^{t_m,X_{t_m}^{i,N},i,N}  
        - X_{T}^{s,X_{t_m}^{i,N},i,N}  
        \big|^4 \Big]
        ~
         \bE \Big[
        \big|     
        X_{T,x_i,x_j}^{t_m,X_{t_m}^{i,N},i,N}  
        \big|^4 \Big] }
        \\
        \nonumber
        &\qquad + 
        \frac{K}{N^2}   \sqrt{ 
          \bE \Big[
        \big|     
        X_{T}^{t_m,X_{t_m}^{j,N},j,N}  
        - X_{T}^{s,X_{t_m}^{j,N},j,N}  
        \big|^4 \Big]
        ~
         \bE \Big[
        \big|     
        X_{T,x_i,x_j}^{t_m,X_{t_m}^{j,N},j,N}  
        \big|^4 \Big] }
        \\
        \nonumber
        &\qquad  + \frac{K}{N} \sum_{k=1,k\notin\{i,j\}}^N
           \sqrt{ 
          \bE \Big[
        \big|     
        X_{T}^{t_m,X_{t_m}^{k,N},k,N}  
        - X_{T}^{s,X_{t_m}^{k,N},k,N}  
        \big|^4 \Big]
        ~
         \bE \Big[
        \big|     
        X_{T,x_i,x_j}^{t_m,X_{t_m}^{k,N},k,N}  
        \big|^4 \Big] }
        \\ \label{eq: partial2 u results 11}
        &
        \leq 
          K(s-t_m)e^{-2\lambda_2 (T-s)}e^{-2\lambda_4(T-t_m)} \bigg( 
        \frac{1}{N^2 \cdot  N^{2 \hco((i,j))-2}} 
        +
        \frac{1}{N} \sum_{k=1,k\notin\{i,j\}}^N        \frac{ 1}{N^{ 2\hco((i,j,k))-2}}  \bigg). 
    \end{align}
Similarly, for \eqref{eq: uxx term2}, using Proposition~\ref{prop: second var extension}, we have    
    \begin{align}
    \nonumber 
    \eqref{eq: uxx term2}&\leq 
    K  \sqrt{   
     \bE \Big[\Big|       \Big(
       X_{T,x_i,x_j}^{t_m,X_{t_m}^{i,N},i,N}
    - X_{T,x_i,x_j}^{s,X_{t_m}^{i,N},i,N}
    \Big) \Big|^4  \Big] 
    \bE \Big[\Big|   
    \partial_{x_i} g( \bodX_T^{ s,\bodX_{t_m}^{N},N}  )
    \Big|^4 \Big]  }
    \\ \nonumber
    & \quad+K  \sqrt{   
     \bE \Big[\Big|       \Big(
       X_{T,x_i,x_j}^{t_m,X_{t_m}^{j,N},j,N}
    - X_{T,x_i,x_j}^{s,X_{t_m}^{j,N},j,N}
    \Big) \Big|^4  \Big] 
    \bE \Big[\Big|   
    \partial_{x_j} g( \bodX_T^{ s,\bodX_{t_m}^{N},N}  )
    \Big|^4 \Big]  }
    \\ \nonumber
    & \quad + NK \sum_{k=1,k\notin\{i,j\}}^N
     \sqrt{   
     \bE \Big[\Big|       \Big(
       X_{T,x_i,x_j}^{t_m,X_{t_m}^{k,N},k,N}
    - X_{T,x_i,x_j}^{s,X_{t_m}^{k,N},k,N}
    \Big) \Big|^4  \Big] 
    \bE \Big[\Big|   
    \partial_{x_k} g( \bodX_T^{ s,\bodX_{t_m}^{N},N}  )
    \Big|^4 \Big]  }   
    \\
    \label{eq: partial2 u results 22}
    &\leq 
    K(s-t_m) e^{-2\lambda_4(T-t_m)}\bigg( 
     \frac{1}{N^2 \cdot  N^{2 \hco((i,j))-2}} 
     +
     \frac{1}{N} \sum_{k=1,k\notin\{i,j\}}^N        \frac{  1}{N^{ 2\hco((i,j,k))-2}} \bigg).
    \end{align} 
{

For \eqref{eq: uxx term3}--\eqref{eq: uxx term5}, we will make use of the following bound, which follows from exploiting the Lipschitz property of $\partial_{x_k,x_k'}g$ from Assumption~\ref{assum:main_weak error} and applying Lemma~\ref{lemma: 4 moment difference}  
\begin{align*}
    \bE\Big[  \Big|   \partial_{x_k,x_{k'}} g( \bodX_T^{ t_m,\bodX_{t_m}^{N},N}  ) 
    - 
    \partial_{x_k,x_{k'}} g( \bodX_T^{ s,\bodX_{t_m}^{N},N}  )   \Big| ^4 \Big]
     \leq \frac{K}{N^{4\hco( (k,k'))}} e^{-4\lambda_2(T-s)}.
\end{align*}    
Hence, similarly to \eqref{eq: partial2 u results 11}--\eqref{eq: partial2 u results 22}, partitioning the sum based on values taken by $\hco( (i,j,k,k') )$  
\begin{align}
    \nonumber 
    \eqref{eq: uxx term3}
    &\leq 
    K  \bigg( 
      \sum_{    \substack{ k,k'\in\{1,\dots,N\}, \\
      \hco( (i,j,k,k') )= \hco((i,j))  }} 
      + N  \sum_{ \substack{ k,k'\in\{1,\dots,N\}, \\
      \hco( (i,j,k,k') )= \hco((i,j))+1 }}  
       + 
         N^2 \sum_{ \substack{ k,k'\in\{1,\dots,N\}, \\ \hco( (i,j,k,k') )= \hco((i,j))+2 }}  
       \bigg)
    \\
    \nonumber 
    &\quad \cdot 
    \sqrt{  
     \bE \Big[\Big|
      \partial_{x_k,x_{k'}} g( \bodX_T^{ t_m,\bodX_{t_m}^{N},N}  ) 
    - 
    \partial_{x_k,x_{k'}} g( \bodX_T^{ s,\bodX_{t_m}^{N},N}  ) 
     \Big|^4\Big]   
    \bE \Big[ \Big|
    X_{T,x_i}^{t_m,X_{t_m}^{k,N},k,N}  X_{T,x_j}^{t_m,X_{t_m}^{k',N},k',N}
    \Big|^4 \Big] }
    \\
    \nonumber 
    &\leq 
     K(s-t_m) e^{-2\lambda_2 (T-s)}  e^{-4\lambda_1 (T-t_m) }  
     \\ \nonumber
     &
     \qquad\qquad
     \cdot  
     \Bigg( 
      \sum_{    \substack{ k,k'\in\{1,\dots,N\}, \\
      \hco( (i,j,k,k') )= \hco((i,j))  }}\frac{1}{N^{ 2 \hco( (k,k' ))  }}
      \frac{   1 }{      N^{   2 (\hco( (i,k))+\hco( (j,k')) -2 ) )   }  } 
      \\ \nonumber
      &
      \qquad\qquad \qquad\qquad
      + N  \sum_{ \substack{ k,k'\in\{1,\dots,N\}, \\
      \hco( (i,j,k,k') )= \hco((i,j))+1 }} \frac{1}{N^{ 2 \hco( (k,k' ))  }}
      \frac{   1 }{      N^{   2 (\hco( (i,k))+\hco( (j,k')) -2 ) )   }  }
      \\ \label{eq: huge cauchy schwarz expansion}
      &   
         \qquad\qquad\qquad\qquad
       +  N^2 \sum_{ \substack{ k,k'\in\{1,\dots,N\}, \\ \hco( (i,j,k,k') )= \hco((i,j))+2 }}         
      \frac{1}{N^{ 2 \hco( (k,k' ))  }}
      \frac{   1 }{      N^{   2 (\hco( (i,k))+\hco( (j,k')) -2 ) )   }  }\Bigg)
        \\ 
        \label{eq: summation term 6.1}
    &\leq 
         \frac{K(s-t_m) e^{-2\lambda_2 (T-s)}  e^{-4\lambda_1 (T-t_m) }}{N^{2 \hco( (i,j))}}  \cdot
        \sum_{m=0}^2 N^{m}  \sum_{ \substack{ k,k'\in\{1,\dots,N\}, \\ \hco( (i,j,k,k') )= \hco((i,j))+m }}  \frac{1}{N^{2m}}
     \\
    \label{eq: partial2 u results 33}
    &\leq 
         \frac{K(s-t_m) e^{-2\lambda_2 (T-s)}  e^{-4\lambda_1 (T-t_m) }}{N^{2 \hco( (i,j))}}.
\end{align}
The term \eqref{eq: huge cauchy schwarz expansion} is derived by applying the Cauchy--Schwarz inequality, which then requires the $L^8$-moments of $ X_{T,x_i}^{t_m,X_{t_m}^{k,N},k,N}$ and $  X_{T,x_j}^{t_m,X_{t_m}^{k',N},k',N} $   
obtained from Lemma~\ref{lemma: first variation bound noiid} (that holds without integrability requirements on $X_{t_m}^{k,N},X_{t_m}^{k',N}$ since \eqref{first_var_process noiid} is a linear ODE with bounded coefficients starting from either $1$ or $0$), and noting we can unify the bounds in Lemma~\ref{lemma: first variation bound noiid} as (note that $\lambda_1 < \lambda$ and $ \hco((i,k))$ is either 0 or 1)
\begin{align*}
    \bE\big[\,|X_{T,x_i}^{t_m,X_{t_m}^{k,N},k,N}|^8\big] 
    \leq
    \frac{K}{N^{8( \hco((i,k))-1)}}e^{-8\lambda_1(T-t_m)}.
\end{align*}
To obtain \eqref{eq: summation term 6.1}, we used the following bounds, which can be confirmed by checking cases 
\begin{align*}
    2\hco\big( (k,k')\big)
    &+ 2 \big(\hco( (i,k))+\hco( (j,k')) -2 \big)
    \\
    &
    \geq \begin{cases}
        2+2(1+ \hco( (i,j)) -2) = 2 \hco( (i,j)),
        & \hco( (i,j,k,k') )= \hco((i,j)) ,
        \\
        4+2(1+2-2)=6 \geq 2+2 \hco( (i,j)),
        & \hco( (i,j,k,k') )= \hco((i,j))+1,~k\neq k' ,
        \\
        2+2(2+2-2)=6 \geq 2+2 \hco( (i,j)),
        & \hco( (i,j,k,k') )= \hco((i,j))+1,~k = k',
        \\
        4+ 2(2+2 -2)=8\geq 4+2 \hco( (i,j)),
        & \hco( (i,j,k,k') )= \hco((i,j))+2.
    \end{cases}
\end{align*}
In \eqref{eq: partial2 u results 33}, we used \eqref{eq-aux:countingNumbers} to ensure the summation  term in \eqref{eq: summation term 6.1} is $\cO(1)$. We establish bounds for \eqref{eq: uxx term4} and \eqref{eq: uxx term5} in a similar fashion and obtain 
\begin{align}
       \label{eq: partial2 u results 44} 
       \eqref{eq: uxx term4} + \eqref{eq: uxx term5}
       &\leq
       \frac{K(s-t_m) e^{-2\lambda_4 (T-s)}   }{N^{2 \hco( (i,j))}} .        
\end{align}
Substituting \eqref{eq: partial2 u results 11}--\eqref{eq: partial2 u results 44} into \eqref{eq: uxx term1}--\eqref{eq: uxx term5} yields
    \begin{align}
    &\bE \Big[       \big|   \partial_{x_i,x_j}^2u(t_m,\bodX_{t_m}^{N})- \partial_{x_i,x_j}^2u(s,\bodX_{t_m}^{N})     \big|^2 \Big]
    \label{eq: partial u res 2}
    \leq ~ 
    \frac{K(s-t_m)e^{-2\lambda_4 (T-s)}}{N^{2 \hco((i,j))}} .
\end{align}
}

 \textit{Part 2.3: Collecting the estimates for \eqref{eq: B0 extended 1}--\eqref{eq: B0 extended 3}.}
 Consequently, substituting \eqref{eq: partial u res 1} and \eqref{eq: partial u res 2} into \eqref{eq: B0 extended 1}--\eqref{eq: B0 extended 3} and using the bounds for the derivatives of the function $B$ in \eqref{eq: prop of deriv of B}, we conclude that  
there exists some constant $\lambda_0\in(0,\lambda)$ such that for all $m$ and $s\in[t_m,t_{m+1}]$,
 \begin{align}
   \nonumber 
     \bE \Big[  &
            L(t_m,\bodX_{t_m}^{N})
        -   L(s,\bodX_{t_m}^{N})
           \Big]
           \\ \nonumber
           &
            \leq 
            K \sum_{i,j=1}^{N} 
            \Bigg(  
            \frac{\sqrt{s-t_m} }{N^{  \hco(  (i,j))-1 }\cdot N} e^{-\lambda_2(T-s)}
            +
            \frac{\sqrt{s-t_m} }{N^{  \hco(  (i,j))-1} N^{ \hco((i,j)) } } e^{-\lambda_4(T-s)}
            \\ \nonumber 
            &
            \hspace{5cm }
           +
             \frac{\sqrt{s-t_m} }{N^{  \hco(  (i,j))-1} \cdot N} e^{-\lambda_2(T-s)}
            \Bigg)
            \\
            &
         \label{eq: b0 term 3 result}
            \leq 
            K h^{1/2} e^{-\lambda_0(T-s)}\sum_{i,j=1}^{N}  \frac{1}{N^{  \hco(  (i,j)) }} 
           \leq 
           K h^{1/2} e^{-\lambda_0 (T-t_{m+1})},
\end{align}
where once again the final inequality arises from recalling \eqref{eq-aux:countingNumbers}.\\

 \textit{Part 3: Concluding the proof.}
Substituting the results of  \eqref{eq: b0 term 1 result}, \eqref{eq: b0 term 2 result 3}, \eqref{eq: b0 term 3 result} to \eqref{eq:  rb sum term1}--\eqref{eq:  rb sum term3}, we conclude that 
\begin{align}
    \nonumber
    \sum_{m=0}^{M-1} &
      \bE \Big[   \int_{t_m}^{t_{m+1}}
        \Big(  L(t_m,\bodX_{t_m}^{N,h})
        -  L(s,\bodX_{s}^{N})
        \Big)
      \dd s \Big]
       \leq  Kh^{3/2}\sum_{m=0}^{M-1} e^{-\lambda_0 (T-t_{m+1})}  \leq Kh^{1/2} .
\end{align} 
Therefore, for the left-hand side of \eqref{eq: B0 terms results}, there exists some positive constants $\lambda_0 \in (0,\lambda),K$ such that  
\begin{align*}
     h^2\bE \Big[  &\sum_{m=0}^{M-1} L(t_m,\bodX_{t_m}^{N,h})  \Big]
    \leq  
     Kh^{3/2}+Khe^{-\lambda_0 T}.
\end{align*}
\end{proof}

\begin{remark}[On losing $1/2$ in the rate of convergence]
\label{rem:Losing1/2 conv rate}
Letting $T\rightarrow\infty$ in \eqref{eq: B0 terms results},  and temporarily ignoring higher-order remainder terms, we have that the weak error is of order $3/2$. In \cite{leimkuhler2014long} this term satisfies the bound (with $K$ depending on the dimension $N$)
\begin{align*}
    h^2\bE \Big[  \sum_{m=0}^{M-1} L(t_m,\bodX_{t_m}^{N,h})  \Big]
    \leq  
     Kh^{2}+Khe^{-\lambda_0 T},
\end{align*}
and thus a weak order of $2$ is attained.

The loss of $1/2$ in the convergence rate in our estimates occurs in the final step of \eqref{eq: b0 term 1 result}. In \cite{leimkuhler2014long} this term is estimated with the weak error (starting on \eqref{eq:  rb sum term1}), whilst here, we are not able to do so. In fact, to recover the missing $1/2$-rate by following using the arguments \cite{leimkuhler2014long} one  would need to show that $L$ in \eqref{eq:def of B0} satisfies condition (3) of Assumption~\ref{assum:main_weak error} -- see additionally our Remark \ref{rem:H2-diff}. At present, this is an open question.
\end{remark}

\subsection{Analysis of residual terms}
\label{sec:Remainder terms R} 

A close inspection of the proof of the main result in \cite{leimkuhler2014long} shows that there are 8 remainder terms which need to be analyzed and we do so in the next lemma. 
We will make use of the following helpful abbreviation: for $m \in \lbrace 0, \ldots, M-1 \rbrace$, we define using \eqref{eq:def of func B} and \eqref{eq: def : non-Markov Euler scheme}
\begin{align}
\nonumber 
 \Delta X_{t_m}^{i,N,h} 
 := 
 &
 X^{i,N,h}_{t_{m+1}} - X^{i,N,h}_{t_{m}}
 \\
 \label{eq-aux:DeltaXfrom nME Scheme}
 =
 & 
 - \Big(  \nabla U(X^{i,N,h}_{t_{m}}) + \frac{1}{N}\sum_{j=1}^{N} \nabla V(X_{t_{m}}^{i,N,h} - X_{t_m}^{j,N,h}) \Big)  h   + \frac\sigma2
    (\Delta W_{m}^i + \Delta W_{m+1}^i)   
\\ \nonumber
=
& 
B_i(\boldsymbol{X}^{N,h}_{t_m})h + \frac \sigma 2 
(\Delta W_{m}^i + \Delta W_{m+1}^i). 
\nonumber
\end{align}
Further, we define the continuous extension of $\boldsymbol{X}^{N,h}_{t_m}$: for all $  s\in [0,h]$  
\begin{align}
 \label{eq-aux: continuous  extension Scheme}
 X_{t_m+s}^{i,N,h} 
 := 
 &
 X^{i,N,h}_{t_{m}}
+
 \Big(  \nabla U(X^{i,N,h}_{t_{m}}) + \frac{1}{N}\sum_{j=1}^{N} \nabla V(X_{t_{m}}^{i,N,h} - X_{t_m}^{j,N,h}) 
 +\frac\sigma{2h} \Delta W_{m}^i
 \Big)  s   + \frac \sigma 2 
\Delta W_{m+1,s}^i   
\\ 
=
& X_{t_m}^{i,N,h} + 
\color{black}
\Big( 
B_i(\boldsymbol{X}^{N,h}_{t_m})+ \frac{\sigma \Delta W_m^i}{2h} \Big) s + \frac \sigma 2 
\Delta W_{m+1,s}^i, \nonumber \\
\Delta W^{i}_{m,s} & = {W}^{i}_{t_{m-1}+s} - {W}^{i}_{t_{m-1}} ,
    \text{ for } m>0. \label{eq-aux: def of dWms notion}
\end{align}
Equivalently, we could write
\begin{align*}
    \Delta \boldsymbol{X}_{t_m}^{N,h} 
 &:= 
 \boldsymbol{X}^{N,h}_{t_{m+1}} - \boldsymbol{X}^{N,h}_{t_{m}}
 =
B(\boldsymbol{X}^{N,h}_{t_m})h + \frac \sigma 2 
( \Delta \boldsymbol{W}_m +  \Delta \boldsymbol{W}_{m+1}),
\\
\boldsymbol{X}_{t_m+s}^{N,h}
&:= 
\boldsymbol{X}^{N,h}_{t_{m}}
+
\Big( 
B(\boldsymbol{X}^{N,h}_{t_m}) + \frac{\sigma \Delta \boldsymbol{W}_{m,h}}{2 h}  \Big)s
+
\frac{\sigma}{2 } \Delta \boldsymbol{W}_{m+1,s},
\end{align*}
where $\Delta \boldsymbol{X}_{t_m}^{N,h} \coloneqq ( \Delta X_{t_m}^{1,N,h}, \ldots,  \Delta X_{t_m}^{N,N,h})$, $ \Delta \boldsymbol{W}_m \coloneqq(\Delta W^1_m,\ldots,\Delta W^N_m)$ and $ \Delta \boldsymbol{W}_{m,s} \coloneqq(\Delta W^1_{m,s},\ldots,\Delta W^N_{m,s})$ (for $m >0$). \color{black}
Instead of dealing with the expression above, we rewrite the scheme in a different way (see \cite[p7]{leimkuhler2014long}). For  $m \in \lbrace 0, \ldots, M-1 \rbrace$ and $i\in \{ 1,\dots,N\}$, with $\hat{X}^{i,N,h}_{t_0}=X^{i,N}_{t_0}$, define  
\begin{align}
\label{eq: def : aux Euler scheme 0}  
\hx^{i,N,h}_{t_{m+1}} 
    & 
    = \hx^{i,N,h}_{t_{m}} +  \sigma 
    \Delta W_{m}^i 
    \\\nonumber
    & - \Big( \nabla U( \hx^{i,N,h}_{t_{m}}+  \frac{\sigma}{2}   
    \Delta W_{m}^i) 
    + \frac{1}{N}\sum_{j=1}^{N} \nabla V\big( \hx_{t_{m}}^{i,N,h}+  \frac{\sigma}{2}  
    \Delta W_{m}^i -  \hx_{t_m}^{j,N,h}-  \frac{\sigma}{2}   
    \Delta W_{m}^j\big)  \Big)  h,
    \\   \nonumber
    \hat{\boldsymbol{X}}^{N,h}_{t_{m+1}} 
    & \color{black}
    = \hat{\boldsymbol{X}}^{N,h}_{t_{m}} + \sigma  \Delta \boldsymbol{W}_m
    + B\big( \hat{\boldsymbol{X}}^{N,h}_{t_{m}} + \frac\sigma2  \Delta \boldsymbol{W}_m \big)h,
\end{align}
so that $X^{i,N,h}_{t_{m}}=\hx^{i,N,h}_{t_{m}}+ \sigma \Delta W_{m}^i/2$, for all $i \in \lbrace 1, \ldots, N \rbrace $ 
 (or $\boldsymbol{X}^{N,h}_{t_{m}}=\hat{\boldsymbol{X}}^{N,h}_{t_{m}} + \sigma  \Delta \boldsymbol{W}_m/2$).

 \color{black} 
Now, for $s\in[0,h)$  and $m \in \lbrace 1, \ldots, M \rbrace$, we define the following auxiliary process  
\begin{align} \label{process:taylor}
\overline{X}^{i,N,h}_{t_{m-1}+s} & = \hx^{i,N,h}_{t_m} + \frac\sigma 2 \Delta W_{m,s}^i 
    ,
    \quad 
     \overline{X}^{i,N,h}_{t_{m-1}}
     = \hx^{i,N,h}_{t_m},
     \quad 
    \\
    \overline{\bodX}_{t_{m-1}+s}^{N,h}&= \hat{\bodX}_{ {t_m}}^{N,h} +
     \frac\sigma 2 \Delta \boldsymbol{W}_{m,s}, 
     \quad 
      \overline{\bodX}_{t_{m-1}}= 
      \hat{\bodX}_{ {t_m}}^{N,h}, \nonumber 
\end{align} 
where we remark that this form is used in the proof of Lemma~\ref{lemma: analysis of the residual term} (e.g., in \eqref{eq: r4 term 1.2}). The moment stability of these auxiliary schemes is discussed in the appendix, see Lemma~\ref{prop: moment bound extension}. 
We refer the reader to \cite{vilmart2015postprocessed} for different versions of such schemes, coined there 
 `postprocessed schemes', achieving higher-order weak convergence in the ergodic setting.
\begin{lemma}
    \label{lemma: analysis of the residual term}
    Let   Assumption~\ref{assum:main_weak error} hold and let $\xi \in L^{10}(\Omega,\mathbb{R})$. Then for the remainder term $R$ in \eqref{eq: weak error expansion-ori1}, there exists a positive constant $K$ independent of $h,T,M$ and $N$, such that   
     \begin{align*}
          \mathbb{E} \left[ \sum_{m=0}^{M-1} R(t_m,\boldsymbol{X}^{N,h}_{t_m}) \right]
          \leq 
          Kh^{3/2}. 
     \end{align*}
\end{lemma} 
\begin{proof}
Recall that the remainder term $r(t_m,\cdot)$ in \cite[Equation~(3.17)]{leimkuhler2014long} (and corresponding to our $R(t_m,\cdot)$ term in \eqref{eq: weak error expansion-ori1}) is given as a linear combination of remainders denoted as $h^3 r_i(t_m,\cdot)$ for $i \in \lbrace 1,\ldots,8 \rbrace$ and appearing in Equations (3.8), (3.10)--(3.16) in \cite[p7-9]{leimkuhler2014long}, { \color{black}   where they correspond to the remainder terms of the It\^o-Taylor expansions in \eqref{eq:def of B0} -- we refer to the terms $\{R^i_{t_m}\}_{i=1}^8$  in this proof  and Appendix~\ref{appendix_aux_RemainderSection6.2} for similar detailed forms.}  
In our case, we derive bounds for   $\mathbb{E} [ R(t_m,\boldsymbol{X}^{N,h}_{t_m}) ]\leq \sum_{i=1}^8 |\mathbb{E}[ R^i_{t_m} ]|$, where each $R^i_{t_m}$ term corresponds to the $h^3r_i(t_m, \cdot)$ terms  ($i \in \lbrace 1,\ldots,8 \rbrace$) of \cite[p7-9]{leimkuhler2014long}.  

It may not be immediately transparent how the $h^3 r_i(t_m,\cdot)$ terms correspond to our $R_{t_m}^i$ terms. We explicitly present the derivations of $\bE [R_{t_m}^1 ]$, $\bE [R_{t_m}^4 ]$ and $\bE [R_{t_m}^6 ]$, which we feel encapsulate the techniques and proof methodology. The derivation of the residual terms $R_{t_m}^2,\,R_{t_m}^3,\,R_{t_m}^5,\,R_{t_m}^7,\,R_{t_m}^8$, can be found in Appendix \ref{appendix_aux_RemainderSection6.2}. In this proof, let $K$ be a constant independent of $m,M,N,T,h$ which may change from line to line.\\

\textit{Part 1:}
Recall the $\Delta X$ notation introduced in \eqref{eq-aux:DeltaXfrom nME Scheme}. The remainder term $h^{3}r_1$ in Equation~(3.8) of \cite[p7]{leimkuhler2014long} is established from the Taylor expansion with Lagrange’s form of the remainder term. That is, for all $m$, there exists some $\rho_m \in (0,1)$ such that 
\begin{align*}
    \bE \big[          R^1_{t_m}      \big]
     &=  
     K\bE \bigg[              
     \sum_{\gamma \in \Pi_{5}^N} 
     \Delta X^{\gamma_1,N,h}_{t_m} \dots \Delta X^{\gamma_5,N,h}_{t_m}~
     \partial^{5}_{ \gamma_1,\dots,\gamma_5      }  
     u(t_{m+1},\bodX_{m,{\rho_m }}^{N,h})  \bigg] 
\end{align*} 
where $\bodX_{m,{\rho_m }}^{N,h}:= {\rho_m } \bodX_{t_m}^{N,h} + (1-{\rho_m }) \bodX_{t_{m+1}}^{N,h}$.\\

We deal with $|\bE[R^1_{t_m}]|$ via H\"older's inequality to isolate the $\Delta X$ terms from the derivatives of $u$ term, 
\begin{align}
\nonumber
    \big|  
    \mathbb{E} \big[         R^1_{t_m}   \big]    \big|    
    &\leq 
    K  \sum_{\gamma \in \Pi_{5}^N}   
    \sqrt{  
    \bE\Big[       \big| 
     \Delta X^{\gamma_1,N,h}_{t_m} \dots \Delta X^{\gamma_5,N,h}_{t_m}
       \big|^2  \Big]  
      \bE\Big[   ~    \Big| 
      \partial^{5}_{ x_{ \gamma_1 },\dots,x_{  \gamma_5}      }~
     u(t_{m+1},\bodX_{m,{\rho_m }}^{N,h}) 
      \Big|^2  \Big]  } 
      \\
      &\leq 
       K h^{5/2} \sum_{\gamma \in \Pi_{5}^N} 
       e^{ -\lambda_0 (T-t_{m+1})}N^{-\hco(\gamma)} 
       \leq K h^{5/2} e^{ -\lambda_0 (T-t_{m+1})}, 
       \label{eq-aux:our R1 vs Leimkuhler's r1 new}
\end{align}
which holds for some $\lambda_0 \in (0,\lambda)$. We derive \eqref{eq-aux:our R1 vs Leimkuhler's r1 new} using Lemma~\ref{lemma: higher-order derivative of u} to treat the derivatives of $u$, while for the $\Delta X$ terms, we use their explicit forms \eqref{eq-aux:DeltaXfrom nME Scheme} and the fact that $\nabla U,\nabla V$ are of linear growth. In particular, the assumption $\xi \in L^{10}(\Omega,\bR)$, allows to use Proposition~\ref{prop:basic_estimates11}  and Proposition~\ref{prop:basic_estimates22} to control the $L^{10}$-moments of the processes.{ The moment properties of the increments extract the leading order $h$ terms}; providing the $h^{5/2}$ leading order. Note that for $R^1_{t_m}$ we incur a loss of $h^{1/2}$ in the leading term while in \cite{leimkuhler2014long} there is no such loss -- this issue has already appeared in the proof of Lemma~\ref{lemma: analysis of the b0 term summation integration} and is discussed in Remark \ref{rem:Losing1/2 conv rate} (and Remark \ref{rem:H2-diff}). Critically, \cite{leimkuhler2014long} has a $|\cdot|_\infty$- norm bound for the 5-th derivative of $u$ while here we are only able to bound it in expectation. We are forced to use H\"older's inequality.
 \\
 
 Similar to \cite[Equation~(3.12)]{leimkuhler2014long}, we show how the $R_{t_m}^4$ (corresponding to $r_4(t_m,\cdot)h^3$ in  \cite[Equation~(3.12)]{leimkuhler2014long}) is generated and give the exact form of $\bE[ R_{t_m}^4]$. Expanding  out the increments $\Delta X^{\gamma_i,N,h}_{t_m}$, $i \in \{1,2,3\}$, and recall the definition of $\Delta W_{m,2h}$ in \eqref{eq-aux: def of dWms notion}, we have
\begin{align}
    \nonumber 
     \sum_{ \gamma \in \Pi_3^N  }&
     \bE\Big[ \Delta X^{\gamma_1,N,h}_{t_m} \Delta X^{\gamma_2,N,h}_{t_m} \Delta X^{\gamma_3,N,h}_{t_m}~
     \partial^{3}_{ x_{  \gamma_1},x_{  \gamma_2}, x_{  \gamma_3}     }  
     u(t_{m+1},\bodX_{t_m}^{N,h})  \Big] 
     \\
     \label{eq: r4 term1}
     = &
     K 
     \sum_{ \gamma \in \Pi_3^N  }
     \bE\Big[ \Delta W_{m,2h}^{\gamma_1} \Delta W_{m,2h}^{\gamma_2}\Delta W_{m,2h}^{\gamma_3}~
     \partial^{3}_{ x_{  \gamma_1},x_{  \gamma_2}, x_{  \gamma_3}     }  
     u(t_{m+1},\bodX_{t_m}^{N,h})  \Big]
     \\
     \label{eq: r4 term2}
     &+ K
     h
     \sum_{ \gamma \in \Pi_3^N  }
     \bE\Big[ B_{\gamma_1}( {\bodX}_{t_m}^{N,h})   \Delta W_{m,2h}^{\gamma_2}\Delta W_{m,2h}^{\gamma_3}~
     \partial^{3}_{ x_{  \gamma_1},x_{  \gamma_2}, x_{  \gamma_3}     }  
     u(t_{m+1},\bodX_{t_m}^{N,h})  \Big]
     \\
     \label{eq: r4 term3}
     &+
      K h^2
     \sum_{ \gamma \in \Pi_3^N  }
     \bE\Big[ B_{\gamma_1}( {\bodX}_{t_m}^{N,h})  B_{\gamma_2}( {\bodX}_{t_m}^{N,h})\Delta W_{m,2h}^{\gamma_3}~
     \partial^{3}_{ x_{  \gamma_1},x_{  \gamma_2}, x_{  \gamma_3}     }  
     u(t_{m+1},\bodX_{t_m}^{N,h})  \Big]
     \\
     \label{eq: r4 term4}
     &+
     K h^3
     \sum_{ \gamma \in \Pi_3^N  }
     \bE\Big[  B_{\gamma_1}( {\bodX}_{t_m}^{N,h})  
     B_{\gamma_2}( {\bodX}_{t_m}^{N,h})
     B_{\gamma_3}( {\bodX}_{t_m}^{N,h})
     ~
     \partial^{3}_{ x_{  \gamma_1},x_{  \gamma_2}, x_{  \gamma_3}     }  
     u(t_{m+1},\bodX_{t_m}^{N,h})  \Big].
\end{align}
 We first deal with the term \eqref{eq: r4 term1}, by applying an It{\^o}-Taylor expansion (see \cite[Section 5.1, p.163--164]{KloedenPlaten1992SDENumericsbook}) around $\hat{\bodX}_{{t_m}}^{N,h}$ to obtain
\begin{align}
\label{eq: r4 term 1.1}
&    \eqref{eq: r4 term1} \nonumber 
\\ 
   & = 
    K \sum_{ \gamma \in \Pi_3^N  }
     \bE\Big[ \Delta W_{m,2h}^{\gamma_1} \Delta W_{m,2h}^{\gamma_2}\Delta W_{m,2h}^{\gamma_3}~
     \partial^{3}_{ x_{  \gamma_1},x_{  \gamma_2}, x_{  \gamma_3}     }  
     u(t_{m+1},\hat{\bodX}_{{t_m}}^{N,h})  \Big]
     \\
 \label{eq: r4 term 1.2}
     &+ K  \sum_{ \gamma \in \Pi_4^N  } 
     \bE \bigg[\int_{t_m}^{t_m+h} 
     \partial_{x_{  \gamma_1}}\bigg(\Delta W_{m,2h}^{\gamma_2} \Delta W_{m,2h}^{\gamma_3}\Delta W_{m,2h}^{\gamma_4}~
     \partial^{3}_{ x_{  \gamma_2},x_{  \gamma_3}, x_{  \gamma_4}     }  
     u(t_{m+1},\overline{\bodX}_{q_1}^{N,h})    \bigg) 
     \dd  W^{\gamma_1}_{q_1} \bigg]
     \\
      \label{eq: r4 term 1.3}
     &+ K  \sum_{ \gamma \in \Pi_4^N  } 
     \bE \bigg[\int_{t_m}^{t_m+h} 
     \partial^2_{x_{  \gamma_1},x_{  \gamma_1}}\bigg(\Delta W_{m,2h}^{\gamma_2} \Delta W_{m,2h}^{\gamma_3}\Delta W_{m,2h}^{\gamma_4}~
     \partial^{3}_{ x_{  \gamma_2},x_{  \gamma_3}, x_{  \gamma_4}     }  
     u(t_{m+1},\overline{\bodX}_{q_1}^{N,h})    \bigg) 
     \dd  {q_1} \bigg].
\end{align}
One observes that $\eqref{eq: r4 term 1.1}=0$ since for $\gamma \in \Pi_3^N$, regardless of the value of $\hat{\cO}(\gamma)$, there will always be an odd power of a Brownian increment $\Delta W_{m,2h}$ presented in \eqref{eq-aux: def of dWms notion}, independent of the Brownian increment contained in the $\hat{\bodX}_{{t_m}}^{N,h}$ term. The term \eqref{eq: r4 term 1.3} however, does not vanish since the $\Delta W_{m,2h}$ term is not independent of $\overline{\bodX}_{q_1}^{N,h}$, because the latter contains the Brownian increment $W_{q_1}-W_{t_m}$.

Applying a further It{\^o}-Taylor expansion around $\hat{\bodX}_{t_m}^{N,h}$, and splitting the zeroth order term into the cases $\hat{\cO}(\gamma) \in\{1,2\}$ yields  
     \begin{align}
     \label{eq: removed from remainder 0 }
     & \eqref{eq: r4 term 1.2}= \nonumber 
     \\
     &
     K h^2  \sum_{ i=1   } ^N
     \bE \bigg[  
     \partial^{4}_{ x_{  i},x_{ i},x_{  i}, x_{  i}    }  
     u(t_{m+1},\hat{\bodX}_{t_m }^{N,h})    
      \bigg]
      +
      K h^2  \sum_{ i,j=1,~i\neq j   }^N
     \bE \bigg[  
     \partial^{4}_{ x_{  i},x_{ i},x_{  j}, x_{  j}    }  
     u(t_{m+1},\hat{\bodX}_{t_m}^{N,h})    
      \bigg]
     \\\label{eq: R4 comb1}
     &+ K \sum_{ \gamma \in \Pi_5^N  } 
     \bE \bigg[\int_{t_m}^{t_m+h} \int_{t_m}^{q_1} 
     \partial_{x_{  \gamma_1},x_{  \gamma_2}} \bigg(\Delta W_{m,2h}^{\gamma_3} \Delta W_{m,2h}^{\gamma_4}\Delta W_{m,2h}^{\gamma_5}~
     \partial^{3}_{ x_{  \gamma_4},x_{  \gamma_4}, x_{  \gamma_5}     }  
     u(t_{m+1},\overline{\bodX}_{q_2}^{N,h})    \bigg) 
      \dd  W^{\gamma_1}_{q_2}
     \
     \dd  W^{\gamma_2}_{q_1} \bigg]
     \\
     \label{eq: R4 comb2}
     &+K \sum_{ \gamma \in \Pi_5^N  } 
     \bE \bigg[\int_{t_m}^{t_m+h} \int_{t_m}^{q_1} 
     \partial^{3}_{x_{  \gamma_1},x_{  \gamma_1},x_{  \gamma_2}} \bigg(\Delta W_{m,2h}^{\gamma_3} \Delta W_{m,2h}^{\gamma_4}\Delta W_{m,2h}^{\gamma_5}~
     \partial^{3}_{ x_{  \gamma_4},x_{  \gamma_4}, x_{  \gamma_5}     }  
     u(t_{m+1},\overline{\bodX}_{q_2}^{N,h})    \bigg) 
      \dd   {q_2}
     \
     \dd  W^{\gamma_2}_{q_1} \bigg],
\end{align}
where the two summation terms in \eqref{eq: removed from remainder 0 }  correspond to the second and third summation in \cite[(3.12)]{leimkuhler2014long}  and they do not contribute to the remainder term $ R^4_{t_m}$. 
\color{black}
Similarly for \eqref{eq: r4 term2}, we have  
\begin{align}\label{eq: removed from remainder 1 }
   & \eqref{eq: r4 term2}
    =  K h 
     \sum_{ i,j=1  }^N
     \bE\Big[ B_{i}( \hat{\bodX}_{{t_m}}^{N,h})   
     \partial^{3}_{ x_{  i},x_{ j}, x_{  j}     }  
     u(t_{m+1},\hat{\bodX}_{{t_m}}^{N,h})  \Big]
     \\
     \label{eq: R4 comb4}
     &+ Kh
      \sum_{ \gamma \in \Pi_4^N  } 
     \bE \bigg[\int_{t_m}^{t_m+h}
     \partial_{x_{  \gamma_1}} \bigg(
     B_{\gamma_2}( \overline{\bodX}_{q_1}^{N,h})
      \Delta W_{m,2h}^{\gamma_3}\Delta W_{m,2h}^{\gamma_4}~
     \partial^{3}_{ x_{  \gamma_2},x_{ \gamma_3}, x_{  \gamma_4}     }  
     u(t_{m+1},\overline{\bodX}_{q_1}^{N,h})  \bigg) 
     \dd   W^{\gamma_1}_{q_1} \bigg]
     \\  \label{eq: R4 comb5}
     &+ K h
     \sum_{ \gamma \in \Pi_4^N  } 
     \bE \bigg[\int_{t_m}^{t_m+h}
     \partial^2_{x_{  \gamma_1},x_{  \gamma_1}} \bigg(
     B_{\gamma_2}( \overline{\bodX}_{q_1}^{N,h})
      \Delta W_{m,2h}^{\gamma_3}\Delta W_{m,2h}^{\gamma_4}~
     \partial^{3}_{ x_{  \gamma_2},x_{ \gamma_3}, x_{  \gamma_4}     }  
     u(t_{m+1},\overline{\bodX}_{q_1}^{N,h})  \bigg) 
     \dd  {q_1} \bigg],
\end{align}
where   \eqref{eq: removed from remainder 1 }  corresponds to the first summation in \cite[(3.12)]{leimkuhler2014long}, thus also does not contribute to $R_{t_m}^4$. Leimkuhler et al. isolates and separates the terms \eqref{eq: removed from remainder 0 } and \eqref{eq: removed from remainder 1 } from $h^3 r_4$, since these terms cleverly cancel with other corresponding lower order terms in the expansion of other $r_i$'s when everything is summed over.
\color{black}

As for \eqref{eq: r4 term3}, we have  
\begin{align}
    \label{eq: 6.2 zero term} 
   & \eqref{eq: r4 term3}
    =  K h^2  
     \sum_{ \gamma \in \Pi_3^N  }
     \bE\Big[ B_{\gamma_1}( \hat{\bodX}_{{t_m}}^{N,h})  B_{\gamma_2}( \hat{\bodX}_{{t_m}}^{N,h})\Delta W_{m,2h}^{\gamma_3}~
     \partial^{3}_{ x_{  \gamma_1},x_{  \gamma_2}, x_{  \gamma_3}     }  
     u(t_{m+1},\hat{\bodX}_{{t_m}}^{N,h})  \Big]
    \\ 
    \label{eq: R4 comb6}
     &+   K h^2
      \sum_{ \gamma \in \Pi_4^N  } 
     \bE \bigg[\int_{t_m}^{t_m+h}
      \partial_{x_{  \gamma_1}}\bigg(
     B_{\gamma_2}( \overline{\bodX}_{q_1}^{N,h})
      B_{\gamma_3}( \overline{\bodX}_{q_1}^{N,h})\Delta W_{m,2h}^{\gamma_4}~
     \partial^{3}_{ x_{  \gamma_2},x_{ \gamma_3}, x_{  \gamma_4}     }  
     u(t_{m+1},\overline{\bodX}_{q_1}^{N,h})  \bigg) 
     \dd   W^{\gamma_1}_{q_1} \bigg]
     \\
     \label{eq: R4 comb7}
     &+  K h^2 
      \sum_{ \gamma \in \Pi_4^N  } 
     \bE \bigg[\int_{t_m}^{t_m+h}
     \partial^2_{x_{  \gamma_1},x_{  \gamma_1}} \bigg(
     B_{\gamma_2}( \overline{\bodX}_{q_1}^{N,h})
       B_{\gamma_3}( \overline{\bodX}_{q_1}^{N,h})
       \Delta W_{m,2h}^{\gamma_4}~
     \partial^{3}_{ x_{  \gamma_2},x_{ \gamma_3}, x_{  \gamma_4}     }  
     u(t_{m+1},\overline{\bodX}_{q_1}^{N,h})  \bigg) 
     \dd  {q_1} \bigg],
\end{align} 
where the first term in this expansion \eqref{eq: 6.2 zero term} is zero, following the same reasoning as for \eqref{eq: r4 term 1.1}.

For the residual term $R^4_{t_m}$, we have that via direct application of Hölder's inequality, It\^{o}'s isometry and Fubini's Theorem, we have
\begin{align*}
     \mathbb{E}& \big[         R^4_{t_m}       \big]
      =
     \eqref{eq: r4 term4}+
     \eqref{eq: r4 term 1.3}+
     \eqref{eq: R4 comb1} + 
     \eqref{eq: R4 comb2}+
     \eqref{eq: R4 comb4}+
     \eqref{eq: R4 comb5}+
     \eqref{eq: R4 comb6}+
     \eqref{eq: R4 comb7}
     \\
     &\leq 
     Kh^3
     \sum_{ \gamma \in \Pi_3^N  }
     \bE\Big[  B_{\gamma_1}( {\bodX}_{t_m}^{N,h})  
     B_{\gamma_2}( {\bodX}_{t_m}^{N,h})
     B_{\gamma_3}( {\bodX}_{t_m}^{N,h})
     ~
     \partial^{3}_{ x_{  \gamma_1},x_{  \gamma_2}, x_{  \gamma_3}     }  
     u(t_{m+1},\bodX_{t_m}^{N,h})  \Big]
     \\
     & \quad + Kh^{3/2}
       ~\int_{t_m}^{t_m+h} 
     ~ \bigg(
    \bE \bigg[ ~\bigg| 
    \sum_{ \gamma \in \Pi_4^N  } \partial^5_{x_{  \gamma_1},x_{  \gamma_1},x_{  \gamma_2},x_{  \gamma_3},x_{ \gamma_4 }} u(t_{m+1},\color{black}
\overline{\bodX}_{q_1}^{N,h}
\color{black} )
    ~\bigg|^2 ~~\bigg]   \bigg)^{1/2} ~ \dd q_1  
    \\
    &\quad+ Kh^{2}\bigg(   \int_{t_m}^{t_m+h}
    \bE \bigg[ ~\bigg| 
    \sum_{ \gamma \in \Pi_5^N  } \partial^5_{x_{  \gamma_1},\dots,x_{ \gamma_5 }} u(t_{m+1},\color{black}
\overline{\bodX}_{q_1}^{N,h}
\color{black}  )
    ~\bigg|^2 ~~\bigg]    ~ 
    \dd q_1  \bigg)^{1/2}
    \\
    &\quad+  Kh^{2}\bigg( \int_{t_m}^{t_m+h}
    \bE \bigg[ ~\bigg| 
    \sum_{ \gamma \in \Pi_5^N  } \partial^6_{x_{  \gamma_1},x_{  \gamma_1},x_{  \gamma_2},x_{\gamma_4},x_{\gamma_4},x_{ \gamma_5 }} u(t_{m+1},\color{black}
\overline{\bodX}_{q_1}^{N,h}
\color{black} )
    ~\bigg|^2 ~~\bigg]      ~ 
    ~ \dd q_1  \bigg)^{1/2}
    \\
    &\quad+  Kh^{2}\bigg( \int_{t_m}^{t_m+h}
    \bE \bigg[ ~\bigg| 
    \sum_{ \gamma \in \Pi_4^N  } \partial_{x_{  \gamma_1}} \bigg(
     B_{\gamma_2}(\color{black}
\overline{\bodX}_{q_1}^{N,h}
\color{black})
     \partial^{3}_{ x_{  \gamma_2},x_{ \gamma_3}, x_{  \gamma_4}     }  
     u(t_{m+1},\color{black}
\overline{\bodX}_{q_1}^{N,h}
\color{black} )  \bigg)
    ~\bigg|^2 ~~\bigg]      ~ \dd q_1  \bigg)^{1/2}
     \\
    &\quad+  Kh^{3/2} \int_{t_m}^{t_m+h}
    \bigg(\bE \bigg[ ~\bigg| 
    \sum_{ \gamma \in \Pi_4^N  } 
    \partial^2_{x_{  \gamma_1},x_{  \gamma_1}} \bigg(
     B_{\gamma_2}(\color{black}
\overline{\bodX}_{q_1}^{N,h}
\color{black})
     \partial^{3}_{ x_{  \gamma_2},x_{ \gamma_3}, x_{  \gamma_4}     }    
     u(t_{m+1},\color{black}
\overline{\bodX}_{q_1}^{N,h}
\color{black} )  \bigg)
    ~\bigg|^2 ~~\bigg]   \bigg)^{1/2}   ~ \dd q_1  
     \\
    &\quad+   Kh^{5/2}\bigg( \int_{t_m}^{t_m+h}
    \bE \bigg[ ~\bigg| 
    \sum_{ \gamma \in \Pi_4^N  } \partial_{x_{  \gamma_1}} \bigg(
     B_{\gamma_2}( \color{black}
\overline{\bodX}_{q_1}^{N,h}
\color{black})
     B_{\gamma_3}( \color{black}
\overline{\bodX}_{q_1}^{N,h}
\color{black})
     \partial^{3}_{ x_{  \gamma_2},x_{ \gamma_3}, x_{  \gamma_4}     }  
     u(t_{m+1},\color{black}
\overline{\bodX}_{q_1}^{N,h}
\color{black} )  \bigg)
    ~\bigg|^2 ~~\bigg]      ~ \dd q_1  \bigg)^{1/2} 
     \\
    &\quad+  Kh^{5/2} \int_{t_m}^{t_m+h}
    \bigg(\bE \bigg[ ~\bigg| 
    \sum_{ \gamma \in \Pi_4^N  } 
    \partial^2_{x_{  \gamma_1},x_{  \gamma_1}} \bigg(
     B_{\gamma_2}(\color{black}
\overline{\bodX}_{q_1}^{N,h}
\color{black})
     B_{\gamma_3}( \color{black}
\overline{\bodX}_{q_1}^{N,h}
\color{black})
     \partial^{3}_{ x_{  \gamma_2},x_{ \gamma_3}, x_{  \gamma_4}     }    
     u(t_{m+1},\color{black}
\overline{\bodX}_{q_1}^{N,h}
\color{black}  )  \bigg)
    ~\bigg|^2 ~~\bigg]~\bigg)^{1/2}    \dd q_1 . 
\end{align*}
As for $R_{t_m}^6$, recalling the definition of the operator $\mathcal{L}_N$ 
\begin{align*}
    \mathcal{L}_N u 
    &= \sum_{i=1}^{N} 
    \Big( 
    B_i \partial_{x_i} u 
    + \frac{1}{2} \sigma^2 \partial^2_{x_i, x_i} u 
    \Big) 
    ,
\end{align*}
the remainder term $r_6(t_m,\cdot)h^3$ in \cite[Equation~(3.14)]{leimkuhler2014long} (apply $\mathcal{L}_N$ three times) corresponds to  
\begin{align*}
   \mathbb{E}\big[           R^6_{t_m}       \big]   
    &=  \bE \bigg[ \int_{t_m}^{t_m+h}
    \int_{t_m}^{q_1}  \int_{t_m}^{q_2}
     (\mathcal{L}_N)^3 u ( t_{m+1},{\bodX}_{q_3}^{N,h} ) 
     \dd q_3 ~ \dd q_2 ~ \dd q_1 ~\bigg]
     \\
    &= \bE \bigg[
     \int_{t_m}^{t_m+h}
    \int_{t_m}^{q_1}  \int_{t_m}^{q_2}
    \Bigg( K   
    \sum_{ \gamma \in \Pi_3^N } 
     B_{\gamma_1}(  {\bodX}_{q_3}^{N,h})
    B_{\gamma_2}(  {\bodX}_{q_3}^{N,h})
    B_{\gamma_3}(  {\bodX}_{q_3}^{N,h})
    \partial^3_{x_{  \gamma_1},x_{  \gamma_2},x_{ \gamma_3 } }
    u(t_{m+1},{\bodX}_{q_3}^{N,h}  )
    \\
    &\qquad \qquad + K
    \sum_{ \gamma \in \Pi_3^N }
     \partial^2_{x_{  \gamma_3},x_{  \gamma_3}}
     \bigg( 
     B_{\gamma_1}(  {\bodX}_{q_3}^{N,h})
    B_{\gamma_2}(  {\bodX}_{q_3}^{N,h})
    \partial^2_{x_{  \gamma_1},x_{  \gamma_2} }
    u(t_{m+1},{\bodX}_{q_3}^{N,h}  )
     \bigg)
     \\
    &\qquad \qquad + K
    \sum_{ \gamma \in \Pi_3^N }
     \partial^4_{x_{  \gamma_2},x_{  \gamma_2},x_{  \gamma_3},x_{  \gamma_3}}
     \bigg( 
     B_{\gamma_1}(  {\bodX}_{q_3}^{N,h})
    \partial_{x_{  \gamma_1} }
    u(t_{m+1},{\bodX}_{q_3}^{N,h}  )
     \bigg)
     \\
    &\qquad \qquad + K
     \sum_{ \gamma \in \Pi_6^N, \hco(\gamma)\leq 3 } \partial^6_{x_{  \gamma_1},\dots,x_{\gamma_6  }} u(t_{m+1},{\bodX}_{q_3}^{N,h}  ) \Bigg) 
     \dd q_3 ~ \dd q_2 ~ \dd q_1 ~\bigg].
    \end{align*}
Once again, a direct application of Hölder's inequality and Fubini's Theorem
    \begin{align*}
\big| \mathbb{E} \big[           R^6_{t_m}       \big]\big| 
    &\leq 
    Kh^2 \int_{t_m}^{t_m+h}
        \bigg(\bE\bigg[  ~  \bigg|~
    \sum_{ \gamma \in \Pi_3^N } 
     B_{\gamma_1}(  \color{black} {\bodX}_{q_1}^{N,h} \color{black})
    B_{\gamma_2}(  \color{black} {\bodX}_{q_1}^{N,h} \color{black})
    B_{\gamma_3}(  \color{black} {\bodX}_{q_1}^{N,h} \color{black})
    \partial^3_{x_{  \gamma_1},x_{  \gamma_2},x_{ \gamma_3 } }
    u(t_{m+1}, \color{black} {\bodX}_{q_1}^{N,h} \color{black}  )
    \\
    &\qquad \qquad +
    \sum_{ \gamma \in \Pi_3^N }
     \partial^2_{x_{  \gamma_3},x_{  \gamma_3}}
     \bigg( 
     B_{\gamma_1}(  \color{black} {\bodX}_{q_1}^{N,h} \color{black})
    B_{\gamma_2}(  \color{black} {\bodX}_{q_1}^{N,h} \color{black})
    \partial^2_{x_{  \gamma_1},x_{  \gamma_2} }
    u(t_{m+1}, \color{black} {\bodX}_{q_1}^{N,h} \color{black}  )
     \bigg)
     \\
    &\qquad \qquad +
    \sum_{ \gamma \in \Pi_3^N }
     \partial^4_{x_{  \gamma_2},x_{  \gamma_2},x_{  \gamma_3},x_{  \gamma_3}}
     \bigg( 
     B_{\gamma_1}(  \color{black} {\bodX}_{q_1}^{N,h} \color{black})
    \partial_{x_{  \gamma_1} }
    u(t_{m+1}, \color{black} {\bodX}_{q_1}^{N,h} \color{black}  )
     \bigg)
     \\
    &\qquad \qquad +
     \sum_{ \gamma \in \Pi_6^N, \hco(\gamma)\leq 3 } \partial^6_{x_{  \gamma_1},\dots,x_{\gamma_6  }} u(t_{m+1},\color{black}
{\bodX}_{q_1}^{N,h}
\color{black} ) 
     ~\bigg|^2
     \bigg]  \bigg)^{1/2} ~ \dd q_1 .
\end{align*}

\textit{Part 2: Other remainder terms.}  
We now present the dominations of the other residual terms, whose explicit expressions can be found in Appendix \ref{appendix_aux_RemainderSection6.2}. These, once again, are a result of applying Hölder's inequality, It\^{o}'s isometry and Fubini's Theorem.
We have
\begin{align*} 
\big|    
    \mathbb{E} \big[          R^2_{t_m}       \big]   \big| 
    &\leq 
    K h^{2}  \bigg( ~ \int_{t_m}^{t_m+h} 
     \bE  \bigg[  ~  \bigg|
    \sum_{ \gamma \in \Pi_5^N } \partial^5_{x_{  \gamma_1},\dots,x_{\gamma_5  }} u(t_{m+1},\color{black} \overline{\bodX}_{q_1}^{N,h} \color{black}  ) ~ \bigg|^2~  \bigg]~ \dd q_1~     \bigg)^{1/2}
    \\
    & \quad +
    K h^{5/2} \bigg( ~ \int_{t_m}^{t_m+h}   
     \bE  \bigg[  ~  \bigg|
    \sum_{ \gamma \in \Pi_6^N  } \partial^6_{x_{  \gamma_1},\dots,x_{\gamma_6  }} u(t_{m+1},\color{black} \overline{\bodX}_{q_1}^{N,h} \color{black}  ) ~ \bigg|^2~  \bigg]~ \dd q_1~     \bigg)^{1/2},
    \\ 
      \big|    
    \mathbb{E}  \big[          R^3_{t_m}        \big]    \big| 
    &\leq  
     K h^{2}  \bigg( ~\int_{t_m}^{t_m+h}   
     \bE  \bigg[  ~  \bigg|
    \sum_{ \gamma \in \Pi_3^N } \partial_{x_{  \gamma_1}}\bigg(
      \partial_{x_{  \gamma_3}} B_{\gamma_2 }   ( \color{black} \overline{\bodX}_{q_1}^{N,h} \color{black} )
     \partial^2_{ x_{  \gamma_2},x_{\gamma_3  } } u(t_{m+1},\color{black} \overline{\bodX}_{q_1}^{N,h} \color{black})
     \bigg) ~ \bigg|^2~  \bigg]~ \dd q_1~     \bigg)^{1/2}
     \\
    &\quad +  
     K h^{2}  \bigg( ~\int_{t_m}^{t_m+h}   
     \bE  \bigg[  ~  \bigg|
    \sum_{ \gamma \in \Pi_3^N } \partial^2_{x_{  \gamma_1},x_{  \gamma_1}}\bigg(
      \partial_{x_{  \gamma_3}} B_{\gamma_2 }   ( \color{black} \overline{\bodX}_{q_1}^{N,h} \color{black} )
     \partial^2_{ x_{  \gamma_2},x_{\gamma_3  } } u(t_{m+1},\color{black} \overline{\bodX}_{q_1}^{N,h} \color{black})
     \bigg) ~ \bigg|^2~  \bigg]~ \dd q_1~     \bigg)^{1/2}
     \\
    &\quad +  
     K h^{2}  \bigg( ~\int_{t_m}^{t_m+h}   
     \bE  \bigg[  ~  \bigg|
    \sum_{ \gamma \in \Pi_3^N } \partial_{x_{  \gamma_1}}\bigg(
        B_{\gamma_2 }   ( \color{black} \overline{\bodX}_{q_1}^{N,h} \color{black} )
     \partial^2_{ x_{  \gamma_2},x_{\gamma_3  },x_{\gamma_3  } } u(t_{m+1},\color{black} \overline{\bodX}_{q_1}^{N,h} \color{black})
     \bigg) ~ \bigg|^2~  \bigg]~ \dd q_1~     \bigg)^{1/2}
     \\
    &\quad +  
     K h^{2}  \bigg( ~\int_{t_m}^{t_m+h}   
     \bE  \bigg[  ~  \bigg|
    \sum_{ \gamma \in \Pi_3^N } \partial^2_{x_{  \gamma_1},x_{  \gamma_1}}\bigg(
      B_{\gamma_2 }   ( \color{black} \overline{\bodX}_{q_1}^{N,h} \color{black} )
     \partial^2_{ x_{  \gamma_2},x_{\gamma_3  } ,x_{\gamma_3  }} u(t_{m+1},\color{black} \overline{\bodX}_{q_1}^{N,h} \color{black})
     \bigg) ~ \bigg|^2~  \bigg]~ \dd q_1~     \bigg)^{1/2}
     \\
      &\quad + 
    K h^{2}  \bigg( ~\int_{t_m}^{t_m+h}   
     \bE  \bigg[  ~  \bigg|
    \sum_{ \gamma \in \Pi_5^N } \partial^5_{x_{  \gamma_1},\dots,x_{\gamma_5  }} u(t_{m+1},\color{black} \overline{\bodX}_{q_1}^{N,h} \color{black}  ) ~ \bigg|^2~  \bigg]~ \dd q_1~     \bigg)^{1/2}
    \\
    &\quad +
    K h^{5/2} \bigg( ~  \int_{t_m}^{t_m+h}  
     \bE  \bigg[  ~  \bigg|
    \sum_{ \gamma \in \Pi_6^N  } \partial^6_{x_{  \gamma_1},\dots,x_{\gamma_6  }} u(t_{m+1},\color{black} \overline{\bodX}_{q_1}^{N,h} \color{black}  ) ~ \bigg|^2~  \bigg]~ \dd q_1~     \bigg)^{1/2} 
    ,
    \\
      \big|\mathbb{E}  \big[         R^5_{t_m}        \big] \big|&\leq 
     Kh^{5/2}
      \sum_{ \gamma \in \Pi_4^N  }
      \bigg(
       \bE \bigg[ ~\bigg|
         \partial^{4}_{ x_{  \gamma_1},x_{  \gamma_2}, x_{  \gamma_3}    , x_{  \gamma_4}   }  
     u(t_{m+1},\bodX_{t_m}^{N,h})
       ~\bigg|^2 ~\bigg]
       ~ \bigg)^{1/2}
     \\
     &\quad 
     + Kh^{2}
      \bigg( ~\int_{t_m}^{t_m+h}  
    \bE \bigg[ ~\bigg| 
    \sum_{ \gamma \in \Pi_5^N  } \partial^5_{x_{  \gamma_1},\dots,x_{ \gamma_5 }} u(t_{m+1},\color{black} \overline{\bodX}_{q_1}^{N,h} \color{black}  )
    ~\bigg|^2 ~\bigg]   ~ \dd q_1 ~ \bigg)^{1/2}
    \\
    &\quad  + Kh^{2}
     \bigg( \int_{t_m}^{t_m+h}  ~
    \bE \bigg[ ~\bigg| 
    \sum_{ \gamma \in \Pi_5^N  } \partial^6_{x_{  \gamma_1},x_{  \gamma_1},x_{  \gamma_2},x_{  \gamma_3},x_{  \gamma_4},x_{ \gamma_5 }} u(t_{m+1},\color{black} \overline{\bodX}_{q_1}^{N,h} \color{black}  )
    ~\bigg|^2 ~\bigg]     ~ \dd q_1~\bigg)^{1/2},
    \\
      \big|\mathbb{E} \big[          R^7_{t_m}        \big] \big|
     &\leq  K h^{2} 
    \bigg( ~ \int_{t_m}^{t_m+h}
    \bE \bigg[  ~ \bigg|  ~ 
     \sum_{ \gamma \in \Pi_6^N, \hco(\gamma)\in\{4,5\} } \partial^6_{x_{  \gamma_1},\dots,x_{\gamma_6  }} u(t_{m+1},\color{black} \overline{\bodX}_{q_1}^{N,h} \color{black}  )
    ~ \bigg|^2~\bigg]
     ~ \dd q_1~ \bigg)^{1/2}
     \\
     & 
     +      
     K h^{2} 
    \bigg( ~ \int_{t_m}^{t_m+h}
    \bE \bigg[  ~ \bigg|  ~ 
     \sum_{ \gamma \in \Pi_5^N, \hco(\gamma)=4  } \partial^6_{x_{  \gamma_1},\dots,x_{\gamma_6  }} u(t_{m+1},\color{black} \overline{\bodX}_{q_1}^{N,h} \color{black}  )
    ~ \bigg|^2~\bigg]
     ~ \dd q_1~ \bigg)^{1/2}
    ,
    \\
      \big|  
    \mathbb{E}  \big[         R^8_{t_m}     \big]     \big|  
    &
    \leq  K h^{2}
    \bigg( ~ \int_{t_m}^{t_m+h} 
    \bE \bigg[ ~ \bigg|  ~
    \sum_{ \gamma \in \Pi_5^N,
    \hco(\gamma)\leq 3 } \partial^5_{x_{  \gamma_1},\dots,x_{ \gamma_5 }} u(t_{m+1},\color{black} \overline{\bodX}_{q_1}^{N,h} \color{black}  )
     ~ \bigg|^2~ \bigg]
    \dd q_1~ \bigg)^{1/2} 
    \\
    &+
    K h^{2}
     \int_{t_m}^{t_m+h} 
     ~\bigg(\bE \bigg[ ~ \bigg|  ~
    \sum_{ \gamma \in \Pi_6^N,
    \hco(\gamma)\leq 3 } \partial^6_{x_{  \gamma_1},\dots,x_{ \gamma_6 }} u(t_{m+1},\color{black} \overline{\bodX}_{q_1}^{N,h} \color{black}  )
     ~ \bigg|^2~ \bigg]\bigg)^{1/2} 
    \dd q_1~ 
    \\
    \quad 
    &
    +
    K h^{2}\bigg( ~ \int_{t_m}^{t_m+h} 
    \\
    &  ~ 
    \bE \bigg[ ~ \bigg|  ~ 
    \sum_{ \substack{    \alpha,\beta \in \bigcup_{k=0}^{5} \Pi_k^N, ~ \hco(\alpha \bigcup \beta )=3 
     \\      
     |\alpha|+|\beta|=5,~|\beta|\geq 2 
     } 
     }
     \partial^{ |\alpha| }_{x_{  \alpha_1},\dots,x_{ \alpha_{|\alpha|} }}B_{\beta_1}(   \color{black} \overline{\bodX}_{q_1}^{N,h} \color{black}     )
     ~ \partial^{|\beta| }_{x_{  \beta_1},\dots,x_{ \beta_{|\beta| } }} u(t_{m+1},\color{black} \overline{\bodX}_{q_1}^{N,h} \color{black}  )
    \bigg) ~ 
    ~ \bigg|^2~ \bigg]
    \dd q_1~ \bigg)^{1/2}.
\end{align*} 
Similar to \eqref{eq-aux:our R1 vs Leimkuhler's r1 new} and the proof of Lemma~\ref{lemma: analysis of the b0 term summation integration}, we can use the linear growth of $B_j$ in combination with the assumption $\xi \in L^{10}(\Omega,\bR)$ and the Cauchy--Schwarz inequality to establish $L^{10}$-estimates of the $B_j$ (see, Proposition~\ref{prop:basic_estimates11} and Lemma~\ref{prop: moment bound extension}) which, for instance, are employed in the estimation of $R^3_{t_m},R^5_{t_m},R^8_{t_m}$. An application of Lemma~\ref{lemma: higher-order derivative of u} allows to control the moments of derivatives of the solution, $u$, to the Kolmogorov backward equation. Hence, all the above expectations are well-defined and bounded by $K (h^{5/2}+h^3) e^{ -\lambda_0 (T-t_{m+1})}$.

\medskip

\textit{Part 3: Collecting the estimates.} For the eight residual terms $R^1_{t_m},\ldots,R^8_{t_m}$, we have
\begin{align*}
    \sum_{i=1}^8  \big|    
    \mathbb{E}  \big[          R^i_{t_m}       \big]    \big| 
    \leq 
    K  (h^{5/2} +h^3)  e^{-\lambda_0 (T-t_{m+1})}.
\end{align*}  
Combining all of the above estimates, we conclude
\begin{align*}
    \sum_{m=0}^{M-1}
    \mathbb{E} \big[ R(t_m,\boldsymbol{X}^{N,h}_{t_m})     \big]  
    \leq 
    \sum_{m=0}^{M-1}
    \sum_{i=1}^8
     \big|    
    \mathbb{E}  \big[         R^i_{t_m}      \big]   \big|  
    \leq 
    \sum_{m=0}^{M-1}
    K (h^{5/2} +h^3) e^{-\lambda_0 (T-t_{m+1})}
    \leq Kh^{3/2}. 
\end{align*}
\end{proof}


\section{Numerical illustration}
\label{sec:Numericalexample}
 We illustrate the performance of the non-Markovian Euler scheme \eqref{eq: def : non-Markov Euler scheme} with a simple linear MV-SDE example. Consider the following mean-field equation:
 \begin{align}
    \label{eq: example: linear 1}
    {  \color{black}      
     \dd X_t = \Big(  -\alpha \big(X_t - \bE [ X_t ] \big) -X_t\Big) \dd t +
     \sigma \dd W_t } , 
     \quad 
     X_0 \in L^{10} (\Omega, \bR),
 \end{align}
where $\alpha,\sigma>0$. Explicit calculations yield $\bE[X_t]=\bE[X_0]e^{-t}$ and thus this process admits the following stationary distribution 
\begin{align}
    \label{eq: example: linear 1 density}
   \mu^{*}(  x) =  Z \exp\Big(-\frac{\alpha +1}{\sigma^2}  x^2   \Big),
\end{align}  
where $Z$ is the renormalization constant such that $\int_{\bR} \mu^{*}(x) \mathrm{d}x=1$. 
We use the following recipe to compute errors in the numerical experiments.  
\begin{enumerate}
    \item Choose a large enough domain $[a,b]$ such that $\big(F_{\textrm{cdf}}(b)-F_{\textrm{cdf}}(a)\big)>1-10^{-6}$, where $F_{\textrm{cdf}}$ denotes the cumulative density function (CDF) of the invariant distribution $\mu$ of  \eqref{eq: example: linear 1 density}.  
    \item Split the domain $[a,b]$ into $N_{\textrm{bins}}$ equally spaced bins and compute the true density in each bin. These values are denoted $ (\mu_{i}^{\textrm{true}})_{i\in \{1,\dots, N_{\textrm{bins}}\}}$ and obtained using numerical integration. 
    Additionally, the first bin and the value $\mu_{1}^{\textrm{true}}$ takes the interval $(-\infty,a]$ into account while the last bin (and the value  $\mu_{N_{\textrm{bins}}}^{\textrm{true}}$) takes the interval $[b,\infty)$. In this way, one has 
    $\sum_{i=1}^{ N_{\textrm{bins}}  } \mu_{i}^{\textrm{true}} = 1$. 
    \item {\color{black}We simulate $N$ particles up to time $T$ using different time-stepping schemes 
 and compute an approximation of the density (using a histogram approach) denoted by $ (\mu_{i}^{\textrm{proxy}})_{i\in \{1,\dots, N_{\textrm{bins}}\}}$ based on the simulated paths.}
 {\color{black}     
    We choose $f(x)=x\cdot \1_{x\geq 0}$ for  $x\in \bR$ as the test function in the weak error test in  Figure \ref{fig: eg2 }. }
 
    \item As in \cite{leimkuhler2014long}, we compute the relative entropy error and the $L_2$-Error as  
        \begin{align*}
        \textrm{ Relative Entropy Error } 
        & =  
        \sum_{i=1}^{ N_{\textrm{bins}}  } 
        \mu_{i}^{\textrm{true}} 
        \ln \Big( \frac{\mu_{i}^{\textrm{true}}}{\mu_{i}^{\textrm{proxy}}}
            \Big)
            \\
        \textrm{ $L_2$-Error } 
        & = \sqrt{  \sum_{i=1}^{ N_{\textrm{bins}}  }      | \mu_{i}^{\textrm{true}}- \mu_{i}^{\textrm{proxy}} |^2 }.
        \end{align*} 
        
     \item We also show  a  PoC  result by  computing the $L_2$-Error for different sizes $N$ of particles  with fixed  time $T$ and timestep $h$ (chosen independently of  $T,M,N$).    
\end{enumerate}
The goal of this simple example is to simulate the IPS associated with \eqref{eq: example: linear 1} up to $T=9$ using the classical Euler method and the non-Markovian Euler method, respectively, and compare the results to \eqref{eq: example: linear 1 density}. Following the recipe above, we set the domain as $[a,b]=[-1.8,1.8]$ and split it into $72$ bins.  
\begin{figure}[h!bt]
\centering 
    \begin{subfigure}{.30\textwidth}
			\centering
 			\includegraphics[scale=0.25]{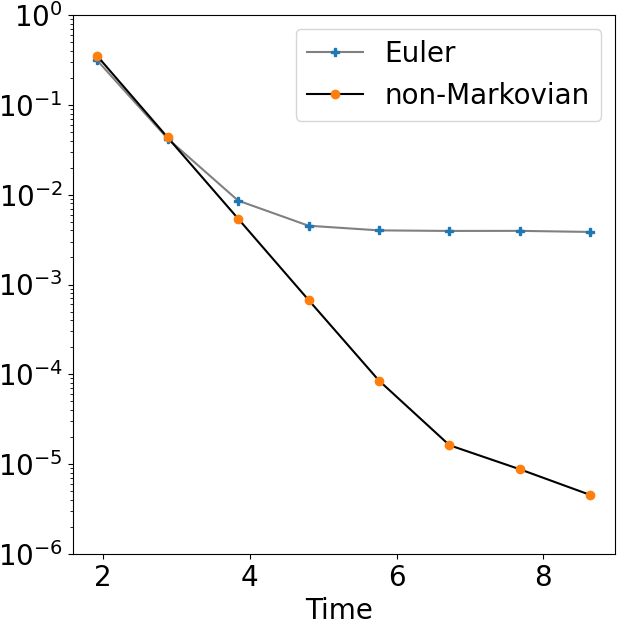}
			\caption{Rel.~Entropy Error}
		\end{subfigure}%
  \ \ 
		\begin{subfigure}{.30\textwidth}          
			\centering
 			\includegraphics[scale=0.25]{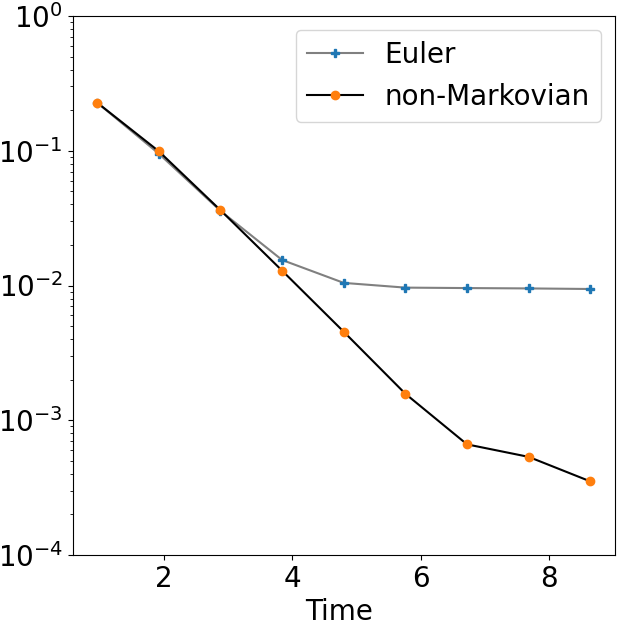}
			\caption{$L_2$-Error}
		\end{subfigure}	
  \ \ 
           \begin{subfigure}{.31\textwidth}          
			\centering
 			\includegraphics[scale=0.25]{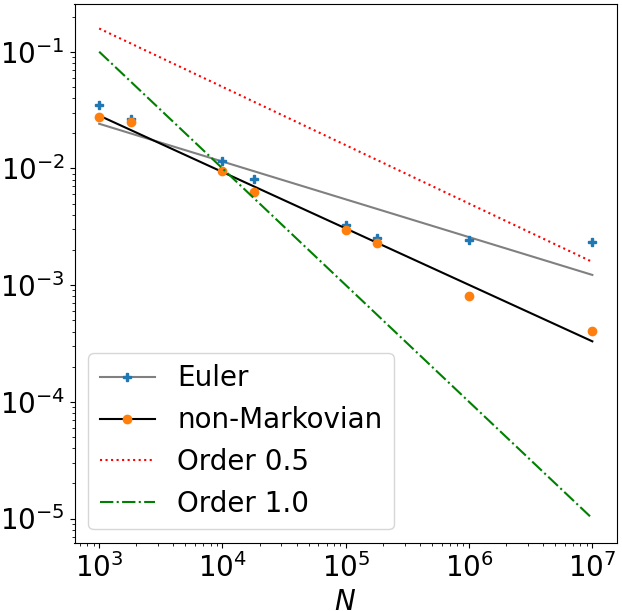}
			\caption{PoC (log-)$L_2$-Error, $T=9$}
		\end{subfigure}	
    \caption{Simulation of the linear MV-SDE \eqref{eq: example: linear 1} with $\alpha =0.5, \sigma = 0.8$, $N=10^7, h = 0.16$, and $X_0 \sim \mathcal{N}(\pi,1)$ (a normal distribution with mean value $\pi$ and variance $1$). Both schemes run on the exact same samples of the initial condition and Brownian increments.   
     (a) Entropy Error of the Euler method and non-Markovian method in log-scale over time.  
     (b) $L_2$-Error of the Euler method and non-Markovian method in log-scale over time.
     (c) $L_2$-Error in particle size $N$ of the Euler method and non-Markovian method in log-scale with respect to different number of particles $N$ at $T=9$. 
     }
    \label{fig: eg1 }
\end{figure}

Figure \ref{fig: eg1 } (a) and (b) show that the non-Markovian method (uniformly) outperforms the Euler method in the approximation of the stationary distribution $\mu$ in both error metrics and the difference in performance becomes more evident as time increases. This is expected from the result in Theorem \ref{theorem: main weak convergence}, since the first order term decays exponentially as $T$ increases. As in \cite[Fig.2]{leimkuhler2014long} the numerical errors of the methods plateau as $T$ increases with the non-Markovian method performing several orders of magnitude better.

Figure \ref{fig: eg1 } (c) shows (for a fixed timestep $h=0.04$) that the PoC $L_2$-Error of the non-Markovian method decays consistently as $N$ increases with a rate of approximately $0.5$ (the expected strong PoC $\cO(1/\sqrt{N})$ rate). The error of the Euler method plateaus for $N>10^5$ as the error from the time-discretization dominates the particle error; see also Table \ref{tab:my-table 22} below for more information.
  
\color{black}
\begin{figure}[h!bt]
\centering 
    \begin{subfigure}{.30\textwidth}
			\centering
 			\includegraphics[scale=0.25]{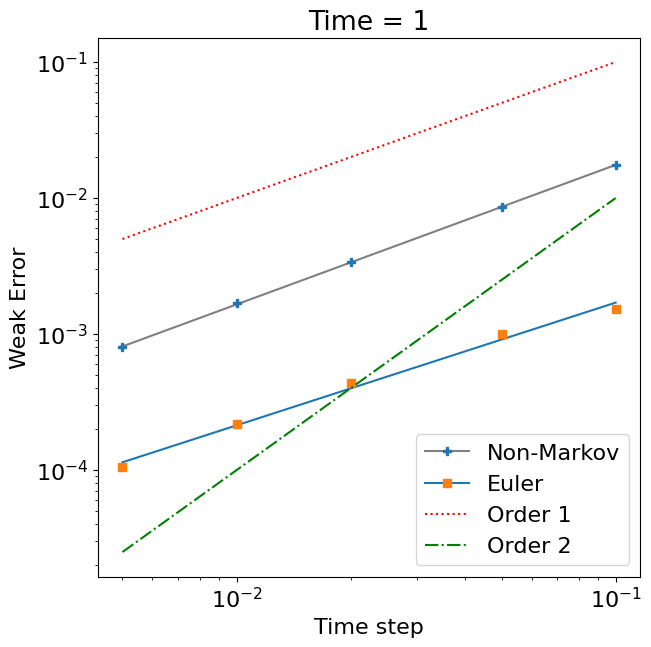}
			\caption{  $T=1$, Slope = $1.02$}
		\end{subfigure}%
  \ \ 
		\begin{subfigure}{.30\textwidth}          
			\centering
 			\includegraphics[scale=0.25]{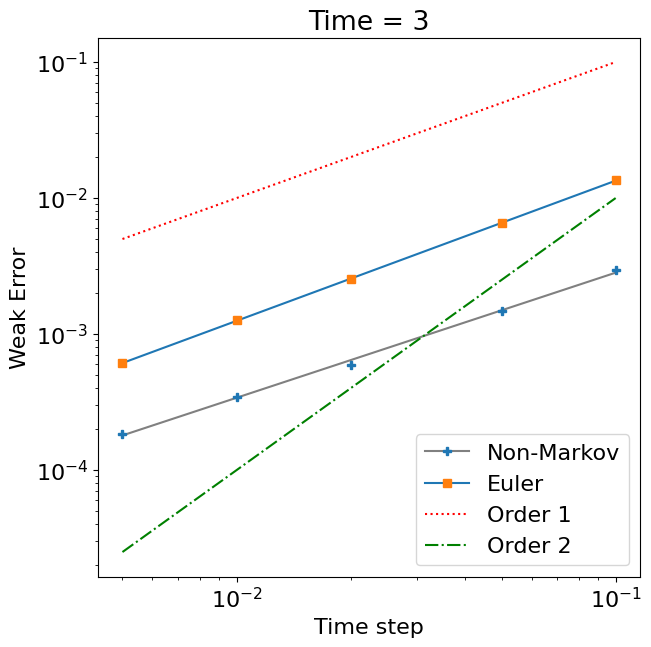}
			\caption{ $T=3$, Slope = $0.92$}
		\end{subfigure}	
  \ \ 
           \begin{subfigure}{.31\textwidth}          
			\centering
 			\includegraphics[scale=0.25]{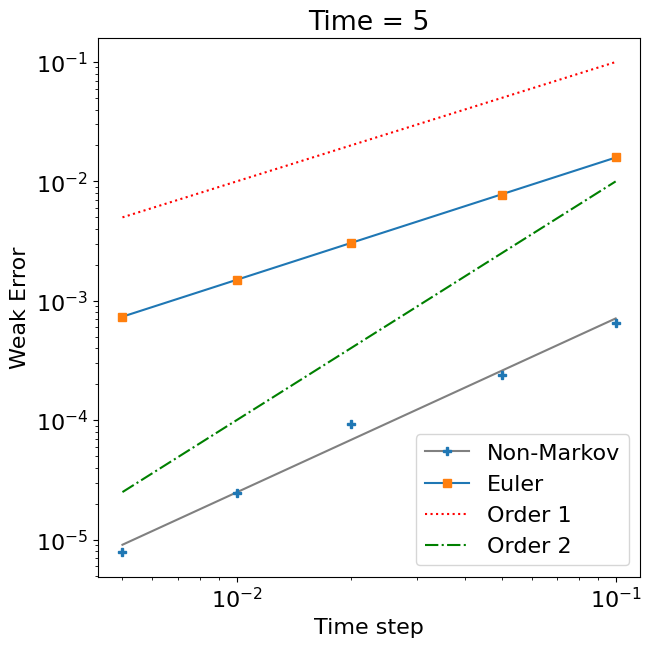}
			\caption{  $T=5$, Slope = $1.45$}
		\end{subfigure}	
    \caption{
    \color{black}
    Simulation of the linear MV-SDE \eqref{eq: example: linear 1} with $\alpha =0.5, \sigma = 0.8$, $N=10^7$, and $X_0 \sim \mathcal{N}(1,1)$ (a normal distribution with mean value $1$ and variance $1$) with different choices of time step $h\in \{0.005,0.01,0.02,0.05,0.1\}$ at different time $T\in\{1,3,5\}$; the simulation at $T=7$ or $T=9$ is not different from that in (c) at $T=5$, we thus do not present them. 
    Both schemes run on the exact same samples of the initial condition and Brownian increments.   
    We show the slope of the regression line of the non-Markovian method.  
     }
    \label{fig: eg2 }
\end{figure}

For Figure \ref{fig: eg2 }, we take for the weak error the test functional $\bE[X_T \cdot \mathbbm{1}_{X_T>0} ]$ -- this is a non-symmetric map as $\bE[X_T]=0$ since $X_T$ follows a symmetric distribution with zero mean.  
As the time increases, we see the weak error rate also increase from around $1.0$ to $1.5$. An improvement in the weak error rate is observed when increasing the number of particles from $10^5$ to $10^7$, as shown in later results. Nevertheless, the non-Markovian method outperforms the Euler method in terms of weak convergence.  
\color{black}


\begin{table}[h!bt]
\centering 
\begin{tabular}{ccccccllll}
\hline
\multirow{2}{*}{$\alpha$} & \multirow{2}{*}{$\sigma$} & \multirow{2}{*}{$a$}    & \multirow{2}{*}{$b$}   & \multirow{2}{*}{$N_{\textrm{bins}}$} & \multirow{2}{*}{h} & \multicolumn{2}{c}{Entropy Error}                      & \multicolumn{2}{c}{$L_2$-error}                           \\ \cline{7-10} 
                       &                        &                       &                      &                        &                           & \multicolumn{1}{c}{Euler}    & \multicolumn{1}{c}{NM} & \multicolumn{1}{c}{Euler}    & \multicolumn{1}{c}{NM} \\ \hline
\multirow{4}{*}{$0.5$}   & \multirow{4}{*}{$0.8$}   & \multirow{4}{*}{$-1.8$} & \multirow{4}{*}{$1.8$} & \multirow{4}{*}{$72$}    & $0.04$                      & \multicolumn{1}{l}{2.33E-04} & 4.71E-06                & \multicolumn{1}{l}{2.37E-03} & 3.56E-04                \\ \cline{6-10} 
                       &                        &                       &                      &                        & 0.16                      & \multicolumn{1}{l}{3.84E-03} & 4.33E-06                & \multicolumn{1}{l}{9.47E-03} & 3.37E-04                \\ \cline{6-10} 
                       &                        &                       &                      &                        & 0.24                      & \multicolumn{1}{l}{9.26E-03} & 4.40E-06                & \multicolumn{1}{l}{1.47E-02} & 3.25E-04                \\ \cline{6-10} 
                       &                        &                       &                      &                        & 0.48                      & \multicolumn{1}{l}{4.31E-02} & 3.25E-06                & \multicolumn{1}{l}{3.18E-02} & 2.92E-04                \\ \hline
\multirow{4}{*}{$0.3$}   & \multirow{4}{*}{$1.5$}   & \multirow{4}{*}{$-3.0$}   & \multirow{4}{*}{$3.0$}   & \multirow{4}{*}{$120$}   & $0.04$                      & \multicolumn{1}{l}{1.84E-04} & 7.85E-06                & \multicolumn{1}{l}{1.44E-03} & 3.81E-04                \\ \cline{6-10} 
                       &                        &                       &                      &                        & 0.16                      & \multicolumn{1}{l}{2.98E-03} & 5.94E-06                & \multicolumn{1}{l}{5.88E-03} & 2.82E-04                \\ \cline{6-10} 
                       &                        &                       &                      &                        & 0.24                      & \multicolumn{1}{l}{6.84E-03} & 6.07E-06                & \multicolumn{1}{l}{9.00E-03} & 3.19E-04                \\ \cline{6-10} 
                       &                        &                       &                      &                        & 0.48                      & \multicolumn{1}{l}{3.08E-02} & 5.72E-06                & \multicolumn{1}{l}{1.95E-02} & 3.09E-04                \\ \hline
\end{tabular}
\caption{Simulation results for MV-SDE \eqref{eq: example: linear 1} with $N=10^7,T=8.64$ and different choices of parameters using the non-Markovian (NM) Euler and standard Euler method, respectively. 
(As for Fig.~\ref{fig: eg1 }: $X_0 \sim \mathcal{N}(\pi,1)$ and both schemes run on the exact same samples of the initial condition and Brownian increments.)
}
\label{tab:my-table 11}
\end{table}

Table \ref{tab:my-table 11} shows that the non-Markovian method has a significantly better approximation accuracy compared to the Euler method under different choices of timesteps and model parameters. The Euler method produces larger errors as the timestep increases and the non-Markovian method yields stable results across all choices for the timestep. 

The results in Table \ref{tab:my-table 22} show (at fixed timestep $h=0.04$) the entropy error and the $L_2$-Error of the non-Markovian method decaying as the number of particles $N$ increases. However, the error of the Euler method remains stable for $N>10^5$ (i.e., there is a plateau). Due to computational limitations, we are not able to show results beyond $N>10^7$. The terminal time $T=8.64$ is chosen for convenience only (due to the smaller timestep $h=0.04$).

\begin{table}[h!bt]
\centering 
\begin{tabular}{ccccccclll}
\hline
\multirow{2}{*}{$\alpha$} & \multirow{2}{*}{$\sigma$} & \multirow{2}{*}{$a$}    & \multirow{2}{*}{$b$}   & \multirow{2}{*}{$N_{\textrm{bins}}$} & \multirow{2}{*}{$N$} & \multicolumn{2}{c}{Entropy Error}                      & \multicolumn{2}{c}{$L_2$-Error}                           \\ \cline{7-10} 
                       &                        &                       &                      &                        &                    & \multicolumn{1}{c}{Euler}    & \multicolumn{1}{c}{NM} & \multicolumn{1}{c}{Euler}    & \multicolumn{1}{c}{NM} \\ \hline
\multirow{5}{*}{$0.5$}   & \multirow{5}{*}{$0.8$}   & \multirow{5}{*}{$-1.8$} & \multirow{5}{*}{$1.8$} & \multirow{5}{*}{$72$}    & $10^{3}$                & \multicolumn{1}{c}{-}        & \multicolumn{1}{c}{-}  & \multicolumn{1}{l}{2.89E-02} & 3.28E-02                \\ \cline{6-10} 
                       &                        &                       &                      &                        & $10^4$                 & \multicolumn{1}{c}{-}        & \multicolumn{1}{c}{-}  & \multicolumn{1}{l}{1.01E-02} & 1.04E-02                \\ \cline{6-10} 
                       &                        &                       &                      &                        & $10^5$                 & \multicolumn{1}{l}{8.21E-04} & 4.83E-04                & \multicolumn{1}{l}{4.29E-03} & 3.10E-03                \\ \cline{6-10} 
                       &                        &                       &                      &                        & $10^6$                 & \multicolumn{1}{l}{2.74E-04} & 4.66E-05                & \multicolumn{1}{l}{2.31E-03} & 1.26E-03               \\ \cline{6-10} 
                       &                        &                       &                      &                        & $10^7$                 & \multicolumn{1}{l}{2.33E-04} & 4.71E-06                & \multicolumn{1}{l}{2.37E-03} & 3.56E-04                \\ \hline
\end{tabular}
\caption{ 
Simulation results for MV-SDE \eqref{eq: example: linear 1} with $h=0.04$ and $T=8.64$ for increasing numbers of particles $N$. (As for Fig.~\ref{fig: eg1 }: $X_0 \sim \mathcal{N}(\pi,1)$ and both schemes run on the exact same samples of the initial condition and Brownian increments.) 
}
\label{tab:my-table 22}
\end{table}


\appendix
\section{Appendix }

\subsection{Proof of Proposition~\ref{prop:basic_estimates11}}
\label{appendix_proof_basic11}

\begin{proof}[Proof of Proposition~\ref{prop:basic_estimates11}]

Assumption~\ref{assum:main} is sufficient to guarantee the existence of the unique stationary distribution; see  \cite{cattiaux2008probabilistic} (under their Assumption~(A')). The uniform PoC result follows from \cite[Theorem 1.2]{Malrieu2001}. This addresses the proposition's last two statements.

The system's wellposedness (as an SDE in $\bR^N$) follows \cite{Malrieu2001}. For completeness, we present a short proof for the moment stability of the IPS, highlighting that the constant $K$ appearing in the RHS of the inequality is independent of $t$ and $N$.

Let $p \geq 2$. Performing similar calculations as in \cite[Appendix A]{chen2022SuperMeasureIMA}  and applying Gronwall's inequality, we deduce that there exists some positive constant $K $ independent of $N,t \geq 0$, such that for any  $i \neq j$ 
\begin{align}
    \label{eq: contraction particle system}
    \mathbb{E}[\,|X^{i,N}_t-X^{j,N}_t|^{p}] \leq K    < \infty.
\end{align} 
Employing It\^{o}'s formula, we deduce for any $t \geq 0$ , $p\geq 2$ and $\kappa \in [0,\lambda)$
\begin{align*}
    & e^{p\kappa t}|X^{i,N}_t|^{p} - |\xi^{i}|^{p} 
    \\
    &\leq \int_{0}^{t} p \bigg( \kappa |X_s^{i,N}|^2 + \big(    X_s^{i,N} 
      \big)   \cdot	\Big(  
    -\nabla U(X_s^{i,N}) - \frac{1}{N}\sum_{j=1}^N \nabla V(X_s^{i,N}-X_s^{j,N})  \Big)  \bigg) e^{p \kappa s} |X^{i,N}_s|^{p-2}  \mathrm{d}s 
    \\
    & \quad + \sigma^2\frac{p(p-1)}{2}\int_{0}^{t} e^{p \kappa s} |X_s^{i,N}|^{p-2} \mathrm{d}s + \sigma p\int_{0}^{t}  e^{p \kappa s} |X_s^{i,N}|^{p-2}  X_s^{i,N}  \mathrm{d}W_s^i . 
\end{align*}
Taking expectations on both sides, using the inequality: for $q \in \lbrace 1, 2 \rbrace$
\begin{align*}
    a^{p-q}b^q \leq \frac{p-q}{p} \varepsilon a^{p} + \frac{q}{p \varepsilon^{(p-q)/q}} b^{p},
    \qquad \textrm{for any $\varepsilon >0$ and $a,b >0$}, 
\end{align*}
assumption \ref{assum:main} and \eqref{eq: contraction particle system} yields
\begin{align*}
    \mathbb{E}&\left[  e^{p \kappa  t}|X^{i,N}_t|^{p} \right] - \mathbb{E}\left[|\xi^{i}|^{p} \right]  
    \\
    & \leq  p \int_{0}^{t} ( \kappa-\lambda)  \mathbb{E}\left[ e^{p \kappa 
 s}|X^{i,N}_s|^{p}  \right] \mathrm{d}s 
 \\
 & \qquad \qquad 
    + 
      \int_{0}^{t} e^{p \kappa  s }
    \Big( \frac{K}{N}\sum_{j=1}^N\bE \Big[  |X^{i,N}_s|^{p-1}~|X^{i,N}_s-X^{j,N}_s|  \Big] \Big) 
    \mathrm{d}s
    + 
    K \int_{0}^{t} e^{p \kappa  s } \mathbb{E} \left[  |X_s^{i,N}|^{p-2} \right] \mathrm{d}s 
    \\
    & \leq  
    p \int_{0}^{t} ( \kappa-\lambda+\varepsilon)  \mathbb{E}\left[ e^{\kappa p s}|X^{i,N}_s|^{p}  \right] \mathrm{d}s 
    + 
    K \int_{0}^{t}   e^{p \kappa s }   \dd s 
    \leq K   e^{p \kappa t } ,
\end{align*}
for some positive constant $K$ and $\varepsilon>0$ arbitrarily small. 
\end{proof}
We provide the following auxiliary results:
\begin{proposition}
\label{prop:basic_estimates11xx} 
 Let the assumptions and set up of Proposition~\ref{prop:basic_estimates11} hold and let the IPS be given as in \eqref{ModelIPS}. Then there exists $K \geq 0$, independent of $t,s$ and $N$, such that for any   $s\geq t \geq 0,~s-t< 1$

\begin{align*}
   \max_{i \in \lbrace 1, \ldots, N \rbrace } 
    \bE \Big[  
        \big|X_{s}^{i,N}-X_{t}^{i,N}   \big|^2 
        \Big] \leq K(s-t).
\end{align*}
\end{proposition}
\begin{proof}
    By  It\^o's formula and an application of the Cauchy--Schwarz and Jensen inequalities, we have that for any $s\geq t \geq 0,~(s-t)< 1$, $i\in\{1,\dots,N\}$,
    \begin{align*}
        \bE \Big[  &
        \big|X_{s}^{i,N}-X_{t}^{i,N}   \big|^2 
        \Big]
        =  \bE \bigg[ \bigg| 
        \int_t^s 
             -\nabla U(X_u^{i,N}) - \frac{1}{N}\sum_{j=1}^N \nabla V(X_u^{i,N}-X_u^{j,N}) \dd u
        + \sigma \int_t^s \dd  W_u^i
        \bigg|^2\bigg]
        \\
        & \leq K  (s-t)\int_t^s \bE \Big[  	\big|  \nabla U(X_u^{i,N})\big|^2\Big] + \frac{1}{N}\sum_{j=1}^N\bE \Big[ \big| \nabla V(X_u^{i,N}-X_u^{j,N})\big|^2	\Big]  \dd u 
        + K (s-t) 
        \\
        &\leq  K(s-t) \Big( 1 + \int_t^s \bE \Big[  	\big|  X_u^{i,N} \big|^2\Big] + 
        \frac{1}{N}\sum_{j=1}^N
        \bE \Big[  	\big|  X_u^{j,N} \big|^2\Big]\dd u \Big)         
        \leq K(s-t), 
    \end{align*}
where we used Proposition~\ref{prop:basic_estimates11} in the last estimate.      
\end{proof}

\subsection{Proof of Proposition~\ref{prop:basic_estimates22}}\label{appendix_proof_basic22} 
For clarity, we prove each of the proposition's statements separately. The main argument requires care as an analysis by cases is needed, but within each case the estimation procedure is standard. 
\begin{proof}[Proof of Proposition~\ref{prop:basic_estimates22} -- Statement (1): Moment estimates]

For the scheme defined in \eqref{eq: def : non-Markov Euler scheme}, we introduce the following notations for any $m \in \lbrace 0, \ldots, M\rbrace$,
\begin{align}
\label{eq-auc:1step of scheme}
    b_m^{i,N,h} \coloneqq -\nabla U(X^{i,N,h}_{t_{m}}) - \frac{1}{N}\sum_{j=1}^{N} \nabla V(X_{t_{m}}^{i,N,h} - X_{t_m}^{j,N,h}),
    \quad
    \Delta \overline{W}_m^{i,N,h} \coloneqq  \frac{\Delta W_{m}^i + \Delta W_{m+1}^i}{2}. 
\end{align}
Assumption~\ref{assum:main} implies that 
\begin{align*}
       \big( X_{t_{m}}^{i,N,h} - X_{t_m}^{j,N,h}    
      \big)   \cdot  &	\big(  
    ~b_m^{i,N,h}-  b_m^{j,N,h}  \big) 
    \leq 
    -\lambda 
    \big|    X_{t_{m}}^{i,N,h} - X_{t_m}^{j,N,h}  \big| ^2	,
    \\
   \big|~b_m^{i,N,h}-  b_m^{j,N,h} \big|^2  &\leq 
    2\big|\nabla U(X^{i,N,h}_{t_{m}})-\nabla U(X^{j,N,h}_{t_{m}}) \big|^2
    \\
    &\quad 
        + \frac{2}{N}\sum_{k=1}^{N}\big| \nabla V(X_{t_{m}}^{i,N,h} - X_{t_m}^{k,N,h})- \nabla V(X_{t_{m}}^{j,N,h} - X_{t_m}^{k,N,h})\big|^2
    \\
    & 
    \leq 2(\lambda^2+K_V^2)  |    X_{t_{m}}^{i,N,h} - X_{t_m}^{j,N,h}  | ^2
    \leq K  |    X_{t_{m}}^{i,N,h} - X_{t_m}^{j,N,h}  | ^2.
\end{align*}
An inspection of the above inequalities indicates that in order to establish the $L^p$-moments for $X_{t_m}^{i,N,h}$, we need $L^p$-estimates on the local differences $X_{t_m}^{i,N,h}-X_{t_m}^{j,N,h}$ (a term that appears from the interaction kernel). This proof is split accordingly; \textit{Part 1} deals with the latter while \textit{Part 2} with the former.
\smallskip

\textit{Part 1: Moments of local differences uniformly bounded in time.}
We first prove that for all $ i,j\in  \lbrace 1, \ldots, N \rbrace$, $m \in \lbrace 0, \ldots, M\rbrace, p\geq 2$, with $p$ an even number, that $\mathbb{E}\big[  |X_{t_m}^{i,N,h}-X_{t_m}^{j,N,h} |^{p}\big] $ is uniformly bounded in time. Note that due to the nature of the scheme, $ X_{t_{m}}^{i,N,h}$ is not independent of $\Delta \overline{W}_{m}^{i,N,h}$, and thus we analyze the different cases as time evolves, i.e., $m\in\{0,1,2\}$ and $m\geq 3$ below. (This same procedure will be used in \textit{Part 2} of the proof.)  
\medskip

\noindent
\textit{Case: $m=0$.} 
For any $ i,j\in \lbrace 1, \ldots, N \rbrace$, we have
\begin{align}
    \label{eq: apdx diff res m0}
    \bE \big[ |X_{t_0}^{i,N,h}-X_{t_0}^{j,N,h} |^{p} \big]
    &\leq K \bE \big[    |X_{t_0}^{i,N,h}  |^{p} + |X_{t_0}^{j,N,h}  |^{p}\big]
    =  K \bE\big[|\xi|^{p}\big].
\end{align}
\noindent
\textit{Case: $m=1$.} 
We get
\begin{align*}
    |X_{t_1}^{i,N,h}-X_{t_1}^{j,N,h} |^{2}
    & 
    =
    |X_{t_0}^{i,N,h}+b_0^{i,N,h}h-X_{t_0}^{j,N,h}-b_0^{j,N,h}h |^{2}
    +
     |\sigma\Delta \overline{W}_0^{i,N,h}-\sigma\Delta \overline{W}_0^{j,N,h}|^2
    \\
    &\quad+ 2
    \big(    X_{t_0}^{i,N,h}+b_0^{i,N,h}h-X_{t_0}^{j,N,h}-b_0^{j,N,h}h
     \big)   \cdot  	\big(   
    \sigma\Delta \overline{W}_0^{i,N,h}-\sigma\Delta \overline{W}_0^{j,N,h}
    \big) 
    \\
    & \leq
    |X_{t_0}^{i,N,h}-X_{t_0}^{j,N,h}|^{2} 
    \big(1-2\lambda h + Kh^2 \big)
     +
     |\sigma\Delta \overline{W}_0^{i,N,h}-\sigma\Delta \overline{W}_0^{j,N,h}|^2
    \\
    &\quad+ 2
    \big(   X_{t_0}^{i,N,h}+b_0^{i,N,h}h-X_{t_0}^{j,N,h}+b_0^{j,N,h}h
       \big)   \cdot  	\big(   
    \sigma\Delta \overline{W}_0^{i,N,h}-\sigma\Delta \overline{W}_0^{j,N,h}
    \big).
\end{align*}
Taking the power $p/2$ and expectation on both sides, we deduce that there exist positive constants $K_{p,1}$, $K_{p,2}$, $K$, $\kappa$ (all are independent of $h,T,M$ and $N$) such that  
\color{black}  
\begin{align*}    
\nonumber
    \bE \big[ |X_{t_1}^{i,N,h}-&X_{t_1}^{j,N,h} |^{p} \big]
    \\
    &
    \leq 
    \bE \bigg[ \Big|  
    |X_{t_0}^{i,N,h}-X_{t_0}^{j,N,h}|^{2} 
    \big(1-2\lambda h + Kh^2 \big)
     +
     |\sigma\Delta \overline{W}_0^{i,N,h}-\sigma\Delta \overline{W}_0^{j,N,h}|^2 
      \\
    &\quad+ 2
    \big(   X_{t_0}^{i,N,h}+b_0^{i,N,h}h-X_{t_0}^{j,N,h}+b_0^{j,N,h}h
       \big)   \cdot  	\big(   
    \sigma\Delta \overline{W}_0^{i,N,h}-\sigma\Delta \overline{W}_0^{j,N,h}  \big)
    \Big|^{p/2} \bigg] 
    \\
    &
    \leq
     \bE \big[ |X_{t_0}^{i,N,h}-X_{t_0}^{j,N,h} |^{p} \big] \big(1-2\lambda h + Kh^2 \big)^{p} 
     \\
     & 
     \quad + h \big( K_{p,1}\bE \big[ |X_{t_0}^{i,N,h}-X_{t_0}^{j,N,h} |^{p} \big]+ K_{p,2} \big)
     \\
    \nonumber
    & \quad +  p
    \bE \big[ |\sigma\Delta \overline{W}_0^{i,N,h}-\sigma\Delta \overline{W}_0^{j,N,h}|^2   \cdot   |X_{t_0}^{i,N,h}-X_{t_0}^{j,N,h} |^{p-2}  \big(1-2\lambda h + Kh^2 \big)^{p-2} \big]
    \\ \nonumber
    &\leq (1-K_{p,1} h ) \bE \big[ |X_{t_0}^{i,N,h}-X_{t_0}^{j,N,h} |^{p} \big]+~ K_{p,2} h\big(1+  \bE\big[|\xi|^{p-2}\big]  \big)
    \\
    &\leq
    K\big(  h+ \bE\big[|\xi|^{p }\big] + \bE\big[|\xi|^{p-2}\big]  \big)  
    \leq  
    K e^{\kappa t_1} \big( 1+\bE\big[|\xi|^{p}\big]e^{-  {\kappa t_1} }   \big),
\end{align*}
where we keep the $\cO(1)$ and $\cO(h)$ terms by expanding the formula in the first inequality and bound the $o(h)$ terms by $\cO(h)$ terms via Young's inequality.
\color{black}
Note that in this case (and the case $m=2$) the factor $e^{- \kappa t_1}$ ($e^{- \kappa t_2}$ for $m=2$) is only added to make it consistent with the estimates obtained for a general $m$. \\ \\
\noindent
\textit{Case: $m=2$.} 
Based on the calculations above, there exist positive constants $ K, {\kappa}$ (both are independent of $h,T,M$ and $N$) such that  
\begin{align}
    \label{eq: apdx diff res m2}
     \bE \big[ |X_{t_2}^{i,N,h}-X_{t_2}^{j,N,h} |^{p} \big]
    & 
    \leq   K e^{\kappa t_2}(1+\bE\big[|\xi|^{p}\big]e^{- {\kappa} t_2} ).
\end{align}
Note that below we will show that the constant on the right-hand side will not blow up as $m$ increases.  \\ \\
 \noindent
\textit{Case: $m \geq 3$.} 
More generally, for $m \geq 3$, we have 
\begin{align}
    \label{eq: m3 square expression}
    |X_{t_m}^{i,N,h}-X_{t_m}^{j,N,h} |^{2}
    & 
    \leq \big(1-2\lambda h + Kh^2 \big)
    |X_{t_{m-1}}^{i,N,h}-X_{t_{m-1} }^{j,N,h}|^{2} 
     +
    \sigma^2|\Delta \overline{W}_{m-1}^{i,N,h}-\Delta \overline{W}_{m-1}^{j,N,h}|^2
    \\
    \nonumber
    &\quad+ 2
     \big(   X_{t_{m-1}}^{i,N,h}+b_{m-1}^{i,N,h}h-X_{t_{m-1}}^{j,N,h}+b_{m-1}^{j,N,h}h
     \big)   \cdot  	\big(   
    \sigma\Delta \overline{W}_{m-1}^{i,N,h}-\sigma\Delta \overline{W}_{m-1}^{j,N,h}
    \big) .
\end{align}
Since for $m\geq 3$, $ X_{t_{m-1}}^{i,N,h}$ is not independent of $\Delta \overline{W}_{m-1}^{i,N,h}$, we further expand $ X_{t_{m-1}}^{\cdot,N,h}$ to get 
\begin{align*}
    \big(  &  X_{t_{m-1}}^{i,N,h} +b_{m-1}^{i,N,h}h-X_{t_{m-1}}^{j,N,h}+b_{m-1}^{j,N,h}h
      \big)   \cdot  	\big(   
    \sigma\Delta \overline{W}_{m-1}^{i,N,h}-\sigma\Delta \overline{W}_{m-1}^{j,N,h}
    \big)
    \\
    & 
    \leq
    \big(    X_{t_{m-2}}^{i,N,h} -X_{t_{m-2}}^{j,N,h} 
    + b_{m-2}^{i,N,h}h - b_{m-2}^{j,N,h}h
    + \sigma\Delta \overline{W}_{m-2}^{i,N,h}-\sigma\Delta \overline{W}_{m-2}^{j,N,h}
      \big)   \cdot  	\big(   
    \sigma\Delta \overline{W}_{m-1}^{i,N,h}-\sigma\Delta \overline{W}_{m-1}^{j,N,h}
    \big) 
    \\
    & \quad +
   Kh  \big|  X_{t_{m-1}}^{i,N,h}-X_{t_{m-1} }^{j,N,h} \big| 
     \big|
      \sigma\Delta \overline{W}_{m-1}^{i,N,h}-\sigma\Delta \overline{W}_{m-1}^{j,N,h}
    \big|. 
\end{align*}
Define the following local quantities: for all $i,j \in \lbrace 1, \ldots, N\rbrace	$  
\begin{align*}
    G_{m,1}^{i,j} &= \big(1-2\lambda h + Kh^2 \big)
    |X_{t_{m-1}}^{i,N,h}-X_{t_{m-1} }^{j,N,h}|^{2} 
    ,\quad
    G_{m,2}^{i,j} =
    \sigma^2|\Delta \overline{W}_{m-1}^{i,N,h}-\Delta \overline{W}_{m-1}^{j,N,h}|^2
    \\
     G_{m,3}^{i,j} & = 2\big(   X_{t_{m-2}}^{i,N,h} -X_{t_{m-2}}^{j,N,h}
     + \sigma\Delta \overline{W}_{m-2}^{i,N,h}-\sigma\Delta \overline{W}_{m-2}^{j,N,h}
      \big)   \cdot  	\big(   
    \sigma\Delta \overline{W}_{m-1}^{i,N,h}-\sigma\Delta \overline{W}_{m-1}^{j,N,h}
    \big) 
    ,
    \\
    G_{m,4}^{i,j} &=2
     \big(     b_{m-2}^{i,N,h}h - b_{m-2}^{j,N,h}h
     \big)   \cdot  	\big(   
    \sigma\Delta \overline{W}_{m-1}^{i,N,h}-\sigma\Delta \overline{W}_{m-1}^{j,N,h}
    \big) 
    \\
    G_{m,5}^{i,j} &= 2 Kh  \big|  X_{t_{m-1}}^{i,N,h}-X_{t_{m-1} }^{j,N,h} \big| 
     \big|
      \sigma\Delta \overline{W}_{m-1}^{i,N,h}-\sigma\Delta \overline{W}_{m-1}^{j,N,h}
    \big|.
\end{align*}
Using these local quantities, we can express the estimate \eqref{eq: m3 square expression} as follows:
\begin{align*}
    |X_{t_m}^{i,N,h}-X_{t_m}^{j,N,h} |^{2}
    & 
    \leq \big( G_{m,1}^{i,j}+ G_{m,2}^{i,j}+ G_{m,3}^{i,j}+ G_{m,4}^{i,j}+ G_{m,5}^{i,j} \big) .
\end{align*}
Now, taking the power of $p/2$ and expectations on both sides, in combination with Young’s inequality, we have for  some positive constants $K_{p,1}, K_{p,2}, K$ ( both are independent of $h,T,M$ and $N$)  such that  
\begin{align}
\nonumber
   \mathbb{E}  \Big[ &\big( G_{m,1}^{i,j}+ G_{m,2}^{i,j}+ G_{m,3}^{i,j}+ G_{m,4}^{i,j}+ G_{m,5}^{i,j} \big)^{p/2}  \Big] 
   \\ \nonumber 
   &
    =  \sum_{l=0}^{p/2} \binom{p/2}{l} \mathbb{E}  \Big[
   (G_{m,1}^{i,j})^{p/2-l} (G_{m,2}^{i,j}+ G_{m,3}^{i,j}+ G_{m,4}^{i,j}+ G_{m,5}^{i,j})^{l} \Big]
   \\
   \nonumber 
   & \leq \mathbb{E} \big[   | G_{m,1}^{i,j}|^{p/2} \big]  
    \\
    \label{eq: g terms 15}
    &
   \qquad + p\mathbb{E}    \Big[    | G_{m,1}^{i,j}|^{p/2-1} (G_{m,2}^{i,j} +G_{m,3}^{i,j}+G_{m,4}^{i,j}+ G_{m,5}^{i,j}) 
    + 
    | G_{m,1}^{i,j}|^{p/2-2}
    |G_{m,3}^{i,j}|^2
    \Big] 
    \\
    \label{eq: g terms 15-2}
    &\qquad 
    +    h^{3/2}  \Big(    K_{p,1}\bE \big[ \big|  X_{t_{m-1}}^{i,N,h}-X_{t_{m-1} }^{j,N,h} 
     \big|^{p} \big]  + K_{p,2}\bE \big[  \big|  X_{t_{m-2}}^{i,N,h}-X_{t_{m-2} }^{j,N,h} 
     \big|^{p}\big]       +K   \Big),
\end{align}
where we used the fact that 
\begin{align*}
    & \mathbb{E}  \Big[(G_{m,1}^{i,j})^{p/2-2} (G_{m,2}^{i,j}+ G_{m,3}^{i,j}+ G_{m,4}^{i,j}+ G_{m,5}^{i,j})^{2}  \Big]
    \\
    &\leq  
    \mathbb{E}    \Big[   
    | G_{m,1}^{i,j}|^{p/2-2}
    |G_{m,3}^{i,j}|^2
    \Big] 
    +
     h^{3/2}  \Big(    K\bE \big[ \big|  X_{t_{m-1}}^{i,N,h}-X_{t_{m-1} }^{j,N,h} 
     \big|^{p} \big]  + K\bE \big[  \big|  X_{t_{m-2}}^{i,N,h}-X_{t_{m-2} }^{j,N,h} 
     \big|^{p}\big]       +K   \Big)
     ,
\end{align*}
\color{black}
and the terms for $l \in \{3,\dots, p/2\} $ are also bounded by the  term in \eqref{eq: g terms 15-2}.  
The \eqref{eq: g terms 15} terms are estimated next: 
\color{black}
\begin{align*}
    & \bE \big[  (G_{m,1}^{i,j})^{p/2-1} G_{m,5}^{i,j} \big] 	
      \\
      & \leq Kh \mathbb{E}\left[
    \big|  X_{t_{m-2}}^{i,N,h}-X_{t_{m-2} }^{j,N,h} 
     +  b_{m-2}^{i,N,h}h - b_{m-2}^{j,N,h}h
    + \sigma\Delta \overline{W}_{m-2}^{i,N,h}-\sigma\Delta \overline{W}_{m-2}^{j,N,h}
     \big|^{p-1}
     \big|  \sigma\Delta \overline{W}_{m-1}^{i,N,h}-\sigma\Delta \overline{W}_{m-1}^{j,N,h} \big| \right]
     \\
     &
     \leq  Kh\bE \big[   	
     \big|  X_{t_{m-2}}^{i,N,h}-X_{t_{m-2} }^{j,N,h} 
     \big|^{p-1} \big|  \sigma\Delta \overline{W}_{m-1}^{i,N,h}-\sigma\Delta \overline{W}_{m-1}^{j,N,h} \big|\big]  + 
      K h^{3/2} \bE \big[   	
     \big|  X_{t_{m-2}}^{i,N,h}-X_{t_{m-2} }^{j,N,h} 
     \big|^{p-1}\big]  +Kh^{3/2}
     \\
     &\leq K h^{3/2} \bE \big[   	
     \big|  X_{t_{m-2}}^{i,N,h}-X_{t_{m-2} }^{j,N,h} 
     \big|^{p-1}\big]  + Kh^{3/2}.
\end{align*}
Also, note that we have the bound 
\begin{align*}
     \bE \big[  (G_{m,1}^{i,j})^{p/2}  \big] 	
      &\leq (1-Kh)   \bE \big[  	
     \big|  X_{t_{m-1}}^{i,N,h}-X_{t_{m-1} }^{j,N,h} 
     \big|^{p} \big].
\end{align*} 
In the following, $\varepsilon>0$ (arbitrarily small)  and  $K>0$ (both are independent of $h,T,M$ and $N$) will denote positive constants appearing due to the application of Young's inequality. In addition recall that $X_{t_{m-2}}^{\cdot,N,h}  $ is independent of $W_{m-1}^{\cdot,N,h}$. Consequently, we derive  
\begin{align*}
     \bE \big[ & ( G_{m,1}^{i,j})^{p/2-1} G_{m,2}^{i,j}  \big] 	
     \leq  
      (1-Kh)  \bE \big[  	
         \big|X_{t_{m-1}}^{i,N,h}-X_{t_{m-1} }^{j,N,h} 
     \big|^{p-2}  \sigma^2|\Delta \overline{W}_{m-1}^{i,N,h}-\Delta \overline{W}_{m-1}^{j,N,h}|^2\big] 
      \\
     &\leq 
      (1-Kh)  \bE \big[  	
       (K +\varepsilon\big|X_{t_{m-1}}^{i,N,h}-X_{t_{m-1} }^{j,N,h} 
     \big|^{p}) \sigma^2|\Delta \overline{W}_{m-1}^{i,N,h}-\Delta \overline{W}_{m-1}^{j,N,h}|^2\big]  
     \\
     &\leq
     Kh +\varepsilon
     \bE \big[  	
     \big|  X_{t_{m-2}}^{i,N,h}-X_{t_{m-2} }^{j,N,h} 
     +  b_{m-2}^{i,N,h}h - b_{m-2}^{j,N,h}h
    + \sigma\Delta \overline{W}_{m-2}^{i,N,h}-\sigma\Delta \overline{W}_{m-2}^{j,N,h}
     \big|^{p} \sigma^2|\Delta \overline{W}_{m-1}^{i,N,h}-\Delta \overline{W}_{m-1}^{j,N,h}|^2\big]
     \\
     &\leq
     Kh + \varepsilon
      \bE \big[  	
     \big|  X_{t_{m-2}}^{i,N,h}-X_{t_{m-2} }^{j,N,h}  
     \big|^{p} \sigma^2|\Delta \overline{W}_{m-1}^{i,N,h}-\Delta \overline{W}_{m-1}^{j,N,h}|^2\big]
     + K h^{3/2} 
     \\
     &\quad+
     K \bE \big[  	
     \big|   b_{m-2}^{i,N,h}h - b_{m-2}^{i,N,h}h  
     \big|^{p}\sigma^2|\Delta \overline{W}_{m-1}^{i,N,h}-\Delta \overline{W}_{m-1}^{j,N,h}|^2\big]
     \\
     &\quad+
     K \bE \big[  	
     \big|  \sigma\Delta \overline{W}_{m-2}^{i,N,h}-\sigma\Delta \overline{W}_{m-2}^{j,N,h} 
     \big|^{p}  \sigma^2|\Delta \overline{W}_{m-1}^{i,N,h}-\Delta \overline{W}_{m-1}^{j,N,h}|^2\big]
     \leq Kh +  \varepsilon h \bE \big[  	
     \big|  X_{t_{m-2}}^{i,N,h}-X_{t_{m-2} }^{j,N,h}  
     \big|^{p} \big].
\end{align*} 
Similarly, we have  
\begin{align*}
     \bE \big[ & ( G_{m,1}^{i,j})^{p/2-1} (G_{m,3}^{i,j}+G_{m,4}^{i,j})  \big] 
     \leq
     Kh +  \varepsilon h \bE \big[  	
     \big|  X_{t_{m-1}}^{i,N,h}-X_{t_{m-1} }^{j,N,h}  
     \big|^{p} \big]
     \\
     &
      \quad + 2(1-Kh) \bE \big[  	
     \big|  X_{t_{m-1}}^{i,N,h}-X_{t_{m-1} }^{j,N,h} 
     \big|^{p-2}  
     \big(   X_{t_{m-1}}^{i,N,h} -X_{t_{m-1}}^{j,N,h}
       \big)   \cdot  	\big(   
    \sigma\Delta \overline{W}_{m-1}^{i,N,h}-\sigma\Delta \overline{W}_{m-1}^{j,N,h}
     \big)
     \big] 
     \\
     &
     \leq K
     \bE \Big[ 
     \big|  \sigma\Delta \overline{W}_{m-2}^{i,N,h}-\sigma\Delta \overline{W}_{m-2}^{j,N,h} 
     \big|^{p-1}
     \big(    X_{t_{m-1}}^{i,N,h} -X_{t_{m-1}}^{j,N,h}
      \big)   \cdot  	\big(  
    \sigma\Delta \overline{W}_{m-1}^{i,N,h}-\sigma\Delta \overline{W}_{m-1}^{j,N,h}
    \big)
     \Big] 
     \\&\quad +K
     \bE \Big[ 
     \big|    X_{t_{m-2}}^{i,N,h}-X_{t_{m-2} }^{j,N,h}  
     \big|^{p-1}
     \big(   \sigma\Delta \overline{W}_{m-2}^{i,N,h}-\sigma\Delta \overline{W}_{m-2}^{j,N,h}
     \big)   \cdot  	\big(  
    \sigma\Delta \overline{W}_{m-1}^{i,N,h}-\sigma\Delta \overline{W}_{m-1}^{j,N,h}
     \big) 
     \Big]  
     \\&\quad +
     K
     \bE \Big[ 
        \big|  b_{m-2}^{i,N,h}h - b_{m-2}^{i,N,h}h  
     \big|^{p-1}
     \big( \sigma\Delta \overline{W}_{m-2}^{i,N,h}-\sigma\Delta \overline{W}_{m-2}^{j,N,h}
      \big)   \cdot  	\big(  
    \sigma\Delta \overline{W}_{m-1}^{i,N,h}-\sigma\Delta \overline{W}_{m-1}^{j,N,h}
     \big) 
     \Big] 
      \\&\quad 
     +  \varepsilon h \bE \big[  	
     \big|  X_{t_{m-1}}^{i,N,h}-X_{t_{m-1} }^{j,N,h}  
     \big|^{p} \big] +
     Kh + Kh^{3/2}
     \\
     &
     \leq K
     \bE \Big[ 
     \big|  \sigma\Delta \overline{W}_{m-2}^{i,N,h}-\sigma\Delta \overline{W}_{m-2}^{j,N,h} 
     \big|^{p-1}
     \big(     X_{t_{m-3}}^{i,N,h} -X_{t_{m-3}}^{j,N,h}
     \big)   \cdot  	\big(  
    \sigma\Delta \overline{W}_{m-1}^{i,N,h}-\sigma\Delta \overline{W}_{m-1}^{j,N,h}
     \big) 
     \Big] 
     \\&\quad +K
     \bE \Big[ 
     \big|    X_{t_{m-3}}^{i,N,h}-X_{t_{m-3} }^{j,N,h}  
     \big|^{p-1}
     \big(  \sigma\Delta \overline{W}_{m-2}^{i,N,h}-\sigma\Delta \overline{W}_{m-2}^{j,N,h}
      \big)   \cdot  	\big(   
    \sigma\Delta \overline{W}_{m-1}^{i,N,h}-\sigma\Delta \overline{W}_{m-1}^{j,N,h}
    \big) 
     \Big] 
     \\&\quad 
     +  \varepsilon h \bE \big[  	
     \big|  X_{t_{m-1}}^{i,N,h}-X_{t_{m-1} }^{j,N,h}  
     \big|^{p} \big] +
     Kh + Kh^{3/2}
     \\
     &  
     \leq  
     Kh + Kh^{3/2} + \varepsilon h \bE \big[  	
     \big|  X_{t_{m-3}}^{i,N,h}-X_{t_{m-3} }^{j,N,h}  
     \big|^{p} \big] + 
     \varepsilon h \bE \big[  	
     \big|  X_{t_{m-1}}^{i,N,h}-X_{t_{m-1} }^{j,N,h}  
     \big|^{p} \big] ,
\end{align*}
where we expanded the second term in the second inequality, used Young’s inequality and the fact that $X_{t_{m-3}}^{\cdot,N,h}  $ is independent of $W_{m-2}^{\cdot,N,h}$. For the last term, we have 
 \begin{align*}
     \bE &\big[  ( G_{m,1}^{i,j})^{p/2-2} |G_{m,3}^{i,j}|^2  \big] 	
     \\
     &
      \leq 
      K \bE \big[  	
     \big|  X_{t_{m-1}}^{i,N,h}-X_{t_{m-1} }^{j,N,h} 
     \big|^{p-4}  \big( 
     \big|  X_{t_{m-2}}^{i,N,h} -X_{t_{m-2}}^{j,N,h}
       \big|^2 
       +  \big|  
    \sigma\Delta \overline{W}_{m-2}^{i,N,h}-\sigma\Delta \overline{W}_{m-2}^{j,N,h}
     \big|^2
       \big)
      \\
      & \qquad \quad
       \cdot  	\big|  
    \sigma\Delta \overline{W}_{m-1}^{i,N,h}-\sigma\Delta \overline{W}_{m-1}^{j,N,h}
     \big|^2
     \big] 
      \\
     & 	\leq
      K \bE \big[  	
     \big|  X_{t_{m-2}}^{i,N,h}-X_{t_{m-2} }^{j,N,h} 
     +  b_{m-2}^{i,N,h}h - b_{m-2}^{j,N,h}h
    + \sigma\Delta \overline{W}_{m-2}^{i,N,h}-\sigma\Delta \overline{W}_{m-2}^{j,N,h} 
     \big|^{p-4}  
     \\
     & \qquad \quad \cdot 
     \big( 
     \big|  X_{t_{m-2}}^{i,N,h} -X_{t_{m-2}}^{j,N,h}
       \big|^2 
       +  \big|  
    \sigma\Delta \overline{W}_{m-2}^{i,N,h}-\sigma\Delta \overline{W}_{m-2}^{j,N,h}
     \big|^2
       \big)\cdot  	\big|  
    \sigma\Delta \overline{W}_{m-1}^{i,N,h}-\sigma\Delta \overline{W}_{m-1}^{j,N,h}
     \big|^2
     \big] 
      \\
     & 	\leq
      Kh    + \varepsilon h \bE \big[  	
     \big|  X_{t_{m-2}}^{i,N,h}-X_{t_{m-2} }^{j,N,h}  
     \big|^{p} \big]
     . 
\end{align*}
Hence, there exist positive constants $K_{p,3}, K_{p,4},K_{p,5}$ (all independent of $h,T,M$ and $N$)  
 satisfying $K_{p,3}>2 K_{p,4} $ 
 (by carefully choosing the constants in Young's inequality) such that  
\begin{align}
\nonumber
    \bE &\big[ |X_{t_m}^{i,N,h}-X_{t_m}^{j,N,h} |^{p} \big]
    \\
    \nonumber
    &
    \leq 
    (1-K_{p,3} h ) \bE \big[ |X_{t_{m-1}}^{i,N,h}-X_{t_{m-1}}^{j,N,h} |^{p} \big]
    \\ 
    \nonumber
    & \qquad \qquad 
    + K_{p,4} h \big( \bE \big[ |X_{t_{m-2}}^{i,N,h}-X_{t_{m-2}}^{j,N,h} |^{p} \big]
    + \bE \big[ |X_{t_{m-3}}^{i,N,h}-X_{t_{m-3}}^{j,N,h} |^{p} \big]
    \big) 
    +K_{p,5} h \\
    \label{eq: i j  pre  difference }
    &\leq
    K
    + Ke^{2\kappa}(1+\bE\big[|\xi|^{p}\big]  )e^{-(K_{p,3}-2K_{p,4})t_{m}/3},
\end{align}
for some constant $K$ (independent of $h,T,M$ and $N$). {\color{black}In the last estimate, we used Lemma~\ref{prop: sequence inequality 1} with $c_1 \equiv (1-K_{p,3} h ), c_2 \equiv K_{p,4} h, c_3 \equiv K_{p,4} h$, $C \equiv K_{p,5} h$ and the moment bounds for the differences process at $\{t_0,t_1,t_2\}$ in \eqref{eq: apdx diff res m0} and \eqref{eq: apdx diff res m2}}. Note that the condition $0<c_1 + c_2 + c_3  <1$ is satisfied for our choice $h\in(0,\min \{1/2\lambda,1\})$.

We conclude that there exist some positive constants $K,  {\kappa}$ (both independent of $h,T,M$ and $N$), such that for all $i,j \in \lbrace 1, \ldots, N \rbrace$, and $m\geq 0$
\begin{align}
    \label{eq: differences bound for shceme}
    \bE &\big[ |X_{t_m}^{i,N,h}-X_{t_m}^{j,N,h} |^{p} \big]
    \leq
    K\big(1+ \bE\big[|\xi|^{p}\big]e^{ -{\kappa} t_m } \big).
\end{align}
\smallskip

\noindent
\textit{Part 2: Moments are uniformly bounded in time.}
Let $p \geq 2$ be given. We now prove that for all $ i\in \lbrace 1, \ldots, N \rbrace, m\geq 0,$ $   \mathbb{E}\big[  |X_{t_m}^{i,N,h}|^{p}\big]$ is uniformly bounded in time. As in \textit{Part 1}, we separately consider the cases $m\in\{0,1,2\}$ then $m\geq 3$.

Let $ i\in \lbrace 1, \ldots, N \rbrace$ be arbitrary and set $m=0$. Then we have by assumption on the initial data 
\begin{align*}
    \bE\big[  |X_{0}^{i,N,h}|^{p}\big]
    = \bE\big[|\xi|^{p}\big] < \infty.
\end{align*}
\noindent
\textit{Case 1: $m=1$.}  
Due to Assumption~\ref{assum:main} and Jensen's inequality, it follows that 
\begin{align*}
    |X_{t_1}^{i,N,h} |^{2}
    & 
    =
    |X_{t_0}^{i,N,h}+b_0^{i,N,h}h  |^{2}
    +
     |\sigma\Delta \overline{W}_0^{i,N,h}|^2
    +
   2 \big(   X_{t_0}^{i,N,h}+b_0^{i,N,h}h 
     \big)   \cdot  	\big(   
    \sigma\Delta \overline{W}_0^{i,N,h}
    \big) 
    \\
    & \leq
    (1-2\lambda h+ Kh^2)|X_{t_0}^{i,N,h} |^{2}
    + 2 \big(    X_{t_0}^{i,N,h}+b_0^{i,N,h}h 
      \big)   \cdot  	\big(   
    \sigma\Delta \overline{W}_0^{i,N,h}
     \big) +  |\sigma\Delta \overline{W}_0^{i,N,h}|^2
    \\
    &\quad +
     \frac{2K_V h}{N}\sum_{j=1}^{N} | X_{t_{0}}^{i,N,h} - X_{t_0}^{j,N,h} ||X_{t_0}^{i,N,h} |
     + \frac{K_V^2 h^2}{N}\sum_{j=1}^{N} | X_{t_{0}}^{i,N,h} - X_{t_0}^{j,N,h} |^2.
\end{align*}
 
Raising to the power $p/2$ and taking the expectation on both sides above, an application of \eqref{eq: differences bound for shceme}, Jensen's inequality and Young's inequality shows that there exist  positive constants  $K_{p,6},\kappa$ (both independent of $h,T,M$ and $N$) such that 
\begin{align*}
    \bE \big[  |X_{t_1}^{i,N,h} |^{p} \big] 	
    &\leq
    e^{ \kappa t_1} K_{p,6}  \big( 1 + \bE\big[|\xi|^{p}\big] e^{- {\kappa} t_1 }
      \big).
\end{align*}
\textit{Case 2: $m=2$.} 
There exist positive constants  $K_{p,6}, \kappa$ (both are independent of $h,T,M$ and $N$) such that
\begin{align*}
    \bE \big[  |X_{t_2}^{i,N,h} |^{p} \big] 	
    &\leq
    e^{ \kappa t_2} K_{p,6}  \big( 1 + \bE\big[|\xi|^{p}\big] e^{- {\kappa} t_2 }
      \big).
\end{align*}
We will show that the constant on the right-hand side does not blow up as $m$ increases. 
\\ \\
\noindent
\textit{Case 3: $m \geq 3$.}  
We have  
\begin{align*}
    |X_{t_{m}}^{i,N,h} |^{2}
    & 
    =
    |X_{t_{m-1}}^{i,N,h}+b_{m-1}^{i,N,h}h  |^{2}
    +
    \sigma^2|\Delta \overline{W}_{m-1}^{i,N,h}|^2
    +2
    \big(  X_{t_{m-1}}^{i,N,h}+b_{m-1}^{i,N,h}h  
      \big)   \cdot  	\big(   
    \Delta \overline{W}_{m-1}^{i,N,h}
    \big)
    \\
    &\leq 
    (1-2\lambda h+ Kh^2)|X_{t_{m-1}}^{i,N,h} |^{2}
    +
    \sigma^2|\Delta \overline{W}_{m-1}^{i,N,h}|^2
    +2
    \big(  X_{t_{m-1}}^{i,N,h}+b_{m-1}^{i,N,h}h  
      \big)   \cdot  	\big(   
    \Delta \overline{W}_{m-1}^{i,N,h}
    \big)
    \\
    &\quad 
    + \frac{2K_V h}{N}\sum_{j=1}^{N} | X_{t_{m-1}}^{i,N,h} - X_{t_{m-1}}^{j,N,h} ||X_{t_0}^{i,N,h} |
     + \frac{K_V^2 h^2}{N}\sum_{j=1}^{N} | X_{t_{m-1}}^{i,N,h} - X_{t_{m-1}}^{j,N,h} |^2.
\end{align*}
Since for $m\geq 3$, $ X_{t_{m-1}}^{i,N,h}$ is not independent of $\Delta \overline{W}_{m-1}^{i,N,h}$, we further expand $ X_{t_{m-1}}^{i,N,h}$ and estimate
\begin{align*}
    \big(    X_{t_{m-1}}^{i,N,h}+b_{m-1}^{i,N,h}h 
      \big) 
      & 
      \cdot \big(   
    \Delta \overline{W}_{m-1}^{i,N,h}
    \big) 
    \\ 
    & 
    \leq h|b_{m-1}^{i,N,h}||\Delta \overline{W}_{m-1}^{i,N,h}|
    +  \big(   X_{t_{m-2}}^{i,N,h}+b_{m-2}^{i,N,h}h + \Delta \overline{W}_{m-2}^{i,N,h}
       \big)   \cdot  	\big(   
    \Delta \overline{W}_{m-1}^{i,N,h}
    \big).
\end{align*}
Note that $X_{t_{m-2}}^{\cdot,N,h} $ is independent of $W_{m-1}^{\cdot,N,h}$. Similar to the analysis of the first part, we define the following local quantities: for all $i \in \lbrace 1, \ldots, N\rbrace	$ ,

\begin{align*}
     &G_{m,1}^{i}
     = 
     (1-2\lambda h+ Kh^2)|X_{t_{m-1}}^{i,N,h} |^{2},
    \\
    &
    G_{m,2}^{i}
    = \sigma^2|\Delta \overline{W}_{m-1}^{i,N,h}|^2, 
    \\
    &
     G_{m,3}^{i}= 2\big(   X_{t_{m-2}}^{i,N,h}+b_{m-2}^{i,N,h}h + \Delta \overline{W}_{m-2}^{i,N,h}
       \big)   \cdot  	\big(   
    \Delta \overline{W}_{m-1}^{i,N,h}
    \big),
    \\ 
    &
    G_{m,4}^{i}
    =2h|b_{m-1}^{i,N,h}||\Delta \overline{W}_{m-1}^{i,N,h}|,
    \\
    &
    G_{m,5}^{i}=\frac{2K_V h}{N}\sum_{j=1}^{N} | X_{t_{m-1}}^{i,N,h} - X_{t_{m-1}}^{j,N,h} ||X_{t_{m-1}}^{i,N,h} |
     + \frac{K_V^2 h^2}{N}\sum_{j=1}^{N} | X_{t_{m-1}}^{i,N,h} - X_{t_{m-1}}^{j,N,h} |^2,
\end{align*}
and note that
\begin{align*}
|X_{t_{m}}^{i,N,h} |^{2}
    & 
    \leq  \big( G_{m,1}^{i}+ G_{m,2}^{i}+ G_{m,3}^{i}+ G_{m,4}^{i}+ G_{m,5}^{i} \big).
\end{align*}
Raising to the power $p/2$ and taking the expectation on both sides, we have
\begin{align*}
\nonumber
   \mathbb{E}  \Big[ &\big( G_{m,1}^{i}+ G_{m,2}^{i}+ G_{m,3}^{i}+ G_{m,4}^{i}+ G_{m,5}^{i} \big)^{p/2}  \Big] 
     \\ \nonumber 
     &=  \sum_{l=0}^{p/2} \binom{p/2}{l} \mathbb{E}  \Big[
   (G_{m,1}^{i})^{p/2-l} (G_{m,2}^{i}+ G_{m,3}^{i}+ G_{m,4}^{i}+ G_{m,5}^{i})^{l} \Big]
   \\
    & \leq \mathbb{E} \big[   | G_{m,1}^{i}|^{p/2} \big]  
    + p\mathbb{E}    \Big[    ( G_{m,1}^{i})^{p/2-1} (G_{m,2}^{i} +G_{m,3}^{i}+G_{m,4}^{i}+ G_{m,5}^{i}) 
    + 
    | G_{m,1}^{i}|^{p/2-2}
    |G_{m,3}^{i}|^2
    \Big] 
    \\
    \nonumber
    &\quad 
    +    h^{3/2}  \Big(    K_{p,1}\bE \big[ \big|  X_{t_{m-1}}^{i,N,h}  
     \big|^{p} \big]  + K_{p,2}\bE \big[  \big|  X_{t_{m-2}}^{i,N,h} 
     \big|^{p}\big]       +K   \Big).
\end{align*}
Using \eqref{eq: i j  pre  difference }, we obtain the following estimate:
\begin{align*}
     \bE \big[ & (G_{m,1}^{i})^{p/2-1} G_{m,5}^{i}   \big] 
     \\
     &
     \leq
     K  \bE \bigg[  	
     \big|  X_{t_{m-1}}^{i,N,h}  
     \big|^{p-2} 
     \Big( 
     \frac{2K_V h}{N}\sum_{j=1}^{N} | X_{t_{m-1}}^{i,N,h} - X_{t_{m-1}}^{j,N,h} ||X_{t_{m-1}}^{i,N,h} |
     + \frac{K_V^2 h^2}{N}\sum_{j=1}^{N} | X_{t_{m-1}}^{i,N,h} - X_{t_{m-1}}^{j,N,h} |^2
     \Big)
     \bigg]
     \\
     &
      \leq h \Big(   \varepsilon     \bE \big[  	
     \big|  X_{t_{m-1}}^{i,N,h}  
     \big|^{p} \big]     
     +  \frac{K}{N}\sum_{j=1}^{N}
     \bE \big[  	
     \big|  X_{t_{m-1}}^{i,N,h}-X_{t_{m-1} }^{j,N,h}  
     \big|^{p} \big]
     \Big)
     \\
     &\leq
     \varepsilon h\bE \big[  	
     \big|  X_{t_{m-1}}^{i,N,h}  
     \big|^{p} \big]  
     + Kh (1+ \bE\big[|\xi|^{p}\big] e^{-(K_{p,3}-2K_{p,4})t_{m}/3}).
\end{align*}
 The analysis of the other terms works similarly as in the proof of the first part. Hence,  
we conclude  that 
there exists positive constants $K_{p,7},K_{p,8},K_{p,9}$  (both independent of $h,T,M$ and $N$)  
 satisfying $K_{p,7}>2 K_{p,8} $ and $e^{-(K_{p,3}-2K_{p,4})h}\neq 1-(K_{p,7}-2K_{p,8})h > 0$ (for sufficiently small $h$ independent of  $T,M,N$), such that 
\begin{align*}
    \bE \big[  |X_{t_m}^{i,N,h} |^{p} \big] 
    &\leq
     (1-K_{p,7} h ) \bE \big[  |X_{t_{m-1}}^{i,N,h} |^{p} \big] 
     +
     K_{p,8} h   \big( \bE \big[  |X_{t_{m-2}}^{i,N,h} |^{p} \big]
     + \bE \big[  |X_{t_{m-3}}^{i,N,h} |^{p} \big]
     \big)
     \\
     &
     \hspace{6cm}+
     K_{p,9} h
      \big(1+ \bE\big[|\xi|^{p}\big]e^{ -(K_{p,3}-2K_{p,4} )t_m } \big)
     \\
     &
      \leq K
    + K (1+\bE\big[|\xi|^{p}\big]  ) \big( e^{-(K_{p,7}-2K_{p,8})t_{m}/3}
    + e^{-(K_{p,3}-2K_{p,4})t_{m}/3}
    \big),
\end{align*}
for some constant $K$ (independent of $h,T,M$ and $N$). In the last estimate, we used the second statement in Lemma~\ref{prop: sequence inequality 1} with $c_1 \equiv (1-K_{p,7} h ), c_2 \equiv K_{p,8} h, c_3 \equiv K_{p,8} h$, $c_4 = K_{p,9}\bE\big[|\xi|^{p}\big] h$, $c_5 = K_{p,3}-2K_{p,4}$,  $C \equiv K_{p,9} h$ and the moment bounds at $\{t_0,t_1,t_2\}$.   
\end{proof} 

\begin{proof}[Proof of Proposition~\ref{prop:basic_estimates22} -- Statement (2): the $L^2$-strong error]
Consider the IPS described in  \eqref{ModelIPS} and recall the auxiliary scheme of  \eqref{eq: def : aux Euler scheme 0}. 
For all $m \in \lbrace 0, \ldots, M-1 \rbrace$ with $\hat{X}_{t_0}^{i,N,h} = X^{i,N}_{t_0}$, we define 
\begin{align}
\nonumber 
    \hx^{i,N,h}_{t_{m+1}} = \hx^{i,N,h}_{t_{m}} - \Big( \nabla U( \hx^{i,N,h}_{t_{m}}+   \frac{\sigma}{2} 
    \Delta W_{m}^i) + \frac{1}{N}\sum_{j=1}^{N} \nabla V( \hx_{t_{m}}^{i,N,h}+  \frac{\sigma}{2} 
    \Delta W_{m}^i -  \hx_{t_m}^{j,N,h}-  \frac{\sigma}{2} 
    \Delta W_{m}^j)  \Big)  h   +  \sigma 
    \Delta W_{m}^i ,
\end{align}
 where $t_m := mh$,  $T:=Mh$, and $\Delta W_{m}^i = W^i_{t_{m}} - W^i_{t_{m-1}}$. Recall that $X^{i,N,h}_{t_{m}}=\hx^{i,N,h}_{t_{m}}+ \sigma \Delta W_{m}^i/2$.  
We compute the difference terms:
    \begin{align}
    \label{eq: se: Definition of Delta i,m}
        &\Delta_{i,m}: =
        X_{t_m}^{i,N} - \hx^{i,N,h}_{t_{m+1}}
        \\
         &=
        X_{t_m}^{i,N} - \bigg( X_{t_{m-1}}^{i,N} - \Big( \nabla U(  X_{t_{m-1}}^{i,N} +  \frac{\sigma}{2}  
    \Delta W_{m}^i) + \frac{1}{N}\sum_{j=1}^{N} \nabla V(  X_{t_{m-1}}^{i,N} +  \frac{\sigma}{2}  
    \Delta W_{m}^i -  X_{t_{m-1}}^{j,N} -  \frac{\sigma}{2} 
    \Delta W_{m}^j)  \Big)  h   +  \sigma 
    \Delta W_{m}^i\bigg)
   \notag \\
    &\quad 
    +\bigg( X_{t_{m-1}}^{i,N} - \Big( \nabla U(  X_{t_{m-1}}^{i,N} +  \frac{\sigma}{2} 
    \Delta W_{m}^i) + \frac{1}{N}\sum_{j=1}^{N} \nabla V(  X_{t_{m-1}}^{i,N} +  \frac{\sigma}{2} 
    \Delta W_{m}^i -  X_{t_{m-1}}^{j,N} -  \frac{\sigma}{2} 
    \Delta W_{m}^j)  \Big)  h   +  \sigma 
    \Delta W_{m}^i\bigg)
  \notag \\
    &\quad -
    \bigg(\hx^{i,N,h}_{t_{m}} - \Big( \nabla U( \hx^{i,N,h}_{t_{m}}+  \frac{\sigma}{2} 
    \Delta W_{m}^i) + \frac{1}{N}\sum_{j=1}^{N} \nabla V( \hx_{t_{m}}^{i,N,h}+  \frac{\sigma}{2} 
    \Delta W_{m}^i -  \hx_{t_m}^{j,N,h}-  \frac{\sigma}{2} 
    \Delta W_{m}^j)  \Big)  h   +  \sigma 
    \Delta W_{m}^i\bigg) \notag \\
    & =: R^{i,1}_{t_m} + R^{i,2}_{t_m}. \notag
    \end{align}
Note that here we match $ X_{t_m}^{i,N} $ with $ \hx^{i,N,h}_{t_{m+1}}$ instead of $\hx^{i,N,h}_{t_{m}}$. 

We estimate the above terms separately and collect all the estimates an the end. For the first term, taking squares and expectations yields 
\begin{align*}
 \mathbb{E}&\big[|R^{i,1}_{t_m}|^2\big]  
  \\
  &=  \bE\bigg[  \Big|  ~ X_{t_m}^{i,N} - \bigg( X_{t_{m-1}}^{i,N} - \Big( \nabla U(  X_{t_{m-1}}^{i,N} + \frac{\sigma}{2}  
    \Delta W_{m}^i) 
    \\
    & \quad\quad\quad  \quad\quad+
    \frac{1}{N}\sum_{j=1}^{N} \nabla V(  X_{t_{m-1}}^{i,N} +  \frac{\sigma}{2} 
    \Delta W_{m}^i -  X_{t_{m-1}}^{j,N} -  \frac{\sigma}{2}  
    \Delta W_{m}^j)  \Big)  h   +  \sigma 
    \Delta W_{m}^i\bigg) ~\Big|^2\bigg]
        \\
    & = \bE\bigg[ \Big|-\int_{t_{m-1} }^{t_m} 
    ~\Big(
       \nabla U(  X_{s}^{i,N}) - \nabla U(  X_{t_{m-1}}^{i,N} +  \frac{\sigma}{2} 
    \Delta W_{m}^i)
      \\
    & \quad\quad\quad  \quad\quad+
       \frac{1}{N}\sum_{j=1}^{N} \nabla V(  X_{s}^{i,N} -  X_{s}^{j,N}   )
    -
    \frac{1}{N}\sum_{j=1}^{N} \nabla V(  X_{t_{m-1}}^{i,N} +  \frac{\sigma}{2} 
    \Delta W_{m}^i -  X_{t_{m-1}}^{j,N} -  \frac{\sigma}{2}  
    \Delta W_{m}^j)  \Big)\dd s
        ~\Big|^2\bigg] 
    \\
    &\leq 
    2h \int_{t_{m-1} }^{t_m}  \bE \Big[  
    \big| \nabla U(  X_{s}^{i,N}) - \nabla U(  X_{t_{m-1}}^{i,N} + \frac{\sigma}{2}  
    \Delta W_{m}^i)\big|^2 \Big]\dd s 
    \\
    &\quad +
    2h \int_{t_{m-1} }^{t_m}   \frac{1}{N}\sum_{j=1}^{N}\bE \Big[
   \big| \nabla V(  X_{s}^{i,N} -  X_{s}^{j,N}   )-
    \nabla V(  X_{t_{m-1}}^{i,N} + \frac{\sigma}{2} 
    \Delta W_{m}^i -  X_{t_{m-1}}^{j,N} -  \frac{\sigma}{2} 
    \Delta W_{m}^j)
    \big|^2 ~\dd s
    \Big]
    \\
    &\leq K h \int_{t_{m-1} }^{t_m}  \bigg( \bE \Big[  \big|  X_{s}^{i,N}-X_{t_{m-1}}^{i,N}\big|^2\Big]   
    +
     \bE \Big[  \big| \frac{\sigma}{2}  
    \Delta W_{m}^i \big|^2\Big] 
    + 
     \frac{1}{N}\sum_{j=1}^{N} \Big(
     \bE \Big[  \big|  X_{s}^{j,N}-X_{t_{m-1}}^{j,N}\big|^2\Big]   
     +
     \bE \Big[  \big| \frac{\sigma}{2}  
    \Delta W_{m}^j \big|^2\Big] 
     \Big)~
    \bigg)
    \dd s 
     \\
    &\leq K h \int_{t_{m-1} }^{t_m}  \Big( \max_{i \in \lbrace 1, \ldots, N \rbrace}\bE \Big[  \big|  X_{s}^{i,N}-X_{t_{m-1}}^{i,N}\big|^2\Big]   +
    h 
    \Big)
    \dd s ~\leq Kh^2  \big(1 + \mathbb{E}\big[\,|\xi|^{2}\big] e^{-\kappa  t_m}\big), 
\end{align*}
where we used Jensen's inequality, the Lipschitz continuity of the potentials and Proposition~\ref{prop:basic_estimates11}, $\kappa,K >0$ are independent of $h,T,M$ and $N$. 
Next, we consider the second term $R^{i,2}_{t_m} $. Taking squares and expectations, we have that 
\begin{align*}  
\mathbb{E}& \big[|R^{i,2}_{t_m}|^2\big]  \\
=& \bE\bigg[   
    \bigg| ~\Big(X_{t_{m-1}}^{i,N} - \Big( \nabla U(  X_{t_{m-1}}^{i,N} +  \frac{\sigma}{2}      \Delta W_{m}^i) 
    + 
    \frac{1}{N}\sum_{j=1}^{N} \nabla V(  X_{t_{m-1}}^{i,N} +  \frac{\sigma}{2}  
    \Delta W_{m}^i -  X_{t_{m-1}}^{j,N} -  \frac{\sigma}{2}  
    \Delta W_{m}^j)  \Big)  h   
    \Big)
    \\
    &\quad -
    \Big(\hx^{i,N,h}_{t_{m}} - \Big( \nabla U( \hx^{i,N,h}_{t_{m}}+  \frac{\sigma}{2}  
    \Delta W_{m}^i) + \frac{1}{N}\sum_{j=1}^{N} \nabla V( \hx_{t_{m}}^{i,N,h}+  \frac{\sigma}{2}  
    \Delta W_{m}^i -  \hx_{t_m}^{j,N,h}-  \frac{\sigma}{2}  
    \Delta W_{m}^j)  \Big)  h   
    \Big)~\bigg|^2\bigg]
    \\
    &\leq \bE \Big[    \big|  X_{t_{m-1}}^{i,N} - \hx^{i,N,h}_{t_{m}}\big|^2    \Big] 
    -
    h\bE \Big[    	\Big(     X_{t_{m-1}}^{i,N} - \hx^{i,N,h}_{t_{m}}   \Big)   \cdot	\Big(  
     \nabla U(  X_{t_{m-1}}^{i,N} +  \frac{\sigma}{2}     \Delta W_{m}^i) -
     \nabla U( \hx^{i,N,h}_{t_{m}}+  \frac{\sigma}{2}     \Delta W_{m}^i)
      \Big)   \Big] 
    \\
    &\quad -
    \frac{ h}{N}\sum_{j=1}^{N}\bE \Big[  \Big(   	    X_{t_{m-1}}^{i,N} - \hx^{i,N,h}_{t_{m}}
       \Big)   \cdot	\Big(   
    \nabla V(  X_{t_{m-1}}^{i,N} +  \frac{\sigma}{2}  
    \Delta W_{m}^i -  X_{t_{m-1}}^{j,N} -  \frac{\sigma}{2}  
    \Delta W_{m}^j)  
    \\
    &\qquad \qquad\qquad\quad -
     \nabla V( \hx_{t_{m}}^{i,N,h}+  \frac{\sigma}{2}  
    \Delta W_{m}^i -  \hx_{t_m}^{j,N,h}-  \frac{\sigma}{2}   
    \Delta W_{m}^j)
       \Big)\Big] 
     \\
    &\quad
    + 2 h^2 
    \bE \Big[       \big|       \nabla U(  X_{t_{m-1}}^{i,N} +  \frac{\sigma}{2}    \Delta W_{m}^i) -
     \nabla U( \hx^{i,N,h}_{t_{m}}+ \frac{\sigma}{2}      \Delta W_{m}^i)  \big|^2 \Big]  
    \\
    &\quad + \frac{ 2h^2}{N}\sum_{j=1}^{N}  \bE \Big[       \big|  \nabla V(  X_{t_{m-1}}^{i,N} +  \frac{\sigma}{2}  
    \Delta W_{m}^i -  X_{t_{m-1}}^{j,N} -  \frac{\sigma}{2}  
    \Delta W_{m}^j)
    \\
    &\qquad \qquad\qquad\quad
    -
     \nabla V( \hx_{t_{m}}^{i,N,h}+  \frac{\sigma}{2}  
    \Delta W_{m}^i -  \hx_{t_m}^{j,N,h}-  \frac{\sigma}{2}  
    \Delta W_{m}^j)       \big|^2 \Big] 
    \\
    & \leq \bE \Big[    \big|  X_{t_{m-1}}^{i,N} - \hx^{i,N,h}_{t_{m}}\big|^2    \Big] 
      -
    h\bE \Big[    	\Big(     X_{t_{m-1}}^{i,N} - \hx^{i,N,h}_{t_{m}}   \Big)   \cdot	\Big(  
     \nabla U(  X_{t_{m-1}}^{i,N} +  \frac{\sigma}{2}     \Delta W_{m}^i) -
     \nabla U( \hx^{i,N,h}_{t_{m}}+  \frac{\sigma}{2}     \Delta W_{m}^i)
      \Big)   \Big] 
    \\
    &\quad + Kh^2  \big(1 + \mathbb{E}\big[\,|\xi|^{2}\big] e^{-\kappa  t_m}\big) -
    \frac{ h}{N}\sum_{j=1}^{N}\bE \Big[  \Big(   	    X_{t_{m-1}}^{i,N} - \hx^{i,N,h}_{t_{m}}
       \Big)   
    \\
    &\quad \qquad  
    \cdot	\Big(   
    \nabla V(  X_{t_{m-1}}^{i,N} +  \frac{\sigma}{2}  
    \Delta W_{m}^i -  X_{t_{m-1}}^{j,N} -  \frac{\sigma}{2}  
    \Delta W_{m}^j) -
     \nabla V( \hx_{t_{m}}^{i,N,h}+  \frac{\sigma}{2}  
    \Delta W_{m}^i -  \hx_{t_m}^{j,N,h}-  \frac{\sigma}{2}   
    \Delta W_{m}^j)
       \Big)\Big], 
\end{align*}
where we used Jensen's inequality,   Propositions \ref{prop:basic_estimates11} and the Statement (1) in  Proposition~\ref{prop:basic_estimates22}.
We further estimate
\begin{align*}
     & \mathbb{E}[|R^{i,2}_{t_m}|^2]  \leq (1-\lambda h) \bE \Big[    \big|  X_{t_{m-1}}^{i,N} - \hx^{i,N,h}_{t_{m}}\big|^2    \Big]  + Kh^2 \big(1 + \mathbb{E}\big[\,|\xi|^{2}\big] e^{-\kappa  t_m}\big)
    \\
    &\quad 
     -
    \frac{ h}{2N^2}\sum_{i,j=1}^{N}\bE \Big[     \Big(    	     (  X_{t_{m-1}}^{i,N} +  \frac{\sigma}{2}   
    \Delta W_{m}^i -  X_{t_{m-1}}^{j,N} -  \frac{\sigma}{2}   
    \Delta W_{m}^j)-
    ( \hx_{t_{m}}^{i,N,h}+  \frac{\sigma}{2}   
    \Delta W_{m}^i -  \hx_{t_m}^{j,N,h}-  \frac{\sigma}{2}   
    \Delta W_{m}^j)\Big) 
    \\
    &\qquad   \quad   \cdot	\Big(   
    \nabla V(  X_{t_{m-1}}^{i,N} +  \frac{\sigma}{2}   
    \Delta W_{m}^i -  X_{t_{m-1}}^{j,N} -  \frac{\sigma}{2}   
    \Delta W_{m}^j) -
     \nabla V( \hx_{t_{m}}^{i,N,h}+  \frac{\sigma}{2}   
    \Delta W_{m}^i -  \hx_{t_m}^{j,N,h}-  \frac{\sigma}{2}   
    \Delta W_{m}^j)
      \Big) \Big]
    \\
    &\leq (1-\lambda h)  \bE  \Big[       \big|    \Delta_{i,m-1}    \big|^2 \Big]+Kh^2 \big(1 + \mathbb{E}\big[\,|\xi|^{2}\big] e^{-\kappa  t_m}\big) ,
\end{align*}
 with $\Delta_{i,m}$ as defined in \eqref{eq: se: Definition of Delta i,m}.   
We used the \textit{`symmetrization trick' }
 in \eqref{eq: symmetric trick} to handle the convolution term.
 Note that the positive constant $K$ is independent of $h,T,M$ and $N$.
Hence for all $m\in \lbrace 1, \ldots, M-1\rbrace,~ i \in \lbrace 1, \ldots, N\rbrace	$,  there exists a positive constant $K$ such that
\begin{align*}
     \bE \Big[       \big|    \Delta_{i,m}    \big|^2 \Big]  
      &\leq  (1-\lambda h)  \bE \Big[       \big|    \Delta_{i,m-1}    \big|^2 \Big]+Kh^2 \big(1 + \mathbb{E}\big[\,|\xi|^{2}\big] e^{-\kappa  t_m}\big)
      \\
      & = 
      (1-\lambda h)^m  \bE \Big[       \big|    \Delta_{i,0}    \big|^2 \Big]
      +\big(1 + \mathbb{E}\big[\,|\xi|^{2}\big] e^{-\kappa  t_m}\big) Kh^2 \sum_{j=0}^{m-1} (1-\lambda h)^j
      \\
      & 
       \leq 
       \Big( Kh+\frac{Kh}{\lambda} \Big) \big(1 + \mathbb{E}\big[\,|\xi|^{2}\big]  \big),
\end{align*}
where we used $ \bE  \big[       |    \Delta_{i,0}    |^2  \big] \leq K h \mathbb{E}\big[\,|\xi|^{2}\big] $. Recall that $X^{i,N,h}_{t_{m}}=\hx^{i,N,h}_{t_{m}}+ \sigma \Delta W_{m}^i/2$. Using Statement (1) in Proposition~\ref{prop:basic_estimates22} and \eqref{eq: def : non-Markov Euler scheme}, we further have  
\begin{align*}
    \bE \Big[       \big|  X^{i,N,h}_{t_{m}}-\hx^{i,N,h}_{t_{m}}   \big|^2 \Big]
    &= \bE \Big[       \big|\frac{\sigma}{2}   \Delta W_{m}^i  \big|^2 \Big] \leq Kh,
    \\
     \textrm{and}\quad 
     \bE \Big[       \big|  X^{i,N,h}_{t_{m+1}}-X^{i,N,h}_{t_{m}}   \big|^2 \Big] &\leq Kh\big(1 + \mathbb{E}\big[\,|\xi|^{2}\big] e^{-\kappa  t_m}\big).
\end{align*}
Collecting the last 3 estimates, we have for all $m\in \lbrace 1, \ldots, M-1\rbrace,~ i \in \lbrace 1, \ldots, N\rbrace$, 
\begin{align*}
     \bE    \Big[      \big|  X_{t_m}^{i,N} -   X_{t_m}^{i,N,h} \big|^2  \Big] 
     &= 
     \bE    \Big[      \big|  \Delta_{i,m} +( \hx^{i,N,h}_{t_{m+1}} -X^{i,N,h}_{t_{m+1}})
     +(
      X^{i,N,h}_{t_{m+1}}-X^{i,N,h}_{t_{m}}
     )
     \big|^2  \Big] 
     \\
     &\leq 
     K\Big(   \bE \Big[       \big|    \Delta_{i,m}    \big|^2 \Big]
     +    \bE \Big[       \big|  X^{i,N,h}_{t_{m+1}}-\hx^{i,N,h}_{t_{m+1}}   \big|^2 \Big]+
     \bE \Big[       \big|  X^{i,N,h}_{t_{m+1}}-X^{i,N,h}_{t_{m}}   \big|^2 \Big]
     \Big)  
     \\
     &
     \leq Kh \big(1 + \mathbb{E}\big[\,|\xi|^{2}\big] \big).
\end{align*}
As $K$ is independent of the critical quantities $M$ and $N$, maximizing over $i$ and $m$ yields the final result.
\end{proof}

\subsection{Auxiliary results  }

The following statement is an auxiliary result on the differences of SDE starting at different times ($t$ and $s$ with $t\leq s$) at the same point $\boldsymbol{x}$ and is used in the proof of Proposition~\ref{propsition:  first variation bound prop v2}. 
\begin{lemma}
    \label{lemma: 4 moment difference}
    Let the assumptions and setup of Proposition~\ref{propsition:  first variation bound prop v2} hold and let  $r 
 \geq s \geq t \geq 0$, $u\geq 0$ with $s-t<1$. Let the starting positions $x_i\in L^4(\Omega, \bR)$ be $\mathcal{F}_t$-measurable random variables that are identically distributed over all $ i  \in \lbrace 1, \ldots, N \rbrace$.    
    Let {\color{black}$   (\boldsymbol{X}^{t,\boldsymbol{x},N}_s  )_{ s \geq t \geq 0} $ and $ (  \boldsymbol{X}^{s,\boldsymbol{x},N}_r)_{ r \geq s \geq 0}$ }be the solutions of \eqref{eq: def of Bi} starting from $\boldsymbol{x}$ at time $t$ and $s$, respectively. 
       Then there exist some $\lambda_2\in(0,\min \{   \lambda-2K_V,\lambda_1\} )$ and $K>0$ (both are independent of $s,t,N$), such that for any $i\in\{1,\dots,N\}$ 
   \begin{align*}
        \bE \Big[       \big| X^{t,x_i,i,N}_{s+u}  - X^{s,x_i,i,N}_{s+u}     \big|^4 \Big]
      \leq K(s-t)^2 e^{- 4\lambda_2 u}.
    \end{align*}     
\end{lemma}

\begin{proof}
    By  It\^o's formula, we have, for any $s \geq t \geq 0$, $u \geq 0$, $i\in\{1,\dots,N\}$ 
    \begin{align*}
    & \bE \Big[      \big| X^{t,x_i,i,N}_{s+u}  - X^{s,x_i,i,N}_{s+u}     \big|^4 \Big]
    \\
    &\leq 
     \bE \Big[       \big| X^{t,x_i,i,N}_{s}  - X^{s,x_i,i,N}_{s}     \big|^4 \Big]
     \\
     & \qquad -4        
     \int_{0}^u
     \bE \Big[  \big|      X^{t,x_i,i,N}_{s+w}  - X^{s,x_i,i,N}_{s+w}  \big|^2
    \Big(   X^{t,x_i,i,N}_{s+w}  - X^{s,x_i,i,N}_{s+w}
       \Big)   \cdot	\bigg(  ~ \Big( 
    \nabla U(X^{t,x_i,i,N}_{s+w}) -
    \nabla U(X^{s,x_i,i,N}_{s+w})
    \Big) 
     \\
     & \qquad + 
    \Big(   
      \frac{1}{N} \sum_{l=1}^{N}
      \nabla V(X^{t,x_i,i,N}_{s+w}-X^{t,x_l,l,N}_{s+w})   -
      \nabla V(X^{s,x_i,i,N}_{s+w}-X^{s,x_l,l,N}_{s+w})
      \Big) 
    \bigg)   \Big] \dd w 
     \\
     & \leq  \bE \Big[       \big| X^{t,x_i,i,N}_{s}  - X^{s,x_i,i,N}_{s}     \big|^4 \Big] 
     -4\Big(\lambda-\left(1+\frac{3}{4} \right)K_V\Big)  \int_{0}^u
      \bE \Big[  \big| X^{t,x_i,i,N}_{s+w}   - X^{s,x_i,i,N}_{s+w}     \big|^4\Big]  \dd w 
      \\
      &\qquad 
      + 
      \frac{K_V}{ N}  \sum_{l=1}^{N} 
      \int_{0}^u \bE \Big[  \big| X^{t,x_l,l,N}_{s+w}   - X^{s,x_l,l,N}_{s+w}     \big|^4\Big]
      \dd w,
\end{align*}
where we used Young's inequality along with Assumption~\ref{assum:main}. 
From the proof of Proposition~\ref{propsition:  first variation bound prop v2} (see, equation \ref{eq: first var diff initial diff}), we deduce that $ \bE \big[  | X^{t,x_i,i,N}_{s}  - X^{s,x_i,i,N}_{s} |^4 \big] \leq K(s-t)^2$. Using the fact that $\lambda > 2 K_V$, we conclude the claim. 
\end{proof}

The next statement concerning the second-order variation process is similar to Proposition~\ref{propsition:  first variation bound prop v2} which is used in the proof in   Lemma~\ref{lemma: analysis of the b0 term summation integration}.       
\begin{proposition}
    \label{prop: second var extension}
 
    Let the assumptions and set up of Lemma~\ref{lemma: first variation bound noiid} and Proposition~\ref{propsition:  first variation bound prop v2} hold.
    Then there exist some  $\lambda_4\in (0, \min \lbrace \lambda-2K_V, \lambda_3 \rbrace)$ and 
    $K>0$ (both are independent of $h,T,M$ and $N$) 
    such that
     for all $T\geq s\geq t\geq 0$ (with $s-t< 1$)  and $i\in\{1,\dots,N\}$ 
    \begin{align*}
     \mathbb{E}\Big[|X^{t,x_i,i,N}_{T,x_i,x_i}-X^{s,x_i,i,N}_{T,x_i,x_i}|^2\Big] 
     & 
     \leq  K(s-t)  e^{- {2\lambda_4  }(T-s)},
     \\
     \sum_{i,j,k=1, i\neq j\neq k }^{N}  \mathbb{E}\Big[|X^{t,x_i,i,N}_{T,x_j,x_k}-X^{s,x_i,i,N}_{T,x_j,x_k}|^2\Big] 
     & \leq  \frac{K(s-t)}{N } e^{-2\lambda_4 (T-s)},
          \end{align*}
and 
\begin{align*}
    \sum_{i,k=1, i\neq k }^{N}
    \bE \Big[  
    |X^{t,x_i,i,N}_{T,x_k,x_k}-X^{s,x_i,i,N}_{T,x_k,x_k}|^2
    +
    |X^{t,x_i,i,N}_{T,x_i,x_k}-X^{s,x_i,i,N}_{T,x_i,x_k}|^2
    +
    |X^{t,x_i,i,N}_{T,x_k,x_i} 
    & -X^{s,x_i,i,N}_{T,x_k,x_i}|^2
     \Big] 
     \\
     &
     \leq  K(s-t) e^{-2\lambda_4 (T-s)}.
    \end{align*}
\end{proposition}
\begin{proof}
 This proof is a combination of the methods used to prove Proposition~\ref{propsition:  first variation bound prop v2} and Lemma~\ref{lemma: second variation bound noiid details} and we streamline the presentation. For any $i,j, k \in \lbrace 1, \ldots, N \rbrace$,  we have that
    \begin{align*}
        |&X^{t,x_i,i,N}_{T,x_j,x_k}-X^{s,x_i,i,N}_{T,x_j,x_k}|^2
        =
        |X^{t,x_i,i,N}_{s,x_j,x_k}-X^{s,x_i,i,N}_{s,x_j,x_k}|^2
        \\
        &\quad \quad 
        -
         2\int_{0}^{T-s}  \Big(   X^{t,x_i,i,N}_{s+u,x_j}-X^{s,x_i,i,N}_{s+u,x_j}
           \Big)   \cdot	\Big(   
         \nabla^{2} U(X^{t,x_i,i,N}_{s+u}) X^{t,x_i,i,N}_{s+u,x_j,x_k} 
         -\nabla^{2} U(X^{s,x_i,i,N}_{s+u})X^{s,x_i,i,N}_{s+u,x_j,x_k}        \Big) \mathrm{d}u 
         \\
        &\quad \quad 
        -
         2\int_{0}^{T}  \Big(     X^{t,x_i,i,N}_{s+u,x_j,x_k}-X^{s,x_i,i,N}_{s+u,x_j,x_k}
           \Big)   \cdot	\Big(   
         \frac{1}{N} \sum_{l=1}^{N}  
        \nabla^2 V(X^{t,x_i,i,N}_{s+u}-X^{t,x_l,l,N}_{s+u})  ( X^{t,x_i,i,N}_{{s+u},x_j,x_k}-X^{t,x_l,l,N}_{{s+u},x_j,x_k})
         \\
        &\quad \quad \quad\qquad \qquad 
        - \frac{1}{N} \sum_{l=1}^{N} \nabla^2 V(X^{s,x_i,i,N}_{s+u}-X^{s,x_l,l,N}_{s+u})  ( X^{s,x_i,i,N}_{{s+u},x_j,x_k}-X^{s,x_l,l,N}_{{s+u},x_j,x_k})
             \Big)   \mathrm{d}u 
         \\
        & \quad \quad 
        -
         2\int_{0}^{T-s}
        \Big(    X^{t,x_i,i,N}_{s+u,x_j,x_k}-X^{s,x_i,i,N}_{s+u,x_j,x_k}
          \Big)   \cdot	\Big(  
        \sum_{l=1}^{N} \sum_{l'=1}^{N}\partial^2_{x_l,x_{l'}}  B_i(\bodX_{u}^{t,\bodx,N}) X^{t,x_l,l,N}_{s+u,x_j}  X^{t,x_{l'},l',N}_{s+u,x_k}
         \\
        &\quad\qquad \qquad\quad \quad 
        - 
        \sum_{l=1}^{N} \sum_{l'=1}^{N}\partial^2_{x_l,x_{l'}}  B_i(\bodX_{u}^{s,\bodx,N}) X^{s,x_l,l,N}_{s+u,x_j}  X^{s,x_{l'},l',N}_{s+u,x_k}
          \Big)
        ~ \dd u.
\end{align*} 

The remaining steps are similar to those in the proof of Lemma~\ref{lemma: second variation bound noiid details} and Proposition~\ref{propsition:  first variation bound prop v2} and we therefore omit a detailed analysis.
\end{proof}



The next statement is a classical result on the stability of SDEs with respect their initial condition.
\begin{lemma}
    \label{lemma: 2 moment difference}
    Let Assumption~\ref{assum:main} hold (with $\lambda > 0$ denoting the convexity parameter) and let $p \geq 2$ be given. Let $(  \boldsymbol{X}^{N}_t)_{t \geq 0}$, and $(  \boldsymbol{Y}^{N}_t)_{t \geq 0}$ be generated from \eqref{ModelIPS} with i.i.d.\ initial states $X^{i,N}_0 \sim \mu, Y^{i,N}_0 \sim \nu \in \mathcal{P}_p(\mathbb{R})$, where $\boldsymbol{X}^{N}_0 =(X^{1,N}_0,\ldots, X^{N,N}_0)$ and $\boldsymbol{Y}^{N}_0 =(Y^{1,N}_0,\ldots, Y^{N,N}_0)$. Then for any $i \in \lbrace 1, \ldots, N \rbrace$, $t \geq 0$, we have  
    \begin{align}
    \label{eq:  2 moment difference }
        \bE \Big[       \big| X^{i,N}_{t}  - Y^{i,N}_{t}    \big|^2 \Big]
      \leq   e^{- 2\lambda t} \bE \Big[ \,  \big| X^{i,N}_{0}  - Y^{i,N}_{0}    \big|^2 \Big].
    \end{align}
\end{lemma}

\begin{proof} 

    This result is classical and we present only a sketch of its proof. By It\^o's formula, we have 
    \begin{align*}
    & \bE \Big[      \big| X^{i,N}_{t}  - Y^{i,N}_{t}    \big|^2 \Big]
    \\
    &\leq 
     \bE \Big[       \big| X^{i,N}_{0}  - Y^{i,N}_{0}    \big|^2 \Big]
     -2        
     \int_{0}^t
     \bE \Big[ 
    \Big(    X^{i,N}_{s}  - Y^{i,N}_{s} 
       \Big)   \cdot	\Big(   
    \nabla U(X^{i,N}_{s} ) -
    \nabla U(Y^{i,N}_{s})
    \Big)   \Big] \dd s
     \\
     &  \qquad -2 \int_{0}^t
     \bE \Big[   
      \Big(      X^{i,N}_{s}  - Y^{i,N}_{s}
        \Big)   \cdot	\Big(   
      \frac{1}{N} \sum_{l=1}^{N}
      \nabla V(X^{i,N}_{s}-X^{l,N}_{s})   -
      \nabla V(Y^{i,N}_{s} -Y^{l,N}_{s})
      \Big)  \Big] 
     \dd s
     \\
     & \leq  \bE \Big[        \big| X^{i,N}_{0}  - Y^{i,N}_{0}    \big|^2 \Big] -2\lambda  \int_{0}^t
      \bE \Big[  \big| X^{i,N}_{s}  - Y^{i,N}_{s}     \big|^2\Big]  \dd s
       \\
     &  \quad 
      - 
      \frac{1}{2N^{2}}  \sum_{i=1}^{N}\sum_{l=1}^{N} 
      \int_{0}^t 
      \bE \Big[   
      \Big(      (X^{i,N}_{s}-X^{l,N}_{s})  - 
      (Y^{i,N}_{s} -Y^{l,N}_{s})
        \Big)   \cdot	\Big(   
      \nabla V(X^{i,N}_{s}-X^{l,N}_{s})   -
      \nabla V(Y^{i,N}_{s} -Y^{l,N}_{s})
      \Big)  \Big] 
      \dd u
      \\
     & \leq  \bE \Big[        \big| X^{i,N}_{0}  - Y^{i,N}_{0}    \big|^2 \Big] -2\lambda  \int_{0}^t
      \bE \Big[  \big| X^{i,N}_{s}  - Y^{i,N}_{s}     \big|^2\Big]  \dd s,
\end{align*}
where we used Assumption~\ref{assum:main}. This estimate allows to deduce the first result \eqref{eq:  2 moment difference }.

\end{proof}

\begin{lemma}
    \label{prop: moment bound extension}
 Let the assumptions and set up of  Proposition~\ref{prop:basic_estimates22} 
    hold with $\xi \in L^{p}(\Omega,\mathbb{R})$ for some given $p \geq 2$. 
    Then for the processes defined in \eqref{eq-aux: continuous  extension Scheme},  
 \eqref{eq: def : aux Euler scheme 0} and \eqref{process:taylor}, respectively, there exists a constant
    $K>0$ (independent of $h,T,M$ and $N$) 
    such that
     for all $m \in \lbrace 0, \ldots, M-1 \rbrace $
    \begin{align*}
    \max_{i \in \lbrace 1, \ldots, N \rbrace } \sup_{s\in[0,h]}\mathbb{E}\big[\,|X_{t_m+s}^{i,N,h}|^{p}\big] 
    \leq K \big(1 + \mathbb{E}\big[\,|\xi|^{p}\big] e^{-\kappa  t_m}\big),
    \\
    \max_{i \in \lbrace 1, \ldots, N \rbrace }  \mathbb{E}\big[\,| \hx^{i,N,h}_{t_{m}} |^{p}\big] 
    \leq K \big(1 + \mathbb{E}\big[\,|\xi|^{p}\big] e^{-\kappa  t_m}\big),
    \\
{\color{black}    \max_{i \in \lbrace 1, \ldots, N \rbrace } \sup_{s\in[0,h)}\mathbb{E}\big[\,|\overline{X}_{t_m+s}^{i,N,h}|^{p}\big] 
    \leq K \big(1 + \mathbb{E}\big[\,|\xi|^{p}\big] e^{-\kappa  t_m}\big).}
    \end{align*}

\end{lemma}

\begin{proof}
Using Proposition~\ref{prop:basic_estimates22} and taking the definition of the processes in  \eqref{eq-aux: continuous  extension Scheme} into account, we have for all $m \in \lbrace 0, \ldots, M-1 \rbrace,~ s\in [0,h]$   and $i\in\{1,\dots,N\}$, 
\begin{align*}
    \bE \big[\,|  X_{t_m+s}^{i,N,h}   |^p  \big]
    \leq &
    K \Big( 
    \bE \big[\,|   X^{i,N,h}_{t_{m}}    |^p  \big]
    +
    \bE \big[\,|   \nabla U(X^{i,N,h}_{t_{m}})    |^p  \big] h^p 
    +
    \frac{1}{N}\sum_{j=1}^{N}\bE \big[\,|   \nabla V(X_{t_{m}}^{i,N,h} - X_{t_m}^{j,N,h})     |^p  \big] h^p 
    \\
    & \quad 
    +
    \bE \big[\,|     \Delta W_{m}^i   |^p  \big] 
    + \bE \big[\,|     \Delta W_{m+1,s}^i   |^p  \big]
    \Big)
    \\
    &\leq 
    K(1+h^p) \Big( \bE \big[\,|   X^{i,N,h}_{t_{m}}    |^p  \big] +  \bE \big[\,|   \nabla U(X^{i,N,h}_{t_{m}})    |^p  \big] +\frac{1}{N} \sum_{j=1}^{N} \bE \big[\,|   \nabla U(X^{j,N,h}_{t_{m}})    |^p  \big] +1\Big)
    \\
    &\leq K \big(1 + \mathbb{E}\big[\,|\xi|^{p}\big] e^{-\kappa  t_m}\big),
\end{align*} 
 where we used Jensen's inequality and the fact that $\nabla U,\nabla V$ are of linear growth (Assumption~\ref{assum:main}).
 
For the second and the last estimate, recall  that $X^{i,N,h}_{t_{m}}=\hx^{i,N,h}_{t_{m}}+ \sigma \Delta W_{m}^i/2$ and $\overline{X}^{i,N,h}_{t_{m-1}+s}  = \hx^{i,N,h}_{t_m} +  \sigma   \Delta W_{m,s}^i/2$. We have that for all $m \in \lbrace 1, \ldots, M -1 \rbrace,~ s\in [0,h)$ (recall $h \in (0,1)$ is sufficiently small)  and $i\in\{1,\dots,N\}$, we have 
\begin{align*}
    \bE \big[\,|  \hx_{t_m}^{i,N,h}   |^p  \big]
    &\leq  K \big( 
    \bE \big[\,|   X^{i,N,h}_{t_{m}}    |^p  \big]
    +
    \bE \big[\,|     \Delta W_{m}^i   |^p  \big]
    \big) \leq 
    K \big(1 + \mathbb{E}\big[\,|\xi|^{p}\big] e^{-\kappa  t_m}\big),
    \\
     \bE \big[\,|  \hx_{t_0}^{i,N,h}   |^p  \big]
     &= 
     \bE \big[\,|  X^{i,N}_{t_0}   |^p  \big]
     = \mathbb{E}\big[\,|\xi|^{p}\big], 
    \\
    \bE \big[\,| \overline{X}^{i,N,h}_{t_{m-1}+s}   |^p  \big]
    &\leq  K \big( 
    \bE \big[\,|   \hx^{i,N,h}_{t_{m}}    |^p  \big]
    +
    \bE \big[\,|     \Delta W_{m,s}^i   |^p  \big]
    \big) \\
    & \leq  K \big(1 + \mathbb{E}\big[\,|\xi|^{p}\big] e^{-\kappa  t_{m-1}} e^{-\kappa h}\big)
    \\
    &
    \leq 
    K \big(1 + \mathbb{E}\big[\,|\xi|^{p}\big] e^{-\kappa  t_{m-1}}\big)
    .
\end{align*} 

\end{proof}

 \begin{lemma} 
\label{lemma: appendix: gronwall}
(Gronwall's inequality). Let $T>0$ and let $\alpha, \beta$ and $u$ be real-valued functions defined on $[0, T]$. Assume that $\alpha$ and $u$ are continuous and that the negative part of $\beta$ is integrable on every closed and bounded subinterval of $[0, T]$.
If $\alpha$ is non-negative and if $u$ satisfies the integral inequality
\begin{align*}
    u(t) \leq \beta(t)+\int_0^t \alpha(s) u(s) \dd s, \quad \forall t \in[0, T],
\end{align*}
then
\begin{align*}
    u(t) \leq \beta(t)+\int_0^t \alpha(s) \beta(s) \exp \left(\int_s^t \alpha(r) \dd r\right) \dd s, \quad \forall t \in[0, T].
\end{align*}
 If we further have that $\beta$ is non-decreasing, then 
\begin{align*}
    u(t) \leq \beta(t) \exp \left(\int_0^t \alpha(s) \dd s\right) , \quad \forall t \in[0, T] .
\end{align*}
\end{lemma}

The following auxiliary result is needed in the proof of Proposition~\ref{prop:basic_estimates22}.

\begin{lemma}
\label{prop: sequence inequality 1}
    Let $c_1,c_2,c_3,c_4,c_5>0, C>0$ be real constants with $c_1+c_2+c_3<1$. Let $(a_n)_{n \in \mathbb{N}}$ be a real-valued sequence satisfying $a_{n+3}\leq c_3a_{n+2}+c_2a_{n+1}+c_1a_{n}+C$ with initial values $0<a_1,a_2,a_3< K$, for some constant $K >0$. Then there exist some constants $K_1,K_2>0$ (both are independent of $n$) such that for all $n \geq 4$ 
\begin{align*}
    a_n\leq K_1+K_2\max\{a_{1},a_{2},a_{3}\} e^{-(1-c_1-c_2-c_3)n/3},
\end{align*}
    Moreover, if $a_{n+3}\leq c_3a_{n+2}+c_2a_{n+1}+c_1a_{n}+C+c_4e^{-c_5n}$ and $(c_1+c_2+c_3)\neq e^{-c_5}$, then there exists constants $K_3,K_4>0$ (both are independent of $n$) such that for all $n \geq 4$ 
\begin{align*}
    a_n\leq K_3+K_4\max\{a_{1},a_{2},a_{3}\} e^{-(1-c_1-c_2-c_3)n/3}.
\end{align*}
    
\end{lemma}

\begin{proof}
    By the condition satisfied by the sequence $(a_n)_{n \in \mathbb{N}}$, we deduce for any $n \geq 1$
    \begin{align}
    \nonumber
        a_{n+3}&\leq c_3a_{n+2}+c_2a_{n+1}+c_1a_{n}+C\leq (c_1+c_2+c_3)\max\{a_{n},a_{n+1},a_{n+2}\}+C,
        \\ \nonumber
         a_{n+4}&\leq c_3a_{n+3}+c_2a_{n+2}+c_1a_{n+1}+C
         \leq (c_1+c_2+c_3)\max\{a_{n+1},a_{n+2},a_{n+3}\}+C
        \\ \nonumber&
        \leq (c_1+c_2+c_3)\max\big\{a_{n+1},a_{n+2}, (c_1+c_2+c_3)\max\{a_{n},a_{n+1},a_{n+2}\}+C\big\}+C
         \\ \nonumber&
         \leq (c_1+c_2+c_3) \max\{a_{n},a_{n+1},a_{n+2}\}+2C,
         \\  \nonumber
          a_{n+5}&
          \leq c_3a_{n+4}+c_2a_{n+3}+c_1a_{n+2}+C
         \leq (c_1+c_2+c_3)\max\{a_{n+2},a_{n+3},a_{n+4}\}+C
         \\&
         \leq (c_1+c_2+c_3) \max\{a_{n},a_{n+1},a_{n+2}\}+3C.
         \label{eq: lemma: a123 leq}
    \end{align}
Hence,
\begin{align*}
    \max\{a_{n+3},a_{n+4},a_{n+5}\}
    &\leq (c_1+c_2+c_3) \max\{a_{n},a_{n+1},a_{n+2}\}+3C.
 \end{align*}
Consequently, adding $\tfrac{3C}{(c_1+c_2+c_3)-1} < 0$ on both sides, we observe
 \begin{align*}
    \max\{a_{n+3}, & a_{n+4},a_{n+5}\}
    + 
    \frac{3C}{(c_1+c_2+c_3)-1}
    \\
    &\leq 
    (c_1+c_2+c_3) \Big(  \max\{a_{n},a_{n+1},a_{n+2}\}+ \frac{3C}{(c_1+c_2+c_3)-1}\Big).
    \end{align*}
Further, we derive for $n \geq 1$ 
    \begin{align*}
    \max\{a_{3n+1}, & a_{3n+2},a_{3n+3}\} + \frac{3C}{(c_1+c_2+c_3)-1}
    \\
    &\leq 
    (c_1+c_2+c_3)^{n} \Big(  \max\{a_{1},a_{2},a_{3}\}+ \frac{3C}{(c_1+c_2+c_3)-1}\Big),
    \end{align*}
    which implies
    \begin{align*}
     \max\{a_{3n+1},a_{3n+2},a_{3n+3}\}
    & \leq  (c_1+c_2+c_3)^{n} \max\{a_{1},a_{2},a_{3}\} + \frac{3C}{1-(c_1+c_2+c_3)} \\
    & \leq e^{-(1-c_1 - c_2- c_3)n} \max\{a_{1},a_{2},a_{3}\}+ K,
\end{align*}
for some $K>0$,
where we used the inequality $e^{x}\geq 1+x$, for any $x \in \mathbb{R}$ and $c_1+c_2+c_3<1$.

Similarly, for the second claim, using the fact that $e^{-c_5(n+2)}<e^{-c_5(n+1)}< e^{-c_5n}$, we derive as in \eqref{eq: lemma: a123 leq}:
\begin{align*}
   \max\{a_{n+3},a_{n+4},a_{n+5}\}
    &\leq 
    (c_1+c_2+c_3)  \max\{a_{n},a_{n+1},a_{n+2}\}+ 3C + 3c_4 e^{-c_5n}.
\end{align*}
Adding  $\frac{3C}{(c_1+c_2+c_3)-1}
    + \frac{3e^{-c_5(n+1)} }{(c_1+c_2+c_3)-e^{-c_5}}$ on both sides, we observe that
 \begin{align*}
    & \max\{a_{3n+1},a_{3n+2},a_{3n+3}\}+ \frac{3C}{(c_1+c_2+c_3)-1}
    + \frac{3e^{-c_5(n+1)} }{(c_1+c_2+c_3)-e^{-c_5}} 
    \\
    &\leq 
    (c_1+c_2+c_3) \Big(  \max\{a_{3n-2},a_{3n-1},a_{3n}\}+ \frac{3C}{(c_1+c_2+c_3)-1}
    + 
    \frac{3e^{-c_5 n} }{(c_1+c_2+c_3)-e^{-c_5}} 
    \Big).
    \end{align*}
 Consequently,   we have   
\begin{align*}
& \max\{a_{3n+1},a_{3n+2} 
,a_{3n+3}\}  
+
\frac{3C}{(c_1+c_2+c_3)-1}
    + \frac{3e^{-c_5(n+1)} }{(c_1+c_2+c_3)-e^{-c_5}} 
    \\
    & \leq 
    (c_1+c_2+c_3)^n \Big(  \max\{a_1,a_2,a_3\}+ \frac{3C}{(c_1+c_2+c_3)-1}
    + 
    \frac{3e^{-c_5} }{(c_1+c_2+c_3)-e^{-c_5}} 
    \Big),
\end{align*}
and therefore
\begin{align*}
    &\max\{a_{3n+1},a_{3n+2},a_{3n+3}\}
    \\
    &\quad 
    \leq e^{-(1-c_1 - c_2- c_3)n} \max\{a_{1},a_{2},a_{3}\}
    +
      \frac{3C}{1-(c_1+c_2+c_3)} +
     \frac{3e^{-c_5} }{|(c_1+c_2+c_3)-e^{-c_5}|} 
     (1-e^{-c_5n})
     \\
    &\quad 
    \leq e^{-(1-c_1 - c_2- c_3)n} \max\{a_{1},a_{2},a_{3}\}+ \tilde{K}_2,   
\end{align*}
for some positive constant $ \tilde{K}_2$.
\end{proof}

\subsection{Auxiliary results from Section \ref{section: proof of B0 terms}}

\color{black}

\begin{lemma}
\label{lemma:ZeroMeanResult}
Recall the operator $L$ given in \eqref{eq: B0 formula} and $\mu^{N,*}(  \boldsymbol{x})$ given by \eqref{eq:NparticleGibbsDistro}, then 
\begin{align}
\label{lemma:statementauxiliaryintegraliszero}
   \bE \Big[ L(0,\olbodX_0^{N})  \Big] = \int_{\bR^N} L(0,\boldsymbol{x}) \mu^{N,*}(\dd \boldsymbol{x}) =0.
\end{align}
\end{lemma}
The proof relies on using integration by parts. The argument is found in \cite[Equation~(3.29)]{leimkuhler2014long} and leverages the known closed form (up to a scaling constant) of the invariant density map $\mu^{N,*}$ and the definition of $L$.

\begin{proof}
  Recall the expression for $B_i(\boldsymbol{x})$ in \eqref{eq:def of func B} 
 \begin{align*}
    B_i(\boldsymbol{x})=B_i(x_1, \ldots,x_N) \coloneqq - \nabla U(x_i) - \frac{1}{N} \sum_{j=1}^{N} \nabla V(x_i - x_j).
\end{align*}
Define $H:\bR^N\to N$ for $\boldsymbol{x}=(x_1, \ldots,x_N)\in \bR^N$ as 
\begin{align*}
    H(\boldsymbol{x}):= 
     \sum_{k=1}^{N} U(x_k) + \frac{1}{2N} \sum_{k=1}^{N}\sum_{\ell=1}^{N} V(x_k - x_\ell), 
\end{align*}
then we can write $\mu^{N,*}(  \boldsymbol{x})$ in  \eqref{eq:NparticleGibbsDistro} as 
$\mu^{N,*}(  \boldsymbol{x})\propto \exp\big( - 2 H(\boldsymbol{x}) / \sigma^2  \big)$. 
Since $V$ is assumed even, we have that $\nabla V$ is odd and thus $\partial_{x_i} H(\boldsymbol{x})=-B_i(\boldsymbol{x})$.

Let us now focus on the operator $L$ given in \eqref{eq: B0 formula}. Note that it can be written as $L=\frac12 (L_1 + \frac{\sigma^2}2 L_2)$ where $L_1$ captures the first summand in $L$ and $L_2$ captures the remaining two terms of $L$ that are multiplied by the weights $\sigma^2/2$. 

To prove \eqref{lemma:statementauxiliaryintegraliszero}, we use integration by parts and show that  
\begin{align*}
\int_{\bR^N} \frac{\sigma^2}2 
& 
L_2(0,\boldsymbol{x}) \mu^{N,*}(\dd \boldsymbol{x}) = -\int_{\bR^N} L_1(0,\boldsymbol{x}) \mu^{N,*}(\dd \boldsymbol{x}) \Leftrightarrow
\\
\int_{\bR^N}
    \frac{\sigma^2}{2}  
    &
    \bigg( \sum_{i,j=1}^{N} \partial_{x_j}B_i(\boldsymbol{x}) 
      \partial_{x_i, x_j}^2 u(t,\boldsymbol{x}) 
     +   \sum_{i,j=1}^{N} \partial^2_{x_j, x_j}B_i(\boldsymbol{x})  \partial_{x_i} u(t,\boldsymbol{x}) \bigg) 
    \exp \Big( - \frac2{\sigma^2} H(\boldsymbol{x}) \Big) 
     \dd \boldsymbol{x}
     \\
     & =-
     \int_{\bR^N}\bigg( \sum_{i,j=1}^{N} 
     B_j(\boldsymbol{x})
     \partial_{x_j}B_i(\boldsymbol{x})  \partial_{x_i} u(t,\boldsymbol{x}) 
    \bigg) 
    \exp \Big( - \frac2{\sigma^2} H(\boldsymbol{x}) \Big) 
     \dd \boldsymbol{x}. 
\end{align*}
Recognizing the sum terms (in the mid-line) as the product rule for derivatives, we have  using integration by parts and the identity $\partial_{x_j} H(\boldsymbol{x})=-B_j(\boldsymbol{x})$
\begin{align*}
& \int_{\bR^N} \frac{\sigma^2}2 
L_2(0,\boldsymbol{x}) \mu^{N,*}(\dd \boldsymbol{x})
\\
& 
= \sum_{i,j=1}^{N}
    \int_{\bR^N}
    \frac{\sigma^2}{2}
    \bigg(          
    \partial_{x_j} \Big(   \partial_{x_j} B_i(\boldsymbol{x})  \partial_{x_i} u(0,\boldsymbol{x}) 
    \Big)   
     \bigg) 
     \exp \Big( - \frac2{\sigma^2} H(\boldsymbol{x}) \Big)
     \dd \boldsymbol{x}
\\
&
= \sum_{i,j=1}^{N}      
     \int_{\bR^{N-1}}
     \Bigg(
     \frac{\sigma^2}{2}
     \bigg(  \partial_{x_j}B_i(\boldsymbol{x})  \partial_{x_i} u(t,\boldsymbol{x}) 
      \bigg) 
     \exp \Big( - \frac2{\sigma^2} H(\boldsymbol{x}) \Big)
     \Bigg)
     \bigg |_{x_j=-\infty}^{x_j=+\infty}  
     \dd \boldsymbol{x}^{-j} 
\\ 
& \qquad 
 -  \sum_{i,j=1}^{N}
     \int_{\bR^N}\Big( 
     \frac{\sigma^2}{2}
     \partial_{x_j}B_i(\boldsymbol{x})  \partial_{x_i} u(t,\boldsymbol{x}) 
    \Big) 
    \partial_{x_j} \bigg (\exp \Big( - \frac2{\sigma^2} H(\boldsymbol{x}) \Big)\bigg)
     \dd \boldsymbol{x}
\\
& =-\sum_{i,j=1}^{N}  
     \int_{\bR^N}\Big(  
     B_j(\boldsymbol{x})
     \partial_{x_j}B_i(\boldsymbol{x})  \partial_{x_i} u(t,\boldsymbol{x}) 
    \Big) 
    \exp \Big( - \frac2{\sigma^2} H(\boldsymbol{x}) \Big)
\dd \boldsymbol{x}
= -\int_{\bR^N} L_1(0,\boldsymbol{x}) \mu^{N,*}(\dd \boldsymbol{x}),
\end{align*}
where the integration over $\dd \boldsymbol{x}^{-j}$ means that we have integrated the variable $x_j$ out and $N-1$ integrations remain; the boundary term (over the $x_j$ variable) is zero due to the boundedness assumptions of the derivatives of $u$ and $B$ and that \eqref{eq:NparticleGibbsDistro} is a density that vanishes at infinity.
\end{proof}

\color{black}

\subsection{Omitted residual terms of Section \ref{sec:Remainder terms R}  }
\label{appendix_aux_RemainderSection6.2}

In this part, we show the  exact expectation form for the residual terms $R^2_{t_m}, R^3_{t_m}, R^5_{t_m}, R^7_{t_m}$ and $R^8_{t_m}$ in Lemma~\ref{lemma: analysis of the residual term}. The positive constant $K$ below is independent of $h,T,M$ and $N$ and may have a different value in each line. 
For the residual term $R^2_{t_m}$, we have 
\begin{align*}
     &\mathbb{E}  \big[          R^2_{t_m}        \big]=  K
     \bE \bigg[  
     \int_{t_m}^{t_m+h} 
    \int_{t_m}^{q_1}  \int_{t_m}^{q_2}
    \sum_{ \gamma \in \Pi_3^N }
     \Delta W_{m,2h}^{\gamma_3}
     \cdot 
     \partial^5_{x_{  \gamma_1},x_{  \gamma_2},x_{  \gamma_2},x_{\gamma_3  },x_{\gamma_3  }} u\Big(t_{m+1},\overline{\bodX}_{ q_3}^{N,h} ) \Big)
     \dd  W^{\gamma_1}_{q_3} \ \dd  {q_2}\    \dd  W^{\gamma_3}_{q_1} \bigg]
     \\
     &+K\bE \bigg[  
     \int_{t_m}^{t_m+h} 
    \int_{t_m}^{q_1}  \int_{t_m}^{q_2}
     \sum_{ \gamma \in \Pi_3^N } 
     \Delta W_{m,2h}^{\gamma_3}
     \cdot 
     \partial^6_{x_{  \gamma_1},x_{  \gamma_1},x_{  \gamma_2},x_{  \gamma_2},x_{\gamma_3  },x_{\gamma_3  }}  u\Big(t_{m+1},\overline{\bodX}_{ q_3}^{N,h} ) \Big)
     \dd  {q_3} \ \dd  {q_2}\  \dd  W^{\gamma_3}_{q_1} \bigg].
\end{align*}
For the residual term $R^3_{t_m}$, we have    
\begin{align*}
     \mathbb{E} & \big[          R^3_{t_m}        \big]
     \\
     =&
      K h
     \bE \bigg[  
     \int_{t_m}^{t_m+h} 
    \int_{t_m}^{q_1}  
\qquad 
    \sum_{ \gamma \in \Pi_3^N }
     \Delta W_{m,2h}^{\gamma_3}
     \cdot 
     \partial_{x_{  \gamma_1}}\bigg(
      \partial_{x_{  \gamma_3}} B_{\gamma_2 }   ( \overline{\bodX}_{ q_2}^{N,h} )
     \partial^2_{ x_{  \gamma_2},x_{\gamma_3  } } u(t_{m+1},\overline{\bodX}_{ q_2}^{N,h})
     \bigg) 
    \dd  W^{\gamma_1}_{q_2}  \    \dd  W^{\gamma_3}_{q_1} \bigg]
     \\
     &+
      K h
     \bE \bigg[  
     \int_{t_m}^{t_m+h} 
    \int_{t_m}^{q_1}   
    \qquad 
    \sum_{ \gamma \in \Pi_3^N }
     \Delta W_{m,2h}^{\gamma_3}
     \cdot 
     \partial^2_{x_{  \gamma_1},x_{  \gamma_1}}\bigg(
      \partial_{x_{  \gamma_3}} B_{\gamma_2 }   (  \overline{\bodX}_{ q_2}^{N,h}   )
     \partial^2_{ x_{  \gamma_2},x_{\gamma_3  } } u(t_{m+1}, \overline{\bodX}_{ q_2}^{N,h} )
     \bigg) 
     \dd   {q_2}  \    \dd  W^{\gamma_3}_{q_1} \bigg]
     \\
     &+ 
      K h
     \bE \bigg[  
     \int_{t_m}^{t_m+h} 
    \int_{t_m}^{q_1}  
    \qquad 
    \sum_{ \gamma \in \Pi_3^N }
     \Delta W_{m,2h}^{\gamma_3}
     \cdot 
     \partial_{x_{  \gamma_1}}
     \bigg(
       B_{\gamma_2 }   (  \overline{\bodX}_{ q_2}^{N,h}  )
     \partial^3_{ x_{  \gamma_2},x_{\gamma_3  } ,x_{\gamma_3  }} u(t_{m+1}, \overline{\bodX}_{ q_2}^{N,h} )
     \bigg) 
    \dd  W^{\gamma_1}_{q_2}  \    \dd  W^{\gamma_3}_{q_1} \bigg]
     \\
     &+
      K h
     \bE \bigg[  
     \int_{t_m}^{t_m+h} 
    \int_{t_m}^{q_1}   
     \qquad 
    \sum_{ \gamma \in \Pi_3^N }
     \Delta W_{m,2h}^{\gamma_3}
     \cdot 
     \partial^2_{x_{  \gamma_1},x_{  \gamma_1}}
     \bigg(
       B_{\gamma_2 }   ( \overline{\bodX}_{ q_2}^{N,h}  )
     \partial^3_{ x_{  \gamma_2},x_{\gamma_3  } ,x_{\gamma_3  }} u(t_{m+1}, \overline{\bodX}_{ q_2}^{N,h})
     \bigg)
     \dd   {q_2}  \    \dd  W^{\gamma_3}_{q_1} \bigg]
     \\
     &+ K
     \bE \bigg[ 
     \int_{t_m}^{t_m+h} 
    \int_{t_m}^{q_1}  \int_{t_m}^{q_2} 
     \qquad 
     \sum_{ \gamma \in \Pi_5^N } 
     \bigg( 
     \Delta W_{m,2h}^{\gamma_4}\Delta W_{m,2h}^{\gamma_5}
     \cdot  
     \partial^5_{x_{  \gamma_1},\dots,x_{\gamma_5  }} u (t_{m+1},\overline{\bodX}_{ q_3}^{N,h}  ) ~ 
     \dd  W^{\gamma_1}_{q_3} \ \dd  W^{\gamma_2}_{q_2}\    \dd  W^{\gamma_3}_{q_1}
     \\&
     \qquad \qquad\qquad
     + \Delta W_{m,2h}^{\gamma_4}\Delta W_{m,2h}^{\gamma_5}
     \cdot  
     \partial^6_{x_{  \gamma_1},x_{  \gamma_1},x_{  \gamma_2} ,x_{  \gamma_3} ,x_{  \gamma_4} ,x_{\gamma_5  }} u (t_{m+1},\overline{\bodX}_{ q_3}^{N,h}   ) ~ 
     \dd   {q_3} \ \dd  W^{\gamma_2}_{q_2}\    \dd  W^{\gamma_3}_{q_1}
     \bigg)~ 
     \bigg]
     \\
     & +  K
     \bE \bigg[ 
     \int_{t_m}^{t_m+h} 
    \int_{t_m}^{q_1} 
     \qquad 
     \sum_{ \gamma \in \Pi_4^N } 
     \bigg( 
     \Delta W_{m,2h}^{\gamma_3}\Delta W_{m,2h}^{\gamma_4}
     \cdot  
     \partial^5_{x_{  \gamma_1},x_{  \gamma_1},x_{  \gamma_2},x_{  \gamma_3},x_{\gamma_4  }} u (t_{m+1},\overline{\bodX}_{ q_2}^{N,h}  ) 
     \dd  {q_2} \    \dd  W^{\gamma_2}_{q_1}
     \\&
     \qquad\qquad\qquad
     +\Delta W_{m,2h}^{\gamma_3}\Delta W_{m,2h}^{\gamma_4}
     \cdot  
     \partial^5_{x_{  \gamma_1},x_{  \gamma_2},x_{  \gamma_2},x_{  \gamma_3},x_{\gamma_4  }} u (t_{m+1},\overline{\bodX}_{ q_2}^{N,h}   ) 
     \dd  W^{\gamma_1}_{q_2} \    \dd  {q_1}
     \\&
     \qquad \qquad\qquad
     + \Delta W_{m,2h}^{\gamma_3}\Delta W_{m,2h}^{\gamma_4}
     \cdot  
     \partial^6_{x_{  \gamma_1},x_{  \gamma_1},x_{  \gamma_2},x_{  \gamma_2},x_{  \gamma_3},x_{\gamma_4  }} u (t_{m+1},\overline{\bodX}_{ q_2}^{N,h}  ) 
     \dd   {q_2}  
     \    \dd   {q_1}
     \bigg)~ 
     \bigg].
\end{align*}
For the residual term $R^5_{t_m}$, we have 
\begin{align*}
     \mathbb{E}  \big[ &        R^5_{t_m}        \big]
      = \sum_{ \gamma \in \Pi_5^N  } 
     \bE \bigg[\int_{t_m}^{t_m+h} 
     \partial_{x_{  \gamma_1}}\bigg(\Delta W_{m,2h}^{\gamma_2} \Delta W_{m,2h}^{\gamma_3}\Delta W_{m,2h}^{\gamma_4}\Delta W_{m,2h}^{\gamma_5}\cdot
   {  \partial^{4}_{ x_{  \gamma_2},x_{  \gamma_3}, x_{  \gamma_4}   , x_{  \gamma_5}  }  }
     u(t_{m+1},\overline{\bodX}_{q_1}^{N,h})    \bigg) 
     \dd  W^{\gamma_1}_{q_1} \bigg]
     \\
     \nonumber
     & \quad + \sum_{ \gamma \in \Pi_4^N  } 
     \bE \bigg[\int_{t_m}^{t_m+h} 
     \partial^2_{x_{  \gamma_1},x_{  \gamma_1}}\bigg(\Delta W_{m,2h}^{\gamma_2} \Delta W_{m,2h}^{\gamma_3}\Delta W_{m,2h}^{\gamma_4}\Delta W_{m,2h}^{\gamma_5}
     \cdot
     \partial^{4}_{ x_{  \gamma_2},x_{  \gamma_3}, x_{  \gamma_4} , x_{  \gamma_5}    }  
     u(t_{m+1},\overline{\bodX}_{q_1}^{N,h})    \bigg) 
     \dd  {q_1} \bigg]
      \\
     & \quad +
     K h
     \sum_{ \gamma \in \Pi_4^N  }
     \bE\Big[  B_{\gamma_1}( {\bodX}_{t_m}^{N,h})  
    \Delta W_{m,2h}^{\gamma_2}
     \Delta W_{m,2h}^{\gamma_3}
    \Delta W_{m,2h}^{\gamma_4}
     \cdot
     \partial^{4}_{ x_{  \gamma_1},x_{  \gamma_2}, x_{  \gamma_3}    , x_{  \gamma_4}   }  
     u(t_{m+1},\bodX_{t_m}^{N,h})  \Big]
     \\
     & \quad +
     K h^2
     \sum_{ \gamma \in \Pi_4^N  }
     \bE\Big[  B_{\gamma_1}( {\bodX}_{t_m}^{N,h})  
     B_{\gamma_2}( {\bodX}_{t_m}^{N,h})
     \Delta W_{m,2h}^{\gamma_3}
    \Delta W_{m,2h}^{\gamma_4}
     \cdot
     \partial^{4}_{ x_{  \gamma_1},x_{  \gamma_2}, x_{  \gamma_3}    , x_{  \gamma_4}   }  
     u(t_{m+1},\bodX_{t_m}^{N,h})  \Big]
     \\
     & \quad +
     K h^3
     \sum_{ \gamma \in \Pi_4^N  }
     \bE\Big[  B_{\gamma_1}( {\bodX}_{t_m}^{N,h})  
     B_{\gamma_2}( {\bodX}_{t_m}^{N,h})
     B_{\gamma_3}( {\bodX}_{t_m}^{N,h})
    \Delta W_{m,2h}^{\gamma_4}
     \cdot
     \partial^{4}_{ x_{  \gamma_1},x_{  \gamma_2}, x_{  \gamma_3}    , x_{  \gamma_4}   }  
     u(t_{m+1},\bodX_{t_m}^{N,h})  \Big]
     \\
     & \quad +
     K h^4
     \sum_{ \gamma \in \Pi_4^N  }
     \bE\Big[  B_{\gamma_1}( {\bodX}_{t_m}^{N,h})  
     B_{\gamma_2}( {\bodX}_{t_m}^{N,h})
     B_{\gamma_3}( {\bodX}_{t_m}^{N,h})
     B_{\gamma_4}( {\bodX}_{t_m}^{N,h})
     \cdot
     \partial^{4}_{ x_{  \gamma_1},x_{  \gamma_2}, x_{  \gamma_3}    , x_{  \gamma_4}   }  
     u(t_{m+1},\bodX_{t_m}^{N,h})  \Big].
\end{align*}
For the residual term $R^7_{t_m}$, we have
\begin{align*}
     \mathbb{E}  \big[          R^7_{t_m}        \big]
     =&  K h
     \bE \bigg[  
     \int_{t_m}^{t_m+h} 
    \int_{t_m}^{q_1}  \int_{t_m}^{q_2}
    \sum_{ \gamma \in \Pi_5^N }
     \partial^5_{x_{  \gamma_1},x_{  \gamma_2},x_{  \gamma_3},x_{  \gamma_4},x_{\gamma_5  }} u (t_{m+1},\overline{\bodX}_{ q_3}^{N,h}    )
     \dd  W^{\gamma_1}_{q_3} \ \dd  W^{\gamma_2}_{q_2}\    \dd  W^{\gamma_3}_{q_1} \bigg]
     \\
       &+  K h
     \bE \bigg[  
     \int_{t_m}^{t_m+h} 
    \int_{t_m}^{q_1}  \int_{t_m}^{q_2}
    \sum_{ \gamma \in \Pi_5^N }
     \partial^6_{x_{  \gamma_1},x_{  \gamma_1},x_{  \gamma_2},x_{  \gamma_3},x_{  \gamma_4}, ,x_{\gamma_5  }} u (t_{m+1},\overline{\bodX}_{ q_3}^{N,h}    )
     \dd   {q_3} \ \dd  W^{\gamma_2}_{q_2}\    \dd  W^{\gamma_3}_{q_1} \bigg]
     \\
       &+  K h
     \bE \bigg[  
     \int_{t_m}^{t_m+h} 
    \int_{t_m}^{q_1}   
    \sum_{ \gamma \in \Pi_4^N }
     \partial^5_{x_{  \gamma_1},x_{  \gamma_1},x_{  \gamma_2},x_{  \gamma_3},x_{\gamma_4  }} u (t_{m+1},\overline{\bodX}_{ q_2}^{N,h}    )
      \dd   {q_2}\    \dd  W^{\gamma_2}_{q_1} \bigg]
     \\
     &+  K h
     \bE \bigg[  
     \int_{t_m}^{t_m+h} 
    \int_{t_m}^{q_1}   
    \sum_{ \gamma \in \Pi_4^N }
     \partial^5_{x_{  \gamma_1},x_{  \gamma_2},x_{  \gamma_2},x_{  \gamma_3},x_{\gamma_4  }} u (t_{m+1},\overline{\bodX}_{ q_2}^{N,h}    )
       \dd  W^{\gamma_1}_{q_2} \ \dd   {q_1}     \bigg]
       \\
     &+  K h
     \bE \bigg[  
     \int_{t_m}^{t_m+h} 
    \int_{t_m}^{q_1}   
    \sum_{ \gamma \in \Pi_4^N }
     \partial^6_{x_{  \gamma_1},x_{  \gamma_1},x_{  \gamma_2},x_{  \gamma_2},x_{  \gamma_3},x_{\gamma_4  }} u (t_{m+1},\overline{\bodX}_{ q_2}^{N,h}    )
       \dd   {q_2} \ \dd   {q_1}     \bigg].
\end{align*}
For the residual term $R^8_{t_m}$, we have 
\begin{align*}
     \mathbb{E}  \big[          R^8_{t_m}        \big]
     &=  K h^2 
     \bE \bigg[  
     \int_{t_m}^{t_m+h} \sum_{ \gamma \in \Pi_3^N  } \partial_{x_{  \gamma_1}} \bigg(
     B_{\gamma_2}( 
\overline{\bodX}_{q_1}^{N,h}
  )
     \partial^{3}_{ x_{  \gamma_2},x_{ \gamma_3}, x_{  \gamma_3}     }  
     u(t_{m+1}, 
\overline{\bodX}_{q_1}^{N,h}
   )  \bigg)   ~ \dd   W^{\gamma_1}_{q_1} \bigg]
     \\
     & \quad +
      K h^2 
     \bE \bigg[  
     \int_{t_m}^{t_m+h} \sum_{ \gamma \in \Pi_3^N  } \partial^2_{x_{  \gamma_1},x_{  \gamma_1}} \bigg(
     B_{\gamma_2}( 
\overline{\bodX}_{q_1}^{N,h}
  )
     \partial^{3}_{ x_{  \gamma_2},x_{ \gamma_3}, x_{  \gamma_3}     }  
     u(t_{m+1}, 
\overline{\bodX}_{q_1}^{N,h}
   )  \bigg)   ~ \dd   {q_1} \bigg]
    \\
    & \quad + K h^2 
     \bE \bigg[  
     \int_{t_m}^{t_m+h} \sum_{ \gamma \in \Pi_3^N  } \partial_{x_{  \gamma_1}} \bigg(
     \partial_{x_{  \gamma_3}}  B_{\gamma_2}( 
\overline{\bodX}_{q_1}^{N,h}
  )
     \partial^{2}_{ x_{  \gamma_2},x_{ \gamma_3}     }  
     u(t_{m+1}, 
\overline{\bodX}_{q_1}^{N,h}
   )  \bigg)   ~ \dd   W^{\gamma_1}_{q_1} \bigg]
     \\
     & \quad +
      K h^2 
     \bE \bigg[  
     \int_{t_m}^{t_m+h} \sum_{ \gamma \in \Pi_3^N  } \partial^2_{x_{  \gamma_1},x_{  \gamma_1}}\bigg(
     \partial_{x_{  \gamma_3}}  B_{\gamma_2}( 
\overline{\bodX}_{q_1}^{N,h}
  )
     \partial^{2}_{ x_{  \gamma_2},x_{ \gamma_3}     }  
     u(t_{m+1}, 
\overline{\bodX}_{q_1}^{N,h}
   )  \bigg)   ~ \dd   {q_1} \bigg]
    \\
     & \quad + K h^2 
     \bE \bigg[  
     \int_{t_m}^{t_m+h} \sum_{ \gamma \in \Pi_3^N  }  
     \partial^{5}_{ x_{  \gamma_1},x_{  \gamma_2},x_{  \gamma_2},x_{ \gamma_3}, x_{  \gamma_3}     }  
     u(t_{m+1}, 
\overline{\bodX}_{q_1}^{N,h}
   )    ~ \dd   W^{\gamma_1}_{q_1} \bigg]
     \\
     & \quad +
      K h^2 
     \bE \bigg[  
     \int_{t_m}^{t_m+h} \sum_{ \gamma \in \Pi_3^N  }   
     \partial^{6}_{ x_{  \gamma_1},x_{  \gamma_1},x_{  \gamma_2},x_{  \gamma_2},x_{ \gamma_3}, x_{  \gamma_3}     }  
     u(t_{m+1}, 
\overline{\bodX}_{q_1}^{N,h}
   )    ~  \dd   {q_1} \bigg].
\end{align*}

\bibliographystyle{amsplain}



\providecommand{\bysame}{\leavevmode\hbox to3em{\hrulefill}\thinspace}
\providecommand{\MR}{\relax\ifhmode\unskip\space\fi MR }
\providecommand{\MRhref}[2]{%
  \href{http://www.ams.org/mathscinet-getitem?mr=#1}{#2}
}
\providecommand{\href}[2]{#2}

\end{document}